\documentclass[english,a4paper,10pt,final]{amsart}
\usepackage[utf8]{inputenc}
\usepackage{tikz-cd,
amsmath,
amscd,
amssymb, 
amsxtra, 
amsthm,
array,
booktabs,
calc,
color,
microtype,
enumitem,
youngtab,
stackrel,
}
\usetikzlibrary{arrows,graphs}

\usepackage[colorlinks=true,unicode]{hyperref}
\usepackage{amsrefs}

\usepackage[all]{xy}

\ifx\undefined\xspace \usepackage{xspace} \fi

\newcommand{\nc}{\newcommand}
\renewcommand{\frak}{\mathfrak}
\providecommand{\cal}{\mathcal}

\renewcommand{\bold}{\mathbf}
\newtheorem*{theorem*}{Theorem}
\numberwithin{equation}{subsection}

\newcommand{\pfname}{Proof.}

\newenvironment{pf}{\vskip-\lastskip\vskip\medskipamount{\it\pfname}}%
                      {$\square$\vskip\medskipamount\par}

\newenvironment{pfof}[1]{\vskip-\lastskip\vskip\medskipamount{\it
    Proof of #1.}}%
                      {$\square$\vskip\medskipamount\par}

\newtheorem{thm}{Theorem}[section]
\newtheorem{theorem}[thm]{Theorem}

\newtheorem{thmnn}{Theorem}

\newtheorem{corollary}[thm]{Corollary}

\newtheorem{prop}[thm]{Proposition}
\newtheorem{proposition}[thm]{Proposition}

\newtheorem{lemma}[thm]{Lemma} 

\theoremstyle{definition}
\newtheorem{definition}[thm]{Definition}

\theoremstyle{definition}

\newtheorem{remark}[thm]{Remark}
\newtheorem{remarks}[thm]{Remarks}

\newtheorem{example}[thm]{Example} 

\newtheorem{question}[thm]{Question}

\nc{\Theorem}[1]{Theorem~{#1}}
\nc{\Th}[1]{({\sl Th.}~#1)}
\nc{\Thd}[2]{({\sl Th.}~{#1} {#2})}
\nc{\Theorems}[2]{Theorems~{#1} and ~{#2}}
\nc{\Thms}[2]{({\it Thms. ~{#1} and ~{#2}})}
\nc{\Lemmas}[2]{Lemma~{#1} and ~{#2}}

\nc{\manga}[6]{({\it Thms. ~ #1, ~ #2, ~ #3,\\ ~ #4, ~ #5, ~ #6})}

\nc{\Prop}[1]{({\sl Prop.}~{#1})}
\nc{\Proposition}[1]{Proposition~{#1}}
\nc{\Propositions}[2]{Propositions~{#1} and ~{#2}}
\nc{\Props}[2]{({\sl Props.}~{#1} and ~{#2})}
\nc{\Cor}[1]{({\sl Cor.}~{#1})}
\nc{\Corollary}[1]{Corollary~{#1}}
\nc{\Corollaries}[2]{Corollaries~{#1} and ~{#2}}
\nc{\Definition}[1]{Definition~{#1}}
\nc{\Defn}[1]{({\sl Def.}~{#1})}
\nc{\Lemma}[1]{Lemma~{#1}} 
\nc{\Lem}[1]{({\sl Lem.} ~{#1})} 
\nc{\Eq}[1]{equation~({#1})}
\nc{\Equation}[1]{Equation~({#1})}
\nc{\Section}[1]{Section~{#1}}
\nc{\Sections}[1]{Sections~{#1}}
\nc{\Sec}[1]{({\sl Sec.} ~{#1})} 
\nc{\Chapter}[1]{Chapter~{#1}}
\nc{\Chapt}[1]{({\sl Ch.}~{#1})}

\nc{\Ex}[1]{{\sl Ex.}~{#1}}
\nc{\Exa}[1]{{\sl Example}~{#1}}
\nc{\Example}[1]{{\sl Example}~{#1}}
\nc{\Examples}[1]{{\sl Examples}~{#1}}
\nc{\Exercise}[1]{{\sl Exercise}~{#1}}

\nc{\Rem}[1]{({\sl Rem.}~{#1})}
\nc{\Remark}[1]{{\sl Remark}~{#1}}
\nc{\Remarks}[2]{{\sl Remarks}~{#1} and ~{#2}}
\nc{\Note}[1]{{\sl Note}~{#1}}

\nc{\Conjecture}[1]{Conjecture~{#1}}
\nc{\Claim}[1]{Claim~{#1}}

\nc \Proof{{  \it Proof. }}

\nc \Dbb{\mathbb D}
\nc{\xmu}{\mu}
\nc{\w}{\omega}
\nc{\xv}{\mbox{\boldmath$x$}}
\nc{\uv}{\mbox{\boldmath$u$}}
\nc{\xiv}{\mbox{\boldmath$\xi$}}
\nc{\bbeta}{\mbox{\boldmath$\beta$}}
\nc{\balpha}{\mbox{\boldmath$\alpha$}}
\nc{\bgamma}{\mbox{\boldmath$\gamma$}}
\nc{\bdelta}{\mbox{\boldmath$\delta$}}
\nc{\bepsilon}{\mbox{\boldmath$\epsilon$}}

\nc \Ab{{\ensuremath{\bold A}}}
\nc \ab{{\ensuremath{\bold a}}}
\nc \bb{{\ensuremath{\bold b}}}
\nc \cb{{\ensuremath{\bold c}}}
\nc \db{{\ensuremath{\bold d}}}
\nc \Bb{{\ensuremath{\bold B}}}
\nc \Gb{{\ensuremath{\bold G}}}
\nc \Qb{{\ensuremath{\bold Q}}}
\nc \Rb{{\ensuremath{\bold R}}} \nc \Cb{{\ensuremath{\bold C}}} 
\nc \Eb{{\ensuremath{\bold E}}}
\nc \eb{{\ensuremath{\bold e}}}
\nc \Db{{\ensuremath{\bold D}}}
\nc \Fb{{\ensuremath{\bold F}}}
\nc \ib{{\ensuremath{\bold i}}}
\nc \jb{{\ensuremath{\bold j}}}
\nc \kb{{\ensuremath{\bold k}}}
\nc \Kb{{\ensuremath{\bold K}}}
\nc \nb{{\ensuremath{\bold n}}}
\nc \rb{{\ensuremath{\bold r}}}
\nc \Ob{{\ensuremath{\bold O}}}
\nc \Pb{{\ensuremath{\bold P}}}
\nc \pb{{\ensuremath{\bold p}}}
\nc \qb{{\ensuremath{\bold q}}}
\nc \SPb{{\ensuremath{\bold {SP}}}}
\nc \Zb{{\ensuremath{\bold Z}}} 
\nc \zb{{\ensuremath{\bold z}}} 
\nc \gb{{\ensuremath{\bold g}}} 
\nc \fb{{\ensuremath{\bold f}}} 
\nc \ub{{\ensuremath{\bold u}}} 
\nc \vb{{\ensuremath{\bold v}}} 
\nc \yb{{\ensuremath{\bold y}}} 
\nc \xb{{\ensuremath{\bold x}}} 
\nc \tb{{\ensuremath{\bold t}}} 
\nc \Xb{{\ensuremath{\bold X}}} 
\nc \Yb{{\ensuremath{\bold Y}}} 
\nc \xib{{\ensuremath{\bold \xi}}} 
\nc \Nb{{\ensuremath{\bold N}}} 
\nc \Hb{{\ensuremath{\bold H}}} 
\nc \wb{{\ensuremath{\bold w}}} 
\nc \Wb{{\ensuremath{\bold W}}} 
\nc \syz{{\mathbf {syz}}}
\nc \bnoll{{\ensuremath{\bold 0}}} 

\nc \mf{\frak m} 
\nc \mh{\hat{\mf}} 
\nc \nf{\frak n}
\nc \Of{\frak O}
\nc \of{\frak o}
\nc \rf{\frak r}
\nc \tf{\frak t}
\nc \mufr{{\mathbf \mu}}
\nc \hf{\frak h} 
\nc \qf{\frak q} 
\nc \bfr{\frak b} 
\nc \kfr{\frak k} 
\nc \pfr{\frak p} 
\nc \af{\frak a }
\nc \cf{\frak c }
\nc \sfr{\frak s} 
\nc \ufr{\frak u} 
\nc \g{\frak g} 
\nc \gA{\g_{\Ao}} 
\nc \lfr{\frak l}
\nc \afr{\frak a}
\nc \gfh{\hat {\frak g}}
\nc \gl{\frak { gl }}
\nc \Sl{\frak {sl}}
\nc \SU{\frak {SU}}
\nc{\Homf}{\frak{Hom}}

\newcommand{\on}{\operatorname}
\newcommand{\adj}[2]{ \overset{#1}{\underset{#2}\rightleftarrows }}
\nc\hankel{\on {Hankel}}
\nc\row{\on {row\ }}
\nc\nullity{\on {nullity }}
\nc\col{\on{col}}
\nc\rowm{\on {Row \ }}
\nc\loc{\on {lc \ }}
\nc\lcm{\on {lcm \ }}
\nc\nullo{\on {null\ }}
\nc\Nul{\on {Nul}}
\nc\Cola{\on {Col }}
\nc \Ann {\on {Ann }}
\nc \Ass {\on {Ass}}
\nc \Coker {\on {Coker}}
\nc \Co{\on C}
\nc \Homo{\on {Hom}}
\nc \Galo{\on {Gal}}
\nc \Ker {\on {Ker}}
\nc \Lef {\on {Lef}}
\nc \Leff {\on {Leff}}
\nc \omod{\on{mod}}
\nc \No {\on N}
\nc \NN {\on {NN}}
\nc \NGo {\on {NG}}
\nc \Oo {\on O}
\nc \ch {\on {ch}}
\nc \cd {\on {cd}}
\nc \rko {\on {rk}}
\nc \Sing {\on {Sing\ }}
\nc \Reg {\on {Reg}}
\nc \RG{\on {R\Gamma}}
\nc \CoI {\on {CI}}
\nc \CoM {\on {CM}}
\nc \Gor {\on {Gor}}
\nc \Type {\on {Type}}
\nc \can {\on {can}}
\nc \Top {\on {T}}
\nc \Tr {\on {Tr}}
\nc \Norm {\on {Norm}}
\nc \rel {\on {rel}}
\nc \tr {\on {tr}}
\nc \sgn {\on {sgn }}
\nc \trdeg {\on {tr.deg}}
\nc \codim {\on {codim }}
\nc \coht {\on {coht}}
\nc \divo {\on {div \ }}
\nc \rot {\on {rot }}
\nc \coh {\on {coh}}
\nc \Clo {\on {Cl}}
\nc \Divo {\on {Div}}
\nc \embdim{\on {embdim}}
\nc \ord{\on {ord}}
\nc \ed{\on {ed}}
\nc \embcodim{\on {embcodim  }}
\nc \qcoh {\on {qcoh}}
\nc \grad {\on {grad}}
\nc \grade {\on {grade}}
\nc \hto {\on {ht}}
\nc \depth {\on {depth}}
\nc \prof {\on {prof}}
\nc \reso{\on {res}}
\nc \Reso{\on {Res}}
\nc \ind{\on {ind}}
\nc \prodo{\on {prod}}
\nc \coind{\on {coind}}
\nc \Con{\on {Con}}
\nc \Crit{\on {Crit}}
\nc \Der{\on {Der}}
\nc \Des{\on {Des}}
\nc \Char{\on {Char}}
\nc \Ch{\on {Ch}}

\nc \Ext{\on {Ext}}
\nc \Eo{\on {E}}
\nc \End{\on {End}}
\nc \ad{\on {ad}}
\nc \Ad{\on {Ad}}
\nc \gr{\on {gr}}
\nc \Fo{\on {F}}
\nc \Gr{\on {Gr}}
\nc \Go{\on {G}}
\nc \GFo{\on {GF}}
\nc \Glo{\on {Gl}}
\nc \PGlo{\on {PGl}}
\nc \Ho{\on {H}}
\nc \CMo{\on {\CM}}
\nc \SCM{\on {SCM}}
\nc \rig{\on {right}}
\nc \lef{\on {left}}
\nc \hol{\on {hol}}
\nc{\sgd}{\on{sgd}}
\nc \supp{\on {supp}}
\nc \ssupp{\on {s-supp}}
\nc \singsupp{\on {singsupp}}
\nc \msupp{\on {msupp}}
\nc \spec{\on {spec}}
\nc \spano{\on {span }}
\nc \Span{\on {Span }}
\nc \Max{\on {Max}}
\nc \Mat{\on {Mat}}
\nc \Min{\on {Min}}
\nc \nil{\on {nil}}
\nc \Mod{\on{Mod}}
\nc \Rad {\on {Rad}}
\nc \rad {\on {rad}}
\nc \rank {\on {rank}}
\nc \range {\on {range}}
\nc \Slo{\on {SL}}
\nc \soc {\on {soc}}
\nc \dt {\on {dt}_Z}
\nc \Irr {\on {Irr}}
\nc \Reo {\on {Re}}
\nc \Imo {\on {Im}}
\nc \SSo{\on {SS}}
\nc \lub{\on {lub}}
\nc \gldim{\on {gl.d.}}
\nc \length{\on {length}}
\nc \pdo{\on {p.d.}} 
\nc \cdo{\on {cd}} 
\nc \fdo{\on {f.d.}} 
\nc \ido{\on {i.d.}} 
\nc \dSSo{\dot {\SSo}}
\nc \So{\on S}
\nc \SOo{\on{ SO}}
\nc \Io{\on I}
\nc \Jo{\on J}
\nc \jo{\on j}
\nc \Ko{\on K}
\nc \PBW{\Ac_{PBW}}
\nc \Ro{\on R}
\nc \To{\on T}
\nc \Ao{\on A}

\nc \Do{{\on D}}
\nc \Bo{\on B}
\nc \Po{\on P}
\nc \Qo{\on Q}
\nc \Zo{\on Z}
\nc \Uo{\on U}
\nc \wt{\on {wt}}
\nc \Uoh{\hat {\Uo}}
\nc \Lo{\on L}
\nc \Loc{\on {Loc}}
\nc{\dop}{\on d}
\nc{\eo}{\on e}
\nc{\ado}{\on{ad}}
\nc{\Tot}{\on{Tot}}
\nc{\Aut}{\on{Aut}}
\nc{\sinc}{\on {sinc}}

\nc{\overrightleftarrows}[2]{\overset{#1}{\underset{#2}{\rightleftarrows}}}

\nc{\CCF}{\cal{CF}}
\nc{\CDF}{\cal{DF}}
\nc{\CHC}{\check{\cal C}}

\nc{\Cone}{\on{Cone}}
\nc{\dec}{\on{dec}}
\nc{\Diff}{\on{Diff}}
\nc{\dirlim}{\underset{\to}{\on{lim}}}
\nc{\dpar}{\partial}
\nc{\dlog}{\on{dlog}}
\nc{\GL}{\on{GL}}
\nc{\glo}{\on{gl}}
\nc{\CGr}{\cal{G}r}
\nc{\pr}{\on{pr}}
\nc{\semid}{|\!\!\!\times}
\nc{\Hom}{\on{Hom}}
\nc \RHom{\on {RHom}}

\nc \Proj{\mathrm {Proj\ }}
\nc \proj{\mathrm {proj}}
\nc{\Id}{\on{Id}}
\nc{\id}{\on{id}}
\nc{\Ima}{\on{Im}}
\nc{\invtimes}{\underset{\gets}{\otimes}}
\nc{\invlim}{\underset{\gets}{\on{lim}}}
\nc{\Lie}{\on{Lie}}
\nc{\re}{\on{Re }}
\nc{\Pic}{\on{Pic }}
\nc{\LPic}{\on{LPic }}
\nc{\Sch}{\on{Sch}}
\nc{\Sh}{\on{Sh}}
\nc{\Set}{\on{Set}}
\nc{\spo}{\on{sp\  }}
\nc{\Spec}{\on{Spec}}
\nc{\mSpec}{\on{mSpec}}
\nc{\Specb}{\bold {Spec}\ }
\nc{\Projb}{\bold {Proj}}
\nc{\Specan}{\on{Specan}}
\nc{\Spo}{\on{Sp}}
\nc{\Mpo}{\on{Mp}}
\nc{\Spf}{\on{Spf}}
\nc{\sym}{\on{sym}}
\nc{\symm}{\on{symm}}
\nc{\rop}{\on{r}}
\nc{\Td}{\on{Td}}
\nc{\Tor}{\on{Tor}}

\nc{\Alg}{\on {Alg}}
\nc{\Artin}{\cal{A}rtin}
\nc{\Dgcoalg}{\cal{D}gcoalg} \nc{\Dglie}{\cal{D}glie}
\nc{\Ens}{\cal{E}ns} \nc{\Fsch}{\cal{F}sch}
\nc{\Groupoids}{\cal{G}roupoids}
\nc{\Holie}{\cal{H}olie}
\nc{\Mor}{\cal{M}or}

\nc \Dd{\mathbb D}
\nc{\CF}{\ensuremath{\cal{F}}}
\nc \Kc{{\ensuremath{\cal K}}}
\nc{\Kzind}[4]{{\ensuremath{{\mathcal K^{#4e} (#3)}_{#1 \rightarrow
        #2}}}}
\nc{\Kz}[3]{{\ensuremath{{\mathcal K^\bullet (#3)}_{#1 \rightarrow
        #2}}}}
\nc{\Kzd}[2]{{\ensuremath{{\mathcal K}^\bullet_{#1 \rightarrow #2}}}}
\nc \Lc{{\ensuremath{\cal L}}}
\nc \lcc{{\mathcal l}} 
\nc \CC{{\ensuremath{\cal C}}} 
\nc \Cc{{\ensuremath {\cal C}}}
\nc{\Dl}[2]{{\ensuremath{{\mathcal D}_{#1 \leftarrow #2}}}}
\nc{\Dr}[2]{{\ensuremath{{\mathcal D}_{#1 \rightarrow #2}}}}
\nc \Dc{\ensuremath{\mathcal D}}
\nc \Ac{{\ensuremath{\cal A}}} 
\nc \Bc{{\ensuremath{\cal B}}}
\nc \Ec{{\ensuremath{\cal E}}}
\nc \Fc{{\ensuremath{\cal F}}}
\nc \Mcc{{\ensuremath{\cal M}}} 
\nc \hM{\hat{\Mcc}} 
\nc \bM{\bar {\Mcc}} 
\nc\hbM{\hat{\bar \Mcc}}  
\nc \Nc{{\ensuremath{\cal N}}}
\nc \Hc{{\ensuremath{\cal H}}} 
\nc \Ic{{\ensuremath{\cal I}}} 
\nc \Jc{{\ensuremath{\cal J}}} 
\nc \Oc{\ensuremath{{\cal O}}}
\nc \Och{\hat{\cal O}} 
\nc \Sc{{\ensuremath{{\cal S}}}}
\nc \Tc{\ensuremath{{\cal T}}} 
\nc \Vc{{\ensuremath{{\cal V}}}} 
\nc{\CA}{{\ensuremath{{\cal A}}}}
\nc{\CB}{{\ensuremath{{\cal B}}}}
\nc{\Fcc}{\ensuremath{{\cal F}}}
\nc{\Gc}{{\ensuremath{{\cal G}}}}
\nc{\CH}{\ensuremath{\mathcal H}}
\nc{\CI}{{\ensuremath{{\cal I}}}}
\nc{\CM}{{\ensuremath{{\cal M}}}}
\nc{\CN}{{\ensuremath{{\cal N}}}}
\nc{\CO}{{\ensuremath{{\cal O}}}}
\nc{\Rc}{{\ensuremath{{\cal R}}}}
\nc{\Pc}{{\ensuremath{{\cal P}}}}
\nc{\CT}{{\ensuremath{\mathcal T}}}
\nc{\CU}{\ensuremath{{\cal U}}}
\nc{\Uc}{\ensuremath{{\cal U}}}
\nc{\Yc}{\ensuremath{{\cal Y}}}
\nc{\CV}{\ensuremath{{\cal V}}}
\nc{\CZ}{\ensuremath{{\cal Z}}}
\nc{\Homc}{\ensuremath{{\cal {Hom}}}}

\nc{\fa}{\frak{a}}
\nc{\fA}{\frak{A}}
\nc{\fg}{\frak{g}}
\nc{\fh}{\frak{h}}
\nc{\fI}{\frak{I}}
\nc{\fK}{\frak{K}}
\nc{\fm}{\frak{m}}
\nc{\fP}{\frak{P}}
\nc{\fS}{\frak{S}}
\nc{\ft}{\frak{t}}
\nc{\fX}{\frak{X}}
\nc{\fY}{\frak{Y}}

\nc{\bF}{\bar{F}}
\nc{\bCP}{\bar{\cal{P}}}
\nc{\bmbox}{\mbox{\bf{m}}}
\nc{\bT}{\mbox{\bf{T}}}
\nc{\hB}{\hat{B}}
\nc{\hC}{\hat{C}}
\nc{\hP}{\hat{P}}
\nc{\htest}{\hat P}

\nc{\nen}{\newenvironment}
\nc{\ol}{\overline}
\nc{\unl}{\underline}
\nc{\ra}{\to}
\nc{\lla}{\longleftarrow}
\nc{\lra}{\longrightarrow}
\nc{\Lra}{\Longrightarrow}
\nc{\Lla}{\Longleftarrow}
\nc{\Llra}{\Longleftrightarrow}
\nc{\hra}{\hookrightarrow}
\nc{\iso}{\overset{\sim}{\lra}}

\nc{\dsize}{\displaystyle}
\nc{\sst}{\scriptstyle}
\nc{\tsize}{\textstyle}
\nen{exa}[1]{\label{#1}{\bf Example.\ } }{}

\nen{rem}[1]{\label{#1}{\em Remark.\ } }{}
\nen{exer}[1]{\label{#1}{\em Exercise.\ } }{}

\theoremstyle{definition}

\theoremstyle{remark}

\nc{\Sats}[1]{Sats~\ref{#1}}
\nc{\Sa}[1]{({\sl Sa.}~\ref{#1})}
\nc{\Kor}[1]{({\sl Kor.}~\ref{#1})}
\nc{\Korollarium}[1]{Korollarium~\ref{#1}}
\nc{\Exe}[1]{{\sl Exempel}~\ref{#1}}
\nc{\Anm}[1]{{\sl Anmärkning}~\ref{#1}}
\nc{\Not}[1]{{\sl Not}~\ref{#1}}

\nc{\Formodan}[1]{Förmodan~\ref{#1}}
\nc{\Pastaende}[1]{Påstående~\ref{#1}}

                      {$\square$\vskip\medskipamount\par}
\begin{document}


\title[$D$-modules and finite maps ]{ $D$-modules and finite maps}
\author{Rolf Källstr{\"om}} \address{Department of Mathematics,
  University of Gävle} \email{rkm@hig.se} \date \today
\maketitle
\tableofcontents

\section{Introduction}
Let $\pi: X\to Y$ be a map of smooth varieties over an algebraically
closed field $k$ of characteristic $0$, $\Dc_X= \Dc_{X/k}$ be the ring
of $k$-linear differential operators on $X$, and $\hol(\Dc_X)$ be the
category of holonomic $\Dc_X$-modules. This work is concerned with the
direct and image functors
  \begin{displaymath}
    \hol(\Dc_X) \stackrel[\pi_+]{\pi^!}{\leftrightarrows} \hol(\Dc_Y) 
\end{displaymath}
mostly in the case when $\pi$ is finite. The main goal is to work out
when $\pi_+(M)$ and $\pi^!(N)$ are semisimple for $M\in \hol(\Dc_X)$
and $N\in \hol (\Dc_Y)$, and also to find the structure of the
decomposition. We discuss separately $\pi_+$ and $\pi^!$, and lastly
we consider the subcategory $\Con(X) $ of connections, by which we
mean $\Dc_X$-modules that are coherent over $\Oc_X$, and study how
they interact with $\pi_+$ and $\pi^!$; in particular the notion of
simply connected varieties is given a new treatment.

\subsection{ The direct image $\pi_+$.} Assume that $\pi$ is finite,
so that in particular  $\pi_+$ is exact \Prop{\ref{flat-module}}. Let $L$ and $K$ be
the fraction fields of $X$ and $Y$, respectively, and $\bar L/L/K$ be
a Galois cover, so that $\bar L/L$ and $\bar L/K$ are Galois, and we
denote their Galois groups $H$ and $G$, respectively. Let $\tilde X$
be the integral closure of $X$ in $\bar L$, so that we have finite
maps $\tilde X \xrightarrow{\tilde \pi} X \xrightarrow{\pi} Y $, and
assume for now that $\tilde X$ is also smooth.

The inertia group $G_{\tilde M}$ is a subgroup of symmetries of the
$\Dc_{\tilde X}$-module $\tilde \pi^!(M) = \tilde M$, and there exists
a central extension $\bar G_{\tilde M}$ of $G_{\tilde M}$ that
actually acts on $\tilde M$ (for details, see
(\ref{inertia-section})). Let $\widehat {\bar G}_{\tilde M}$ be the
  isomorphism classes of finite-dimensional representations of
  $\bar G_{\tilde M}$.
  \begin{theorem*}[A] (Thm. \ref{decomposition-thm} and Cor.
    \ref{non-galois}) Let $M$ be a simple holonomic $\Dc_X$-module.
    \begin{enumerate}
    \item If $\supp M$ is finite over $Y$ (so $\pi$ needs not be
      finite), then $\pi_+(M)$ is semisimple.
      \item 
    Assume that $\pi$ is finite, that $\tilde X$ is smooth, and
    $\supp M = X$. Then
\begin{displaymath}
  \pi_+(M) = \bigoplus_{\chi \in \widehat {\bar G}_{\tilde M} }
 (V_\chi ^*)^H\otimes_k \tilde M_\chi,
\end{displaymath} where $\tilde M_\chi$ is a simple $\Dc_Y$-module,
$V_\chi^*$ is the dual of an irreducible $\bar G$-representation, and   $(V_\chi ^*)^H= \{v' \in V_\chi^* \ \vert \ h
\cdot v' = v', \quad h\in H\}$. The multiplicities of the simple
components are
\begin{displaymath}
  [\pi_+(M): \tilde M_\chi] = \dim_k V^H_\chi = \frac 1{|H|}\Tr_H(\phi_\chi)
  = \frac 1{|H|} \sum_{h\in H}  \phi_\chi(h),
\end{displaymath}
where $\phi_\chi$ is the character of
$ \chi \in \widehat {\bar G}_{\tilde M}$.

    \end{enumerate}
\end{theorem*}
The case when $\tilde X$ is singular is treated in
\Corollary{\ref{non-galois}}, and when $\pi$ is Galois (so that
$H= \{e\}$) \Theorem{\ref{explicit}} provides projection operators on
the isotypical components $V_\chi^*\otimes_k M_\chi$ and on the simple
submodules $M_\chi$ of $\pi_+(M)$.

The proof of \Theorem{(A)} (1) goes as follows when $M$ is torsion
free and $\pi$ is finite. Since $\pi_+$ is exact one can first reduce
to proving that $\pi_+(M)$ is semisimple at the generic point of $Y$.
We then have the map $\pi: \Spec L \to \Spec K$, where $K/k$ is a
field extension of finite type over $k$, and by going to the Galois
cover $\bar L$ one can assume $\pi$ is Galois with Galois group $G$.
If now $M$ is semisimple over the ring of $k$-linear differential
operators $\Dc_L= \Dc_{L/k}$ and finite-dimensional as $L$-vector
space, then the ring extension $\Dc_L[\bar G_M]\otimes_{\Dc_L}M$ is
again semisimple over the skew group ring $\Dc_L[\bar G_M]$
($\bar G_M$ is a central extension of the inertia subgroup of $M$ in
$G$), so that one can finally infer that $M$ is semisimple over
$\Dc_K$ from a well-known descent equivalence.

It has already been established that $\pi_+(M)$ is semisimple in the
much more general situation when $\pi$ is projective but not
necessarily finite, accomplished in a long impressive development that
allowed successively greater classes of holonomic modules, see
\cite{perverse} (see also \cite{cataldo-mogliorini:hodge
  theory}),\cite{m-saito:polarisable}, \cite{sabbah: polarised} and
finally \cite{wild-harmonic-wild-pure}*{Th. 19.4.2}. However, in spite
of the fact that it is an assertion in the realm of characteristic $0$
algebraic varieties, the cited work rely on G.A.G.A., the
Riemann-Hilbert correspondence, and the notions weights and/or
harmonic metrics. One can therefore say that this treatise stems from
a frustration at not finding in the literature an algebraic
description of the decomposition of for example $\pi_+(\Oc_X)$,
avoiding the detour to analytic topology and sheaf cohomology. Our
algebraic proof that $\pi_+(M)$ is semisimple when $\pi$ is finite
makes no distinction between regular and irregular singular holonomic
module (or being of geometric origin), and no auxiliary arithmetic or
analytic notions are used; the algebraic approach also allows for a
description of the decomposition of $\pi_+(M)$ in terms of symmetries
of $M$.

Assume now that a finite group $G$ acts on $X$. Levasseur and Stafford
\cite{levasseur-stafford:semisimplicity,wallach:invariantdiff} employ
a Morita equivalence between modules over the ring $\Dc_X^G$ of
invariant differential operators and modules over the skew group
algebra $\Dc_X[G]$ to decompose the $\Dc_X^G$-module $\pi_*(M)$ for
certain $M$ when $\pi : X \to Y= X^G$ is the invariant map (see
\cite{levasseur-stafford:semisimplicity}*{Th. 3.4}); this has its
geometric counterpart in our use of descent. For such $\pi$ the ring
$\Dc_X^G$ coincides with the subring of liftable differential
operators $\Dc_Y^\pi$ in $\Dc_Y$, but also for general $\pi$ there is
a natural homomorphism of $\Dc^\pi$-modules $ \pi_*(M) \to \pi_+(M)$,
induced by a canonical global trace section of the relative dualizing
module $ \omega_{X/Y}$. An advantage of studying the
$\Dc_Y^\pi$-module $\pi_*(M) $ rather than the $\Dc_Y$-module
$\pi_+(M)$ is that the former is coherent over $\Oc_X$ when $M$ is a
connection, while $\pi_+(M)$ normally is not coherent when the support
of $M$ intersects the ramification locus, the drawback being that it
is harder to work with $\Dc_Y^\pi$ than $\Dc_Y$; one is also usually
more interested in the $\Dc_Y$-module since then the de~Rham
cohomology is available. Still, by work of F. Knop
\cite{knop:gradcofinite} one has a reasonably good knowledge of the
properties of $\Dc_Y^\pi$ also when $\pi$ is uniformly ramified, which
is slightly more general than invariant maps, making it possible to
get, if $M$ is torsion free and coherent over $\Oc_X$ along the
ramification locus of $\pi$, that $\pi_*(M)$ is semisimple over
$\Dc_Y^\pi$ and its simple components are in perfect correspondence
with the simples in the $\Dc_Y$-module $\pi_+(M)$
\Th{\ref{coh-decomp}}.

More specifically consider an invariant map
$X= \Spec \So(V) \to Y= \Spec \So(V)^G $, where $V$ is a
finite-dimensional $k$-vector space with symmetric algebra $\So(V)$,
and $G$ is a subgroup of $ \Glo_k(V)$, so that $\Dc_Y^\pi= \Dc_X^G$.
Again by a theorem of Levasseur and Stafford
\cite{levasseur-stafford:invariantdiff}, $\Dc^\pi_Y$ is generated by
the subalgebras $\So(V)^G, \So(V^*)^G \subset \Dc^G_Y$, and they have
also given a presentation of the $\Dc_Y^\pi$-module $\pi_*(E_\lambda)$
when $G$ is a complex reflection group and $E_\lambda$ is a certain
natural exponential $\Dc_X$- module (in the case of a Weyl group,
$\pi_*(E_\lambda)$ controls the structure of a certain important
$\Dc^\pi_Y$-module on a reductive Lie algebra). We complement (and
reprove) their latter result by giving a presentation of the full
direct image $\Dc_Y$-module $\pi_+(E_\lambda)$
\Th{\ref{semsimplemod}}; this presentation, as that of
$\pi_*(E_\lambda)$, however, has the deficiency that no cyclic
generator is obtained. In this context, \Theorem{\ref{normalbasis}}
adds to the normal basis theorem for Galois field extensions $L/K$
that are non-algebraic over $k$, by showing that cyclic generators of
$\pi_+(L)$ either as $\Dc_K$- or $k[G]$-module are the same.

We give an explicit semisimple decomposition of $\pi_+(\Oc_X)$ when
$G$ is an imprimitive complex reflection subgroup $ G(ed,e,n)$ of
$\Glo(V)$ \Th{\ref{imprimitive}}, and at the same time explicit
realizations of the irreducible $G(ed,e,n)$-representations as
subspaces of a polynomial ring. Previously these representations were
constructed in a more computational way
\citelist{\cite{ariki-hecke}\cite{ariki-koike-hecke}}.

Returning to a general finite map of smooth varieties, $\pi: X\to Y$,
there exists a natural coarsest stratification by locally closed
subsets $\{X_{i}, Y_j, i \in I_j, j \in J\}$ such that one gets
restrictions $X_{i}\to Y_i$ that are étale. To this stratification we
associate factorizations
$\pi : X \xrightarrow{p_{i}} Z_{i} \xrightarrow{q_{i}}Y$ such that
$p_{i}$ is minimal totally ramified along the generic point $x_i$ of
$X_{i}$ and $q_{i}$ is maximally étale along the generic point of
$z_i=p_i(x_i)$; if $X/Y$ is Galois, then $p_i$ is even totally
ramified over $p_i(x_i)$. In fact, we prove the existence of such a
factorization for any point in $X$, which is a result of independent
interest that can be regarded as a refinement of Stein factorization
\Th{\ref{factor-theorem}}. Associated to the factorizations are the
{\it inertia submodules} $\Tc_{i}$ of $N= \pi_+(\Oc_X)$, which are
kernels of certain ``inertia'' trace morphisms
\begin{displaymath}
      \overline{ \Tr}_{i}: N\to N;
\end{displaymath}
when $\pi$ is Galois it indeed arises from a trace map with respect to
the map $p_i$. Now for each stratum $Y_i$ there exists, thanks to the
semisimplicity of $N$, a unique maximal submodule $N_j$ with vanishing
local cohomology $R\Gamma_{Y_j}(N_j)=0$; we say that the $N_j$ are the
{\it canonical submodules} of $N$.
\begin{theorem*}[B](Thm. \ref{can-filt})
  \begin{enumerate}
  \item \begin{displaymath}
      N_j = \bigcap_{i\in I_j } \Tc_{i}.
    \end{displaymath}
  \item If $Y_{j'}$ belongs to the closure of $Y_j$, then
    $N_j \subset N_{j'}$.
  \end{enumerate}
\end{theorem*}
Thus $\{N_j\}_{j \in J }$ forms a filtration, and taking successive
quotients one gets a canonical decomposition of $N$ associated to the
stratification $\{Y_j\}$ of $Y$. \Theorem{(B)} is applied to the
invariant map $\pi: \So(V)^{S_n}\to \So(V)$, where $S_n$ is the
symmetric group acting as permutations of a basis of $V$, and describe
the isotypical decomposition of the canonical submodules $N_j$ of
$ \pi_+(\So(V))$ \Th{\ref{symmetric-canonical}} and also the
isotypical content of successive quotients in the canonical filtration
\Cor{\ref{cor_canonical}}.

\subsection{The inverse image $\pi^!$}

We first have the following general result:
\begin{theorem*}[C](Thm. \ref{semisimple-inv} and Thm.
  \ref{cliffordtheorem} ) Let $\pi: X \to Y$ be a surjective morphism
  of smooth varieties over a field $k$ of characteristic 0, and let
  $N$ be a coherent holonomic $\Dc_Y$-module.
  \begin{enumerate}
        \item
    \begin{enumerate}
    \item Assume that $\pi$ is smooth. Then $ \pi^!(N)$ is semisimple
      if and only if $N$ is semisimple, and if $ \pi^!(N)$ is simple,
      then $N$ is simple.
    \item If $N$ is a semisimple connection, then $\pi^!(N)$ is a
      semisimple connection.
    \item Assume $\pi$ is finite. If $\pi^!(N)$ is semisimple, then
      $N$ is semisimple, and if $ \pi^!(N)$ is simple, then $N$ is
      simple.
    \item Assume that $N$ is semisimple and $N_y$ is of finite type
      over $\Oc_{Y,y}$ for all points $y$ of height $\leq 1$ such that
      the closure of $y$ intersects the discriminant $D_\pi$, i.e.
      $\{y\}^- \cap D_\pi\neq \emptyset$. Then $\pi^!(N)$ is
      semisimple.
    \end{enumerate}
\item    Assume that $\pi$ is finite and that $L/K$ is Galois with Galois
  group $G$. Let $N$ be a simple $\Dc_Y$-module such that
  $i^!(N)\neq 0$. Let $M$ be a simple submodule of $\pi^!(N)$, $G_M$
  be the inertia subgroup of $M$ in $G$, and put $t = [G: G_M]$. Then
  \begin{displaymath}
   j_{!+} j^!(\pi^!(N)) = \bigoplus^t_{i=1} (g_i\otimes M)^e,
 \end{displaymath} for some integer  $e$, where $g_i$ are
 representatives of the cosets $G/G_M$. If
 $N$ is a connection, or $\pi$ is a Galois cover, then one can erase $j_{!+}j^!$
 on the left. Moreover:

 \begin{enumerate}
 \item The integer $e$ divides both the order $|G_M|$ of the inertia group
   and the degree of $\pi$.
 \item $\frac {\rko (N)}{\rko (M)}$ divides the degree of $\pi$.
 \end{enumerate}

  \end{enumerate}    
\end{theorem*} 
Here $i: Y_0\to Y$ is the inclusion of the complement of the
discriminant locus of a finite map $\pi$, and if $X_0 \to Y_0$ is the
corresponding base change, $j: X_0 \to X$ is the projection on the
second factor; $g_i\otimes M$ denotes a certain twist of a
$\Dc_X$-module $M$.

In fact $\pi^!(N)$ is always a semisimple connection when $N$ is a
semisimple connection, which was proven using Hodge theory in
\cite{simpson-higgs}. In \Theorem{\ref{inv-conn}} we give an algebraic
proof of this assertion when $N$ belongs to the class of covering
connections (described below).

To a finite group $G$ there exist finitely generated fields $L$ over
$k$ such that $G$ is isomorphic to a subgroup of the automorphism
group $\Aut (L/k)$. Letting $K$ denote the invariant field $L^G$ there
is a well-known Picard-Vessiot equivalence between the category of
finite-dimensional $k$-representations of $G$ and the category of
$\Dc_K$-modules $M$ such that the inverse image $\Dc_L$-module
$L\otimes_K M$ is isomorphic to $ L^n$- say $M$ is ($L$-)étale trivial
- where $n= \rank_K M$; this is described in detail in
(\ref{finite-groups}), where $k$ does not have to be algebraically
closed, contrary to the usual statement. In this equivalence the
inverse (direct) image of $\Dc$-modules correspond to restriction
(induction) of representations with respect to an inclusion of groups
\Prop{\ref{ind-direct}}, so that branching problems for
representations of groups can be translated to decomposition problems
for inverse and direct images of $\Dc$-modules, and {\it vice versa},
certain decomposition problems for $\Dc$-modules can be studied as a
branching problem in group theory. It is thus no surprise that many
classical results in the representation theory of finite groups can be
extended to results about inverse (direct) images of $\Dc$-modules,
and we present here a collection: \Propositions{\ref{simpledirect}
}{{\ref{cyclic-extension}}}, \Theorem{\ref{cliffordtheorem}},
\Corollary{\ref{ito-isaacs}}, \Theorem{\ref{clifford-pairing}},
\Corollaries{\ref{cor-isaacs}, {\ref{irr-crit}}}{\ref{prime-cor}}. But
since the $\Dc$-module categories are bigger, these results imply
corresponding ones in representation theory and not the other way
around, so that although the proofs are sometimes inspired by
representation theory, they do not depend on it. For instance, the
semisimplicity of $\pi_+(M)$ can be deduced from the Picard-Vessiot
equivalence when $M$ is étale trivial, but there exist a great many
semisimple modules $M$ of interest whose differential Galois group is
non-finite.

The category of {\it covering} $\Dc_X$-modules consists of torsion
free modules (over $\Oc_X$) that generically decompose into a direct
sum of rank $1$ modules after taking an inverse image over a finite
field extension, and thus includes the category of étale trivial
modules. Here the main result is that connections which moreover also
are covering modules form a semisimple tensor category which is stable
under arbitrary inverse images \Thms{\ref{alg-semi}}{\ref{inv-conn}}.
It is proven that étale trivial connections occur as submodules of
connections of the form $\pi_+(\Oc_{X_M})$ for some étale map
$\pi: X_M \to X$ \Th{\ref{riemann}}.

To understand the covering $\Dc$-modules one first needs to consider
the category $\Io(X)$ of $\Dc_X$-modules of rank $1$, which is
therefore rather thoroughly studied, where a main tool is an extension
to an algebraic context K. Saito's notion of logarithmic forms and
residues \Prop{\ref{res-basics}}. In \Theorem{\ref{class-rank-1}} we
compare $\Io(X_0)$ to $\Io(X)$ when $X_0$ is an affine subvariety of a
smooth projective variety $X$. A several variable extension of the
classical residue theorem \Th{\ref{residue-theorem}} is used to get a
kind of purity result for regular singular connections of rank $1$
\Cor{\ref{triv-ranke-1}}, and also to define a degree on $(1,1)$-Hodge
classes \Cor{\ref{hodge-degree}}. Now a $\Dc_K$-module of rank $1$ is
determined by a closed 1-form $\gamma \in \Omega_K^{cl}$,
$M= M_\gamma$, and one quite generally can ask what type of field
extension $L/K$ is required to make $L\otimes_K M_\gamma $ trivial.
Requiring that $L/K$ be finite, \Theorem{\ref{etale-trivial-th}} is a
classification of the category of étale trivial $\Dc_K$-modules
$M_\gamma$ of rank $1$. If one goes only slightly beyond finite field
extensions, the condition on $\gamma$ that
$E\otimes_KM_\gamma \cong E$ for some {\it elementary} field extension
$E/K$ (one adds logarithms and exponentials) can be characterized by a
well-known Liouvillan condition. We provide a new uniform proof of
this classical result, allowing non-regular field extensions and also
several variables \Th{\ref{liouville}}.

\subsection{Connections and simply connected varieties} When the only
regular singular connections $M$ on a quasi-projective variety $X$ are
the trivial ones, so that $M\cong \Oc_X^n$, we say $X$ is {\it simply
  connected} (s.c.); there is also the apparently weaker condition
that all étale trivial connections are trivial, and one says $X$ is
étale simply connected. We prove that the smooth locus of a rational
normal projective varieties is s.c. \Prop{\ref{raional-dsc}} and that
for a smooth proper map with connected fibres $X\to Y$, $X$ is s.c.\/
if and only if $Y$ and a closed fibre $X_y$ is s.c.
\Prop{\ref{sc-proper}}. Next we give a version of the
Grothendieck-Lefschetz theorems, stated for connections instead of
fundamental groups. This is concerned with an inclusion
$\phi: Y \to X$ of smooth projective varieties and the functor
$\phi^! : \Con (X)\to \Con (Y)$.
\begin{theorem*}[D](Thm. \ref{groth-lef})
  \begin{enumerate}
  \item If $Y$ is a complete intersection in $X$, then $\phi^!$ is fully
    faithful.
  \item If (1) holds and moreover $\dim Y \geq 2$, then $\phi^!$ is an
    equivalence of categories.
  \end{enumerate}
\end{theorem*}
Grothendieck's theorem is about a comparison of the étale fundamental
groups of $Y$ and $X$, and therefore étale trivial connections, while
\Theorem{(D)} allows general connections (but see
\Remark{\ref{rh-proof-g-l}}).

A somewhat related idea to analyze when a variety is s.c. is to use a
notion of {\it differential coverings}, which is a family of maps
$ p_\lambda : C_\lambda \to X$, $\lambda \in \Lambda$, such that there
exists a dense subset of smooth points $X_0\subset X$ so that each $x$
in $X_0$ is cut out by the $C_\lambda$, meaning that the tangent space
of $x$ is spanned by the image of the tangent spaces of the
$C_\lambda$ such that $x\in p_\lambda(C_\lambda)$ (see
\Lemma{\ref{cut-lemma}}). Say moreover that $X$ is differentially
simple if $X$ can be provided with a differential covering such that
each $C_\lambda$ s.c.. The main result here is:
\begin{theorem*}[E](Thm. \ref{dsc-thm}) Differentially simple
  varieties are simply connected.
\end{theorem*}
This can be used to generalize the known result that smooth projective
rationally connected are simply connected, by allowing normal
quasi-projective varieties, and at the same the proof is entirely
algebraic; the earlier proof for smooth projective varieties is based
on Hodge theory \Cor{\ref{rat-conn-cor}}. One also gets that if a
normal variety is dominated by a simply connected variety that also
has a ``good'' differential covering, then it is simply connected
\Cor{\ref{dom-cor}}. This generalizes earlier results by removing a
properness assumption.

It is perhaps worth mention also that in \Section{\ref{relative-can}}
we study the relative canonical module $\omega_{X/Y}$ and its interaction
with rings of differential operators.

This work has a long history which for a time involved a collaboration
with my good friend Rikard Bögvad, resulting in a spin-off paper
\cite{kallstrom-bogvad:decomp} that can be suitable parallel reading
to the present one. I want to express to him my sincere appreciation
for his -here mostly hidden- contribution. I also want to thank Claude
Sabbah for very valuable remarks.

\section{Operations on \texorpdfstring{$\Dc$}{D}-modules over finite
  maps}
First we describe the structure of the canonical module $\omega_{X/Y}$
of a finite flat map $\pi : X\to Y$ as module over liftable
derivations and relate this to the right $\Dc_X$- and $\Dc_Y$-modules
$\omega_X$ and $\omega_Y$, respectively. Then the inverse and direct
image functors $(\pi_+, \pi^!)$ of $\Dc$-modules over such finite maps
are treated. For the direct image $\pi_+$ we describe how the ordinary
sheaf direct image $\pi_*$ is a subfunctor, taking modules to modules
over the liftable differential operators.

\subsection{The Jacobian ideal and the étale locus}\label{sec:2.1}
Unless explicitly mentioned other\-wise, in this paper $\pi: X\to Y$
denotes a finite surjective map of smooth varieties over an
algebraically closed field $k$ of characteristic $0$. Let $\Omega_{X}$
and $\Omega_Y$ be the sheaves of Kähler differentials over $k$, on $X$
and $Y$ respectively, and $\Omega_{X/Y}$ the relative Kähler
differentials, appearing as the cokernel of the pull-back morphism
\begin{equation}
 \label{eq:relkaehler}
 0\to \pi^*(\Omega_Y)\xrightarrow{p} \Omega_X \to \Omega_{X/Y}\to 0.
\end{equation}
Taking exterior products $\omega_X= \det \Omega_X$ and $\omega_Y = \det \Omega_Y$, we get a
homomorphism of invertible sheaves $\lambda_\pi :=\det p : \pi^*(\omega_Y)\to \omega_X$, defining a
global section $\lambda_\pi $ of
$Hom_{\Oc_Y}(\omega_Y, \omega_X) = \omega_X\otimes_{\Oc_X} \pi^{*}(\omega_Y^{-1})$, which we call
the Jacobian section, and also an isomorphism
\begin{displaymath}
  \pi^*(\omega_Y ) \cong \Imo (\det p) = J_\pi \omega_X .
\end{displaymath}
\begin{remark}\label{rem-fitting} The Fitting ideal $F_0(\Omega_{X/Y})= (J_\pi)$, where $(\det p)
  (dy)= J_\pi dx$,
  is a principal ideal since $X/k$ and $Y/k$ are smooth, and is
  independent of the choice of basis $dy$ and $dx$ of $\omega_Y$ and
  $\omega_X$, respectively. Besides the presentation
  (\ref{eq:relkaehler}) of $\Omega_{X/Y}$ one often computes
  $F_0(\Omega_{X/Y})$ from a presentation of $\pi$ as a complete
  intersection, so that $j: X=V(I) \hookrightarrow Z$ where $Z/Y$ is a
  smooth map of smooth varieties and $I$ is an ideal locally generated
  by a regular sequence, resulting in the presentation
  \begin{displaymath}
    I/I^2 \to j^*(\Omega_{Z/Y}) \to \Omega_{X/Y}\to 0.
  \end{displaymath}
  For example, working locally when $I=(f)$ is a principal ideal, then
  $F_0(\Omega_{X/Y}) = (\bar f')$, where $\bar f'$ is the image in
  $\Oc_X$ of a $Y$-relative derivative $f'$ of $f$  ($d_{Z/Y}(f) = f' dz$ if $z$ is a relative coordinate of $Z/Y$).
\end{remark}
The ramification divisor $B_\pi$ is the divisor corresponding to $J_\pi$, so that
$ F_0(\Omega_{X/Y}) =\Oc_X(-B_\pi)$, $\pi^*(\omega^{-1}_Y) \otimes_{\Oc_X} \omega_X \cong \Oc_X(B_\pi)$,
and
\begin{displaymath}
  B_\pi = \sum_{x\in X, \hto (x)=1 } \ell({\Omega_{X/Y,x}}) x,
\end{displaymath}
where $ \ell({\Omega_{X/Y,x}}) = \nu_x(J_\pi) $ is the length of the $\Oc_{X,x}$-module
$\Omega_{X/Y,x}$, and $\nu_x$ is the canonical discrete valuation of the ring $\Oc_{X,x}$. The
discriminant (branch locus) $D_\pi$ is the image of the divisor $B_\pi$, which is again a divisor as $\pi$ is
finite.  Letting $i : Y_0= Y\setminus D_\pi\to Y $ be the open inclusion, we have the base change
diagram
 \begin{displaymath}\tag{BC}
   \xymatrix{
     X_0 \ar[r] ^j\ar[d]^{\pi_0}  
     & X\ar[d]^\pi\\
     Y_0 \ar[r]^i & Y,
   }
  \end{displaymath}
  where of course the point is that $\pi_0$ is étale. The ramification locus $B_\pi$ is in general a
  proper subset of $X\setminus j(X_0)$.
  \subsection{Differential operators}\label{diffoperators}
Let $T_{Y}$ be the tangent sheaf of $k$-linear
  derivations of $\Oc_Y$, so that $T_Y$ is both an $\Oc_Y$ -module and
  Lie algebra, such that
  $[\partial_1, a\partial_2] = \partial(a)\partial_2 + a[\partial_1,
  \partial_2] $ when $a\in \Oc_Y$ and $\partial_1, \partial_2 $ are
  sections in $T_Y$, and there is a natural notion of $T_Y$-module.
  The diagonal action of $T_Y$ on $\Omega_{Y}$, commonly called the
  Lie derivative (see e.g. \citelist{\cite{kallstrom:preserve}*{Sec.
      2.1} \cite{hotta-takeuchi-tanisaki}*{Sec 2.1}}), induces a
  diagonal action on the determinant bundle $\omega_Y$ such that the
  {\it negative \/ } of this action results in a structure of {\it
    right} module $\omega_Y$ over $T_Y$. It is well known that if $Y$
  is smooth then its ring of differential operators equals the subring
  that is generated by $T_Y$, $\Dc_Y = \Dc(T_Y)$, so that
  $T_Y$-modules are the same as $\Dc_Y$-modules; see e.g.
  \cite{kallstrom-tadesse:liehilbert}*{Prop. 2.2}. Therefore
  $\omega_Y$ is a right $\Dc_Y$-module (see \cite[]{borel:Dmod}), but
  take notice that if $Y$ is singular and $\Dc_Y$ is not generated by
  first order differential operators one cannot make this conclusion.

  Let $\pi: X\to Y$ be a finite surjective morphism of smooth
  varieties and $d\pi : T_X\to \pi^*(T_Y)$ be its tangent morphism.
  Then $\pi_*(d\pi): \pi_*(T_X)\to \pi_*\pi^*(T_Y) $ and also the
  canonical map $T_Y \to \pi_*\pi^*(T_Y) $ are injective, so we can
  define $T^\pi_Y =T_{Y} \cap \pi_*(T_X)$ as a subsheaf both of
  $\pi_*\pi^*(T_Y)$, $T_Y$ and $\pi_*(T_X)$. We call $T^\pi_Y$ the
  subsheaf of liftable tangent vector fields. The sheaf of {\it
    liftable differential operators} $\Dc_Y^\pi$ is defined by
\begin{displaymath}
  \Dc^\pi_Y = \{P \in \Dc_Y \ \vert \ P \cdot  \pi_*(\Oc_X)\subset \pi_*(\Oc_X)\},
\end{displaymath}
where $\Dc_Y$ acts on the image of
$\pi_*(\Oc_X)\to (\pi_*(\Oc_X)_\eta$ at the generic point $\eta$,
since $\Dc_{Y,\eta}^\pi = \Dc_{Y,\eta}$. Thus
$T^\pi_Y \subset \Dc^\pi_Y$ can regarded as subsheaves both of $\Dc_Y$
and $\pi_*(\Dc_X)$.

 Since $\pi_0$ is étale we have
$T_{Y_0} = T_{Y_0}^{\pi_0} = j^*(T^\pi_Y)$, so that over $Y_0$ all
differential operators are liftable and generated by liftable tangent
vector fields,
\begin{displaymath}
  \Dc_{Y_0} = \Dc_{Y_0}^{\pi_0} = \Dc( T^{\pi_0}_{Y_0}). 
\end{displaymath}

When $X/Y$ is ramified, so that $Y \neq Y_0$, we usually have
$\Dc(T^\pi_Y)\neq \Dc^\pi_Y$. For example, $\Oc_Y$ may be simple over
$\Dc^\pi_Y$, whereas the defining ideal of the discriminant locus
$D_\pi$ is always a proper $\Dc(T^\pi_Y)$-submodule of $\Oc_Y$ (see
\cite{kallstrom:liftingder}).
\subsection{The relative canonical bundle and the isomorphism
  \texorpdfstring{$\eta$}{eta}}
  \label{relative-can}
\subsubsection{About the relative canonical module} It is well-known
(see e.g \cite[Cor 8.3]{hartshorne-res}) that there exists an
isomorphism
\begin{equation}\label{C}
\eta: \omega_{X/Y} \to \omega_Y^{-1}\otimes \omega_X,
\end{equation}
where $\omega_{X/Y} $ is the relative dualizing sheaf, but I have been
unable to see any very ``conceptual'' understanding of such an
isomorphism in the literature. One isomorphism is described in
\Proposition{\ref{dual-iso}}, which is moreover linear over the action
by liftable derivations $T^\pi_Y$ (neither side in the isomorphism is
provided with a natural structure of $\Dc_Y^\pi$-module). The map
$\eta$ is defined by selecting canonical global sections $\tr$ and
$\lambda $ in the source and target such that $\eta(\tr)= \lambda$,
where the first is a global section of
$ Hom_{\Dc_X^\pi}(\pi_*(\Oc_X), \Oc_Y) $ and the latter a global
section of $ Hom_{\Dc_Y^\pi, right}(\omega_Y, \pi_*(\omega_X))$, and
we get a kind of conceptual understanding of $\eta$ by proving that
$\lambda$ is the Poincaré dual of $\tr$ \Prop{\ref{poin-dual-dual}}.

The pairs
$\Oc_X, \Oc_Y$ and the canonical sheaves $\omega_X, \omega_Y$ are
left- and right $\Dc_Y^\pi$-modules, respectively. In
\Corollary{\ref{cor-duality}} we prove that if $\pi$ is uniformly
ramified there exists an isomorphism
\begin{displaymath}
  Hom_{\Dc_Y^\pi}(\pi_*(\Oc_X), \Oc_Y) \cong  Hom_{\Dc_Y^\pi}(\omega_Y,\pi_*( \omega_X)),
\end{displaymath}
and that both sheaves are rank 1 constant local systems, where on the
left(right) we have homomorphisms of left(right) modules.

\subsubsection{Details} Let $\qcoh (\Oc_X) $ be the category of quasi-coherent $\Oc_X$-modules.
The direct image functor for the category of quasi-coherent modules, $\pi_*:
\qcoh(\Oc_X)\to \qcoh (\Oc_Y)$ has the right adjoint functor $\pi^{!'}  (M)= \pi^*
Hom_{\Oc_Y}(\pi_*(\Oc_X),M) $, and the relative dualizing sheaf is 
\begin{equation}\label{A}
    \omega_{X/Y} = \pi^{!'}(\Oc_Y) = \bar \pi ^*(Hom_{\Oc_Y}( \pi_* (\Oc_X), \Oc_Y)), 
\end{equation}
where $\bar \pi : (X, \Oc_X)\to (Y, \pi_*(\Oc_X))$ is the canonical
flat map of ringed spaces\footnote{We write $\pi^{!'}$ since $\pi^!$
  will later be used for inverse images of $\Dc$-modules.}. Since
$\pi_* (\Oc_X)$ and $\Oc_Y$ are left modules over $\pi_*(\Dc_X)$ and
$\Dc_Y$, respectively, it follows that $\omega_{X/Y}$ is a left module
over $ \Dc(T^\pi_Y)$, where the action is induced by the diagonal
action of $T^\pi_Y$, so that one may - perhaps naively - think that
$\omega_{X/Y}$ is even a $\Dc_Y^\pi$-module. This is, however, not
true, as indicated by the fact that $\Dc_Y^\pi$ need not be generated
by first order differential operators; see Example
\ref{ex:nostructure}. Similarly, the sheaf
$Hom_{\Oc_X}(\pi^*(\omega_Y), \omega_X)$ is a $\Dc(T^\pi_Y)$-module
but in general not a $\Dc_Y^{\pi}$-module.

As already mentioned, we give a construction of an isomorphism $\eta$
in (\ref{C}) which moreover is compatible with the differential
structure, and also show that the trace map
\begin{equation} \label{T}
  \Tr : \pi_*(\omega_{X/Y})\to \Oc_Y, \quad \lambda \mapsto \lambda (1),
\end{equation}
where $\lambda (1)$ is the evaluation at $1$ homomorphism of a section
$\lambda$ of the right side of (\ref{A}), is a split surjective
homomorphism of $\Dc(T^\pi_Y)$-modules.

In \Section{\ref{sec:2.1}} is defined a canonical global Jacobian section $\lambda_\pi$ of the right
side of (\ref{C}), which also is the image of $1$ in the induced homomorphism
\begin{equation}\label{S}
  \Theta: \Oc_X\to  \omega_X\otimes_{\Oc_X}\pi^*(\omega_Y^{-1})=(J_{\pi})^{-1}=\Oc_X(B_\pi)
\cong  \omega_{X/Y}.
\end{equation}
We will see in \Proposition{\ref{epinymous}} how the morphism
(\ref{S}) is important for relating the direct image of a
$\Dc$-module with the ordinary sheaf direct image.

Next we describe a canonical {\it trace section} $\tr_\pi$ of the left
hand side of  (\ref{C}). It is clearer to do this first in
the affine case, so that $X= \Spec B$ and $Y= \Spec A$, and $\pi$ is a finite
flat morphism $A \to B$ of smooth $k$-algebras, and we will define a
canonical element $\tr_\pi $ of $\omega_{B/A} = Hom_A(B, A)$. It will
coincide with the more common definition of the trace of an element
$b$, acting $A$-lineary as multiplication map, using an $A$-basis of
$B$; see also \cite{huneke-leuschke} for this construction of
$\tr_\pi$, which in fact gives the trace of homomorphisms of any free
$A$-module. We have canonical isomorphisms
  \begin{equation}\label{I}
    \End_A (\omega_{B/A}) \cong \Hom_A(B^*\otimes_A B, A)\cong \Hom_A(\End_A(B),
    A)
\end{equation}
where the first isomorphism is adjunction and the second follows since
$B/A$ is flat and of finite presentation. Now let
$\tr^0\in \Hom_A(\End_A(B), A) $ be the image of the identity element in
the left side of (\ref{I}), $\mu : B \to \End_A(B)$ be the map defined
by $\mu(b_1) (b_2)= b_1 b_2$, $b_1, b_2 \in B$, and put
\begin{displaymath}
  \tr_{\pi}= \tr^0 \circ \mu : B \to A.
\end{displaymath}
This canonical construction of $\tr_\pi$ will make it clear that
$ \tr_{\pi}$ is a $\Dc_A^\pi$-linear homomorphism; moreover, since all
constructions are canonical it will also be obvious that the at first
locally defined sections $\tr_\pi$ glue to a global section of
$\omega_{X/Y}$ also when $X$ and $Y$ are non-affine.

The global sections $\lambda_\pi$ and $\tr_\pi$ correspond in the
sought isomorphism (\ref{C}).

Let $\Dc_X^\pi$ be  the subring of $\Dc_X$ that is generated by $\Dc_Y^\pi$ and
$\Oc_X$,  and $\Dc_X(T_Y^\pi)$ be the subring that is generated by and
$\Dc_Y(T_Y^\pi)$ and $\Oc_X$. Thus $\Dc_X(T^\pi_Y) \subset \Dc_X^\pi \subset
\Dc_X$.
\begin{proposition}\label{dual-iso}
  \label{dual-prop} Let $\pi: X\to Y$ be a finite and generically smooth morphism of smooth schemes
  over a field of characteristic $0$.
  \begin{enumerate}
  \item There is a canonical isomorphism of $\Dc_X(T^\pi_Y)$-modules
    \begin{displaymath}
      \eta: \omega_{X/Y} \to   \omega_X\otimes_{\Oc_X}\pi^*(\omega_Y^{-1}),
    \end{displaymath}
    such that $\eta(\tr_\pi) = \lambda_\pi$, where $\tr_\pi$ and
    $\lambda_\pi$ are described above. Here $\tr_\pi $ is a global
    section of
    $ \pi^*Hom_{\Dc_Y^\pi}(\pi_*(\Oc_X), \Oc_Y)) \subset \omega_{X/Y}
    $
    and $\lambda_\pi$ is a global section of
    $\pi^*(Hom_{\Dc_Y^\pi}(\omega_Y, \pi_*(\omega_X) ))\subset
    \omega_X \otimes_{\Oc_X}\pi^*(\omega^{-1}_Y)$.
  \item The homomorphisms in (\ref{S}) and (\ref{T}) are
    $\Dc(T^\pi_Y)$-linear.
  \end{enumerate}
\end{proposition}

We will at some places below use such an isomorphism $\eta$ to
identify $\omega_{X/Y}$ with
$\omega_X \otimes_{\Oc_X}\pi^*(\omega_Y^{-1})$. In particular, it can
be used to relate the structure of the direct image of a $D$-module to
the ramification of a map.

\begin{remark}
  \begin{enumerate}
  \item I have not seen a {\it direct} construction of the trace
    morphism $\omega_X\otimes_{\Oc_X}\pi^*(\omega_Y^{-1}) \to \Oc_Y$.
    Here it depends on the isomorphism $\eta$ in (\ref{C}) and the
    trace in (\ref{T}). Another indirect construction is to factorise
    $X/Y$ into a closed embedding and a smooth morphism.
  \item If a differential operator in $\Dc_Y$ preserves both $\Oc_Y$ and $\pi_*(\omega_{X/Y})$,
    regarded as submodules of the stalk at the generic point of $X$, then the map in (\ref{T}) is
    compatible with this action. A similar remark can be made regarding the isomorphism $\eta$.
  \item In \Section{\ref{d-module-motive}} we explain in a more conceptual way the existence of the
    isomorphism $\eta$, using $\Dc$-module constructions. \Corollary{\ref{cor-duality}} contains
    more precise information regarding this isomorphism.
  \end{enumerate}
\end{remark}

As a preparation for the proof of \Proposition{\ref{dual-iso}} we recall
some results related to discrete valuation rings, in a slightly more
general situation than is required here. Let $R \to S$ be an inclusion
of discrete valuation rings such that $S$ is of finite type over $R$,
hence free over $R$, and assume that the residue field extension
$k_S/k_R$ is separable; put $n= \dim_R S$. By \cite{serre:corps}*{Ch.
  III, \S 6, Prop. 12} there exists an element $z$ in $S$ whose
residue class in $k_S$ is a primitive element over $k_R$ and such that
the set of powers $\{1, z, \dots , z^{n-1}\}$ is a basis of $S$ over
$R$, and if $f \in R[X]$ is the minimal polynomial of $z$, then
$\{\frac {z^i}{f'(z)} \tr_{S/R}\}$ is a basis of the $R$-module
$\omega_{S/R}= Hom_R(S, R)$. The element $\tau \in \omega_{S/R}$
defined by $\tau ( z^i) =\delta_{i,n-1}$, sometimes called the Tate
trace \cite{tate:genus}, depends on the choice of element $z$, in
contrast to the canonically defined $\tr_{S/R}$.
\begin{lemma}\label{tate} 
  Let $f$ be the minimal polynomial of the element $z$ that is
  described above. Then the element $\tau$ forms a basis of the
  $S$-module $Hom_R(S, R)$ and
  \begin{displaymath}
    \tr_{S/R} = f'(z) \tau. 
  \end{displaymath}
\end{lemma}
\begin{proof} Let $\{\tau_i\}$ be the dual basis of $\{z^i\}$ and write $f= z^n + b_{n-1}z^{n-1}+
  \cdots + b_0 $.  Then the relations
  \begin{displaymath}
  z\tau_{i} = \tau_{i-1} -  b_{i}\tau_{i}
\end{displaymath}
imply that $S \tau = S\tau_{n-1} = Hom_R(S, R) $.  The relation between $\tau$ and $\tr_{S/R}$ is
called Euler's formula, see \cite{serre:corps}*{Ch. III, \S 6-7}.
\end{proof}

\begin{pfof} {\Proposition{\ref{dual-iso}}}
  Assume first that $X/Y$ is affine, given by a homomorphism of smooth $k$-algebras $\pi: A\to B$,
  where $B$ is free over $A$, $K \to B':=K\otimes_A B$ is étale, where $K$ is the fraction field of
  $A$.  Since $X_0/Y_0$ is étale it follows that $B'$ is a product of fields.  It suffices to
  consider one of these fields at a time, so that in the following we assume that $B'= K(B)$, which
  we denote by $L$, and hence we have a finite field extension $L/K$.

  The trace section $\tr_\pi $: All $A$-modules in (\ref{I}) are in
  fact $T_A^\pi$-modules under various diagonal actions, where
  $T^\pi_A = T_A \cap T_B$, the isomorphisms are
  $\Dc(T^\pi_A)$-linear, and $\tr^0$ is a $T_A^\pi$-invariant since it
  is the image of the invariant element 1 in the left side of (I), so
  that
\begin{displaymath}
  \tr^0 \in \Hom_{A}(\End_A(B), A)^{T^\pi_A} = \Hom_{\Dc(T^\pi_A)}(\End_A(B), A).
\end{displaymath}
Putting $\Dc^\pi = \Dc_A \cap \Dc_B$ we prove that $\tr_\pi=\tr^0 \circ \mu$ is $\Dc^\pi$-linear.
If $S$ is a multiplicative system in $A$, then the trace morphism behaves well under localisation
\begin{displaymath}
  S \tr_{\pi}S^{-1}=   \tr_{BS^{-1}/AS^{-1}}  : BS^{-1} \to AS^{-1}.
\end{displaymath}
In particular, $\tr_{L/K}: L \to K$ is the localisation of $\tr_\pi= \tr_{B/A}$.  Since $K\otimes_A
\omega_{B/A} = \omega_{L/K} = \Hom_K(L, K)$ is a left $\Dc_K$-module under the diagonal action, and
the localisation of $\mu$ results in a $\Dc_K$-linear homomorphism $L \to \End_K(L)$ (notice that
$K\otimes_A \Dc(T^\pi_A)= \Dc(T_K^\pi)= \Dc_K$), it follows that the composed map $\tr_{L/K}: L \to
K $ is $\Dc_K$-linear.  Since the natural maps $B \to L$ and $A\to K$ evidently are
$\Dc^\pi$-linear, and the localisation diagram that occurs does commute, it follows that $\tr_{\pi
}$ is $\Dc^\pi$-linear. Assume now that $X$ and $Y$ are not affine. Since $\tr_\pi$ localise as
described above, the locally defined canonical elements glue to a global
$ \Dc_Y^\pi$-linear homomorphism $\tr_{\pi}:  \pi_*(\Oc_X)\to
\Oc_Y$. This proves the assertion about $\tr_\pi$  in  (1).

The Jacobian section $\lambda_\pi$: The pull-back map $p: B\otimes_A\Omega_{A/k}\to \Omega_{B/k}$
commutes with the action of $T^\pi_A$, which therefore defines a $\Dc(T^\pi_A)$-linear map
\begin{displaymath}
  \lambda_{\pi}= \det p \in \Hom_{\Dc(T^\pi)} (\omega_{A/k}, \omega_{B/k}),
\end{displaymath}
where $\omega_{A/k}$ and $\omega_{B/k}$ are provided with the right
$\Dc(T^\pi_A)$-module structures which are described in
(\ref{diffoperators}). Again if $S$ is a multiplicative system, we
have
\begin{displaymath}
  S\lambda_\pi S^{-1} = \lambda_{BS^{-1}/AS^-1} :  \omega_{A/k}S^{-1} = \omega_{AS^{-1}/k} \to
  \omega_{B/k}S^{-1}=  \omega_{BS^{-1}/k}
\end{displaymath}
Now $\omega_{A/k}$ and $ \omega_{B/k}$ are even right $\Dc^\pi$-modules, and by the above identity
$\lambda_{\pi}$ localises to a $K\otimes_A\Dc(T^\pi_A) = \Dc_K$-linear map $\omega_{K/k}\to
\omega_{L/k}$, so it is in particular $\Dc^\pi$-linear. Since the localisation diagram that occurs
is commutative, it follows that $\lambda_\pi\in\Hom_{\Dc^\pi}(\omega_{A/k}, \omega_{B/k})$.  When
$X$ and $Y$ are not affine then the locally defined map $\lambda_{\pi}$ glue to a global
$\Dc^\pi_Y$-linear homomorphism $\pi^{-1}(\omega_Y) \to \omega_X$.  There is a canonical isomorphism
$Hom_{\Dc(T^\pi_Y)}(\pi^{-1}(\omega_Y), \omega_X) =\Hom_{\Dc(T_Y^\pi)}(\Oc_X, \omega_X
\otimes_{\Oc_X} \pi^*(\omega_Y))$ that sends $\lambda_\pi$ to the homomorphism $\Theta$ in
(\ref{S}).  This completes the proof of the last sentence of (1).

Selecting a basis of $\omega_A$ and $\omega_B$ defines a basis
$\{\sigma \}$ of the $B$-module $\Hom_{A}(\omega_A, \omega_B )$. (In a
regular system of coordinates $\{x_i\}$ and $\{y_i\}$, we can select
the bases $dy_1 \wedge \cdots \wedge dy_{d} $ and
$dx_1 \wedge \cdots \wedge dx_{d} $, respectively, and
$\sigma (dy_1 \wedge \cdots \wedge dy_{d}) = dx_1 \wedge \cdots \wedge
dx_{d}$.) We then have $\lambda_{\pi} = J_{\pi} \sigma$, where
$(J_{\pi})= F_0(\Omega_{B/A})$, as described
in \Section{\ref{sec:2.1}}.

Let $p$ be a prime of height $1$ in $\Spec B$, $q= \pi(p)$, and put $R= A_q$ and $S= B_p$, so we
have a map $R \to S$ of discrete valuation rings, where $k_{S}\to k_{R}$ is finite and separable.
Then $F_0(\Omega_{S/R})= (f') = (J_\pi)_p$ (see \Remark{\ref{rem-fitting}}), and by \Lemma{\ref{tate}}
$\tr_{S/R} = f'(x) \tau_p$ where the Tate trace $\tau_p$ is a basis of $(\omega_{B/A})_p$.  It
follows that we get the isomorphism
\begin{displaymath}
  \eta_p : (\omega_{B/A})_p \to (Hom_A(\omega_A, \omega_B))_p, \quad  b_p  \frac {(\tr_\pi)_p
  }{f'(x)}\mapsto b_p \frac {(\lambda_\pi)_p}{f'(x)}.
\end{displaymath}
Since $\tr_\pi$ and $\lambda_\pi$ are global sections of the locally
free $\Oc_X$-modules (of rank 1) $\omega_{X/Y}$ and
$Hom_{\pi^{-1}(\Oc_Y)}(\pi^{-1}(\omega_Y), \omega_X)$, respectively,
and $\Oc_X$ satisfies Serre's condition $(S_2)$, so that sections of
either sheaf are determined by its values in stalks at points of
height $\leq 1$, it follows that we in fact get an isomorphism
$\eta: \omega_{S/R}\to\Hom_R(\omega_R,\omega_S)$. Since $\lambda_\pi$
and $\tr_\pi$ are $T_Y^\pi$-invariant sections it follows also that
$\eta$ is an isomorphism of $\Dc(T_Y^\pi)$-modules. This finishes the
proof of the first sentence in (1). It is straightforward to see that
the diagonal action of $T^\pi_Y$ on the left and right sides of
(\ref{T}) and (\ref{S}), respectively, makes (\ref{T}) and
(\ref{S}) into homomorphisms of $\Dc(T^\pi_Y)$-modules, showing (2).
\end{pfof}

Next is an example showing that the relative canonical bundle is not
preserved by $\Dc^\pi_Y$. First note that there is only one possible
action on $\omega_{X/Y}$ of a liftable differential operator. The
reason is that $\Dc_{Y_0} = j^*(\Dc_Y^\pi)$ (see
(\ref{diffoperators})), which is generated by derivations, so that
$j^* (\omega_{X/Y})$ is a $j^*(\Dc^\pi)$-module. In general, the image
of a section $P\in \Dc^\pi_Y$ in $ j_*j^*(\Dc_Y^\pi)$ acts on
$j_*j^*(\omega_{X/Y})$, so that $P$ will act on $\omega_{X/Y}$ only if
under this action it preserves the submodule
$\omega_{X/Y} \subset j_*j^*(\omega_{X/Y})$. Now we can proceed with
the example.
  \begin{example}
\label{ex:nostructure} 
Let $\pi : A= k[y_1, y_2]\to B= k[x_1, x_2]$, $y_1= x_1 +x_2 $, $y_2= x_1 x_2$, so that $A$ is the
invariant ring under symmetric group $S_2$, and let $L/K$ be the corresponding field extension. We
have $J_{\pi}= x_1- x_2$, and $Y_0$ in the digram \thetag{BC} is defined by
$J_{\pi}^2=(x_1-x_2)^2 = y_1^2- 4y_1y_2\neq 0$. First observe that differential operators in $\Dc_L$
do not act on $\omega_{L/K}$, while $\Dc_K$ does act on $\omega_{L/K}$ since $T_K$ acts diagonally
and $T_K$ generates $\Dc_K$.  The differential operator
\begin{displaymath}
  P\in \partial_{x_1}\partial_{x_2} = \partial_{y_1}^2  + \partial_{y_2} + y_1 \partial_{y_1}\partial_{y_2} + y_2 \partial_{y_2}^2
\end{displaymath}
is liftable, since it is symmetric, so that $P$ is generated by derivations of $K$ and therefore 
acts on $\omega_{L/K}$.  We have $\omega_{B/A} = B \tau = B/J_{\pi}$ (using the isomorphism $\eta$),
where $\tau(a_1+a_2(x_1-x_2))=a_2$, writing $B= A\oplus A(x_1-x_2)$. Now 
$P \cdot 1/J_{\pi} = (\frac{-2}{J_{\pi}^2})(J_{\pi})^{-1}\notin \omega_{B/A} $.  As an illustration
we verify this by instead acting on $\tau \in \omega_{B/A}$.  We have
$\partial_{y_1} \cdot (x_1- x_2) = \frac {1}{2(x_1-x_2)} \partial_{y_1} (y_1^2- 4 y_1 y_2)= \frac
{y_1 - 2 y_2 }{( y_1^2- 4 y_1 y_2)} (x_1- x_2) $,
so that $(\partial_{y_1}\tau) (a_1+a_2(x_1-x_2)) = a_2 \frac {y_1 - 2 y_2}{(y_1^2- 4 y_1 y_2)}$.  A
straightforward tedious computation gives
\begin{displaymath}
  (P \cdot \tau) (a_1 + a_2(x_1-x_2)) = \frac{-2a_2}{J_{\pi}^2} \in K, 
\end{displaymath}
and again $P\tau \in \omega_{L/K}\setminus \omega_{B/A}$.
\end{example}


\subsection{Direct and inverse images}\label{direct-inverse-sec}
Inverse and direct image functors of $\Dc$-modules are treated in 
\cites{borel:Dmod,bjork:analD,hotta-takeuchi-tanisaki}, and we will
only concentrate on a few points pertaining to {\it finite surjective
  maps} $X\to Y$.

For all constructions of inverse and direct images of $\Dc$-modules one makes use of the
$(\Dc_Y, \Dc_Y)$- bimodule $\Dc_Y$ by pulling it back to $X$ in two different ways: Using
the {\it left} adjoint of $\pi_*$ and the {\it left} $\Oc_Y$-module $\Dc_Y$, put
\begin{displaymath}
\Dc_{X\to Y}:=   \pi^*(\Dc_Y)  = \Oc_X \otimes_{\pi^{-1}{\Oc_Y}}(\pi^{-1}(\Dc_Y)),  
\end{displaymath}
This is a $(\Dc_X, \pi^{-1}(\Dc_Y))$-bimodule, described in e.g. in
\cite{borel:Dmod}*{VI,4.2}, \cite{bjork:analD}*{II.3},
\cite{hotta-takeuchi-tanisaki}*{Def. 1.3.1}.

Using the {\it right} adjoint of $\pi_*$ and the {\it right}
$\Oc_Y$-module structure on $\Dc_Y$, put
\begin{align*}
  \Dc_{Y \leftarrow X} &:= \pi^!(\Dc_Y) = \pi^!(\Oc_X)\otimes_{\Oc_X} \pi^*(\Dc_Y) =
  \omega_{X/Y}\otimes_{\Oc_X} \Dc_{X\to Y}\\
  & = \bar \pi^{*}Hom_{\Oc_Y}(\pi_*(\Oc_X), \Dc_Y) = \Oc_X\otimes_{\bar \pi^{-1}(\pi_*(\Oc_X))}
  \pi^{-1}(Hom_{\Oc_Y}(\pi_*(\Oc_X), \Dc_Y))
\end{align*}
(recall that $X/Y$ is finite). This is a
$(\pi^{-1}(\Dc_Y), \Dc_X)$-bimodule, where the two module structures
here are more tricky to understand, and will therefore be explained.
This is best understood using a finite injective map of smooth
$k$-algebras $A\to B$, instead of sheaves on schemes, and thus by
describing below the following $(\Dc_A, \Dc_B)$-bimodule
\begin{align*}\tag{BM}
  \Dc_{B \leftarrow A}&=\Hom_A(B, \Dc_A) = \Dc_A \otimes_A  \Hom_A(B,A) =
                        \Dc_A \otimes_A \omega_{B/A} \\
                      & = \Dc_A  \otimes_A\omega_B \otimes_A\omega_A^{-1} = \omega_A^{-1}\otimes_A
                        \Dc^{op}_A\otimes_A \omega_A \otimes_A \omega_B\otimes_A \omega_A^{-1}  \\ &=  \omega_A^{-1}\otimes_A
                        \Dc^{op}_A \otimes_A \omega_B
                       = 
                        \omega_B  \otimes_A\Dc_A \otimes_A \omega_A^{-1}.
\end{align*}
On the first line the {\it right} $A$-module $\Dc_A$ is used for the
$A$-homomorphisms $ B \to \Dc_A$ and the tensor products; the
structure of all occurring vector spaces as $(A,B)$-bimodules should
be evident. On the second line the first equality sign follows from
the isomorphism in \Proposition{\ref{dual-iso}}, and the second equality
follows since the opposite ring $\Dc_A^{op}$ of $\Dc_A$ is isomorphic
to $ \omega_A\otimes_A \Dc_A \otimes_A \omega_A^{-1}$. The last
isomorphism is defined by
$\eta \otimes P^{(op)} \otimes \xi\mapsto \xi \otimes P \otimes \eta
$,
$\eta \in \omega_A^{-1}$, $\xi \in \omega_B$ and
$P^{(op)} \in \Dc_A^{(op)}$. See also
\cite{hotta-takeuchi-tanisaki}*{Lemma 1.3.4 }. As for the
$(\Dc_A, \Dc_B)$-bimodule structure of \thetag{BM}, the right $\Dc_A$-module $\Dc_A$
is used for the right $\Dc_B$-module structure on
$ \Dc_{B \leftarrow A}=\Hom_A(B, \Dc_A)$, which is determined by the
action of derivations $\partial \in T_B$,
\begin{displaymath}
(\lambda \cdot \partial )(b) = \lambda (\partial (b)) + \sum_i \lambda
(b_ib)\cdot \partial_i,  \quad \lambda \in\Hom_A(B, \Dc_A) ,
\end{displaymath}
where $\sum_i b_i\otimes\partial_i$ is the image of $\partial$ in $B\otimes_A
T_A$.  The left action of $\Dc_A$ on $\Dc_{B \leftarrow A}$ is defined by $(P \cdot
\lambda) (b)= P (\lambda (b))$, where $P \in \Dc_A$.

On the other hand, the right $\Dc_A$-module $\Dc_A$ is used for the
left $\Dc_A$-action on
$\omega_B\otimes_A \Dc_A \otimes_A \omega_A^{-1}$, with respect to the
diagonal action on $\Dc_A \otimes_A \omega_A^{-1}$ and trivial action
on $\omega_B$; the left $\Dc_A$-module $\Dc_A$ is used for the right
$\Dc_B$-action on $\omega_B\otimes_A \Dc_A \otimes_A \omega_A^{-1}$,
with respect to the diagonal action on $\omega_B\otimes_A \Dc_A$ and
trivial action on $\omega_A^{-1}$ (see
\cite{hotta-takeuchi-tanisaki}*{Lemma 1.2.9}).

All the maps in \thetag{BM} are the natural ones, and they are isomorphisms of
$(\Dc_A, \Dc_B)$-bimodules.
\begin{lemma}\label{finite-bimod}
  The $(\Dc_B, \Dc_A)$- bimodule $\Dc_{A\to B}$ and
  $(\Dc_A, \Dc_B)$-bimodules $\Dc_{B\leftarrow A}$ are finite both as
  $\Dc_A$-module and $\Dc_B$-module.
\end{lemma}
\begin{proof}
  Let $\Pc^n_B $ ($\Pc_{B/A}^n$) be the module of (relative) principal
  parts, so that $\Dc_B =\lim_n\Hom_B(\Pc_B^n, B)$ and
  $\Dc_{A\to B} = \lim_n\Hom_B(\Pc^n_{B/A}, B)$, where $\lim_n$
  denotes direct limits. The exact sequence
  $0 \to B\otimes_A\Pc^n_{A}\to \Pc^n_B\to \Pc^n_{B/A}\to 0 $ results in the dual exact
  sequence of $\Dc_B$-modules
    \begin{equation}
      \label{eq:princiapal}
          0 \to     \Dc_B \to \Dc_{A\to B}\to \lim_n Ext^1_{B}(\Pc^n_{B/A},
    B)\to 0. 
  \end{equation}
  See \cite{EGA4:4}; notice that $\Pc^n_{A}$ and $\Pc^n_B$ are locally
  free over $A$ and $B$, respectively, and that $\supp \Pc^n_{B/A}$
  belongs to the ramification locus in $X$. The right side of
  (\ref{eq:princiapal}) belongs to the Grothendieck group of local
  cohomology groups of the $\Dc_B$-module $B$ along closed subspaces
  of $\supp \Pc^n_{B/A}$; hence by Bernstein's theorem it is of finite
  type. Therefore $\Dc_{A\to B}$ is finite over $\Dc_B$. It is
  evidently finite as right $\Dc_A$-module. The corresponding
  assertions for
  $\Dc_{B \leftarrow A}= \omega_{B/A}\otimes_B \Dc_{A\to B}$ follow
  since the functor $\omega_{B/A}\otimes_B\cdot $ is an equivalence
  between $(\Dc_B, \Dc_A )$- and $(\Dc_A, \Dc_B)$-bimodules.
\end{proof}
\begin{example}\label{right-left} We illustrate how to compute with the
  bimodule $\Dc_{B\leftarrow A}$ in \thetag{BM}, in its last
  incarnation, which thus is isomorphic to the first definition
  because of \Proposition{\ref{dual-iso}}. So let $A\to B$ be a finite map
  of regular local rings that are essentially of finite type, and
  choose regular systems of parameters $\{y_i\}$ and $\{x_i\}$ of $A$
  and $B$, respectively, so that the derivations are described by
  $\partial_A = \sum_{i=1}^n \partial_A (y_i)\partial_{y_i} \in T_A$ and
  $\partial_B = \sum_{i=1}^n \partial_B (x_i)\partial_{x_i} \in T_B $,
  where $\partial_{x_i}(x_j)= \delta_{ij}$, and
  $ \partial_{y_i}(x_j)= \delta_{ij}$. Put also
  $dx = dx_1\wedge \cdots \wedge dx_n $,
  $dy = dy_1 \wedge \cdots \wedge dy_n$, and
  $dy^{-1} = \partial_{y_1}\wedge \cdots \wedge \partial_{y_n}$ . The
  right $\Dc_A$-action on $\omega_A$ is determined by
  \begin{displaymath}
  (a dy ) \cdot \partial_A  = (-\tr_y (\partial_A)
  - \partial_A(a))dy,
\end{displaymath}
where $\tr_{y} (\partial_A) = \sum_{i=1}^n \partial_{y_i}(\partial_A (y_i ))$, and the left action on
$\Dc_{A\leftarrow B}$ is determined by
  \begin{displaymath}
    \partial_A (dx \otimes P\otimes dy^{-1}) = - dx \otimes P\cdot (\partial_A + \tr_y (\partial_A))\otimes  (dy^{-1})
  \end{displaymath}
The right action of $\Dc_B$ is determined by
\begin{displaymath}
  (dx \otimes P\otimes dy^{-1}) \partial_B =  - \tr_x(\partial_B)dx \otimes P \otimes (dy)^{-1} -
  \sum_i \partial_B(y_i) dx\otimes \partial_{y_i}P \otimes (dy)^{-1}.
\end{displaymath}
If we have a {\it liftable} derivation $\partial \in T^\pi_A= T_A \cap T_B$ (we use the same symbol
for the derivation of $A$ and its lift to a derivation of $B$), so that
$\partial = \sum_{i=1}^n \partial (x_i)\partial_{x_i} = \sum_{i=1}^n \partial (y_i)\partial_{y_i}$,
then the right and left actions are related as follows:
\begin{align*}
  &  \partial  (dx \otimes P\otimes dy^{-1}) -   (dx \otimes P\otimes dy^{-1}) \partial \\ & = - dx
                                                                                             \otimes P (\partial + \tr_y (\partial))\otimes  (dy^{-1})  +  \tr_x(\partial)dx \otimes P \otimes (dy)^{-1} +
                                                                                             dx\otimes \partial P \otimes (dy)^{-1}\\
  &= dx \otimes [\partial, P] \otimes (dy)^{-1}  + \tr_x(\partial)dx \otimes P \otimes (dy)^{-1} -dx
    \otimes P \tr_y (\partial)\otimes  (dy^{-1}) \\
  & = dx \otimes [\partial + \tr_y(\partial), P] \otimes (dy)^{-1}  + (\tr_x(\partial) -\tr_y (\partial))dx \otimes P \otimes (dy)^{-1}.
\end{align*}
Localising to a map $A \to B_{J_\pi}$ we have $T_A= T_A^\pi$ so that in particular the partial
derivatives $\partial_{y_i}$ are liftable, and we have $dx= J_\pi^{-1}dy$, hence
$dx \cdot \partial_{y_i} = J_\pi^{-1} dy \cdot \partial_{y_i} = \partial_{y_i}(J_\pi)J_\pi^{-1} dx
$, and since $dx \cdot \partial_{y_i} = dx (\tr_x(\partial_{y_i}))$, we have
\begin{displaymath}
  \tr_x(\partial_{y_i}) = \sum_{i=1}^n \partial_{x_i}(\partial_{y_i} (x_i)) =
\frac {\partial_{y_i}(J_\pi)}{J_\pi},
\end{displaymath}
so that
\begin{align*}
  \partial_{y_i} (dx \otimes P \otimes dy^{-1}) & = (dx \otimes P \otimes dy^{-1})\cdot \partial_y  +
  (dx \otimes P \otimes dy^{-1}) \frac {\partial_{y_i}(J_\pi)}{J_\pi}  \\ &= (dx \otimes P \otimes
  dy^{-1}) (J_\pi^{-1}\partial_{y_i} J_\pi).
\end{align*}
\end{example}

The ring $j_+(\Dc_{X_0}) $ can be regarded either as a $(\Dc_X, \pi^{-1}(\Dc_Y))$ - or
$(\pi^{-1}(\Dc_Y), \Dc_X)$-bimodule in a natural way, and as such we denote it
$j_+(\Dc_{X_0}) _{(i)}$ and $j_+(\Dc_{X_0}) _{(ii)}$, respectively. We have then injective
homomorphisms
\begin{align*}
f&:  \Dc_{X\to Y}=\Oc_X\otimes_{\pi^{-1}(\Oc_Y)}\pi^{-1}(\Dc_Y)   \to  j_+(\Dc_{X_0})_{(i)}, \quad \phi \otimes P \mapsto \phi\tilde P,  \\
g&:  \Dc_{X\leftarrow Y} = \omega_{X/Y}\otimes_{\Oc_X}  \Dc_{X\to Y}  \to  j_+(\Dc_{X_0})_{(ii)},
   \quad \lambda_0 \tr_\pi \otimes P \mapsto  \lambda_0 \tilde P,
\end{align*}
where $\tilde P$ is the unique lift of a differential operator $P$ on
$Y$ to a differential operator on $X_0$, and $\tr_\pi$ is the
canonical global section of $\omega_{X/Y}$ (see (\ref{S}) in
(\ref{relative-can})). Since $j^*(\Dc_{X\to Y})= j^*(\Dc_X)$ and the
canonical map $j^*(\Oc_X)\to j^*(\omega_{X/Y})$ is an isomorphism it
follows that the restriction of $f$ and $g$ to $X_0$ are isomorphisms.

Let $\coh(\Dc_X)$ be the category of coherent left $\Dc_X$-modules.
The {\it inverse image} functor is defined
\begin{align*}
  \pi^!&: \coh(\Dc_Y ) \to \coh(\Dc_X), \\
  M &\mapsto   \Dc_{X\to Y}\otimes_{\pi^{-1}(\Dc_Y)}\pi^{-1}(M) =  \Oc_X\otimes_{\pi^{-1}(\Oc_Y)}\pi^{-1}(M), 
\intertext{and the {\it direct image} functor }
  \pi_+&: \coh(\Dc_X ) \to \coh(\Dc_Y), \\
  M &\mapsto \pi_+(M)= \pi_*(\Dc_{ Y\leftarrow X }\otimes_{\Dc_X}M).
\end{align*}

\begin{remark}\label{invim-remark}
  The general definition of the inverse image is the derived functor
  $\pi^!(M)=\Dc_{X\to Y}\otimes^L_{\pi^{-1}(\Dc_Y)} \pi^{-1}(M)
  [\dop_{X}- \dop_Y]$, where $\dop_ X$ is the dimension of $X$,
  simplifies to the above expression when $\pi$ is finite and
  surjective. For the same reason we will see in
  \Proposition{\ref{flat-module}} that we do not need the derived
  version of the direct image functor. See \cite{borel:Dmod }*{VI, 4.2
    and, 5.1}
\end{remark}

Since $B_\pi$ and $D_\pi$ are divisors, it follows that the direct image functors $j_+$ and $i_+$
for the open embeddings are exact. Assume that $M$ contains no section whose support is contained in
$B_\pi$. Then
\begin{displaymath}
  \pi_+(M) \subset i_+i^+(\pi_+(M)) = i_+(\pi_0)_+ j^+(M),
\end{displaymath}
hence if $U$ is an open set, the action of a section $P\in \Dc_Y(U)$
on $\pi_+(M)(U)$ is by lifting its restriction to a section of
$\Dc_{Y_0}(U\cap Y_0)$ to a differential operator in
$j_+(\Dc_{X_0})(U)$, and then acting on $M(\pi^{-1}(U))\subset
j_+j^+(M)(\pi^{-1}(U)) = j^+(M)(\pi_0^{-1}(U\cap Y_0) )$, so that in
effect
\begin{equation}\label{dir-im}
  \pi_+(M) = \Dc_Y \pi_*(\omega_{X/Y}\otimes_{\Oc_X}M) \subset i_+(\pi_0)_+ j^+(M).
\end{equation}
\begin{example}
  Let $\pi : A= k[y_1, \ldots , y_n] \to B= k[x_1, \ldots x_n]$ be an injective polynomial map and
  $M$ be  a $\Dc_B$-modules such that  $B_{J_\pi}\otimes_B M = M$. Then $\pi_+(M)= M$ as $A$-module,
  and the action of
  $\partial_{y_i}$ is 
  \begin{displaymath}
    \partial_{y_i}\cdot m   =  J_\pi \partial _{y_i} J_\pi^{-1} m
  \end{displaymath}
\end{example}

Let $\pi: X\to Y$ be a morphism finite type, where $X$ and $Y$ are
smooth and $\coh^f(\Dc_X)$ be the category of coherent $\Dc_X$-modules
such that the restriction $\pi: \supp M\to Y$ is a finite morphism.
Thus if $\pi$ is finite, then $\coh(\Dc_X) =\coh^f(\Dc_X)$.

\begin{proposition}\label{flat-module}
  The direct image functor $\pi_+ : \coh^f_c(\Dc_X)\to \coh_c(\Dc_Y)$,
  $M\mapsto \pi_+(M)$, is exact. Assume that $\pi$ is finite. Then
  $\pi_+ : \coh(\Dc_X)\to \coh(\Dc_Y) $ is exact and the left
  $\Dc_X$-module $\Dc_{X\to Y}$ and the right $\Dc_X$-module
  $\Dc_ {Y \leftarrow X}$ are both flat.
\end{proposition}
The earliest reference to this well-known result is in
\cite{hotta-ka}, which in turn refers to
\cite{kashiwara:systemsmicrodifferential}, but I was unable to find a
proof in either source; for a published proof, see
\cite{bjork:analD}*{Th. 2.11.10}. Still, due to its importance and
also to avoid references to analytic structures, a complete proof is
included.
\begin{proof} Since the ordinary sheaf direct image functor $\pi_*$ is exact on the category of
  sheaves with finite support over $Y$, and the functor $M \to \Dc_{Y \leftarrow X} \otimes_{\Dc_X}M
  = \omega_{X/Y}\otimes_{\Oc_X} \pi^*(\Dc_Y) \otimes_{\Dc_Y} M$ takes modules of finite support over
  $Y$ to sheaves with finite support over $Y$, it follows that
  \begin{displaymath}
    \pi_+(M) = \pi_*(\omega_{X/Y}\otimes_{\Oc_X}    \pi^*(\Dc_Y) \otimes^L_{\Dc_Y} M).
\end{displaymath}
Since $X$ is of finite type over $k$ there exists a regular immersion
$ i : X \to \Ab^l_k$, where $\Ab^l_k$ is the $l$-dimensional affine
space over $k$. Put $Z= \Ab^l_k\times_k Y$ and define the composed map
\begin{displaymath}
  f : X \to X \times_kY \to Z,
\end{displaymath}
where the first is the graph embedding and the second is
$ i \times id : X \times_kY \to Z$. Letting $p : Z \to Y$ be the
projection on the second factor we then have $\pi = p \circ f$, so
that $\pi_+ (M) = p_+(f_+(M))$, where $M_1=f_+(M)$ is a $\Dc_Z$-module
such that restriction of $p$ to $\supp M_1$ is finite. Since $f$ is a
closed embedding and hence exact by Kashiwara's theorem, it suffices
to prove that the functor $p_+: \coh^f(\Dc_Z)\to Y$ is exact. Now
putting $Z_i=Y \times \Ab^i $ ($Z_0 = Y $) the map $p $ can be
factorized $p= p_1\circ p_2\circ \cdots \circ p_l $, where
$p_i : Z_i\to Z_{i-1}$ is induced by the projection map
$\Ab^i\to \Ab^{i-1}$, hence we have an isomorphism of derived functors
$p_+ = (p_1)_+\circ (p_2)_+\circ \cdots \circ (p_l)_+ $ (see
\cite{hotta-takeuchi-tanisaki}*{Prop. 1.5.21}) where $(p_i)_+$ defines
a functor $\coh^f(\Dc_{Z_i})\to \coh^f(\Dc_{Z_{i-1}})$. It
suffices therefore to see that $(p_1)_ +$ is exact, and this follows
if the homology $H^\bullet ((p_1)_+(M))$ is concentrated in one degree
when $\supp M$ is finite over $Y$. Now
$(p_1)_+(M) = \Omega^\bullet_{Z_1/Y}(M) = \Omega^\bullet_{\Ab^1}(M)$,
the relative de~Rham complex \cite{hotta-takeuchi-tanisaki}*{Prop
  1.5.28}, and letting $t$ be a regular parameter for the affine space
$\Ab_k^1$ and $\partial_t$ is a basis for the relative derivations
$T_{Z_1/Y}$, satisfying $\partial_t (t)=1$, only the following two
homology groups occur
\begin{displaymath}
  M^{\partial_t} = \{m \ \vert \ \partial_t \cdot m =0 \} \quad \text{
    and } \quad \frac 
  M{\partial_t M}. 
\end{displaymath}
Thus it is enough to prove that $M^{\partial_t} =0$, so assume on the
contrary that there exists a non-zero section $m$ of $M$ such that
$\partial_t m=0$. Since the support $\supp m$ of $m$ is finite over
the closed subset $Y_m= \pi_1(\supp m)\subset Y$ there exists a
non-zero polynomial $p(t)\in \Oc_{Y_m}[t]$, such that $p(t)m=0$, which
we can select to have minimal degree in $t$. Since
$0=\partial_t (p(t)m)=\partial_t (p(t))m$ and
$\deg_t \partial_t(p) < \deg_t p$, we get that $\partial_t (p)=0$ and
hence $p\in \Oc_{Y_m}$. Since by assumption $m\neq 0$ it follows that
$pm \neq 0$, we get a contradiction. Therefore $m=0$.

If $\pi$ is finite, then both the functor
$\pi_+(\cdot) = \pi_*(\Dc_{Y \leftarrow X} \otimes_{\Dc_X}(\cdot )) $
and $\pi_*$ are exact. This implies that the right $\Dc_X$-module
$\Dc_{Y \leftarrow X}$ is flat. Switching left for right is an exact
functor, hence $\Dc_{X\to Y}$ is also a flat left $\Dc_X$-module.
\end{proof}

\subsection{Duality  and finite maps}
First it is recalled how a duality operation $\Dbb$ gives rise to the
category of holonomic $\Dc$-modules, next is presented an accessible
proof that $\pi_+ \Dbb = \Dbb \pi_+$ when $\pi$ is finite, and the
trace map $\Tr : \pi_+\pi^! \to \id$ is described. The isomorphism
$\eta$ in (\ref{C}) in \ref{relative-can} is explained also using the
duality $\Dbb$. The appendix contains an account for the properties of
the minimal extension functor, which is a mechanism for getting simple
holonomic $\Dc$-modules as extensions of simple modules from a locally
closed subset.

\subsubsection{Duality}\label{poincare}
Let $d_X$ denote the dimension of the smooth variety $X$.  On the derived category $D_c(\Dc_X)$ of
complexes of $\Dc_X$-modules with coherent homology one can define the functor
\begin{displaymath}
\Dbb_X : D_c(\Dc_X)\to D_c(\Dc_X),   M\mapsto RHom_{\Dc_X}(M, \Dc_X [d_X])\otimes_{\Oc_X}\omega_X^{-1}
\end{displaymath}
which is contravariant and satisfies $\Dbb_X^2 = \id$, so it is a duality. The functor $\Dbb_X$ does not
preserve the full subcategory $\coh(\Dc_X)\subset D_c(\Dc_X)$, so therefore say that a module
$M\in \coh(\Dc_X)$ is {\it holonomic} if $\Dbb_X( M) $ is concentrated in degree $0$. Denote by
$\hol (\Dc_X)\subset \coh(\Dc_X)$ the full subcategory of holonomic $\Dc_X$-modules.

If $X= \Spec L$, where $L/k$ is a field extension of finite type and
$\Dc_L = \Dc_{L/k}$ is the ring of $k$-linear differential operators,
then $\hol (\Dc_X)= \Mod_{fd}(\Dc_{L/k})$, the category of
$\Dc_L$-modules that are of finite dimension over $L$.

  The fundamental fact is that if $M$ is holonomic, then $\Dbb_X (M)$ is again holonomic, so that
  $\Dbb_X $ restricts to a duality functor
\begin{displaymath}
\Dbb_X : \hol (\Dc_X)\to \hol (\Dc_X), \quad M \mapsto Ext^{d_X}_{\Dc_X}(M, \Dc_X)\otimes_{\Oc_X}
  \omega^{-1}_X.
\end{displaymath}
We need to know how to compute $\Dbb_X$ in an important case. Say that
a $\Dc$-module $M$ is a {\it connection} if it is coherent over
$\Oc_X$. Let $\Con (\Dc_X) $ be the category of
connections\footnote{Thus in this work all connections are
  integrable.}. It is well-known that connections are locally free
over $\Oc_X$, and if $M$ is a connection then
$M^* = Hom_{\Oc_X}(M, \Oc_X)$ is again a connection. The following
proposition shows in particular that
$\Con (\Dc_X)\subset \hol (\Dc_X)$.
\begin{proposition}\label{dual-con} If $M\in \Con(\Dc_X)$, then
  \begin{displaymath}
    \Dbb_X (M) =M^*,
  \end{displaymath}
  where $\Dc_X$ acts diagonally on the right side, so that $\Dbb_X (M)$ is again a connection.
  \end{proposition}
  A proof is included only to clarify this well-known result, as it is
  more or less proven in \cite{hotta-takeuchi-tanisaki}*{Example
    2.6.10}, but two subtle details can be more emphasized. One is
  having to do with the following lemma, and the second with certain
  vertical isomorphisms in [loc cit] that exist due to the fact that
  $\Dc_X$ is generated in degree 1.
  \begin{lemma}\label{basic} Let $F$ be a locally free $\Dc_X$-module
    of rank $m$ and $M$ be a $\Dc_X$-module which is locally free over
    $\Oc_X$ of rank $n$, then the $\Dc_X$-module $F\otimes_{\Oc_X}M$,
    provided with the usual diagonal action of $\Dc_X$, is a locally
    free $\Dc_X$-module of rank $nm$.
      \end{lemma}
      \begin{proof} One may assume that $X=\Spec A$ for some
        $k$-algebra $A$ such that the $A$-module of derivations $T_A$
        is free (e.g. $A$ is local and $k$-smooth); let the
        derivations $\partial_{x_i}$ form an $A$-basis of $T_A$. It
        suffices now to prove that $\Dc_A\otimes_A M$ is free of rank
        $r$ as $\Dc_A$-module, when $M$ is a $\Dc_A$-module which is
        free of rank $r$ as $A$-module. So assume that
        $M= \oplus_{i=1}^r Am_i$ and that we have a relation
        \begin{displaymath}
          \sum P_i (Q_i\otimes m_i) =0. 
        \end{displaymath}
        Recall that $\Dc_A$ is free as right $A$-module with basis
        $\partial^\alpha= \partial_{x_1}^{\alpha_1}\partial_{x_2}^{\alpha_2}\cdots \partial_{x_n}^{\alpha_n}$.
        The order $o(P)$ of
        $P= \sum_\alpha \partial^\alpha p_\alpha\in \Dc_A$ is the
        maximal value of $|\alpha| = \sum_i \alpha_i$ such that
        $p_\alpha \neq 0$. Now
        $P_i (Q_i\otimes m_i) = (P_i Q_i)\otimes m_i + \sum_{j=1}^r
        R_{ij}\otimes m_j $, where $o(R_{ij})< o(P_i Q_i)$. Letting
        $I$ be the set of indices such that $o(P_i Q_i)$ attains it
        maximal value $t$ we get
\begin{displaymath}
  \sum_{i\in I} (P_i Q_i\otimes m_i) + \sum (T_i\otimes m_i) =0,
\end{displaymath}
where $o(T_i) < t$, $i\in I$. This implies that if $i\in I$ and
$P_iQ_i = \sum_\alpha \partial^\alpha s_\alpha$, $s_\alpha \in A$,
then $\sum_{|\alpha| = t} \partial^\alpha s_\alpha =0$, and hence
$P_iQ_i =0$, and as $\Dc_A$ does not have zero-divisors, we get
$P_i=0$.
      \end{proof}
  \begin{pfof}{\Proposition{\ref{dual-con}}}
    See first \cite{hotta-takeuchi-tanisaki}*{Example 2.6.10}, which we
    complement as follows.  The right-most tensor product over $\Oc_X$ in the
    Koszul (or Spencer) complex $\Kc^\bullet (M)=\Dc_X \otimes_{\Oc_X}
    \wedge^\bullet T_X \otimes_{\Oc_X} M = \Kc^\bullet (\Oc_X)\otimes_{\Oc_X} M$
    of $M$ is taken with respect to the left $\Oc_X$-module structure on
    $K^\bullet (\Oc_X)$, the $\Dc_X$-module structure is given by the diagonal
    action of $T_X$, and the differential is induced from the usual
    $\Dc_X$-linear differential of $\Kc^\bullet (\Oc_X)$. It follows from
    \Lemma{\ref{basic}} that $\Kc^\bullet (M)$ is locally free over $\Dc_X$, so
    that $\Kc^\bullet (M)$ is a resolution of $M$ by acyclic objects. We have
    (as detailed below)
    \begin{align*}
      & RHom_{\Dc_X}(M, \Dc_X) = Hom_{\Dc_X}( \Kc^\bullet(M), \Dc_X) = 
      Hom_{\Oc_X}(\Kc^\bullet(M), \Dc_X)^{T_X} \\
      &=Hom_{\Oc_X}( \Kc^\bullet(\Oc_X) , Hom_{\Oc_X}(M, \Dc_X))^{T_X} =
      Hom_{\Dc_X}( \Kc^\bullet(\Oc_X) , Hom_{\Oc_X}(M,
      \Dc_X))\\ & = Hom_{\Dc_X}( \Kc^\bullet(\Oc_X) , Hom_{\Oc_X}(M,\Oc_X) \otimes_{\Oc_X}\Dc_X) 
       = \Omega_X^\bullet (M^*\otimes_{\Oc_X}\Dc_X).
    \end{align*}
    The last entry on the first line denotes the $T_X$-invariants with
    respect to the diagonal action on
    $Hom_{\Oc_X}(\Kc^\bullet(M), \Dc_X) $, and the isomorphism relies
    on the fact that $\Dc_X$ is generated by $\Oc_X$ and $T_X$. The
    first equality on the second line is adjunction for
    $\Oc_X$-modules, noting that it is an isomorphism of
    $\Oc_X$-modules (before taking $T_X$-invariants). In the last
    line, $\Omega_X^\bullet (M^*\otimes_{\Oc_X}\Dc_X)$ is the de~Rham
    complex of the free $\Dc_X$-module $M^*\otimes_{\Oc_X}\Dc_X$
    \Lem{\ref{basic}}. Since $\Omega^\bullet_X(\Dc_X)$ is a resolution
    of $\omega_X$ [loc cit, Lemma 1.5.27] it follows that
    $\Omega_X^\bullet (M^*\otimes_{\Oc_X}\Dc_X)$ is a resolution of
    $M^* \otimes_{\Oc_X}\omega_X [-d_X]$. This gives
    $\Dbb_X( M) \cong M^*$.
  \end{pfof}


  It is a general fact that $\pi^!$ defines a functor $\hol(\Dc_Y)\to \hol(\Dc_X)$ (see
  \cite{hotta-takeuchi-tanisaki}*{Th. 3.2.3}), therefore one also  gets the functor $\pi^+ = \Dbb_X \pi^!
  \Dbb_Y : \hol(\Dc_Y)\to \hol(\Dc_X)$.
\begin{lemma}\label{con-lemma}
  If $M\in \Con(\Dc_Y)$, then $\pi^!(M)= \pi^+(M) $. 
\end{lemma}
In fact, it suffices that $\pi$ be non-characteristic with respect to $M$ for this identity to hold
\cite{hotta-takeuchi-tanisaki}*{Th 2.7.1}. 
\begin{proof}
  We apply \Proposition{\ref{dual-con}} and notice that $\pi^!(M^*)$ is also a
  connection,
  \begin{displaymath}
    \pi^+(M) =   \Dbb_X  \pi^!\Dbb_Y M = (\pi^!(M^*))^* = \pi^!(M), 
\end{displaymath}
where the last equality follows since $M$ is locally free over $\Oc_Y$ and $X/Y$ is flat.
\end{proof}
\begin{proposition}\label{etale-finite}
  Let $\pi: X\to Y$ be a finite morphism of smooth varieties and $N$
  be a holonomic $\Dc_X$-module.
\begin{enumerate}
  \item If $\pi_+(N)=0$, then $N=0$. 
  \item If $\pi_+(N)$ is a connection, then $N$ is a connection.
  \item The following are equivalent for a non-zero connection $N$:
      \begin{enumerate}
      \item $\pi_+(N)$ is a connection.
      \item $\pi$ is étale.
      \end{enumerate}
  \end{enumerate}
\end{proposition}
\begin{proof}
  We can assume that $X$ and $Y$ are affine, so that we have a finite
  map $\pi: A\to B$ of smooth $k$-algebras. The sequence
  (\ref{eq:princiapal}) also gives the exact sequence
  $0 \to \Dc_B \to \Dc_{B \leftarrow A}\to \omega_{B/A}\otimes_B
  \lim_n Ext^1_{B}(\Pc^n_{B/A}, B)\to 0 $, so that tensoring with the
  $\Dc_B$-module $N$ results in the exact sequence
  \begin{displaymath}
0 \to     \pi_*(N) \to \pi_+(N)=\Dc_{B\leftarrow A}\otimes_{\Dc_B}N  \to
\omega_{B/A}\otimes_B \lim_n Ext^1_{B}(\Pc^n_{B/A},
    B)\otimes_{\Dc_B} N \to 0.
  \end{displaymath}
  This implies (1). If the middle term is finite over $A$, then $N$ is
  a finite $B$-module, giving (2). (3): If $\pi$ is étale, then
  $\Dc_{A\to B} \cong \Dc_B$, which gives that $\pi_+(N)$ is a
  connection, proving (a). If $\pi$ is not étale, then the right hand
  side of the above exact sequence is a non-zero torsion $A$-module,
  hence is not a connection, and, since the image of a morphism from a
  connection to a $\Dc$-module again is a connection, $\pi_+(N)$
  cannot be a connection. This proves (b).
\end{proof}
\subsubsection{Duality and finite morphisms}

\begin{theorem}\label{adj-triple-th} Let $\pi: X\to Y$ be a finite map
  of smooth varieties. Then $\Dbb_Y \pi_+ = \pi_+ \Dbb_X$ and we have the adjoint triple $(\pi^+,
  \pi_+, \pi^!)$.  If moreover $\pi$ is smooth, then $\pi^! = \pi^+$ and we have the adjoint triple
  of functors $(\pi_+, \pi^+, \pi_+)$, i.e. $\pi_+$ is both the left and right adjoint functor of
  $\pi^+$.
\end{theorem}
There exists an isomorphism of functors $\pi_+\Dbb_X = \Dbb_Y \pi_+$
for any proper map $\pi$ of smooth varieties, but since the proof is
fairly involved (see
\citelist{\cite{borel:Dmod}*{Prop. 9.6}\cite{hotta-takeuchi-tanisaki}*{Th. 2.7.2}
  \cite{bjork:analD}*{Th. 2.11.3}}) we give a more accessible proof
for a finite map, based on \Proposition{\ref{flat-module}}, paying
special attention to the role of \Proposition{\ref{dual-iso}}. Proofs of
the other assertions in \Theorem{\ref{adj-triple-th}} can also be
found in [loc. cit].

\begin{lemma}\label{lemma-to-dual}
  Let $M_1$ and $M_2$ be  left $\Dc_B$-modules.
  Then  
$    (    M_1\otimes_B \omega_B )\otimes_{\Dc_B} M_2 = (\omega_B\otimes_B
    M_2)\otimes_{\Dc_B} M_1$ as $B$-modules.
\end{lemma}
\begin{proof}
  We know that $M_1\otimes_B \omega_B$ and $\omega_B\otimes_B M_2$ are right $\Dc_B$-modules, where
  the $\Dc_B$-action is determined by the action of derivations $\delta $ of $B$, according to
  $(m_1\otimes \eta) \cdot \delta = (- \delta \cdot m_1)\otimes \eta + m_1\otimes (\eta \cdot
  \delta)$
  where $\eta \cdot \delta$ is the negative of the Lie derivative of $\eta$ along
  $\delta$. Similarly,
  $(\eta\otimes m_2) \cdot \delta =\eta \otimes (- \delta \cdot m_2) + (\eta \cdot \delta)\otimes
  m_2$.
  Therefore both sides makes sense.  The map
  $\phi: ( M_1\otimes_B \omega_B )\otimes_{\Dc_B} M_2 \to (\omega_B\otimes_B M_2)\otimes_{\Dc_B}
  M_1$,
  $\phi( m_1 \otimes \eta \otimes m_2) = \eta \otimes m_2 \otimes m_1$ is well-defined, since by the
  above description a simple computation shows that
  $\phi (((m_1\otimes \eta) \cdot \delta) \otimes m_2) = \phi((m_1\otimes \eta)\otimes (\delta\cdot
  m_2)) $. Clearly, $\phi$ is an isomorphism.
\end{proof}
\begin{pfof}{\Theorem{\ref{adj-triple-th}}}
  $\Dbb_Y \pi_+ = \pi_+ \Dbb_X $: Since $\pi_*(\Oc_X)$ is locally free over $\Oc_Y$, and also in
  order to see the main steps clearly we work instead with a finite homomorphism $A\to B$ of smooth
  $k$-algebras, so that $B$ is free over $A$.  First note that $ \Dc_{A \leftarrow B} $ is a free
left $\Dc_A$-module, so that 
\begin{align*}
  RHom_{\Dc_A}( \Dc_{B \leftarrow A} , \Dc_A) &=\Hom_{\Dc_A}(\Dc_A \otimes_A \omega_{B/A}, \Dc_A) =
 \Hom_A(\omega_{B/A}, A)\otimes_A \Dc_A \\
  & = B \otimes_A \Dc_A=  \Dc_{A\to B}, \tag{*}
\end{align*}
where the  equalities are  canonical isomorphisms of $(\Dc_B, \Dc_A)$-bimodules. 
Also, since $\Dc_{A \to B}$ is a free right $\Dc_A$-module, we have  
\begin{align*}
    RHom_{\Dc_A}( \Dc_{A \to B} , \Dc_A) &=\Hom_{\Dc_A}( B\otimes_A\Dc_A, \Dc_A) =
 \Hom_A(B, A)\otimes_A \Dc_A  & = \\
  & = \omega_{B/A}\otimes_A\Dc_A= \Dc_{B \leftarrow A}\tag{*'},
\end{align*}
where the  equalities are canonical isomorphisms of $(\Dc_A,
\Dc_B)$-bimodules. 
If $M$ is a holonomic $\Dc_B$-module  we have canonical isomorphisms of complexes of  $\Dc_A$-modules  (as detailed below)
  \begin{align*}
    & \Dbb_A (\pi_+(M))[-d_A]= RHom_{\Dc_A}( \Dc_{A \leftarrow B} \otimes_{\Dc_B} M, \Dc_A)\otimes_A
    \omega_A^{-1}\\ & = RHom_{\Dc_A}( \Dc_{A \leftarrow B} \otimes^L_{\Dc_B} M, \Dc_A)\otimes
    \omega_A^{-1}\tag{1}
    \\
    & = RHom_{\Dc_B}(M, RHom_{\Dc_A} ( \Dc_{A \leftarrow B}, \Dc_A))\otimes_A \omega_A^{-1} \tag{2}\\
    & = RHom_{\Dc_B}(M,  \Dc_{A\rightarrow B})\otimes_A \omega_A^{-1}  \tag{3}\\
    & = RHom_{\Dc_B}(M, \Dc_B\otimes^L_{\Dc_B}\Dc_{A\rightarrow B} )\otimes_A
    \omega_A^{-1} \tag{4} \\
& = RHom_{\Dc_B}(M, \Dc_B )\otimes^L_{\Dc_B}  \Dc_{A\rightarrow B} \otimes_A
    \omega_A^{-1} \\ &= (\Dbb_B(M)[-d_B]\otimes_B \omega_B) \otimes_{\Dc_B}  \Dc_{A\rightarrow B} \otimes_A
    \omega_A^{-1}\tag{5} \\
&= (\omega_B\otimes_B  \Dc_{A\rightarrow B} \otimes_A
    \omega_A^{-1})\otimes_{\Dc_B} (\Dbb_B(M)[-d_B])  \tag{6} \\ 
&= \Dc_{A\leftarrow B} \otimes_{\Dc_B}\Dbb_B(M)[-d_B]  = \pi_+(\Dbb_B(M)[-d_B]).\tag{7}
  \end{align*}
  \thetag{1,4,5} follow from \Proposition{\ref{flat-module}};
  \thetag{2} is ordinary adjunction in the derived sense; \thetag{3}
  follows from \thetag{*}; \thetag{6} follows from
  \Lemma{\ref{lemma-to-dual}}, noting that
  $\omega_B\otimes_B \Dc_{A\rightarrow B} \otimes_A \omega_A^{-1}$ is
  a $(\Dc_A, \Dc_B)$-bimodule; \thetag{7} follows from
  \thetag{BM} in \Section{\ref{direct-inverse-sec}}. Finally, we note
  that $d_A= d_B$.

  $(\pi_+, \pi^+)$ is an adjoint pair of functors:  Let $D(\Dc_B)$ and $D(\Dc_A)$ be the derived
  categories of bounded complexes of  $\Dc_B$ and $\Dc_A$-modules of finite type, respectively.
Then we have (as detailed below)
  \begin{align*}
    &    \Hom_{D(\Dc_B)}(M, \pi^!(N)) = \Hom_{D(\Dc_B)} (M, \Dc_{A\to B}\otimes_{\Dc_A} N ) \\
    & = \Hom_{D(\Dc_B)} ( M,RHom_{\Dc_A}(\Dc_{B \leftarrow A}, \Dc_A)\otimes ^L_{\Dc_A} N) \tag{1}
    \\
    & = \Hom_{D(\Dc_A)} (M, RHom_{\Dc_A}(\Dc_{B\leftarrow A}, N))\\
    & = \Hom_{D(\Dc_A)} ( \Dc_{B\leftarrow A} \otimes_{\Dc_B}^L M,N) \tag{2}
    \\
    & = \Hom_{D(\Dc_A)} (\Dc_{B\leftarrow A} \otimes_{\Dc_B} M, N )\tag{3}\\
    & = \Hom_{D(\Dc_A)} (\pi_+(M), N).
  \end{align*}
  \thetag{1} follows since $ \Dc_{A\to B}$ is free over $\Dc_A$ and from (*),
  \thetag{2} is ordinary adjunction in the derived sense,  and \thetag{3} follows from \Proposition{\ref{flat-module}}.

  $(\pi^+, \pi_+)$ is an adjoint pair: This follows formally from the
  already proven $\Dbb_Y \pi_+\Dbb_X = \pi_+ $.

  For a proof that $\Dbb_X \pi^!= \pi^!\Dbb_Y $ when $\pi$ is smooth, see
  \citelist{\cite{borel:Dmod}*{VII, Cor. 9.14}
    \cite{hotta-takeuchi-tanisaki}*{Th. 2.1} }
\end{pfof}

\begin{remark}
  Notice that in the proof of $\pi_+ \Dbb = \Dbb \pi_+$ we need
  \Proposition{\ref{dual-iso}} at the step \thetag{7}, in stark contrast
  to the situation for right $\Dc$-modules, where it is not needed;
  see \Section{\ref{d-module-motive}}.
\end{remark}

\subsection{The trace}\label{trace-section}
The trace homomorphism for $A$-modules relative to a finite map
$A\to B$ was discussed in \Section{\ref{relative-can}} for the
$\Dc_A$-module $A$. For general $\Dc_A$-modules $M$ the trace morphism
is
\begin{eqnarray*}\label{trace-morphism}
\Tr : \pi_+ \pi^!(M) &=&\Hom_A(B, \Dc_A)
\otimes_{\Dc_B}B\otimes_AM   \to M, \\    \lambda  \otimes
    b \otimes m &\mapsto &   \lambda(b)\cdot m.
  \end{eqnarray*}
  Since $B$ is locally free over $A$ it follows that $\Tr$ is
  surjective.

The traces in the category of $\Oc_Y$- and $\Dc_Y$-modules are related in a commutative diagram of
$\Dc_Y^\pi$-modules:
  \begin{displaymath}
   \xymatrix{
     \pi_*(\pi^{!'}(M)) \ar[r] \ar[d]^{\Tr}  
     & \pi_+\pi^!(M)\ar[d]^\Tr \\
     M \ar[r]^= & M.
   }
  \end{displaymath}
  where $\pi^{!'}$ is the right adjoint of $\pi_*$ in the category of $\Oc$-modules.

  \begin{proposition}\label{split-conn}
    Let $X \to Y$ be a finite morphism and $M\in \Con (\Dc_Y)$.  The composition 
\begin{displaymath}
  M \to \pi_+ \pi^+ (M) \cong \pi_+\pi^!(M) \to M
\end{displaymath}
is given by  $ m\mapsto \tr_{X/Y}(1) m$. Hence  $M \to \pi_+\pi^+(M)$ is a split homomorphism.
  \end{proposition}
  \begin{proof} The isomorphism follows from
    \Lemma{\ref{con-lemma}}. For the remaining part we can assume that
    $X$ and $Y$ are affine. The map $\psi :M \to \pi_+ \pi^+ (M) $ is
    $m\mapsto \tr_{B/A}\otimes 1\otimes m \in\Hom_A(B,
    \Dc_A)\otimes_{\Dc_B} B \otimes_{A} M$,
    so that composing with $\Tr : \pi_+\pi^!(M) \to M$, we get the map
    $\Tr \circ \psi ( m) =\tr_{B/A}(1)m$.
  \end{proof}

  \subsection{A \texorpdfstring{$\Dc$}\  -module motivation for the
    isomorphism \texorpdfstring{$\eta$}{eta} }\label{d-module-motive}
The existence of the isomorphism (1) in \Proposition{\ref{dual-iso}} seems
coincidental when viewing it from within the category of
$\Oc$-modules. Taking into account the action of the ring of
differential operators we will now construct this isomorphism in a
natural way.

Put $\Dbb_X' (M) = RHom_{\Dc_X}(M, \Dc_X [d_X])$, so that $\Dbb_X(M) =
\Dbb'_X(M)\otimes_{\Oc_X}\omega_X^{-1} $.  This defines functors
\begin{displaymath}
  \Dbb_X' : D^b_{\hol,\lef} (\Dc_X) \to D^b_{\hol, \rig}(\Dc_X), \quad
  \Dbb_X': D^b_{\hol, \rig}(\Dc_X) \to D^b_{\hol,\lef} (\Dc_X),
\end{displaymath}
where $D^b_{\hol,*} (\Dc_X)$ denotes the category of bounded complexes of left
(right) coherent $\Dc_X$-modules with holonomic homology.  Then $\Dbb'_X$ is an
equivalence of categories, $\Dbb'_X \circ \Dbb'_X = \id $, and $\Dbb'_X(\Oc_X) =
\omega_X$. Let
\begin{displaymath}\tag{P}
  \eta' : Hom_{\Dc_Y, \lef}(\pi_+(\Oc_X), \Oc_Y) \cong  Hom_{\Dc_Y,
    \rig}(\Dbb'_Y (\Oc_Y), \Dbb'_Y (\pi_+(\Oc_X))
\end{displaymath}
be the isomorphism that arises from applying  $\Dbb'_Y$. 

The direct image of a right $\Dc_X$-module $N$ is $\pi_+(N)=
\pi_*(N\otimes_{\Dc_X}\Dc_{X\to Y})$. We have then for a left $\Dc_X$-module
\begin{displaymath}\tag{$*$}
  \Dbb'_Y(\pi_+(M)) = \pi_+(\Dbb'_X(M)).
\end{displaymath}
The proof can be read off from the proof that $\Dbb_A\pi_+ =
\pi_+\Dbb_B$ in \Theorem{\ref{adj-triple-th}}, simply be erasing all the occurrences of $\omega_A$
and $\omega_B$, but to see clearly that we do not need the fact that $\omega_{B/A}$ is isomorphic to
$\omega_A^{-1}\otimes_A\omega_B$ we again write down the needed steps and refer to the previous
proof for explanations:
  \begin{align*}
    & \Dbb'_A (\pi_+(M))[-d_A]= RHom_{\Dc_A}( \Dc_{A \leftarrow B} \otimes_{\Dc_B} M, \Dc_A) = RHom_{\Dc_A}( \Dc_{A \leftarrow B} \otimes^L_{\Dc_B} M, \Dc_A)
\\ & = RHom_{\Dc_B}(M, RHom_{\Dc_A} ( \Dc_{A \leftarrow B}, \Dc_A)) = RHom_{\Dc_B}(M,  \Dc_{A\rightarrow B}) \\
    & = RHom_{\Dc_B}(M, \Dc_B\otimes^L_{\Dc_B}\Dc_{A\rightarrow B} ) = RHom_{\Dc_B}(M, \Dc_B
    )\otimes^L_{\Dc_B}  \Dc_{A\rightarrow B}\\ &  =    \Dbb'_B(M)[-d_B]\otimes^L_{\Dc_B}  \Dc_{A\rightarrow B} 
 = \pi_+ (\Dbb'_B(M))[-d_B].
  \end{align*}

  The proof of \Proposition{\ref{dual-iso}} contains two canonical
  homomorphisms, one of left $\Dc(T^\pi_Y)$-modules
  $\tr_\pi: \pi_*(\Oc_X)\to \Oc_Y$ (the trace map) and one of right
  $\Dc(T^\pi_Y)$-modules $\lambda_\pi : \omega_Y \to \pi_*(\omega_X)$
  (the Jacobian map). Here $\tr_\pi$ is a restriction of the trace
  morphism of left $\Dc_Y$-modules
  $\Tr_\pi: \pi_+\pi^!(\Oc_Y)= \pi_+(\Oc_X)\to \Oc_Y$ in the following
  sense
\begin{displaymath}
  \tr_\pi = \Tr_\pi \circ \Theta 
\end{displaymath} 
(see \Proposition{\ref{epinymous}}),  while $\lambda_{\pi}$ can be extended to a
homomorphism of right $\Dc_Y$-modules
\begin{displaymath}
 \hat  \lambda_\pi :  \omega_Y \to \pi_+(\omega_X)= \pi_*(
 \omega_X\otimes_{\Dc_X}\Dc_{X\to Y}),\quad   \nu \mapsto \lambda_\pi(\nu)\otimes 1,
\end{displaymath}
where $\pi_+(\omega_X) $ is the direct image of the right $\Dc_X$-module
$\omega_X$.
\begin{proposition} \label{poin-dual-dual}
We have
\begin{displaymath}
  Hom_{\Dc_Y, \lef}(\pi_+(\Oc_X), \Oc_Y) = k \Tr_\pi, \quad Hom_{\Dc_Y,
    \rig}(\omega_Y, \pi_+(\omega_X))= k \hat \lambda_\pi.
\end{displaymath}
The isomorphism $\eta'$ induces an isomorphism
\begin{displaymath}
\eta':  Hom_{\Dc_Y, \lef}(\pi_+(\Oc_X), \Oc_Y) \cong Hom_{\Dc_Y, \rig}(\omega_Y, \pi_+(\omega_X)),
\end{displaymath}
so that in particular $\eta'(\Tr_\pi) = c \hat \lambda_\pi$, for some
non-zero element $c$ in $k$.  Moreover, $\eta'$ induces the
$\Dc_X(T_Y^\pi)$-linear isomorphism
\begin{displaymath}
  \eta : \omega_{X/Y}\to  \omega_X\otimes_{\Oc_X}\pi^*(\omega_Y^{-1})
\end{displaymath}
as described in \Proposition{\ref{dual-iso}}.
\end{proposition}
\Proposition{\ref{poin-dual-dual}} thus states that the extensions of the canonical
maps $\tr_\pi$ and $\lambda_\pi$ to homomorphisms of $\Dc$-modules are in fact
Poincaré duals of one another, up to a multiplicative constant.
\begin{proof} 
By \Theorem{\ref{adj-triple-th}} 
   \begin{eqnarray*}
     \dim _k\Hom_{\Dc_Y}(\pi_+(\Oc_X), \Oc_Y)&=&\dim _k\Hom_{\Dc_X}(\Oc_X,
     \pi^!(\Oc_Y))\\&=&\dim _k\Hom_{\Dc_X}(\Oc_X,
     \Oc_X) = 1.
  \end{eqnarray*}
  Therefore $\Tr_\pi$ generates the first hom-space; that
  $\hat \lambda_\pi$ generates the second one is proven similarly.
  To see the isomorphism $\eta'$, use \thetag{P},\thetag{$*$}, and the fact that
  $\Dbb'_X(\Oc_X)= \omega_X$ and $\Dbb'_Y(\Oc_Y)= \omega_Y$. That
  $\eta'(\Tr_\pi) = c \hat \lambda_\pi$ follows from what is already proven.

That $\eta'$ restricts to the isomorphism $\eta$ follows  since $\eta$ was
defined by mapping $\tr_\pi$ to $\lambda_\pi$ in the proof of \Proposition{\ref{dual-iso}}.
\end{proof}

\section{Semisimple inverse and direct images }
To begin there is a discussion of descent for $\Dc$-modules, later
applied for the main result in this section, that if $\pi$ is finite,
then $\pi_+$ is a semisimple functor on the category of holonomic
modules, and if one also knows the symmetries of $M$ (the intertia
group) one gets a complete abstract decomposition in terms of
representations of the inertia group. It is discussed when the
semisimplicity of $N$ or $\pi^!(N)$ implies the semisimplicity of the
other, and Clifford's theorem about restrictions of representations of
finite groups is extended to $\Dc$-modules.

The well-known normal basis theorem for Galois field extensions
$L/K/k$, here assuming the characteristic $0$ and that $K/k$ is
transcendent, is complemented by showing that cyclic generators of a
module over the group algebra are the same as cyclic generators over
the ring of differential operators.
\subsection{Descent of $D$-modules}\label{galois-section} 
By a theorem of Grothendieck \cite{SGA1}, if $\pi: X\to Y$ is a finite
surjective morphism of smooth varieties, then the inverse image
functor $N \mapsto \pi^*(N)$ induces an equivalence between the
category $\Mod(\Oc_Y)$ of quasi-coherent $\Oc_Y$-modules and the
category $\Des_Y(\Oc_X)$ of descent data $(M,\phi)$ of quasi-coherent
sheaves $M$ on $X$, where $\phi : p_1^*(M)\cong p_2^*(M)$, and
$p_i :X\times_Y X\to X$ ($i=1,2$) are the two projection maps, and the
isomorphism $\phi$ satisfies a cocycle condition with respect to the
different projections $p_{ij}: X\times_Y X \times_Y X \to X\times_Y X$
(for details, see (\cite{bosch-lutkebohmert-raynaud:neron-models}*{Ch.
  6.1, Th.4 }; it is actually sufficient to require that $\pi$ be a
faithfully flat quasi-compact morphism of schemes when working with
$\Mod(\Oc_Y)$).

Since $\pi$, $p_i$ and $p_{ij}$ are finite maps, the sheaf of liftable
derivations $T^\pi_Y$ maps to $\pi_*(\Dc_X)$,
$(\pi\circ p_i)_*(\Dc_{X\times_Y X })$, and
$(\pi \circ p_{ij}\circ p_i)_*(\Dc_{X\times_Y X \times_Y X})$. Hence
the image generates a subsheaf $T^\pi_X\subset \Dc_X$ (as Lie algebra
and $\Oc_X$-module), and we can consider the subring $\Dc_X(T^\pi_X)$
of $\Dc_X$ that it generates (see also (\ref{diffoperators}));
similarly we have the subrings
$\Dc_{X\times_Y X}(T^\pi_{X\times_YX}) \subset \Dc_{X\times_Y X}$ and
$\Dc_{X\times_Y X\times_Y X}(T^\pi_{X\times_Y X\times_Y X})\subset
\Dc_{X\times_Y X \times_Y X}$. Let $\Des(\Dc(T^\pi_X)) $ be the
subcategory of $\Des_Y(\Oc_X)$ consisting of coherent
$\Dc_X(T^\pi_X)$-modules with descent data $(M,\phi)$, where now
$\phi$ is an isomorphism of $\Dc_X(T^\pi_X)$-modules; notice that the
pull-back $p_i^*(M)$ ($i=1,2$) is naturally a
$\Dc_{X\times_Y X}(T^\pi_{X\times_Y X})$-module when $M$ is a
$\Dc_X(T^\pi_X))$-module, and similarly the pull-backs of $p_i^*(M)$
to $X\times_Y X \times_Y X$ forms
$\Dc_{X\times_Y\times X\times_Y X}(T^\pi_{X\times_Y\times X\times_Y
  X})$-modules.

The inverse
image functor $\pi^!$ defines a functor
\begin{displaymath}
  g^! : \coh(\Dc_Y(T_Y^\pi)) \to \Des_Y(\Dc_X(T^\pi_X)), \quad N \mapsto
  g^!(N)=\pi^!(N),
\end{displaymath}
writing $g^!$ instead of $\pi^!$ since the target category is
different\footnote{Since $X\times_Y X$ and $X\times_Y X \times_Y X$ need
  not be smooth when $\pi$ is ramified, the pull-backs $p^*_i(M)$
  ($i=1,2$) need not form $\Dc$-modules. Assuming $\pi$ is  étale we do
  get $\Dc$-modules, since then $\Dc_X(T^\pi_X)= \Dc_X$.}. The Grothendieck equivalence then implies:
  \begin{lemma} Assume that $\pi$ is a finite morphism of smooth
  varieties. We have an equivalence of categories
\begin{displaymath}
  g^+: \coh(\Dc_Y(T_Y^\pi)) \to \Des_Y (\Dc_X(T_X^\pi))), \quad N\mapsto
  g^+(N)= \pi^!(N).
\end{displaymath}
  \end{lemma}
  The equivalence $g^+$ is compatible with base change, so that in
  particular we get also the equivalence, referring to the diagram
  \thetag{BC},
\begin{displaymath}
  \coh(\Dc_{Y_0}) = \coh(\Dc(T^\pi_{Y_0})) \cong \Des_{Y_0}
  (\Dc_{X_0}(T^\pi_{X_0}))= \Des_{Y_0} (\Dc_{X_0}).
\end{displaymath} Over the generic point $ \coh(\Dc_{K}) \cong \Des_K
(\Dc_{L}) $, where $K$ and $L$ are the fraction fields of $Y$ and $X$,
respectively. We will have great use of this when moreover $L/K$ is
Galois, and the descent category can be recognized as an ordinary
category of modules over a skew group ring. Say that $\pi$ is a {\it Galois
  covering} (a.k.a. trivial torsor) if there exists a finite group $G$ of
$Y$-automorphisms of $X$, defining an action $\cdot : G\times X \to X$
such that the morphism
  \begin{displaymath}
    G\times_k X \to X\times_Y X, \quad (g, x)\mapsto (g\cdot x, x)
  \end{displaymath} is an isomorphism of schemes, where $G$ is
  regarded as a discrete group scheme (see
  \cite{bosch-lutkebohmert-raynaud:neron-models}*{Ch. 6.2, Example
    B}). If $\pi$ is a Galois covering, then $X\to Y$ is étale (so that
  $\Dc(T^\pi_Y)= \Dc_Y$), and conversely, if $B^G$ is the invariant
  ring with respect to a faithful action on $B$, and the associated
  morphism $\pi: \Spec B \to \Spec B^G $ is étale, then $\pi$ is
  a   Galois covering. In particular, $\Spec L \to \Spec K$ is Galois if $K$ is the
  fixed field of a finite group of automorphisms of the field $L$.

  An automorphism $\phi$ of $L$ induces an automorphism $\Dc_L \to \Dc_L$,
  $P \mapsto P^\phi= \phi \circ P \circ \phi ^{-1}$. Let now $G$ be a subgroup of $\Aut(L/k)$. The
  category $\Mod(k[G], \Dc_L)$ of $(k[G], \Dc_L)$-modules consists of $\Dc_L$-modules $M$ which is
  also a $k[G]$-module, such that $g\cdot P m = P^g g\cdot m $ , $m\in M, g\in G$. The (skew) group
  algebra $\Dc_L[G]$ of $G$ with coefficients in $\Dc_L$ consists of functions
  $\sum_{g\in G} P_g g : G \to \Dc_L$, $g\mapsto P_g$ , where the product is
  $\sum_{g_1\in G} P_{g_1} g_1 \cdot \sum_{g_2\in G} Q_{g_2} g_2 = \sum_{g_1, g_2 \in G} (P_{g_1}
  Q^{g_1}_{g_2}) (g_1g_2)$.
  Then $\Mod(k[G], \Dc_L) = \Mod (\Dc_L[G])$, and it is somewhat more convenient to work with
  $\Dc_L[G]$-modules instead of $(k[G],\Dc_L)$-modules.

  The following lemma collects the descent and ascent functors between the categories
  $ \Mod (\Dc_L[G])$ and $\Mod (\Dc_K)$. Denote by $\Dc_L^l$ the ordinary
  $(\Dc_L[G], \Dc_K)$-bimodule structure of $\Dc_L$ and put
  $\Dc_L^r = Hom_{\Dc_L[G]}(\Dc_L^l, \Dc_L[G])$, which is a $(\Dc_K, \- \Dc_L[G])$-bimodule in a
  natural way; we also have $\Dc_L^l = Hom_{\Dc_K}(\Dc_L^r, \Dc_K)$.
\begin{lemma}\label{desc-lemma}
  We have the adjoint triple of functors between the categories $\Mod(\Dc_K)$ and $\Mod(\Dc_L[G])$
  \begin{displaymath}
    (\Dc_L^r\otimes_{\Dc_L[G]}\cdot , \Dc^l_L \otimes_{\Dc_K}\cdot,
    Hom_{\Dc_L[G]}(\Dc^l_L, \cdot)),
  \end{displaymath} where moreover the left and right adjoint functors
  of $\Dc^l_L \otimes_{\Dc_K}\cdot$ are isomorphic, i.e.
  $Hom_{\Dc_L[G]}(\Dc^l_L, \cdot) \cong\Dc_L^r\otimes_{\Dc_L[G]}\cdot $.
  We have $\Dc_L^r\otimes_{\Dc_L[G]} M = M^G$.
\end{lemma}

\begin{proof} Below $M$ and $N$ are left $\Dc_L[G]$- and
  $\Dc_K$-modules, respectively. To begin,
  $(\Dc^l_L\otimes_{\Dc_K}\cdot , Hom_{\Dc_L[G]}(\Dc^l_L, \cdot))$ is
  an adjoint pair:
\begin{displaymath}
  Hom_{\Dc_L}(\Dc^l_L\otimes_{\Dc_K} N, M ) = Hom_{\Dc_K}(N,
  Hom_{\Dc_L[G]}(\Dc^l_L, M)) = Hom_{\Dc_K}(N, M^G).
  \end{displaymath}
 The right $\Dc_K$-module $\Dc_L$
is free (see \Lemma{\ref{basic}}), so that after selecting a basis we can define an isomorphism
of $\Dc_L[G]$-modules which is functorial in $N$
  \begin{displaymath}
    \Dc^l_L\otimes_{\Dc_K} N = Hom_{\Dc_K}(\Dc_L^r,
    \Dc_K)\otimes_{\Dc_K} N \cong Hom_{\Dc_K}(\Dc_L^r, N).
  \end{displaymath} Therefore
\begin{eqnarray*}
  Hom_{\Dc_L[G]}(M, \Dc^l_L\otimes_{\Dc_K}N ) &=& Hom_{\Dc_L[G]}(M,
  Hom_{\Dc_K}(\Dc_L^r, N)) \\ &=&
  Hom_{\Dc_K}(\Dc_L^r\otimes_{\Dc_L[G]} M,N),
\end{eqnarray*} showing that $(\Dc_L^r\otimes_{\Dc_L[G]}\cdot , \Dc_L
\otimes_{\Dc_K}\cdot)$ is an adjoint pair. It remains to see that the
left- and right-adjoints of $\Dc^l _L \otimes_{\Dc_K}\cdot$ are
isomorphic functors (this would be immediate if we know already that
$\Dc_L^l \otimes_{\Dc_K}\cdot$ defines an equivalence of categories).
There is a homomorphism of $\Dc_K$-modules which is functorial in $M$
\begin{align*}
  \Dc_L^r \otimes_{\Dc_L[G]}M &= Hom_{\Dc_L[G]}(\Dc_L^l, \Dc_L[G])
  \otimes_{\Dc_L[G]}M \to Hom_{\Dc_L[G]}(\Dc^l_L, M),\\
\phi \otimes m &\mapsto (P \mapsto \phi (P)m).
\end{align*} 
Since both functors in $M$ (on either side of the arrow) are
additive and exact, to see that it is an isomorphism it suffices to
check the assertion when $M= \Dc_L[G]$, in which case the left hand
side becomes $\Dc_L$ considered as left $\Dc_K$-module (by
multiplication on $\Dc_L$ from the left). For the right hand side we
have isomorphisms of left $\Dc_K$-modules, noticing that the action of
$\Dc_K$ commutes with the action of $G$,
\begin{eqnarray*}
 && Hom_{\Dc_L[G]}(\Dc^l_L, \Dc_L[G]) = (\Dc_L[G])^G = (\Dc_L
  \otimes_K K[G])^G\\ &=& (Hom_K(K, \Dc_L \otimes_K K[G]))^G = Hom_K
  (K[G]^*, \Dc_L)^G \\ & =& Hom_{K[G]}(K[G], \Dc_L) = \Dc_L.
\end{eqnarray*}
We have used here the fact the group ring $K[G]$ is self-dual,
$ Hom_K(K[G], K) \cong K[G]$ as $K[G]$-modules.
\end{proof}

\begin{remark}
  The functors in \Lemma{\ref{desc-lemma}} in fact define equivalences
  of categories, as follows from a suitable Morita theorem (see
  \cite{montgomery:fixed}*{Th. 2.5, Cor 2.6}). Since $\Dc_X$ is simple,
  the Morita theorem gives an equivalence
  $\Mod (\Dc_X[G])\cong \Mod (\Dc_Y^\pi)$ when $Y= X^G$. However, I
  believe it is more conceptual to arrive at the equivalence
  $\Mod (\Dc_L[G])\cong \Mod (\Dc_K)$ below by appealing to descent.
\end{remark}
  \begin{proposition}(Galois descent)\label{morita} Let $\pi : X \to
  Y$ be a
    Galois covering with Galois group $G$. Then
\begin{displaymath}
  \Des_Y (\Dc_X) \cong \Mod(\Dc_X[G]),
\end{displaymath} and we have an equivalence of categories
\begin{displaymath}
  g^+ : \Mod(\Dc_Y) \to \Mod(\Dc_X[G]), \quad N \mapsto g^+(N)=
  \pi^!(N).
\end{displaymath} The quasi-inverse is
\begin{displaymath} g_+: \Mod (\Dc_X[G])\to \Mod (\Dc_Y), \quad
M\mapsto \pi_*(\Dc^r_X\otimes_{\Dc_X[G]}M) = \pi_*(M^G),
\end{displaymath} where $\Dc_X$ is regarded as a $(\Dc_Y,\Dc_X[G])
$-bimodule. In particular, the functor
\begin{displaymath}
  g^+:\Mod (\Dc_{K}) \to \Mod(\Dc_L[G]), \quad N \mapsto g^+(N)=
  \pi^!(N),
\end{displaymath} defines an equivalence of categories, with
quasi-inverse
\begin{displaymath} g_+: \Mod (\Dc_L[G])\to \Mod (\Dc_K), \quad
M\mapsto \Dc^r_L\otimes_{\Dc_L[G]}M = M^G.
\end{displaymath}
\end{proposition}
\begin{proof}
  This is certainly well-known, but I have been unable to find a
  precise reference, and at the same time it would require much space
  to write down a detailed proof. Instead one can look at
  \cite{bosch-lutkebohmert-raynaud:neron-models}*{Ch. 6.2, Example B},
  where it is is proven that a $G$-action on an $X$-scheme $Z$ is
  equivalent to a descent datum on $Z$. In that proof (where in the
  notation of [loc cit] $S=Y, S'=X$ and $X=Z$) one can replace the
  $X$-scheme $Z$ everywhere by, on the one hand, a coherent
  $\Dc_X$-module with a $G$-action, which is the same as a
  $\Dc_X[G]$-module, and on the other hand it can be replaced by a
  coherent $\Dc_X$-module provided with étale descent datum, and all
  arrows between $X$-schemes are replaced by homomorphisms of
  $\Dc_X$-modules.
\end{proof}
\begin{remark} Assume that $X\to Y$ is a Galois covering and $M$ be a
$\Dc_X[G]$-module. We have then a canonical homomorphism of
$\Dc_Y$-modules
\begin{displaymath}
  g_+(M) \cong \pi_* (Hom_{\Dc_X[G]}(\Dc_X^l, M) ) \to \pi_* (M) \cong
  \pi_+(M), \quad \phi \mapsto \phi(1).
\end{displaymath} More concretely this can be identified with  the inclusion $\pi_*(M^G)
\subset \pi_*( M)$.
\end{remark}

If $X/Y$ is Galois then the category $ \Des_Y(\Dc_X)$ is fibred over
$\Mod(\Dc_X)$, i.e. we have the functor
\begin{displaymath} e_+: \Des_Y(\Dc_X) \cong \Mod(\Dc_X[G]) \to
\Mod(\Dc_X),
\end{displaymath} where $e_+(M)$ is the $\Dc_X$-module that remains
after forgetting the $G$-action. In this context Maschke's theorem takes the following
form:
\begin{proposition}\label{maschke} Assume that $L/K$ is Galois with
Galois group $G$.
  \begin{enumerate}
  \item Consider the functor $e_+: \Mod(\Dc_L[G])\to \Mod (\Dc_L)$,
  where $ e_+ (M) $ is $M$
    regarded as $\Dc_L$-module only. Then $e_+(M)$ is semisimple if
    and only if $M$ is semisimple.
  \item Let $\psi : H \to G$ be a homomorphism of finite groups (so
  that $L$ is
    a $\Dc_L[H]$-module using $\psi$). Let $M$ be a $\Dc_L[H]$-module
    such that $M$ is semisimple as $\Dc_L$-module. Then
  \begin{displaymath}
    \Dc_L[G]\otimes_{\Dc_L[H]}M
  \end{displaymath} is a semisimple $\Dc_L[G]$-module.
  \end{enumerate}
\end{proposition}
\begin{proof}
  (1): The proof that $M$ is semisimple when $e_+(M)$ is semisimple is
  proven exactly in the same way as the usual Maschke theorem for
  modules over group rings, by averaging over the finite group $G$.
  Conversely, if $M_1$ is a maximal proper submodule of $e_+(M)$, then
  $gM_1$ is another maximal proper submodule, implying that the
  radical of $e_+(M)$ is a $\Dc_L[G]$-submodule of $M$, hence it is
  $0$ if $M$ is semisimple; therefore $e_+(M)$ is semisimple.

  (2): We have
  \begin{displaymath}
     e_+( \Dc_L[G]\otimes_{\Dc_L[H]}M) = \bigoplus_{g_i \in G/\psi(H)}
     g_i\otimes M
  \end{displaymath} where the notation on the right side is explained
  above \Definition{\ref{inertia-def}}. Since $M$ is semisimple it
  follows that each $ g_i\otimes M$ is semisimple, hence (2) follows
  from (1).
\end{proof}
The duality $\Dbb_L$ on $\Mod_{fd} (\Dc_L)$ restricts to a duality on
the subcategory $\Mod_{fd} (\Dc_L[G])$ of $ \Mod_{fd} (\Dc_L)$. The
descent functors respect this duality.
\begin{lemma}
  \begin{displaymath}
    \Dbb_L g^+ = g^+ \Dbb_K, \quad \text{and}\quad \Dbb_K g_+= g_+
    \Dbb_L .
  \end{displaymath}
\end{lemma}
\begin{proof}
  If $M\in \Mod_{fd} (\Dc_L[G]) $ there exists a $G$-equivariant
  resolution of $M$,
\begin{displaymath}
  \Dbb_K (M^G) = Ext^n_{\Dc_K}(M^G, \Dc_K ) = Ext^n_{\Dc_L}(M, \Dc_L
  )^G = (\Dbb_L (M))^G.
\end{displaymath} Therefore, since $\Dbb_K \pi_+ = \pi_+ \Dbb_L$, we
have $\Dbb_K g_+ = g_+ \Dbb_L$. Moreover, as $g^+ = \pi^+ $ and
$\Dbb_L \pi^+ = \pi^+ \Dbb_K$, it follows that $\Dbb_L g^+ = g^+
\Dbb_K$.
\end{proof}

\subsection{Twisted modules and the inertia group}\label{inertia-section}
If $L/K$ is Galois and $M$ is a simple $\Dc_L$-modules of finite
dimension over $L$, we will decompose $\pi_+(M)$ using translation
symmetries of $M$ with respect to elements in the Galois group. An
automorphism $\phi$ of the field $L$ gives rise to a new
$\Dc_L$-module $M_\phi$, where $M= M_\phi$ as sets, but the action of
$\Dc_L$ is twisted as
\begin{displaymath}
  P\cdot m = P^{\phi^{-1}}m, \quad P\in \Dc_L, \quad m\in M,
\end{displaymath} where on the right the $P^{\phi^{-1}}$-action is
already known. Notice that $M_\phi = M$ as $\Dc_K$-module if $K$ is
fixed by $\phi$. Let $G$ be a subgroup of the Galois group $\Aut_K(L)$
of $L$ over $K$, and consider the $\Dc_L[G]$-module
$\Dc_L[G]\otimes_{\Dc_L} M$. Then as $\Dc_L$-module we have
  \begin{displaymath}
    \Dc_L[G]\otimes_{\Dc_L} M = \bigoplus_{g\in G} M_g,
  \end{displaymath} where we have identified the $\Dc_L$-submodule $g
  \otimes M$ with the twisted module $ M_g$. If $M$ is simple then
  each twisted module $M_g$ is also simple, so we can conclude that  $
\Dc_L[G]\otimes_{\Dc_L} M$ is semisimple, by \Proposition{\ref{maschke}}.

\begin{definition}\label{inertia-def}
  Let $L/K$ be a finite Galois field extension and $M$ be a simple
  $\Dc_L$-module. The {\it
    inertia group} of $M$ over $K$ is the following subgroup of the
    Galois group
    \begin{displaymath}
    G_M= \{g \in \Aut_K (L) \ \vert \ M_g \cong M\}.
  \end{displaymath}
\end{definition}

 The inertia group is important  partly   because of the following
 proposition. We will later see more precisely how  it controls the
 decomposition of $\pi_+(M)$ \Th{\ref{galois-direct}}.  
\begin{proposition}\label{simpledirect} Let $\pi: \Spec L \to \Spec L$
be the morphisms of schemes associated to a finite
  Galois extension $L/K$. If $M$ is a simple $\Dc_L$-module with
  trivial inertia group $G_M= \{e\}$, then $\pi_+(M)$ is a simple
  $\Dc_K$-module.
\end{proposition}
The proof will be presented a little later. If $L/K$ is not Galois it
does not follow that $\pi_+(M)$ is simple. For example, if $L\neq K$
and $G_L =\Aut (L/K)= \{e\},$ then the module $\pi_+(L)$ always
contains the proper submodule $K$.

To give a concrete application we consider the effect of going to
invariant differential operators on natural rank 1 $\Dc$-modules that
can be associated to arrangements of hyperplanes in $\Cb^n$. Put
$B= \Cb[x_1, \ldots , x_n]$ and
$M_\alpha^\beta= \Dc_B\alpha^\beta\subset
B_{\alpha}\otimes_BM^\beta_\alpha$ ,
$\alpha^\beta= \prod_{i\in \Hc} \alpha_i^{\beta_i}$, where the
$\alpha_i$ are linear forms in $\Cb^n$ and $\beta_i \in \Cb$. In
general $M_\alpha^\beta$ is not simple and I do not know a precise
condition on $\beta$ that makes $M_\alpha^\beta$ it so, but if $N$ is
high enough and $n= (n_i)\in \Nb^n$ satisfies $n_i \geq N$, then
$M_\alpha^{\beta + n}$ is simple.\footnote{Conditions that the
  localized $\Dc_B$-module $B_{\alpha} \otimes_{B} M_\alpha^\beta$ be
  simple can be found in \cite{budur-liu-sau-bot:coh-supp-loc}.}

Let $G\subset \Glo_k (V)$ be a subgroup that is generated by
pseudoreflections of the finite-dimensional vector space $V$, and
$B=\So(V)$ the symmetric algebra. Let $\pi : A= B^G \to B $ be the
inclusion of the invariant ring, where $A$ again is a polynomial ring
by the Chevalley-Shephard-Todd theorem. The action of an element $g$
in $G$ takes a linear form $\alpha_i$ to another linear form
$g\cdot \alpha_i$. Let $ \Ac = (\alpha_1, \cdots, \alpha_r)$ be a set
of linear forms that is preserved by $G$ and put
$\alpha^\beta = \prod_{\alpha_i \in \Ac } \alpha_i^{\beta_i}$. Then
$g(\alpha^\beta) = (g\alpha)^\beta = \alpha^{\gamma}$, so that putting
$g \cdot \beta = \gamma $ one can regard $G$ as acting on the
exponents $\beta$ of the linear forms. Then
\begin{displaymath}
 g\otimes M_{\alpha}^\beta \cong  M_{g\cdot \alpha }^\beta \cong M_\alpha^{g\beta}
\end{displaymath}
and $M_\alpha^{g\beta} \cong M_\alpha^\beta$ if and only of
$\beta - g\cdot \beta  \equiv 0 \ (\omod \Zb^n )$. By
\Proposition{\ref{simpledirect}} we can conclude the following result.

\begin{proposition}
  Assume that $M_\alpha^\beta$ is simple. Then the following
  assertions are
  equivalent:
  \begin{enumerate}
  \item $\pi_+(M_\alpha^\beta)$ is simple.
    \item   $\beta - g\cdot \beta \not  \equiv 0 \  (\omod  \Zb^n )$, $g\in
      G$, $g\neq e$.
  \end{enumerate}
\end{proposition}
\begin{pfof}{\Proposition{\ref{simpledirect}}}
  Since $\pi_+(M) = \Dc_L \otimes_{\Dc_L[G]}\Dc_L[G]\otimes_{\Dc_L}M$
  it suffices to prove that $\Dc_L[G]\otimes_{\Dc_L}M$ is a simple
  $\Dc_L[G]$-module, where $G$ is the Galois group of $L/K$
  \Prop{\ref{morita}}. Since $G_M$ is trivial, the simple module
  $1\otimes M$ has multiplicity $1$ in $\Dc_L[G]\otimes_{\Dc_L} M$.
  Let $p: \Dc_L[G]\otimes_{\Dc_L} M \to M = 1\otimes M $ be the
  $\Dc_L$-linear projection on $1\otimes M$. By
  \Proposition{\ref{maschke}} $\Dc_L[G]\otimes_{\Dc_L} M$ is
  semisimple, hence it contains a simple $\Dc_L[G]$-submodule $N$ such
  that $p(N)\neq 0 $. Since $N\subset \Dc_L[G]\otimes_{\Dc_L} M $ is
  semisimple as $\Dc_L$-module and $1\otimes M$ has multiplicity $1$,
  it follows that $1\otimes M \subset N$. Therefore,
  $\Dc_L[G]\otimes_{\Dc_L}M = \Dc_L[G] (1\otimes M) \subset N $, which
  implies the assertion.
\end{pfof}
In general the inertia group $G_M$ of a simple $\Dc_L$-module $M$
\Defn{\ref{inertia-def}} only has a projective action on $M$, so that
we will have to replace $G_M$ by a certain central extension
$\bar G_M$ to make $M$ inte a $\bar G_M$-module. The construction is
as follows. For each element $g $ in $ G_M$ there exists an
isomorphism of $\Dc_L$-modules $\phi(g) : M \to M_g $, and, since $k$
is algebraically closed, $\phi (g) $ is determined up to a factor in
$k^*$. This gives rise to a cocycle
\begin{displaymath} f: G\times G \to k^* , \quad f (g_1, g_2)=
\phi(g_1) \phi(g_1 g_2)^{-1}\phi(g_2),
  \end{displaymath} where $f (g_1, g_2)$ is a $\Dc_L$-linear
  automorphism of the simple module $M$ and is hence determined by an
  element in $k^*$. Putting $n= |G_M|$ it follows that $f^n$ is a
  coboundary, so that the cocycle $f$ determines an element in the
  cohomology group $ H^2(G_M, \mu_n)$, where $\mu_n\subset k^*$ is the
  subgroup of $n$th roots of unity. Also, $f$ determines a central
  extension
  \begin{equation}\label{central-ext} 1 \to \mu_n\to \bar G_M
\xrightarrow{\psi} G_M \to 1
\end{equation} of $G_M$ by $\mu_n$ such that
  $M$ is a module over $\Dc_L[\bar G_M]$. One may call the pair $(\bar
  G_M, \psi)$ the {\it true
    inertia group} of $M$.

  If a $\Dc_L$-module $M$ is not a $\Dc_L[G_M]$-module it will not
  descend to a $\Dc_{L_1}$-module, for $L_1= L^{G_M}$. But there
  exists a Galois extensions $ L' /L/K$ with Galois groups
  $H= \Aut (L'/L)$ and $G' = \Aut (L'/K)$, where $\bar G_M\subset G'$.
  Then if $p : \Spec L' \to \Spec L$ is the map of $L'/L$, it follows
  that $p^!(M)$ is a $\Dc_{L'}[\bar G_M]$-module. Conversely, if $M$
  is a $\Dc_{L'}[\bar G_M]$-module where $\bar G_M \subset G$, then
  the invariant space $M^{H}$ is a $\Dc_L[\bar G_M]$-module, giving
  rise to projective $\Dc_L[ G_M]$-module.

  The map $\psi$ induces a homomorphism of rings
  \begin{displaymath}
    \Dc_L[\bar G_M]\to \Dc_L[G],
  \end{displaymath} and using this homomorphism, $\Dc_L[G]$ becomes a
  right $ \Dc_L[\bar G_M]$-module. 
  \begin{example}
    Let $L/K$ be a finite Galois extension with group $G$ and
    $\Delta $ be a $G$-semiinvariant in $L$, by which we intend that
    $g \cdot \Delta = \lambda(g) \Delta$, where $\lambda $ is a
    character $ G \to k^*$. We put $\mu = \Delta^{1/m}$ and define the
    $\Dc_L$-module $M= \Dc_L \mu$, where $\mu$ can be regarded as an
    abstract generator for $M$ such that
    $\partial \cdot \mu = \frac {\partial(\Delta)}{m \Delta} \mu$,
    $\partial \in T_L$, but also as an element in the field extension
    $L[\Delta^{1/m}]$. Make a choice of function $G \to k^*$,
    $ g\mapsto \lambda(g)^{1/m} \in k^*$, such that
    $ (\lambda(g)^{1/m})^m = \lambda(g)$. Then we get an isomorphism
    of $\Dc_L$-modules
    \begin{displaymath} \phi (g): M \cong M_g, \quad l \mu \mapsto
      \lambda(g)^{1/m} l^g \mu, \quad l\in L,
\end{displaymath} 
and a cocycle $c: G \times G \to \mu_n$,
$c(g_1, g_2) = \lambda(g_1^{-1})^{1/m} \lambda(g_2^{-1})^{1/m}
\lambda(g_1 g_2)^{1/m} $.
It defines a central extension $\mu_m \to \bar G \to G$ and $M$ is a
$\Dc_L[\bar G]$-module. For example, let $\Delta$ the alternating
polynomial in $k[x_1, \dots , x_n]$, $G= S_n$ be the symmetric group,
and $L$ be the fraction field of the polynomial ring
$k[x_1, \dots , x_n]$. Since $S_n$ only has one non-trivial character,
it follows that the above constructed map $\bar S_n \to S_n$ is
non-split, and that the $\Dc_L[\bar S_n]$- module $L (\Delta)^{1/m}$
is not a $\Dc_L[S_n]$-module.
  \end{example}
  Below we consider $\Dc_L[G]$ as a
  $(\Dc_L, \Dc_L[\bar G_M] )$-bimodule, where the left action of
  $\Dc_L$ is the natural multiplication map.

  \begin{proposition}\label{cyclic-extension} Assume that $\pi: \Spec L \to \Spec K$ is Galois
    with cyclic Galois group $G$. Let $M$ be a simple $\Dc_L$-module such
    that $G_M= G$, then $M=\pi^!(N)$ for some simple $\Dc_K$-module
    $N$.
  \end{proposition}
  \begin{remark} By the equivalence in \Theorem{\ref{equivalence}},
    \Proposition{\ref{cyclic-extension}} is an extension of
    \cite{isaacs:character}*{ Th. 11.22}.
  \end{remark}
\begin{proof}
  Since the cohomology group $H^2(G, k^*)=0$ when $G$ is cyclic it
  follows that $M$ is a $\Dc_L[G]$-module. Then the assertion follows
  from \Proposition{\ref{morita}}.
  \end{proof}

\subsection{Decomposition of inverse images}
We know that $j^+(N)= j^!(N)$ is semisimple if $N$ is semisimple and
$j$ is an open immersion. It has more generally been conjectured by M.
Kashiwara \cite{kashiwara:semisimple} that $\pi^!(N)$ be semisimple
when $\pi$ is non-characteristic to $M$ (this notion is explained in
\cite{hotta-takeuchi-tanisaki}*{Sec. 2.4}). Using only algebraic
methods the result below includes a weaker version of this conjecture.
When $\pi$ is Galois \Theorem{\ref{cliffordtheorem}} is a more precise
description of the inverse image in terms of the inertia group.

\begin{theorem}\label{semisimple-inv}
  Let $\pi: X \to Y$ be a surjective morphism of smooth varieties over
  a field $k$ of characteristic 0, and let $N$ be a holonomic
  $\Dc_Y$-module.
  \begin{enumerate}
  \item Assume that $\pi$ is smooth. Then $ \pi^!(N)$ is semisimple if
  and only if $N$ is
    semisimple, and if $ \pi^!(N)$ is simple, then $N$ is simple.
\item If $N$ is a semisimple connection, then $\pi^!(N)$ is a
semisimple connection.
  \item Assume $\pi$ is finite. If $\pi^!(N)$ is semisimple, then $N$
  is semisimple, and if $
  \pi^!(N)$ is simple, then $N$ is simple.
  \item Assume that $N$ is
    semisimple and $N_y$ is of finite type over $\Oc_{Y,y}$ for all
    points $y$ of height $\leq 1$ such that the closure of $y$
    intersects the discriminant $D_\pi$, i.e. $\{y\}^- \cap D_\pi\neq
    \emptyset$. Then $\pi^!(N)$ is semisimple.
  \end{enumerate}
\end{theorem} 
\begin{remark}\label{levelt-put-hoeij}
  \begin{enumerate}
  \item The assertion in \Theorem{\ref{semisimple-inv}}(2) will be
    extended in \Theorem{\ref{inv-conn}} to arbitrary morphisms of
    smooth varieties for a certain class of connections.
  \item Notice that if $\pi$ arises from a finite field extension
    $L/K$, then both (1) and (3) applies. If $\Dc_K$ is replaced by
    the subring $K[\partial]$ that is generated by a single derivation
    $\partial$ of $K/k$, then the assertion that $L\otimes_K N$ is
    semisimple if and only of $N$ is semisimple also follows from
    \citelist{\cite{levelt}*{\S 1, Proposition}
      \cite{put-hoeij:descent-skew}*{Prop. 2.7} }.
  \item In general, $\pi^!(N)$ need not be semisimple when $\pi$ is
    finite and surjective, and $N$ is a simple holonomic module. For
    example, $N$ can coincide with its maximal extension
    $N= i_+i^+(N)$, in the notation of the diagram \thetag{BC},
    implying that $\pi^!(N)= j_+j^+(\pi^!(N)) = j_+ \pi_0^!i^+(N)$,
    while $j_{!+}(\pi_0^!i^+(N))\neq j_+ \pi_0^!i^+(N)$. More
    concretely, let $\pi: A\to B$ be a finite map of smooth local
    $k$-algebras, $x\in \mf_B\setminus \mf_B^2$ and $F\in A[t]$ be
    such that $F(x)=0$ and $F\neq 0$; let $F^\partial$ be the action
    of a derivation $\partial \in T_{A/k}$ on the coefficients of $F$.
    Then $\partial (x) = - F^\partial(x)/F'(x)$ and we get the
    $\Dc_A$-module $N= \Dc_A x$. I am unaware of an algorithm, given
    $F$, to decide whether $N$ is simple, but if $F= t^n - y$ for some
    integer $n>1$, then we write $N= \Dc_A y^{1/n}$, which indeed is a
    simple module. However, $\pi^!(N) = B_{(x)}$, which is not
    semisimple.
  \end{enumerate}
    \end{remark}
\begin{pfof}{ \Theorem{\ref{semisimple-inv}}}
  For the proof of (1) and (2) we factorize $\pi$ by a finite map $f:
  X \to Z =Y \times_k X$ and a projection $p: Z \to Y$. Since $\pi^! =
  f^! \circ p^!$ it follows that it suffices to prove the assertions
  when $\pi$ is either finite or a projection, and since $p^!(N)$ is
  simple if and only if $N$ is simple it suffices to consider finite
  morphisms.

  (1): By the above discussion we can assume that $\pi$ is étale, and
  we have to prove that $\pi^!(N)$ is semisimple when $N$ is simple.
  Then $N$ is the minimal extension of a simple $\Dc_{Y, \xi}$-module
  $N_\xi$ for some point $\xi$ in $Y$, i.e., in the notation of
  \Theorem{\ref{minimalext}}, $N= N(\xi)$. Since $\pi$ is étale it
  follows that $\pi^!(N)$ is a minimal extension,
  \begin{displaymath} \pi^!(N) = \bigoplus_{x\in \pi^{-1}(\xi)}
\pi^!(N)(x).
  \end{displaymath} It suffices therefore to prove that the
  $\Dc_{X,x}$-module $M= \pi^!(N)_x $ is semisimple when $x\in
  \pi^{-1}(\xi)$. Let $L_1= k_{X,x}$ and $K= k_{Y,\xi}$ be the residue
  fields at $x$ and $\xi$ respectively. Since $\supp M =
  \{\mf_{X,x}\}$, by Kashiwara's equivalence it suffices to prove that
  $M_0= M^{\mf_{X,x}}$ is a semisimple $\Dc_{L_1}$-module. We have
  $M_0 = L_1\otimes_{K} N_0 $, where $N_0= N^{\mf_{Y,\xi}} $; since
  $N_\xi$ is simple, it follows, again by Kashiwara's equivalence,
  that $N_0$ is a simple $\Dc_K$-module. Select a field extension
  $L/L_1/K$ such that $L/K$ is Galois, with Galois group $G$; hence
  $L/L_1$ is also Galois and we let $H$ denote its Galois group. In
  the notation of \Section{\ref{galois-section}}
  \begin{displaymath}
    L\otimes_{K} N_0 =\Dc_L\otimes_{\Dc_K} N_0 = e_+(\Dc^l_L
    \otimes_{\Dc_{K}} N_0),
  \end{displaymath} and by \Proposition{\ref{morita}} $\Dc^l_L
  \otimes_{\Dc_{K}} N_0$ is a semisimple $\Dc_L[G]$-module, hence by
  \Proposition{\ref{maschke}, (1), } $L\otimes_{K} N_0$ is semisimple.
  Again since $ \Dc^l_L \otimes_{\Dc_{K} } N_0$ is a semisimple
  $\Dc_L[G]$-module, hence a semisimple
  $\Dc_L[H]$-module\footnote{This is proven in the same way as in the
  proof that $e_+$ is
    a semisimple functor in \Proposition{\ref{maschke}(1)} }, and
\begin{align*}
  M_0 &= \Dc_{L_1} \otimes_{\Dc_{K}} N_0 =
  \Dc^r_L\otimes_{\Dc_L[H]}\Dc^l_L\otimes_{\Dc_{L_1}} \Dc_{L_1}
  \otimes_{\Dc_{K}} N_0 \\
&= \Dc^r_L \otimes_{\Dc_L[H]} \Dc^l_L \otimes_{\Dc_{K}
  } N_0,
\end{align*} \Proposition{\ref{morita}} implies that $M_0$ is
semisimple. Conversely, if $M_0= \Dc_{L_1}\otimes_{\Dc_K}N_0$ is
semisimple, then $\Dc^l_L\otimes_{\Dc_{L_1}} \Dc_{L_1}
\otimes_{\Dc_{K}} N_0$ is a $\Dc_L[G]$-module which is semimple over
$\Dc_L[H]$, and therefore it is semisimple as $\Dc_L[G]$-module too
\Prop{\ref{maschke}}, hence by \Proposition{\ref{morita}} $N_0$ is
semisimple.

Since a $\Dc_L[G]$-module is simple if its restriction to a
$\Dc_L[H]$-module is simple, it follows from the above argument that
$N$ is simple when $\pi^!(N)$ is simple.

(2): By generic smoothness of $\pi$, (1) implies that $\pi^!(N)$ is
generically semisimple. Hence $\pi^!(N)$ is a connection that extends
a semisimple $\Dc_{X_0}$-module on some open subset $X_0$ of $X$, it
follows that $\pi^!(N)$ is semisimple \Prop{\ref{simple-coh}}.

(3): The remaining part to prove, that $N$ is (semi-)simple when
$\pi^!(N)$ is (semi-) simple is postponed until after
\Theorem{\ref{decomposition-thm}}.
  
  (4): Assume that $N$ is simple. Since $X_0/Y_0$ is étale, by (1) it
  follows that $ j^!(\pi^!(N)) = \pi_0^!i^!(N)$ is semisimple when $N$
  is semisimple. Since $N_y$ is of finite type whenever $y$ is a point
  of height $\leq 1$ whose closure intersects $D_\pi$, it follows that
  $\pi^!(N)_x$ is of finite type over $\Oc_{X,x}$ when $x$ is a point
  of height $\leq 1$ whose closure intersects $B_\pi$.
  \Proposition{\ref{simple-coh}} now implies that $\pi^!(N)$ is
  semisimple.
\end{pfof}

As explained in (3), the proof that $N$ is semisimple when $\pi^!(N)$
is semisimple and $\pi$ is finite is postponed to (\ref{contproof}),
since we first need to see that $\pi_+$ is a semisimple functor. To
prove this, in turn, we need the following special case of
\Theorem{\ref{semisimple-inv}}, (3).

\begin{lemma}\label{semisimple-inv-field}
  Make one of the assumptions:
  \begin{enumerate}
  \item $N$ is coherent over $\Oc_X$.
\item $\pi$ is a smooth map.
  \end{enumerate} If $\pi^!(N)$ is semisimple, then $N$  is semisimple.
\end{lemma}

\begin{proof} The case (2) is already covered by
  \Theorem{\ref{semisimple-inv}}. Let $0 \to N_1 \xrightarrow{i} N$ be
  an exact sequence. We then get a diagram (as detailed below)
  \begin{displaymath}
    \xymatrix{
      \pi_+\pi^!(N_1) \ar[r]^\psi \ar@<.5ex>[d]^{\Tr} & \pi_+\pi^+(N)
      \\ N_1 \ar[r]^i \ar@<.5ex>[u]^v & N \ar@<-.5ex>[u]_u.
    }
  \end{displaymath} The assumptions (1) or (2) imply that
  $\pi^!(N)\cong \pi^+(N)$ (this follows as soon as $\pi$ is
  non-characteristic to $N$), so that there exists an isomorphism $f:
  \pi_+\pi^!(N) \cong \pi_+\pi^+(N_1)$ and we therefore get the map
  $\psi= f\circ \pi_+ \pi^!(i)$. The functor $\pi^!$ is exact since
  $\pi$ is flat and $\pi_+$ is exact due to
  \Proposition{\ref{flat-module}}; hence $\psi$ is injective. By the
  already proven part of \Theorem{\ref{semisimple-inv}}, (1-2), the
  module $\pi^!(N)= \pi^+(N)$ is semisimple, hence there exists a map
  $\phi: \pi_+ \pi^+(N)\to\pi_+ \pi^!(N_1) $ such that $\phi \circ
  \psi $ is the identity. Put $j= \frac 1n \Tr \circ \phi \circ u: N
  \to N_1$, where $n$ is the degree of the map $\pi$. Since $\psi
  \circ v = u \circ i$, we have
  \begin{displaymath}
    j \circ i = \frac 1n \Tr \circ \phi \circ u \circ i =\frac 1n \Tr
    \circ \phi \circ \psi \circ v =\frac 1n \Tr \circ v = \id_{N_1}.
  \end{displaymath} Therefore $N$ is semisimple.

\end{proof}

We can give a more detailed description of the inverse image in terms
of the inertia group when the map is Galois. In the formulation we
refer to the diagram \thetag{BC} in (\ref{sec:2.1}).
\begin{theorem}\label{cliffordtheorem}
  Let $\pi: X\to Y$ be a finite morphism of smooth varieties such that
  the extension of fraction fields is Galois with Galois group $G$.
  Let $N$ be a simple $\Dc_Y$-module such that $i^!(N)\neq 0$. Let $M$
  be a simple submodule of $\pi^!(N)$, $G_M$ be the inertia subgroup
  of $M$ in $G$, and put $t = [G: G_M]$. Then
  \begin{displaymath}
   j_{!+} j^!(\pi^!(N)) = \bigoplus^t_{i=1} (g_i\otimes M)^e,
 \end{displaymath} for some integer  $e$, where $g_i$ are
 representatives of the cosets $G/G_M$. If
 $N$ is a connection, or $\pi$ is smooth, then one can erase $j_{!+}j^!$
 on the left. Moreover:

 \begin{enumerate}
 \item The integer $e$ divides both the order $|G_M|$ of the inertia group
   and the degree of $\pi$.
 \item $\frac {\rko (N)}{\rko (M)}$ divides the degree of $\pi$.
 \end{enumerate}

\end{theorem}

We indicate how \Theorem{\ref{cliffordtheorem}} relates to Clifford's
theorem in the representation theory of finite groups. So assume for a
moment that there exists a field extension $L_1/L/K$ such that $L_1/K$
is Galois and $L_1\otimes_K N= L_1^e$ for some integer $e$ (as
$\Dc_{L_1}$-module); in (\ref{finite-groups}) it is explained how such
$\Dc_K$-modules $N$ are closely related to representations of finite
groups. Then $L_1/L$ is also Galois and we let $G_1$ and $H$ be the
Galois groups of $L_1/K$ and $L_1/L$, so that $G= G_1/H$ is the Galois
group of $L/K$. The decomposition of $N$ in
\Theorem{\ref{cliffordtheorem}} now follows from
\Theorem{\ref{equivalence}} and \Proposition{\ref{ind-direct}} (below)
by applying Clifford's theorem to a normal subgroup (see
\cite{curtis-reiner}*{\S 49}).

\begin{proof} Consider first the restriction of $\pi$ to a map
  $\pi_0 : \Spec L \to \Spec K$. Then the restriction of $M$ to
  $\Spec L$ is non-zero; we again denote this restriction by $M$. If
  $M= \Dc_L \sum l_i \otimes n_i \subset L\otimes_KN$, for each
  element $g$ in $G$, $g\otimes M$ is a simple $\Dc_L$-module that can
  be identified with $\Dc_L (\sum g\cdot l_i \otimes n_i)$. Using this
  identification
  \begin{displaymath}\tag{*}
  \sum_{g\in G} g\otimes M \subset L\otimes_KN 
\end{displaymath}
is an inclusion of $\Dc_L[G]$-modules. Since $N$ is simple,
\Proposition{\ref{morita}} implies that the right side is simple over $\Dc_L[G]$,
hence the inclusion is an equality.

Let $[S:S_1]$ denote the multiplicity of a simple module $S_1$ in a
module $S$.

We have, by adjointness,
\begin{displaymath}
  e=   [\pi_0^!(N): g\otimes M] = [(\pi_0)_+(g\otimes M):  N ]= [ (\pi_0)_+(M): N] =
  [\pi_0^!(N): M], 
\end{displaymath}
where the third equality follows since $g\otimes M \cong M$ when
restricted to $\Dc_K$-module. Now the sum in \thetag{*} can be reduced
to a direct sum over representatives of $G/G_M$ as stated, and then
taking minimal extensions from $\Dc_L$-modules to $\Dc_X$-modules.
This proves the first assertion. If $N$ is a connection or $\pi$ is
smooth, then $\pi^!(N)$ is semisimple
\Th{\ref{inv-conn}}, so that $j_{!+}(\pi^!(N))= \pi^!(N)$.

(1): We have $e= [ (\pi_0)_+(M): N] = \dim_k V_\chi $
\Th{\ref{galois-direct}} for some $\chi \in \widehat {\bar G}_M$, and
$\dim_k V_{\chi} $ is a divisor of $| G_M|$ (see
\cite{serre:lin-rep}*{\S 6.5, Prop 17}). The other assertion follows
since $G_M \subset G$ and $\deg \pi = |G|$.

(2): Put $L_0 = L^{G_M}$ and consider the factorisation
$\Spec L \xrightarrow {q} \Spec L_0 \xrightarrow{p} \Spec K $, so that
$\pi = p\circ q$. Then there exists a simple $\Dc_{L_0}$-module $N_0$
such that $p_+(N_0)=N$ and $q^!(N_0)=M$. Since
$\rko (N) = |G/G_M| \rko (N_0)$ it suffices to show that
$\rko (N_0)/ \rko (M)$ divides $|G/G_M|$. This follows from the first
part of the theorem to the map $q$, with $t=1$, so that
$\rko (N_0) / \rko (M) = e$ and therefore the assertion follows from
(1).\footnote{The proofs of \Theorem{\ref{galois-direct}} and
  \Theorem{\ref{inv-conn}} below do not depend on
  \Theorem{\ref{cliffordtheorem}}.}
\end{proof}

By considering compositions of Galois mappings we get the following
analogue of a theorem due to Ito and Isaacs about representations of
finite groups (see \cite{isaacs:character}*{11.29}). If one restricts
to the case $\Lambda= L$ below, using the equivalence in
\Theorem{\ref{equivalence}} it follows that it is equivalent to the
Ito-Isaacs result. 
\begin{corollary}\label{ito-isaacs}
  Let $N$ be a simple $\Dc_K$-module and $\pi : \Spec L\to \Spec K$ be
  a finite Galois field extension such that $\pi^!(N)= \Lambda^e$ for
  som integer $e$ and $\Dc_L$-module $\Lambda$ of rank $1$ over $L$.
  If $\pi$ can be factorised into field extensions
  $K= L_0 \subset L_1 \subset L_2 \subset \cdots \subset L_n = L $
  such that each $L_{i}/L_{i-1}$ is Galois, then $\rko (N)$ divides
  the degree $[L:K]$.
\end{corollary}
\begin{proof}
  Since $\pi^!(M) = \Lambda^e$ it follows that $\rko(M) =1$ in
  \Theorem{\ref{cliffordtheorem}}, and we can apply (2) to each
  successive extension $L_i/L_{i-1}$.
\end{proof}

\begin{theorem}\label{clifford-pairing} Let $M$ be a simple $\Dc_X$-module, where
  $X= \Spec L$. Put $L_1= L^{G_{M}}$ and consider the factorization
  $ X \xrightarrow{p} Z \xrightarrow{q} Y$, $\pi = q \circ p$,
  corresponding to the field extensions $L/ L_1/K$, where $L/K$ is
  Galois. Put
  \begin{displaymath}
    \Ac = \{N_Z\  \vert \ [p^!(N_Z):M ]\neq 0 \}, \quad \Bc = \{ N_Y \
    \vert \ [\pi^!(N_Y): M]\neq 0\}. 
  \end{displaymath}
  \begin{enumerate}
  \item If $N_Z \in \Ac$, then $q_+(N_Z)$ is simple.
    \item The map $N_Z \to q_+(N_Z)$  defines a bijection $\Ac \to
      \Bc$.
      \item If $N_Z\in \Ac$ and $N_Y= q_+(N_Z)$, then $N_Z$ is the
        unique simple component  of $q^!(N_Y)$ that belongs to $\Ac$.
        \item If $q_+(N_Z)= N_Y$, with $N_Z\in \Ac$, then $[p^!(N_Z):
          M]= [\pi^!(N_Y):M]$. 
  \end{enumerate}
\end{theorem}
The proof below is quite parallel to a corresponding statement for
characters of finite groups, see \cite{isaacs:character}*{Th 6.11}.
\begin{proof}
Let $N_Z\in \Ac$ and  $N_Y$ be a simple $\Dc_Y$-module. Since
  \begin{displaymath}
0\neq     [q_+(N_Z), N_Y]= [q^!(N_Y), N_Z], \quad [p^!(N_Z), M] \neq 0, 
  \end{displaymath}
  we get $[\pi^!(N_Y), M] = [p^!(q^!(N_Y)),M]\neq 0$. Hence
  $N_Y\in \Bc$. By \Theorem{\ref{cliffordtheorem}} we have
  $p^!(N_Z) = M^f$ and $\pi^!(N_Y) = \oplus^t ((g_i M))^e$. Since
  $[q^!(N_Y), N_Z]\neq 0$ we get $f\leq e$. Now comparing dimensions
  of vector spaces
\begin{displaymath}
  e\cdot t\cdot  \rko_L M = \rko_K N_Y \leq \rko_K q_+(N_Z) = t \cdot
  \rko_{L_1} N_Z  = f\cdot t\cdot
  \rko_L M  \leq  f\cdot t\cdot  \rko_L M, 
\end{displaymath}
hence equality holds throughout. In particular,
$\rko_K N_Y = \rko _K q_+(N_Z) $ and we conclude $q_+(N_Z)= N_Y$. This
proves (1).  Also,
\begin{displaymath}
  [\pi^!(N_Y), M] = e = f = [q^!(N_Z), M], 
\end{displaymath}
which proves (4). Statement (3) follows from the last equality since
$N^1_Z\in \Ac $ and $N^1_Z \neq N_Z$, and $N^1_Z$ is a constituent of
$q^!(N_Y)$, then
\begin{displaymath}
  [\pi^!(N_Y), M] \geq [\pi^!(N_Z + N^1_Z), M] = [p^!(N_Z), M] +
  [p^!(N^1_Z), M] > [p^!(N_Z), M],
\end{displaymath}
which is a contradiction. The map in (2) is well defined by (1) and
its image lies in $\Bc$ by (4). It is one-to-one by (3). To prove that
it maps onto $\Bc$, let $N_Y \in \Bc$. Since $M$ is a constituent of
$\pi^!(N_Y)$, there must be some simple constituent $N_Z$ of
$q^!(N_Y)$ with $[q_+(N_Z), N_Y]\neq 0$. Thus $N_Z\in \Ac$ and $N_Y$
is a constituent of $q_+(N_Z)$. Therefore $N_Y = q_+(N_Z)$ and the
proof is complete.
\end{proof}

Say that a $\Dc_Y$-module $N_Y$ is {\it primitive} if it is not the
direct image of another module with respect to a non-trivial finite
map. More precisely, if $q_+(N_Z)= N_Y$ for a finite morphism
$q: Z \to Y $ and some module $N_Z$, then $q $ is an isomorphism.
\begin{corollary}\label{cor-isaacs}
  If $N_Y$ is a simple primitive module, then $\pi^!(N_Y) = M^e$ where
  $M$ is a $\Dc_X$-module and $e$ is some integer.
\end{corollary}
\begin{proof}
  By \Theorem{\ref{cliffordtheorem}}
  $\pi^!(N_Y) = \bigoplus_{i=1 }^t (g_i M)^e$ and by
  \Theorem{\ref{clifford-pairing}} $N_Y= q_+(N_Z)$ for some simple
  $\Dc_Z$-module $N_Z$. Since $N_Y$ is primitive, it follows that
  $q=\id$, hence $G_M= G$ and $t=1$.
\end{proof}

\subsection{Decomposition of direct images}

By  equation (\ref{dir-im}) it follows that {\it if} $\pi_+(M)$ is semisimple
and $M$ has no torsion along the ramification locus $B_\pi$, then the
decomposition of $\pi_+(M)$ is determined by its decomposition over
the étale locus $Y_0$ (see \thetag{BC} in (\ref{sec:2.1})). If
moreover $j^+(M)$ is torsion free as $\Oc_{X_0}$-module, and therefore
$(\pi_0)_+ j^+(M)$ is torsion free, the decomposition is determined by
the decomposition of its restriction to the generic point $\eta$ of
$Y$, and the decomposition at any other point is recovered from the
minimal extension of $\pi_+(M)_\eta$.

We first prove the semisimplicity theorem for general semisimple
holonomic modules with finite support and later we describe this
decomposition in terms of Galois groups
\Thms{\ref{galois-direct}}{\ref{non-galois}}, when $M$ is torsion free
and $\pi$ is finite.
\begin{theorem}\label{decomposition-thm}
  Let $\pi : X\to Y$ be a morphism of smooth varieties and $M$ be
  a semisimple holonomic $\Dc_X$-module. If $\supp M $ is finite over
  $Y$, then $\pi_+(M)$ is also semisimple.
\end{theorem}

\begin{corollary}\label{supp-simples} If $M$ is a simple holonomic module and 
  $N$ is a non-zero simple constituent of $\pi_+(M)$, then
  $\supp N = \pi(\supp M)$.
\end{corollary}
\begin{pf} It is evident that $\supp N \subset \pi(\supp M)$. Assume
  on the contrary that $\supp N \neq \pi(\supp M)$. We have
  \begin{displaymath}
    Hom_{\Dc_Y}(N, \pi_+(M)) = Hom_{\Dc_Y}(\pi_+(M), N) =Hom_{\Dc_X}( M, \pi^!(N) ),
\end{displaymath}
where the first equality follows since $\pi_+(M)$ is semisimple
holonomic \Th{\ref{decomposition-thm}} and the second is because of
\Theorem{\ref{adj-triple-th}}. Since
$\supp (\pi^!(N) )\cap \supp M = \pi^{-1}(\supp N) \cap \supp M \neq
\supp M$ and $M$  is simple, it follows that
$N=0$.
\end{pf}

Notice that if $M$ is a non-zero holonomic $\Dc$-module
then the maximal semisimple submodule $\soc (M)$ (the socle) is
non-zero. 

\begin{corollary}\label{soc-cor}
  In the situation of \Theorem{\ref{decomposition-thm}} assume that
  $M$ is a holonomic module such that $\supp M$ is finite over $Y$,
  and which contains no section whose support is contained in the
  ramification locus $B_\pi$. Then
\begin{displaymath}
  \soc(\pi_+(M))= \pi_+ (\soc (M)).
\end{displaymath}
\end{corollary}
I am uncertain whether it is necessary that $M$ be torsion free along
$B_\pi$. At any rate, the proof will show that if $N$ is a simple submodule of
$\pi_+(M)$  such that $\supp N \not \subset D_\pi $ (the discriminant
locus), then $N \subset \pi_+(\soc (M))$.

Before proving \Theorem{\ref{decomposition-thm}} we use it to complete
the proof of (3) in \Theorem{\ref{semisimple-inv}}.\footnote{The proof of
\Theorem{\ref{decomposition-thm}} does not rely on
\Theorem{\ref{semisimple-inv}}}
\begin{pfof}{\Theorem{\ref{semisimple-inv}}, continued}\label{contproof}
  If $\pi^!(N)$ is semisimple, \Theorem{\ref{decomposition-thm}}
  implies that $\pi_+\pi^!(N)$ is semisimple. Since the trace morphism
  \begin{displaymath}
    \Tr : \pi_+\pi^!(N)\to N
  \end{displaymath} is surjective (see (\ref{trace-morphism})), this
  implies that $N$ is semisimple.
  Let and let $0 \to N_1 \to N \to
  N_2 \to 0$ be an exact sequence of coherent $\Dc_Y$-modules. Since
  $\pi$ is flat we get the exact sequence $0 \to \pi^!(N_1)\to
  \pi^!(N)\to \pi^!(N_2)\to 0$. Since moreover $\pi$ is faithfully
  flat, $\pi^!(N_i)\neq 0$ if and only if $N_i\neq 0$, $i=1,2$.
  Therefore, if $\pi^!(N)$ is simple, then $N$ is simple.
\end{pfof}

The proof of \Theorem{\ref{decomposition-thm}} is based on two lemmas.
For the first we refer to a proper morphism of smooth schemes of
finite type over $k$, $\pi: X\to Y$ and a point $\xi$ in $X$ that maps
to $\eta$ in the scheme $Y$. Put $X(\xi) = \Spec \Oc_{X,\xi}$ and
$Y(\eta)= \Spec \Oc_{Y, \eta}$, and let $\pi^\xi: X(\xi)\to Y(\eta)$
be the map that is induced by $\pi$. We get the diagram
 \begin{displaymath}\tag{$*$}
   \xymatrix{
     X(\xi) \ar[r] ^j\ar[d]_{\pi^\xi} & X\ar[d]^\pi\\ Y(\eta)
     \ar[r]^i & Y,
   }
  \end{displaymath}
  where $\pi \circ j = i\circ \pi^\xi$,
 \begin{lemma}\label{min-direct}
   As functors on $\hol(\Dc_{X(\eta)})$ we have
   \begin{displaymath}
     \pi_+ j_{!+} = i_{!+}\circ \pi^\xi_+.
\end{displaymath}
\end{lemma}
\begin{pf} 
  We have canonical maps of functors
   \begin{displaymath}
j_!= \Dbb_X j_+\Dbb_{X_0} \twoheadrightarrow  j_{!+} \hookrightarrow j_+,
\end{displaymath}
where the first is surjective and the second is injective.
Therefore by \Proposition{\ref{flat-module}} we
get
   \begin{displaymath}
\pi_+ j_!= \pi_+ \Dbb_X j_+\Dbb_{X_0} \twoheadrightarrow  \pi_+ j_{!+} \hookrightarrow \pi_+j_+,
\end{displaymath}
where again the first arrow is surjective and the second is
surjective. Since $\pi_+ j_+ = i_+ \pi^\xi_+$ and
$ \pi_+ j_! = i_! \pi^\xi_+$ \Prop{\ref{adj-triple-th}}, we get the
diagram
\begin{displaymath} i_! \circ \pi^\xi_+ \twoheadrightarrow \pi_+\circ
  j_{!+} \hookrightarrow i_+\circ \pi^\xi_+.
\end{displaymath} Since $i_{!+}$ is characterized as the image of $i_!
\to i_+$, this implies that $\pi_+\circ j_{!+}= i_{!+}\pi^\xi_+$.
 \end{pf}

 \begin{pfof}{\Corollary{\ref{soc-cor}}}
  The inclusion $ \pi_+ (\soc (M)) \subset \soc(\pi_+(M))$ follows
  from \Theorem{\ref{decomposition-thm}} and
  \Proposition{\ref{flat-module}}.

  $ \soc(\pi_+(M))\subset \pi_+ (\soc (M))$: We refer to \thetag{BC}
  in (\ref{sec:2.1}).  Let $N$ be a simple submodule of $\pi_+(M)$. The
  assumption that $M$ has no torsion along $B^\pi$ implies that $N$
  has no torsion along the discriminant locus $D_\pi$, so that
  $i^!(N)\neq 0$. Since $\pi_0$ is étale it is easy to see that
  $i^!(N) \subset (\pi_0)_+(\soc j^!(M))$.  Finally,  by \Lemma{\ref{min-direct}}
  \begin{displaymath}
  N  = i_{!+} i^! (N) \subset i_{!+}(\pi_0)_+(\soc j^!(M)) = \pi_+
  j_{!+} (\soc j^!(M))  \subset \pi_+ (\soc (M)).
\end{displaymath}
\end{pfof}
The proof of the next lemma is not presented until after
\Theorem{\ref{galois-direct}}, where it will also be clear that there
are no vicious circles. See also \cite{put-hoeij:descent-skew}*{Prop.
  2.7}, treating the situation in \Remark{\ref{levelt-put-hoeij}},
(2).
 \begin{lemma}\label{ss-field}
   Let $\pi: \Spec L \to \Spec K$ be a finite field extension. Then
   $\pi_+(M)$ is semisimple.
 \end{lemma}
\begin{pfof} {\Theorem{\ref{decomposition-thm}}}
  It suffices to prove that $\pi_+(M)$ is semisimple when $M$ is a
  simple holonomic $\Dc_X$-module. Then $\supp M$ contains a unique
  generic point $\xi$ and $M$ is the unique simple extension of the
  $\Dc_{X,\xi}$-module $ M_\xi$. By generic smoothness and
  \Lemma{\ref{min-direct}} (using the notation in \thetag{$*$})
  \begin{displaymath}
    \pi_+ (M) = \pi_+(j_{!+}(M_{\xi})) = i_{!+}\pi^\xi_+(M_{\xi}),
  \end{displaymath} so it suffices to prove that $\pi^\xi_+(M_{\xi})$
  is semisimple. Let $L_\xi $ denote the residue field of the point
  $\xi$, $K_\eta$ be the residue field of $\eta= \pi(\xi)$, $p: \Spec L_\xi \to \Spec K_\eta$ be the associated map of
  schemes, and $j_0 : \Spec L_\xi \to X(\xi) $ and $i_0 : \Spec K_\eta
  \to Y(\eta)$ be the  inclusions of closed point in
  $X(\xi)$ and $Y(\eta)$. By Kashiwara's theorem $M_\xi= (j_0)_+j_o^!(M)= (j_0)_+(M^{\mf_\xi}_\xi) $,
  where $j_0 : \Spec L_\xi \to X(\xi)$ is the inclusion of the closed
  point, and  
  \begin{displaymath}
    \bar M_\xi =j_0^!(M_\xi)= M_\xi^{\mf_\xi} = \{m \in M_\xi \ \vert\ \mf_\xi \cdot m =0\}
  \end{displaymath}
  is a semisimple $\Dc_{L_\xi}$-module. Since
  $\pi^\xi \circ j_0 = i_0\circ p$, we have
  \begin{displaymath}
    \pi^\xi_+(M _\xi) = (i_0)_+p_+(\bar M_\xi).
  \end{displaymath} Again by Kashiwara's theorem it suffices to see
  that $p_+(\bar M_\xi)$ is a semisimple $\Dc_{K_\eta}$-module when $\bar M_\xi$ is a
  semisimple $\Dc_{L_\xi}$-module. This
  follows from \Lemma{\ref{ss-field}}.
\end{pfof}
\begin{remark}
  Let $\pi: \Spec L \to \Spec K$ be a finite field extension and $M$
  be an étale trivial $\Dc_L$-module, so there there exists a field
  extension $L_1/L /K$ such $L_1/K$ is Galois and $L_1\otimes_L M$ is
  isomorphic to $L_1^n$ for some integer $n$. Put $G= Gal(L_1/K)$ and
  $H= Gal (L_1/L)$. By \Proposition{\ref{ind-direct}} the
  decomposition of $\pi_+(M)$ can be determined by the decomposition
  of the induced representation $\ind_H^G V$, where
  $V= Hom_{\Dc_{L_1}}(L_1 , L_1\otimes_L M) $ is a representation of
  $H$. In general, however, no such (finite) field extension need to
  exist.
\end{remark}

  Given a $\Dc_L[\bar G_M]$-module $M$ and a $\bar G_M$-module $V$, we
  get the $\Dc_L[\bar G_M]$ -module $V\otimes_k M$, where the action
  is
\begin{displaymath}
  ( \sum_{g\in \bar G_M} a_g g ) v\otimes m = \sum_{g\in \bar G_M}
  (g\cdot v )\otimes g a^{g^{-1}}_g \cdot m
\end{displaymath}
($a_g\in \Dc_L$, $a^{g^{-1}}_g = g^{-1}a_g g$, see (\ref{inertia-section})).

For each $\chi \in \widehat {\bar G}_M$ and $\bar G_M$-module $V_\chi$
of class $\chi$ we put
\begin{align*}
  M_\chi = ( V_\chi \otimes_k M)^{\bar G _M } = Hom_{k[\bar G_M]}
  (V_\chi^*, M) .
\end{align*}

\begin{theorem}\label{galois-direct}
  Assume that $L/K$ is Galois with Galois group $G$, and that $M$ is a
  simple $\Dc_L$-module with true inertia group $\bar G_M$.
  \begin{enumerate}
  \item If $M$ is a simple $\Dc_L$-module then $M_\chi$ is isomorphic
  to a
    simple submodule of $\pi_+(M)$.
\item $ \pi_+(M) $ is a semisimple $\Dc_K[\bar G_M]$-module with
decomposition
\begin{displaymath}
    \pi_+(M)= \bigoplus_{\chi \in \widehat {\bar G}_M } V_\chi
    ^*\otimes_k M_\chi,
  \end{displaymath} where the $V_\chi ^*\otimes_k M_\chi$ are simple
  $\Dc_K[\bar G_M]$-modules which are mutually non-isomorphic for
  different $\chi$. Here $\Dc_K$ acts only on the second factor $M_\chi$
  and $\bar G_M$ acts only on the first factor $V_\chi^*$.

  \item The
  multiplicity and ranks are
  \begin{eqnarray*}
    [\pi_+(M): M_\chi] &=& \dim_k V_\chi,\\
    \rank_K M_\chi &=& |G/G_M|\cdot \dim_k V_\chi \cdot \rank_LM.
  \end{eqnarray*}
  The multiplicities $[\pi_+(M): M_\chi]$ divide the degree of $\pi$.
  \end{enumerate}
\end{theorem}
The proof is given in (\ref{proof-main}). Using the theorem we can
first tie up a loose end:
 
 \begin{pfof}{\Lemma{\ref{ss-field}}}
   Let $L_1 /L/K$ be field extensions such that $L_1/L$ and $L_1/K$
   are Galois, and $G_1$ be the Galois group of $L_1/K$; let $\pi_1 :
   \Spec L_1 \to \Spec K$ be the corresponding map. By
   \Lemma{\ref{semisimple-inv-field}} $L_1\otimes_L M$ is semisimple
   and by \Proposition{\ref{flat-module}} there exists an injective
   homomorphism $\pi_+(M) \to \pi_{1+}(L_1\otimes_L M)$ (here
   exactness is evident since $\Spec L_1 \to \Spec L \to \Spec K$ are
   smooth morphisms), so it suffices to prove that the latter module
   is semisimple. This follows from \Theorem{\ref{galois-direct}}.
 \end{pfof}

 The $\Dc_L$-module $L$ is particularly interesting. It is evidently a
 $\Dc_L[G]$-module for any group $G$ of automorphisms of $L$.
\begin{corollary}\label{direct-L} Let $L$ be a field of finite type
over $k$, let $G$ be a
  finite group of automorphisms of $L$, and let $K= L^G $ be the fixed
  field of $G$. Let $\pi : \Spec L \to \Spec K$ be the morphism that
  is associated to the field extension $ K \subset L$. Then
\begin{displaymath}
    \pi_+(L)= \bigoplus_{\chi \in \hat G} V_\chi ^*\otimes_k L_\chi
\end{displaymath} where $L_\chi = Hom_{k[G]}(V_\chi^*, L)$ is a simple
$\Dc_K$-module of and $V_\chi$ is the simple $k[G]$-module of class $\chi
\in \hat G$.  We have
\begin{displaymath}
[\pi_+(L): L_\chi] = \rank _K L_\chi = \dim_k V_\chi.
\end{displaymath}
\end{corollary}
\begin{remarks}
  \begin{enumerate}
  \item When $L/K$ arises from a morphism $\pi: X\to Y$ of smooth
    varieties we get a decomposition
  \begin{displaymath}
    \pi_+(\Oc_X) = \bigoplus_{\chi \in \hat G} V_\chi ^*\otimes_k
    \Oc_{X_\chi},
  \end{displaymath} where $\Oc_{X_\chi} = Hom_{k[G]}(V_\chi^*, \Oc_X)$
is a simple $\Dc_Y$-module.
  \item Similarly to \Corollary{\ref{non-galois}} (discussed below), when $L/K$ is a
  finite non-Galois field extension
    the decomposition of $\pi_+(L)$ can be expressed as the
    $H$-invariants $\tilde \pi_+(\bar L)^H$ of the decomposition of
    $\tilde \pi_+(\bar L)$, where $\bar L$ is the Galois cover of
    $L/K$ and $H$ is the Galois group of $\bar L/L$.
  \end{enumerate}
\end{remarks}

Consider now a $\Dc_L$-module $M$ that is also a $\Dc_L[G_M]$-module
(so that the cocycle $f$ in \Section{\ref{inertia-section}} is a
coboundary), implying that the decomposition of $\pi_+(M)$ is
determined by the regular representation $k[G_M]$ as in
\Theorem{\ref{galois-direct}}. Putting $L_1 = L^{G_M}$ we have the
morphisms $p: \Spec L \to \Spec L_1$ and $q: \Spec L_1 \to \Spec K$,
so that by \Proposition{\ref{morita}} $M= p^!(N)$ for some
$\Dc_{L_1}$-module $N$, where $M$ is a simple $\Dc_L[G_M]$-module if
and only if $N$ is simple. In this situation we also get that the
direct image of a tensor products by $N$ and certain other simple
modules are simple.
\begin{proposition}
  In the notation above, $q_+(L_\chi\otimes_{_L}N)$ is a simple
  $\Dc_K$-module for any simple component $L_\chi$ of $p_+(L)$.
\end{proposition}
\begin{proof}
  The ``projection formula'' (which is evident in this case, for the
  general case, see \cite{hotta-takeuchi-tanisaki}*{Cor. 1.7.5}),
  \Corollary{\ref{direct-L}} and \Theorem{\ref{galois-direct}} implies
\begin{eqnarray*}
  \pi_+(M)&=& q_+(p_+(L)\otimes_{L_1} N) = \bigoplus_{\chi \in \hat {
              G}_M } V^*_\chi\otimes_k q_+(L_\chi\otimes _{L_1}N)
  \\ &=& \bigoplus_{\chi \in \hat { G}_M
         }
         V_\chi^*\otimes_k M_\chi.
\end{eqnarray*}
It follows that $q_+(L_\chi\otimes _{L_1}N)\cong M_\chi$, which is
simple.
\end{proof}

\subsection{The
  normal basis theorem }
Assume that $L/K/k$ are field extensions, where $L/K$ is finite, $K/k$
is of finite type, and $k$ is algebraically closed of characteristic
$0$. It follows from \Corollary{\ref{direct-L}} that if $\mu$ is an
element in $L$, then the following are equivalent:
\begin{enumerate}
\item $\Dc_K\mu $ is simple.
\item $K[G]\mu$ is simple.
\end{enumerate}
If one instead asks that $\mu$ be a generator of $L$ one arrives at a
refinement of the normal basis theorem in this situation.

 \begin{theorem}\label{normalbasis} Keep the assumption on $L/K/k$. The following are
   equivalent for an element $\mu \in L$:
   \begin{enumerate}
   \item $\Dc_K \mu = L$.
     \item $K[G]\cong K[G]\mu = L$.
     \end{enumerate}
     Moreover, there exists such an element $\mu$. 
 \end{theorem}

 \begin{proof} The existence of the element in (1) is a well-known
   fact since $L$ is a holonomic $\Dc_K$-module, and it is not
   essential that $L/K$ be Galois. See \cite{katz-cyclic} for a
   construction of a cyclic vector $\mu$ from a $K$-basis of $L$,
   using only one derivation of $K$. A non-constructive but more
   ``generic'' proof is as follows. The ring $\Dc_K$ is simple and
   since $K/k$ has positive transcendence degree, we have
   $\dim_ K \Dc_K = \infty$, so that for any element $l\in L$ the map
   $\Dc_K \to L$, $P\mapsto Pl$ is not an isomorphism. Then conclude
   from e.g. \cite{kallstrom:arkiv}*{Lem. 2.2}. For a proof of the
   existence of an element as in (2), see \cite{lang:algebra}*{Th.
     13.1} for the classical normal basis theorem.
   
   $(1)\Rightarrow (2)$: We have the following isomorphisms and
   inclusions of $\Dc_K$-modules (as detailed below)
   \begin{eqnarray*}
&&     \bigoplus_{\chi \in \hat G} V_\chi^* \otimes_k L_\chi = L = \Dc_K \mu \subset
     Hom_{k[G]}(k[G]\mu, \Dc_K \mu) \\ &=& Hom_{k[G]}(k[G]\mu,
     \bigoplus_{\chi \in \hat G} V_\chi^*\otimes_kL_\chi )=   \bigoplus_{\chi \in \hat G} Hom_{k[G]}(k[G]\mu,
    V_\chi^* ) \otimes_kL_\chi.
   \end{eqnarray*}
   The first and third equalities follow from the first part of
   \Corollary{\ref{direct-L}}, the inclusion map is that $\mu$ maps to
   the map $\mu \mapsto \mu$. Since the $\Dc_K$-modules $L_\chi$ are
   mutually non-isomorphic it follows that
\begin{displaymath}
V_\chi^*\subset   Hom_{k[G]}(k[G]\mu,    V_\chi^* ) .
\end{displaymath}
Therefore the multiplicity  of the simple module $K\otimes_k V_\chi^*
$  in the  semisimple $K[G]$-submodule $K[G]\mu\subset L$   satisifies
\begin{displaymath}
  \dim_k V_\chi^* \leq  [K[G]\mu: K\otimes_k V_\chi^*] \leq [L:
  K\otimes_k V_\chi^*] = \dim _k V_\chi^*. 
\end{displaymath}
Since again by \Corollary{\ref{direct-L}}
$L=\oplus V_\chi^*\otimes_kL_\chi = \oplus K\otimes_k V_\chi^*\otimes_K L_\chi $,
this implies $(2)$.

$(2)\Rightarrow (1)$:   We have
\begin{displaymath}
  k[G]\mu \subset Hom_{\Dc_K}(\Dc_K \mu, L) = \bigoplus_{\chi \in \hat G}Hom_{\Dc_K}(\Dc_K \mu, L_\chi) \otimes_k V_\chi^*,
\end{displaymath}
where the inclusion map is determined by letting $\mu$ map to the map
$\mu \mapsto \mu$, and the equality follows from
\Corollary{\ref{direct-L}}. This implies the first inequality below
\begin{eqnarray*}
\dim_k V_\chi^* &=& [L: K\otimes_k V_\chi^* ]= [K[G]\mu : K\otimes_k V_\chi^*]= [k[G]\mu:
V_\chi^*]\\ & \leq  &\dim_k Hom_{\Dc_K}(\Dc_K \mu, L_\chi) \leq \dim_k
Hom_{\Dc_K}(L, L_\chi) = \dim_k V_\chi^*,
\end{eqnarray*}
where multiplicites of $K[G]$-modules occur in the second and third
steps. The second inequality follows since $L$ is a semisimple
$\Dc_K$-module. Therefore $[\Dc_K \mu, L_\chi] = \dim_k V_\chi^*$, so
that $\Dc_K \mu = L$ by \Corollary{\ref{direct-L}}.
 \end{proof}

\subsection{Proof of \Theorem{\ref{galois-direct}}}\label{proof-main}
We have two $\Dc_L[\bar G_M]$-modules $\Dc_L[\bar
G_M]\otimes_{\Dc_L}M$ and $ k[\bar G_M]\otimes_k M$, where on the
first $\Dc_L[\bar G_M]$ acts by multiplication from the left and on
the second $\bar G_M$ acts diagonally while $\Dc_L$ acts only on the
second factor. In fact, these modules are isomorphic.
\begin{lemma}\label{simpletensor}
  Let $M$ be a $\Dc_L[\bar G_M]$-module such that $e_+(M)$ is a simple
  $\Dc_L$-module.
  \begin{enumerate}
  \item If $V$ is a simple $\bar G_M$-module, then $V\otimes_kM$ is a
  simple $\Dc_L[\bar
    G_M]$-module.
\item As $\Dc_L[\bar G_M]$-modules
  \begin{displaymath}
    \Dc_L[\bar G_M]\otimes_{\Dc_L}e_+(M) \cong k[\bar G_M]\otimes_k M
    \cong \bigoplus_{\chi \in \widehat {\bar G}_M } V_\chi^{m_\chi}
    \otimes_k M,
  \end{displaymath} where $\Dc_L[\bar G_M]$ acts by multiplication
  from the left in the left side, while on the other two terms $\bar
  G_M$ acts diagonally and $\Dc_L$ acts only on the right factor $M$.
  The $\Dc_L[\bar G_M]$-module $V_\chi \otimes_k M $ is simple, where
  $V_\chi$ is a simple $\bar G_M$-module of dimension $m_\chi$.
  \end{enumerate}
\end{lemma}
\begin{proof}
  (1): Let $N$ be a non-zero simple submodule of $V\otimes_kM$ and
  $\sum v_i\otimes m_i\in N$ be an element such that
  $v_1 \otimes m_1 \neq 0$, and the $m_i$ are linearly independent
  over $k$. Since $e_+(M)$ is simple and $k$ is algebraically closed,
  by Quillen's theorem \cite{dixmier}*{Lem. 2.6.4}
  $\End_{\Dc_L}(e_+(M))=k$, hence by the density theorem there exists
  $P\in \Dc_L$ such that $Pm_i= \delta_{i1}m_1$ ($\delta_{ij}$ is the
  Kronecker delta); hence $v_1\otimes m_1 \in N$. Since
  $\Dc_L (v_1 \otimes m_1 ) = v_1 \otimes \Dc_L m_1 = v_1\otimes M$ it
  suffices now to prove that $(gv_1) \otimes m_1\in N $ when
  $g\in \bar G_M$, since this would imply $N= V\otimes_k M$. We have
  $g \cdot (v_1 \otimes m_1) = (gv_1 \otimes gm_1)\in N$, so that if
  $gm_1 \in k m_1$ this is evidently true. If $g m_1 \not \in km_1$,
  then by the density theorem there exists $Q\in \Dc_L$ such that
  $Q (gv_1 \otimes gm_1) = (gv_1 \otimes Q (gm_1)) = (gv_1)\otimes
  m_1$, and therefore $(gv_1)\otimes m_1 \in N$.

  (2): Define the map
  \begin{displaymath}
    \phi : \Dc_L[\bar G_M]\otimes_{\Dc_L} e_+(M) \to k[\bar
    G_M]\otimes_k M, \quad g\otimes m \mapsto g \otimes g\cdot m,
\end{displaymath} which is well defined since any element in $
\Dc_L[\bar G_M]\otimes_{\Dc_L} M$ is a unique $k$-linear combination
of elements of the form $g\otimes m$. Now
\begin{eqnarray*}&& \phi ( \sum_{h\in \bar G_M} a_h h (g \otimes
    m)) = \phi ( hg \otimes
 a_h^{(hg)^{-1}} \cdot m) \\ &=& \sum_{h\in \bar G_M} hg \otimes (hg
 \cdot a_h^{(hg)^{-1}} \cdot m) = \sum_{h\in \bar G_M} hg \otimes (
 a_h ((hg) \cdot m)) \\ &=& \sum_{h\in \bar G_M} a_h \cdot ((hg
 )\otimes ( hg \cdot m) )= (\sum_{h\in \bar G_M} a_h h ) (g\otimes
 g\cdot m)
\\
 &=& (\sum_{h\in \bar G_M} a_h h ) \phi (g\otimes m) ,
\end{eqnarray*} showing that $\phi$ is a homomorphism. The map $\phi$
is clearly injective, and since $\phi$ is a homomorphism of vector
spaces over $L$ of equal dimensions, it follows that $\phi$ is an
isomorphism. The second isomorphism in (2) follows since $k[\bar G]=
\oplus_{\chi \in \widehat{\bar G}_M} V_\chi^{m_\chi}$, the regular
representation of $\bar G_M$, and that $V_\chi \otimes_k M$ is simple
follows from (1).
\end{proof}

Next lemma is similar to Mackey's irreducibility criterion for induced
representations \cite{mackey:induced}; see also
\Corollary{\ref{irr-crit}}. Put $ G_M^g = G_M \cap (g G_M g^{-1})$ and
$ \bar G_M^g =\psi^{-1}( G_M^g)$ (see (\ref{central-ext})), so that if
$N$ is a semisimple $\Dc_L[\bar G_M]$-module, then both $N$ and
$g\otimes N$ are $\Dc_L[\bar G^g_M]$-modules, which by
\Proposition{\ref{maschke}} are semisimple.
\begin{lemma}\label{tensorlemma}
  Let $L/K$ be a finite Galois extension and $N$ be a simple $\Dc_L[\bar
  G_M]$-module. Make the following assumption:
  \begin{displaymath}
    Hom_{\Dc_L[\bar G^g_M]}(N, g\otimes N)=0.
  \end{displaymath} In other words, if $g \in G\setminus G_M$, then $N
  $ and $g \otimes N$ have no pairwise isomorphic simple component,
  regarded as $\Dc_L[\bar G^g_M]$-modules. Then
  $\Dc_L[G]\otimes_{\Dc_L[\bar G_M]} N$ is a simple $\Dc_L[G]$-module.
\end{lemma}
\begin{proof} The asumption implies that the canonical map
  \begin{align*}
    Hom_{\Dc_L[\bar G_M]} (N,N) &\to Hom_{\Dc_L[G]}
    (\Dc_L[G]\otimes_{\Dc_L[\bar G_M]} N,
                                  \Dc_L[G]\otimes_{\Dc_L[\bar G_M]}
N)\\ &= Hom_{\Dc_L[\bar G_M]} (N,\Dc_L[G]\otimes_{\Dc_L[\bar G_M]} N)
  \end{align*} is an isomorphism, which implies the assertion, since
  $k$ is algebraically closed of characteristic $0$ (this follows from
  Quillen's theorem, see \cite{dixmier}*{ Lem. 2.6.4}).
\end{proof}

The right action of $\Dc_L[G]$ on $\Dc_L[G]$ and right action of
$\Dc_L[\bar G_M]$ on $\Dc_L[\bar G_M]$ gives $Hom_{\Dc_L[\bar G_M]}
(\Dc_L[G], \Dc_L [\bar G_M])$ a structure of $(\Dc_L[G], \Dc_L[\bar
G_M])$-bimodule.
\begin{lemma}\label{extra} There exists an isomorphism as $(\Dc_L[G],
\Dc_L[\bar G_M])$-bimodules
  \begin{displaymath}
    \Dc_L[G] = Hom_{\Dc_L[\bar G_M]} (\Dc_L[G], \Dc_L [\bar G_M]).
  \end{displaymath}
\end{lemma}
\begin{proof}We only write down a sequence of natural isomorphisms,
  leaving to the reader to work out the various actions of $\bar G_M$,
  $\Dc_L$ and $\Dc_L[G]$:
  \begin{align*}
    & Hom_{\Dc_L[\bar G_M]}(\Dc_L[G], \Dc_L[\bar G_M])=
      (Hom_{\Dc_L}(\Dc_L[G], \Dc_L[\bar G_M]))^{\bar G_M}\\ &=
      (Hom_{\Dc_L}(\Dc_L[G],
                                                              \Dc_L)\otimes_k
                                                              k[\bar
                                                              G_M])^{\bar
                                                              G_M} =
                                                              (\Dc_L[G]\otimes_kk[\bar
                                                              G_M])^{\bar
                                                              G_M}\\
    & = Hom_{k[\bar G_M]}(k[\bar G_M]^*, \Dc_L[G])
      = Hom_{k[\bar G_M]}(k[\bar G_M], \Dc_L[G]) = \Dc_L[G].
  \end{align*}
\end{proof}
\begin{pfof}{\Theorem{\ref{galois-direct}}} (a): We can assume that
$M$ is a simple $\Dc_L$-module, so that by \Lemma{\ref{simpletensor}}
  \begin{displaymath}\tag{$*$}
    \Dc_L[\bar G_M]\otimes_{\Dc_L}M = \bigoplus_{\chi \in \widehat
    {\bar G}_M} V_\chi^{m_\chi} \otimes_k M,
\end{displaymath} hence any simple $\Dc_L[\bar G_M]$-submodule $N$
of the above module is of the form $V\otimes_k M$, where $V$ is a
simple $\bar G_M$-module, and thus $e_+(N) = \oplus^n M$ for for some
integer $n$. Let $g\in G$ and assume that there exists a non-zero
homomorphism
  \begin{displaymath} \phi: e_+(N) = \bigoplus^n M \to g\otimes
  e_+(N) = \bigoplus^n g \otimes M.
\end{displaymath}
Then $M \cong g\otimes M$, and therefore $g \in G_M$; hence there are
no non-trivial $\Dc_L$-linear homomorphism $\phi$ when $g\not \in G_M$, and
hence no $\Dc_L[\bar G^g_M]$-linear ones either. Therefore the
assumptions in \Lemma{\ref{tensorlemma}} are satisfied, and hence
$\Dc_L[G]\otimes_{\Dc_L[\bar G_M]} N$ is a simple $\Dc_L[G]$-module.

(b): We now have (as detailed below)
\begin{displaymath}
  \pi_+(M) = \Dc^r_L \otimes_{\Dc_L[\bar G_M]} \Dc_L[\bar G_M]
  \otimes_{\Dc_L} M = \Dc^r_L \otimes_{\Dc_L[\bar G_M]} \bigoplus_{\chi \in \hat {\bar
      G}_M} V_\chi^{m_\chi} \otimes_k M.
\end{displaymath}
The first equality follows since $\pi_+(M)$ is formed simply as the
restriction of the $\Dc_L$-module $M$ to the subring
$\Dc_K \subset \Dc_L $ and since
$\Dc_L \otimes_{\Dc_L[\bar G_M]} \Dc_L[\bar G_M] \cong \Dc_L$ as
$(\Dc_K, \Dc_L) $-bimodule. The second equality follows from
\thetag{$*$}. By (a)
$\Dc_L[G] \otimes_{\Dc_L[\bar G_M]}V_\chi \otimes_k M$ is a simple
$\Dc_L[G]$-module which by \Proposition{\ref{morita}} implies that the
following $\Dc_K$-module is simple
\begin{displaymath} \Dc^r_L\otimes_{\Dc_L[\bar G_M]}( V_\chi \otimes_k
M) =\Dc^r_L\otimes_{\Dc_L[G]} \Dc_L[G] \otimes_{\Dc_L[\bar
  G_M]}(V_\chi \otimes_k M).
\end{displaymath} (c): We have isomorphisms (as detailed below)
\begin{align*} & \Dc^r_L \otimes_{\Dc_L[\bar G_M]}( V_\chi
\otimes_{k}M) = Hom_{\Dc_L[ G]}(\Dc^l_L, \Dc_L[
  G])\otimes_{\Dc_L[\bar G_M ]} (V_\chi \otimes_{k}M) \\ & =
  Hom_{\Dc_L[ G]}(\Dc_L^l, \Dc_L[ G]\otimes_{\Dc_L[\bar G_M] }(V_\chi
  \otimes_{k}M) ) \\ &
=Hom_{\Dc_L[ G]}(\Dc_L^l,
  Hom_{\Dc_L[\bar G_M]}(\Dc_L[ G], \Dc_L[\bar G_M])\otimes_{\Dc_L[\bar G_M]
                                                                  }(V_\chi
                                                                  \otimes_{k}M)
                                                                  )
  \\& =Hom_{\Dc_L[ G]}(\Dc_L^l,
  Hom_{\Dc_L[\bar G_M]}(\Dc_L[ G], V_\chi \otimes_{k}M ) \\&
=Hom_{\Dc_L[\bar G_M]} (\Dc_L^l, V_\chi \otimes_{k}M) \\ & = ( V_\chi
\otimes_k M)^{\bar G _M} = Hom_{k[\bar G_M]} (V_\chi^*, M)= M_\chi.
\end{align*} The first line is by the definition of the bimodule $
\Dc^r_L$ in
\Section{\ref{galois-section}}, the third follows from
\Lemma{\ref{extra}}, and the fifth line is by adjunction. Since
$V_\chi^* \otimes_k M_\chi$ is the $\chi^*$-component of the
$\bar G_M$-module $M$, by (b) this completes the proof of (1) and (2).

(3): The multiplicity $m_\chi =[\pi_+(M): M_\chi]= \dim _k V_\chi$ is
clear. Put $L_1= L^{G_M}$, so that we have field extensions
$p: \Spec L\to \Spec L_1$, $q: \Spec L_1\to \Spec K$, and
$\pi= q\circ p$. Since
\begin{displaymath}
  \pi_+(M) = q_+(p_+(M))= \bigoplus_{\chi \in \widehat {\bar G}_M }
  V_\chi^*\otimes_k q_+(\tilde M_\chi) = \bigoplus_{\chi \in \widehat {\bar G}_M}
  V_\chi^*\otimes_k M_\chi,
\end{displaymath}
we have $M_\chi = q_+(\tilde M_\chi)$. By
\Theorem{\ref{cliffordtheorem}} $p^!(\tilde M_\chi)= M^{m_\chi}$ and
we have
$[L_1:K]= |G/G_M|$, hence
\begin{displaymath}
  \rank_K(M_\chi) = |G/G_M| \rank_{L_1}(\tilde M_\chi) = |G/G_M| m_\chi \rank_L(M).
\end{displaymath}

That $[\pi_+(M): M_\chi] = \dim_k V_\chi$  divides $\deg \pi = |G|$
follows since $\dim_k V_\chi$ divides $ |G|= |\bar G/Z(\bar G)|$
(see \cite{serre:lin-rep}*{\S 6.5, Prop 17}). 
\end{pfof}

\subsection{Decomposition for general finite
  maps}\label{general-finite}
One really wants to decompose $\pi_+(M)$ for any finite map
$\pi: X\to Y$ of smooth varieties, not requiring that the fraction
fields forms a Galois field extension $L/K$ as in
\Theorem{\ref{galois-direct}}. For this purpose we will use the Galois
cover $\bar L/K$ of $L/K$, so that we get the map
$\tilde \pi : \tilde X \to Y $ of $\pi$, where $\tilde X$ is the
integral closure of $Y$ in $\bar L$. The Galois groups $G$ and $H$ of
$\bar L/K$ and $\bar L/L$ act on the normal variety $\tilde X$ so that
$Y = \tilde X^G$ and $X= \tilde X^H$. Then $\tilde \pi = \pi \circ p$,
where $p: \tilde X \to X$ is the invariant map for the $H$-action.
Consider the diagram (as detailed below)
\begin{displaymath}
  \begin{tikzcd}
     \tilde X \arrow[r, "p"] \arrow[rr, bend left, "\tilde \pi"] & X
\arrow[r,"\pi"] & Y \\ \tilde X_r \arrow[u,"r"] \arrow[r,"p_0"]
\arrow[rr, bend right, "\tilde \pi_0"'] & X_r \arrow[u,"h"]
\arrow[r,"\pi_0"] & Y_r \arrow[u,"i"].
  \end{tikzcd}
\end{displaymath} The map $r$ is inclusion of the smooth locus of
$\tilde X$ and $h: X_r = p(\tilde X_r)\to X$ is the inclusion map (so
that $\codim_{\tilde X} (\tilde X \setminus \tilde X_r)\geq 2$ and
$\codim_{X}(X \setminus X_r) \geq 2$), and let $p_0$ be the
restriction of $p$. The sheaf of rings $\Dc_{\tilde X_r}$ is locally
generated by its first order differential operators, hence the inverse
image $p_0^!(h^!(M))$ is a well-defined $\Dc_{\tilde
  X_r}$-module.\footnote{ The ring $\Dc_{\tilde X}$ needs
  not be generated by first order differential operators, which is
  required to make sense of
  $p^!(M)$.}

Let $M$ be a semisimple $\Dc_X$-module. Then $h^!(M)$ is again
semisimple, but since $p_0$ need not be a smooth, $p_0^!(h^!(M))$ need
not be semisimple (see \Example{\ref{levelt-put-hoeij}}). Moreover,
$(p_0)_+p_0^!(h^!(M))= (p_0)_+(\Oc_{\tilde X}) \otimes_{\Oc_X} h^!(M)$
which again need not be a semisimple $\Dc_{X_r}$-module. We therefore
consider the socle of $p_0^!(h^!(M))$,
\begin{displaymath}
  \tilde M = \soc (p_0^!(h^!(M))),
\end{displaymath}
It is clear that $H$ belongs to the true inertia group
$\bar G_{\tilde M}$ of $\tilde M$. By \Corollary{\ref{soc-cor}} the
canonical injective map $h^!(M) \to (p_0)_+p_0^!(h^!(M))$ results in a
map
\begin{displaymath}
  h^!(M) \to \soc ( (p_0)_+p_0^!(h^!(M))) = (p_0)_+ (\tilde M),
\end{displaymath} where the last equality follows from
\Corollary{\ref{soc-cor}}. Put now $Y_r = \pi(X_r)$ and let $i: Y_r
\to Y$ be the open inclusion, where $\codim_{Y}(Y \setminus Y_r)\geq
2$, and let $\pi_0 : X_r \to Y_r$ be the restriction of $\pi$. Then,
since $\pi_+(M)$ is semisimple, 
\begin{displaymath}
  \pi_+(M) = i_{!+}i^!(\pi_+(M)) = i_{!+} (\pi_0)_+h^! (M).
\end{displaymath} Therefore the decomposition of $\pi_+(M)$ is
entirely determined by the decomposition of the $\Dc_{Y_r}$-module
$(\pi_0)_+h^! (M)$. 
In the notation above we now have:
\begin{corollary}\label{non-galois}
  Let $M$ be a simple holonomic $\Dc_X$-module such that $\supp M = X$
  and put. Let
  \begin{displaymath} (\tilde \pi_0)_+(\tilde M) = \bigoplus_{\chi \in
\widehat {\bar G}_{\tilde M} } V_\chi ^*\otimes_k
    \tilde M_\chi,
  \end{displaymath} be the decomposition in
\Theorem{\ref{galois-direct}} with respect to the map $\tilde \pi_0$.
Then
\begin{displaymath}
  \pi_+(M) = \bigoplus_{\chi \in \widehat {\bar G}_{\tilde M} }
  (V_\chi ^*)^H\otimes_k i_{!+}(\tilde M_\chi),
\end{displaymath} where $(V_\chi ^*)^H= \{v' \in V_\chi^* \ \vert \ h
\cdot v' = v', \quad h\in H\}$. The multiplicities of the simple
components are
\begin{displaymath}
  [\pi_+(M): i_{!+}(\tilde M_\chi)] = \dim_k V^H_\chi = \frac 1{|H|}\Tr_H(\phi_\chi)
  = \frac 1{|H|} \sum_{h\in H}  \phi_\chi(h),
\end{displaymath}
where $\phi_\chi$ is  the character of $ \chi   \in \widehat {\bar G}_{\tilde M}$. 
\end{corollary}
\begin{remark}
  When $M$ is a simple module such that $V= \supp M \neq X$ one still
  gets a similar decomposition as in \Corollary{\ref{non-galois}}
  using the Galois cover of $V$, but we leave this generalisation to
  the reader.
\end{remark}
\begin{proof} Put $M_r=h^!(M) $, which is a simple $\Dc_{X_r}$-module,
  so that $M_r= ((p_0)_+p_0^!(M_r))^H$. We then get a map (as detailed
  below) 
  \begin{align*}
    \bigoplus_{\chi \in \widehat {\bar G}_{\tilde M} } &
    (V_\chi^*)^H\otimes_k i_{!+}(\tilde M_\chi) = \bigoplus_{\chi \in
      \widehat {\bar G}_{\tilde M} }i_{!+}(V_\chi^*\otimes_k \tilde
      M_\chi)^H= i_{!+}([(\tilde \pi_0)_+(\soc(p_0^!(M_r))) ]^H)\\
    &\xrightarrow{\phi} i_{!+}
      ([(\pi_0)_+(p_0)_+p_0^!(M_r)]^H) = i_{!+}(\pi_0)_+
      ([(p_0)_+(p_0)^!(M_r)]^H) \\ &= i_{!+} (\pi_0)_+(M_r) =
                                    \pi_+(h_{!+}h^!(M)) = \pi_+(M).
\end{align*}
The first equality follows since $H$ acts trivially on $\tilde M_\chi$
and the second is by \Theorem{\ref{galois-direct}}. The map $\phi$
come by the inclusion of the socle in $p^!_0(M_r)$ and because $\tilde
\pi_+ = (\pi_0)(p_0)_+$. Note that $H$ acts trivially on $X$ and $Y$
so that $(\pi_0)_+(N)^H= (\pi_0)_+(N^H)$ when $N$ is a
$\Dc_X[H]$-module; this implies the equality on the second line. The
penultimate equality follows from \Lemma{\ref{min-direct}}, and the
ultimate is because $M$ is simple with support $X$.
\end{proof}

\section{More decompositions, modules over liftable differential operators}
First we show that the vanishing trace module
$\Tc_\pi= \Ker (\Tr: L\to K)$ for a finite field extensions contains
no invariants with respect to the action of any non-zero derivation,
and also give an easy condition for the simplicity of $\Tc_\pi$. In
the Galois case we give explicit projection operators onto both the
isotypical and simple components of $\pi_+(M)$, knowing characters and
matrices of representations, respectively, of the ``true'' inertia
group of $M$ (see (\ref{central-ext})). Next the composed functor
$p_1^!(p_2)_+$ is studied, where $p_1$ and $p_2$ are finite morphisms
with the same target, using base change, giving in particular a
simplicity criterion for $\pi_+(M)$; this is related to two well-known
theorems by Mackey in representation theory. In the abelian case there
are explicit decompositions of $\pi_+(M)$ and $\pi^!(N)$. The last
subsection is concerned with the morphism $\pi_*(M)\to \pi_+(M)$ of
modules over the subring of liftable differential operators
$\Dc_Y^\pi\subset \Dc_Y$. If $\pi$ is uniformly ramified, then a
semisimple decomposition of $\pi_+(M)$ induces a semisimple
decomposition of $\pi_*(M)$. In the uniformly ramified sitation we
also get more precise information about the isomorphism
$\eta :\omega_{X/Y} \cong \omega_X\otimes_{\pi^{-1}(\Oc_Y)}
\omega_Y^{-1}$ in \Proposition{\ref{dual-iso}}.

\subsection{About the  vanishing trace module}
\subsubsection{No invariants  }
Letting $L/K$ be a finite field extension of characteristic 0 we have
the decomposition
\begin{displaymath}
  \pi_+(L) = K \oplus \Tc_\pi.
\end{displaymath}
We show that if $\partial$ is a non-trivial derivation of $K$, then
$\Tc^{\partial}_\pi=0$, where
$\Tc^\partial_\pi= \{t\in \Tc_\pi \ \vert \ \partial \cdot t =0 \}$.
\begin{proposition}\label{van-trace-prop}
  Let $L_1/L/K$ be field extensions of characteristic $0$, where $L/K$
  is finite, and $\bar L$ be the integral closure of $K$ in $L_1$. Let
  $\partial$ be a derivation of $K$ (and hence of $L$) and
  $\partial_1$ be an extension of $\partial$ to a derivation of $L_1$.
  Then
  \begin{enumerate}
  \item $\partial(L)\cap K = \partial (K).$
  \item If $l_1 \in L_1 \setminus K$, $\partial_1 (l_1) \neq 0$, and
    $\partial_1(l_1)\in K$, then $l_1 \not \in \bar K$. That is, if
    $\partial_1(l_1)\in K$ while $l_1\not \in K$, it follows that
    $l_1$ is transcendental over $K$.
  \item If $K^\partial \neq K$, then  $\Tc_\pi^{\partial} =0$.
  \item
    $ \{l\in \bar K \ \vert \  \partial (l)\in K \}=K$.
  \end{enumerate}
\end{proposition}

Here the assertion (2) is well-known.
\begin{proof} (1): We have $\pi_+(L)= K \oplus \Tc_\pi$
  \Prop{\ref{split-conn}}, where the point is that $\Tc_\pi$ is a
  $\Dc_K$-submodule. If $l= a+b$ where $a\in K$ and $b\in \Tc_\pi$,
  then $\partial (a)\in K$ and $\partial(b)\in \Tc_\pi$; hence if
  $\partial (l) \in K$, then $\partial(b)=0$.

  (2): Assume that $l_1 \in \bar K \setminus K $ and
  $\partial (l_1)\neq 0$. Let $p\in K[x]$ be the minimal polynomial,
  $p(l_1)=0$. One checks that that
  $q = p' + p_{\partial}/\partial(l_1) \in K[x] $ is non-zero. Since
  $q(\gamma) =0$ and $\deg q < \deg p$ this gives a contradiction.
  Therefore $l_1\not \in \bar K$.
  
  (3): Take $a\in K$ such that $\partial(a)\neq 0$, and assume that
  $b\in \Tc_\pi$ satisfies $\partial(b)=0$. If $b\neq 0$, then
  $l= a+ b\not \in K$ and $\partial (l)\in K$. Therefore
  $\partial (a)= \partial (l) =0$ by (2), which gives a contradiction.
  Hence $b=0$. 
  
  (4): This follows from (2).
\end{proof}
\subsubsection{Simple vanishing trace module}
One can ask when the vanishing trace part $\Tc_\pi$ in the
decomposition $\pi_+(\Oc_X)= \Oc_Y \oplus \Tc_\pi$ is simple. This is
a generic property so we can just as well again assume that $\pi$ is a
finite field extension, and \Corollary{\ref{non-galois}} immediately
gives a characterization of this property in terms of representations
of Galois groups.
\begin{corollary}\label{simple-L} Let $L/K$ be a finite field extension where $K$ is a
  finite transcendence over the base field $k$. Let $\bar L/K$ be a
  Galois cover with Galois group $G$, and $H$ be the Galois group of
  $\bar L/L$. The following are equivalent:
  \begin{enumerate}
  \item  $\Tc_\pi$ is a simple $\Dc_K$-module
  \item There exists only one non-trivial irreducible representation
    $V$ of $G$ such that the $H$-invariant space $V^H \neq 0$, and for
    this representation we have $\dim_k V^H=1$.
  \end{enumerate}
\end{corollary}

Combining with the classical branching rule for the symmetric group
results we now explain an application of \Corollary{\ref{simple-L}}.

Let $\bar X= \Spec k[x_1, \ldots , x_n]$ be the affine $n$-space. Let
the symmetric group $S_n$ act by permutating the variables and
$S_{n-1}$ be the subgroup that fixes the variable $x_n$. Put
$X= \bar X^{S_{n-1}}$, $Y = \bar X^{S_n}$, and let $\pi : X\to Y$ be
the map induced by the inclusion of invariant rings
$k[x_1, \ldots , x_n]^{S_n}\to k[x_1, \ldots , x_n]^{S_{n-1}}$. Let
$\Tr : L=k(x_1, \ldots , x_n)^{S_{n-1}}\to K= L^{S_n} $ be the trace
morphism, so that $\Tc_\pi = \pi_+(\Oc_X)\cap \Ker(\Tr)$; here we use the
canonical inclusion $\pi_*(\Oc_X)\subset \pi_+(\Oc_X)$,
$p\mapsto p \tr \otimes 1 \otimes 1$, to form the intersection. Define
the polynomial
\begin{displaymath}
  \delta_n = (n-1)x_n -\sum_{i=1}^{n-1} x_i.
\end{displaymath}

\begin{corollary} We have $\pi_+(\Oc_X) = \Oc_Y\oplus \Tc_\pi$, where
  $\Tc_\pi$ is a simple $\Dc_Y$-module. We have $\delta_n \in \Tc_\pi$, so
  that $\Tc_\pi= \Dc_Y \delta_n$.
\end{corollary}

\begin{proof}
  The branching rule for the inclusion $S_{n-1}\subset S_n$ can be
  described in terms of Young diagrams and ``removable boxes''. There
  is only one Young diagram of size $n$ that does not correspond to
  the trivial representation of $S_n$ and which moreover contains a
  removable box such that the remaining Young diagram of size $n-1$
  corresponds to the trivial representation (this is a hook where the
  first row is of length 2). Therefore the condition (2) in
  \Corollary{\ref{simple-L}} is satisfied, so that $\Tc_\pi$ is simple. It
  remains to see that $\delta_n\in\Tc_\pi $, since then we get
  $\Tc_\pi = \Dc_Y \delta_n$. Let $\bar L$ be the fraction field of
  $k[x_1, \ldots , x_n]$ and $\Tr_{\bar L/L}, \Tr_{\bar L/K}$ be the
  trace of $\bar L/L$ and $\bar L/K$, respectively. Since $x_n$ and
  $\sum_{i=1}^{n-1}x_i$ are $S_{n-1}$-invariant, it follows that
  $\Tr_{\bar L/L}(x_n)= (n-1)! x_n$,
  $\Tr_{\bar L/L}(\sum_{i=1}^{n-1}x_i)= (n-1)! \sum_{i=1}^{n-1}x_i$.
  Moreover,
  \begin{align*}
    \Tr_{\bar L/K} (x_i)&= \sum_{\sigma \in S_n} \sigma \cdot x_i=
                          (n-1)! \sum_{j=1}^n x_j,\\
    \Tr_{\bar L/K} (\sum_{i=1}^{n-1}x_i)&= (n-1)\cdot (n-1)! ( \sum_{i=1}^n x_i),
  \end{align*}
  where the first line implies the second. Since  $\Tr_{\bar L/K}= \Tr_{L/K} \circ \Tr_{\bar L/L}$ it follows that
\begin{displaymath}
n   \Tr_{L/K} (\delta_n) =\Tr_{\bar L/K} ( \delta_n)  = 0.
\end{displaymath}
\end{proof}
\begin{remark}
  The result also follows by applying
  \cite{kallstrom-bogvad:decomp}*{Prop. 5.4}; the latter result also
  shows that $\pi_+(M)$ is multiplicity free whenever $M$ is a simple
  $\Dc_L$-module such that $\bar L\otimes_L M \cong \bar L^m$ for some
  integer $m$ (as $\Dc_{\bar L}$-module). The element $\delta_n$ is
  the Young polynomial that corresponds to the hook-shaped Young
  tableaux where the first row contains the elements $\{1,n\}$.
\end{remark}
\subsection{Explicit decomposition for Galois morphisms} Let
$V^*_\chi$ be the dual of a finite-dimensional $k$-vector space of
dimension $n_{\chi}$,
 forming an irreducible representation of $\bar
G_M$ corresponding to $\chi \in \widehat {\bar G_M}$, and let $\chi$
also denote the corresponding character of $V_\chi$. Put $n_M= |\bar
G_M|$ and after selecting a basis of $V_\chi$, let $r^{\chi}_{\alpha
\beta}(t)$ be the matrix of the action of an element $t$ on $V_\chi$.
Set
\begin{displaymath}
  p^\chi = \frac {n_\chi}{n_M} \sum_{t\in \bar G_M} \chi (t) t
\end{displaymath} and
\begin{displaymath}
  p_{\alpha \beta}^\chi = \frac {n_\chi}{n_M} \sum_{t\in \bar G_M}
  r_{\alpha \beta}(t^{-1})t
\end{displaymath} defining $K$-linear maps $\pi_+(M) \to \pi_+(M)$,
where on the right $t$ denotes the natural action on $\pi_+(M)$.
Moreover, $p^\chi$ is $\Dc_K[\bar G_M]$-linear while $ p_{\alpha
\beta}^\chi$ is $\Dc_K$-linear. Put $M^\chi =\Imo (p^\chi)$ and
$M^{\chi}_{(\alpha)}= \Imo (p^\chi_{\alpha \alpha})$.

\begin{theorem}\label{explicit} Make the same assumptions as in
\Theorem{\ref{galois-direct}}.
  \begin{enumerate}
  \item The map $p^\chi $ is a projection $\pi_+(M) \to M ^\chi $ and
    $M ^\chi \cong V_\chi ^*\otimes_k M_\chi$. The module $M^\chi $ is
    the isotypical component of $M$ both regarded as $K[\bar G_M]$-module and
    as $\Dc_K$-module.
  \item The map $p^\chi_{\alpha \alpha}$ is a projection map of
  $\Dc_K$-modules;
    it is zero on $M^\psi$ when $\psi \neq \chi$. Its image
    $M^{\chi}_{(\alpha)}$ is contained in $M^\chi$ and $M^\chi$ is the
    direct sum of the $M^{\chi}_{(\alpha)}$ for $1\leq \alpha \leq
    n_\chi$. We have $ p^\chi = \sum_{\alpha} p^\chi_{\alpha \alpha}$.
  \item The map $p^\chi_{\alpha \beta }$ is zero on the $M^\psi$,
    $\psi \neq \chi$, as well as on the $M^{\chi}_{(\gamma)}$ for
    $\gamma \neq \beta$; it defines an isomorphism of $\Dc_K$-modules
    from $M^{\chi}_{(\beta)}$ onto $M^{\chi}_{(\alpha)}$.
  \item Let $m_1$ be an element $\neq 0$ of $M^{\chi}_{1}$ and let
    $m_\alpha= p^\chi_{\alpha 1}(m_1)\in M^{\chi}_{(\alpha)}$. The
    $m_\alpha$ generate a $K$-vector subspace $V_\chi(m_1)$ of
    $\pi_+(M)$, stable under $\bar G_M$ and of dimension $n_\chi$. For
    each $s\in \bar G_M$ we have
  \begin{displaymath}
    s \cdot m_\alpha = \sum_{\beta} r_{\beta \alpha}^\chi(s) m_\beta
  \end{displaymath} (in particular $V_\chi(m_1)$ is isomorphic to
$V_\chi$).
\item If $(m_1^{(1)}, \ldots , m_1^{(n_\chi)})$ is a basis of
$M^\chi_1$, we have
  a decomposition
  \begin{displaymath}
    M^\chi = \bigoplus_{\alpha =1}^{n_\chi} V_\chi(m_1^{(\alpha)})
  \end{displaymath} into simple $K[\bar G_M]$-modules (all isomorphic
  to $V_\chi(m_1)$), and
\begin{displaymath}
  M^\chi = \bigoplus_{\alpha =1}^{n_\chi} M^\chi_{(\alpha)}
\end{displaymath} into simple $\Dc_K$-modules.
  \end{enumerate}
\end{theorem}
\begin{remark}\label{remark-isotypical}
  The decomposition $\pi_+(M)= \oplus M^\chi$ is canonical, where the
  modules $M^\chi$ are simple $\Dc_K[\bar G_M]$-modules of
  multiplicity 1, while the decomposition in (5) is non-canonical. Any
  semisimple decomposition of an isotypical component
  $M^\chi\subset M$ (regarded as $\Dc_K$-module) is determined
  explicitly by the projection operators described above, by
  selecting a basis of $V_\chi$.
\end{remark}

\begin{proof}
  We regard $M$ as a $\Dc_K[\bar G_M]$-module of finite dimension as
  vector space over $K$. By \Theorem{\ref{galois-direct}} we have the
  semisimple $\Dc_K$-module $\pi_+(M)= \oplus_{\chi \in \widehat {\bar
  G_M}} V^*_\chi\otimes_k M_\chi$ where the $V_\chi\otimes_k M_\chi$
  are simple mutually non-isomorphic simple $\Dc_K[\bar G_M]$-modules,
  so that the isotypical component, both as $\Dc_K$- and $k[\bar
  G_M]$-module, is $ V^*_\chi\otimes_k M_\chi$. The proof can now be
  read off from the proof of \cite{serre:lin-rep}*{Thm. 8 and Prop.
  8}; note that we are working with a dual representation.
\end{proof}
\subsection{Composition of direct and inverse images}
We consider the setup
\begin{displaymath}
\pi : \Spec L \to \Spec L_i \xrightarrow{p_i} \Spec K, 
\end{displaymath}
where $L/L_i/K$, $i=1,2$ are finite field extensions, and $K/k$ is a
field extension of a finite transcendence degree over a field $k$ of
characteristic $0$. Let $q_i : \Spec (L_1\otimes_K L_2)\to \Spec L_i$
be the projection maps and
\begin{displaymath}
  s_j: \Spec L^{(j)}\to  \bigcup_{j=1}^r \Spec  L^{(j)} =  \Spec (L_1 \otimes_K L_2) 
\end{displaymath}
be the inclusion of the irreducible components $\Spec L^{(j)}$, where
$L^{(j)}\subset L$ are certain subfields that contain both $L_1$ and
$L_2$. Put $r_j= q_2 \circ s_j: \Spec L^{(j)}\to \Spec L_2$ and
$l_j= q_1 \circ s_j : \Spec L^{(j)}\to \Spec L_1$. Next result is
in the spirit of Mackey's restriction theorem for group
representations (see also \Proposition{\ref{mackey}}) .

If $G= \Aut(L/K)$ and $H_1, H_2 $ are subgroups of $G$ we let $S$ be a
set of representatives of the double cosets $H_1\backslash G / H_2$.
For each $s\in S$ put $H(s) = sH_2s^{-1}\cap H_1\subset H_1$. Let
$l_s : \Spec L^{H(s)}\to \Spec L_1$ be defined by the inclusion of
$L_1$ in $L^{H(s)}$. The map $H(s)\hookrightarrow H_2$,
$x \mapsto s^{-1}x s$, determines a homomorphism
$ r_s : \Spec L^{H(s)} \to \Spec L_2$.
\begin{theorem}\label{inv-dir} Let $M$ be a $\Dc_{L_2}$-module.
  \begin{enumerate}
  \item \begin{displaymath}
      (p_1)^!   (p_2)_+ (M ) = \bigoplus^r_{j=1}(l_j)_+ (r_j)^!  (M).
    \end{displaymath}
  \item Assume that $L/L_i$ ($i=1,2$) and $L/K$ are Galois and
    $H_1, H_2$ be subgroups of $G$ such that $L_i = L^{H_1}$. Then, as
    $\Dc_{L_1}$-modules,
      \begin{displaymath}
      (p_1)^!   (p_2)_+ (M ) =    \bigoplus_{s\in S} (l_s)_+ r_s^!(M)= \bigoplus_{s\in S} L^{H(s)}\otimes_{L_2}M.
    \end{displaymath}
  \item Make the assumptions in (2) and assume moreover that $L_i/K$ ($i=1,2$)
    are Galois. Let $L_{12}= L_1L_2$ be the compositum of $L_1$ and
    $L_2$. Then
    \begin{displaymath}
      (p_1)^!   (p_2)_+ (M ) =   \bigoplus_{s\in S} (l_s)_+
      r_s^!(M) = \bigoplus_{s\in S} (L_{12}\otimes_{L_2} M)_s, 
    \end{displaymath}
    where $(L_{12}\otimes_{L_2} M)_s$ is  $ L_{12}\otimes_{L_2} M$ as
    vector space over $K$. The $\Dc_{L_{12}}$-action is
    \begin{displaymath}
      P\cdot  (l\otimes m) = P^s (l\otimes m), \quad P^s = s P s^{-1},
    \end{displaymath}
    and the $\Dc_{L_1}$-action is  determined by the inclusion
    $\Dc_{L_1}\subset \Dc_{L_{12}}$.
  \item Assume that $\pi : \Spec L\to \Spec K$ is Galois with Galois
    group $G$.  If $M$ is a $\Dc_L$-module, then 
         \begin{displaymath}
         \pi^!\pi_+(M)=   \bigoplus_{g\in G} M_g,
       \end{displaymath}
       where $M_g$ is the twisted $\Dc_L$-module (see
       (\ref{inertia-section})).
  \end{enumerate}
\end{theorem}

\begin{proof}
  (1): We have that $\oplus_{j=1}^r (s_j)_+ s_j^! $ is the identity
  operation, so that by the base change theorem (here simply stating
  that $L_2\otimes_KL_1 \otimes_{L_1}M = L_2\otimes_K M$)
  \begin{align*}
    p_2^!(p_1)_+ (M) & = (q_1)_+ q_2^!(M)=  \bigoplus_{j=1}^r (q_1)_+
                       (s_j)_+s_j^! q_2^!(M) = \bigoplus_{j=1}^r (q_1\circ s_j)_+ (q_2
                       \circ s_j)^!(M)  \\&= \bigoplus_{j=1}^r (l_j)_+r_j^!(M).
  \end{align*}

  (2): Define the  homomorphism 
  \begin{displaymath}
\psi :   L^{H_1}\otimes_K L^{H_2} \to  \bigoplus_{s\in S} L^{H(s)},
\quad x\otimes y \mapsto (x sys^{-1})_{s\in S},
\end{displaymath}
where we note that $L^{H_1} \subset L^{H(s)}$, and
$x \mapsto sxs^{-1}$ defines a map $H(s)\to H_2$, and therefore a map
$L^{H_2}\to L^{H(s)}$, $x \mapsto sx s^{-1}$. To see that $\psi$ is an isomorphism it
suffices to see that it induces an isomorphism of $K$-vector spaces.
As $K$-vector spaces we have $L^{H_1} = K[H_1{\backslash } G]$,
$L^{H_2} = K[G/H_2]$, and $L^{H(s)} = K[G/H(s)]$. We can now factor
$\psi$ by natural maps
\begin{align*}
\psi &:   K[H_1{\backslash } G] \otimes_K K[G/H_2] \xrightarrow{a}  \bigoplus_{s\in
  S} K[G] (\bar 1 \otimes \bar s)\xrightarrow{b}  \bigoplus_{s\in S}
  K[G/H(s)],\\  & \bar x\otimes \bar y   \mapsto (\bar x s\bar y
  s^{-1})_{s\in S},  
\end{align*}
where it is straightforward to see that $a$ and $b$ define
isomorphisms of $K$-vector spaces. The assertion now follows from (1),
where $L^{(s)}= L^{H(s)}$.

(3): Now $H(s) = H_1\cap H_2$ and we have isomorphisms
\begin{displaymath}
 L_{12}= L^{H_1\cap H_2} \to    L^{(s)},\quad x \mapsto sxs^{-1}.
\end{displaymath}
This isomorphism also induces an isomorphism
$\Dc_{L_{12}}\to \Dc_{L^{(s)}}$, $P \mapsto P^s = sPs^{-1}$. Now the
assertion follows from (2).
  
(4): Here $H_1= H_2= \{e\}$ are  trivial groups, so the assertion follows from (3).
\end{proof}

There is also a counterpart to Mackey's irreducibility criterion which
at the same time expands the scope of
\Proposition{\ref{simpledirect}}. Let $\pi: \Spec L \to \Spec K$ be a
finite field extension of positive transcendence degree (over $k$) and
characteristic $0$, and let $\bar \pi : \Spec \bar L \to \Spec K $ be
a Galois cover of $\pi$ with Galois group $G$. Let $H$ be the Galois
group of $\bar L/L$, $S$ be a set of representatives of
$H\backslash G / H$, and put $H(s) = sHs^{-1}\cap H$ for $s\in S$.
Below we regard $\bar L^{H(s)}\otimes_{L} M$ as a $\Dc_L$-module by
the inclusion $L \subset L^{H(s)}$.
\begin{corollary}\label{irr-crit}
  Let $M$ be a simple $\Dc_L$-module of finite dimension over $L$.
  The following conditions are equivalent:
  \begin{enumerate}
  \item  $\pi_+(M)$ is  simple.
  \item  The $\Dc_L$-modules $M$ and  $\bar L^{H(s)}\otimes_{L} M $ are
    disjoint for each $s\in S\setminus H$ (so that
    $Hom_{\Dc_L}(M, \bar L^{H(s)}\otimes_{L} M )=0$).
  \end{enumerate}
\end{corollary}
\begin{proof} We can work in the Grothendieck group $K_1$ of
  $\Dc$-modules, using the natural pairing
  $<>: K_1 \otimes_{\Zb} K_1 \to \Zb $ determined by requiring for
  simple modules $N_i$ that $<[N_1], [N_2]> $ equals $1$ when
  $[N_1]= [N_2]$ and otherwise $0$. Then $\pi_+(M)$ is simple if and
  only if $ <[\pi_+(M)], [\pi_+(M)]>=1$. By adjointness and
  \Proposition{\ref{inv-dir}},(2), with $H= H_1 = H_2$ being the
  Galois group of $\bar L/L$ (and $p_1= p_2 = \pi$), we get
  \begin{align*}
    <[\pi_+(M)], [\pi_+(M)]>&= < [M], [\pi^!\pi_+(M)]> = \sum_{s\in S}
                              <[M], [\bar L^{H(s)}\otimes_L M] >\\ &
                                                                     =1  + \sum_{ s\in S
                                                                     \setminus
                                                                     H}
                                                                     <[M],
                                                                     [\bar L^{H(s)}\otimes_L M] >.
  \end{align*}
  This implies the assertion.
\end{proof}
\subsection{Abelian extensions}\label{abelian-ext} It can be
difficult to compute the central extension $\bar G_M$ of the intertia
group $G_M\subset G$ of a simple $\Dc_L$-module $M$ corresponding to a
Galois extension $\pi : \Spec L \to \Spec K $, with Galois group $G$,
and then also its irreducible representations $V_\chi$,
$\chi \in \widehat {\bar G}_M$, appearing in
\Theorem{\ref{galois-direct}}. Suppose now instead that we know  one
simple constituent $N$ of $\pi_+(M)$ and also the complete decomposition
\begin{displaymath}
  \pi_+(L) = \bigoplus_{\chi \in \hat G}  V_\chi^* \otimes_k L_\chi.
\end{displaymath}
Then 
\begin{displaymath}
  \pi_+\pi^!(N) = \pi_+(L) \otimes_K
  N= \bigoplus_{\chi \in \hat G } V_\chi^* \otimes_k
  L_\chi\otimes_K  N,     
\end{displaymath}
so that if $M$ is a simple component of $\pi^!(N)$ then each simple
module in $\pi_+(M)$ is isomorphic to a submodule of some
$L_\chi\otimes_K N$. Here the latter module is semisimple but need not
be simple when $\pi$ is non-abelian.

\begin{proposition} Assume that $\pi: \Spec L \to \Spec K$ is abelian
  and $M$ be a simple holonomic $\Dc_L$-module. There exist integers
  $m$ and $e$  such that: 
  \begin{enumerate}
  \item 
    \begin{displaymath}
      \pi_+(M) = \bigoplus_{i=1}^m (L_{\chi_i}\otimes_K N)^e,
    \end{displaymath}
    where $L_\chi\otimes_K N$ are simple $\Dc_K$-modules. Here $N$ is
    a simple submodule of $\pi_+(M)$ and the $L_{\chi_i}$ are simple
    submodules of $\pi_+(L)$.
    \item  $e^2\cdot  m = |G_M|$, the order of the
      inertia group of $M$.
      \item   $ \rank_K (N) = \frac {\rank_L (M)}{e\cdot  m}$ and
        $\rank_K (L_\chi) =1$.
  \end{enumerate}
\end{proposition}
\begin{proof}
  In this situation $L_\chi \otimes_K N$ is simple
  since $N$ is simple and $\rank_K L_\chi = \dim V_\chi =1$. Therefore for some
  integers $m_\chi$
  \begin{displaymath}
    \pi_+(M) = \bigoplus (L_{\chi}\otimes N)^{ m_\chi}
  \end{displaymath}
  where the sum runs over certain $\chi \in \hat G$. Moreover, 
\begin{displaymath}
\pi^!(L_\chi \otimes_K N) = \pi^!(L_\chi)\otimes_L \pi^!(N) = L
\otimes_L \pi^!(N) = \pi^!(N).
\end{displaymath}
Hence the multiplicity
$m_\chi =  [\pi_+(M): L_\chi \otimes_K N] = [\pi^!(N): M]= e$ is independent of
$\chi$.
Thus, by \Theorem{\ref{inv-dir}}, (4), 
\begin{displaymath}
\bigoplus_{G/G_M} \bigoplus_{G_M} M_g = \bigoplus_{g\in G} M_g = \pi^!\pi_+(M) = \bigoplus_{i=1}^m \bigoplus_{g\in G/G_M} M_g^{e^2}, 
\end{displaymath}
so that equaling the multiplicity of $M$ in both sides, we get
$|G_M| = e^2m$, proving (2). Also (3)   is now evident. 
\end{proof}

\begin{example}(Kummer extensions) Consider the field extension
  $\pi: K \to L= K(a_1^{1/n}, \ldots , a_r^{1/n})$, where $a_i$
  algebraically independent elements in $K$. It can be factorized into
  cyclic extension
  $\pi_i : K_{i-1}=K(a_1^{1/n}, \ldots , a_{i-1}^{1/n})\subset K_i =
  K(a_1^{1/n}, \ldots , a_i^{1/n})$. Then, putting $x_i = a_i^{1/n}$,
  \begin{align*}
    \pi_+ (L) &= (\pi_1)_+\cdots (\pi_{r-1})_+(\pi_r)_+(L) =
                (\pi_1)_+\cdots (\pi_{r-1})_+ (\bigoplus_{i=0}^{n-1} \Dc_{K_{r-1}}x_r^i)
    \\
              &=
                (\pi_1)_+\cdots (\pi_{r-2})_+(\bigoplus
                (\pi_{r-1})_+(\Dc_{K_{r-1}}x_r^i))\\ &=
                                                       (\pi_1)_+\cdots (\pi_{r-2})_+(\bigoplus_{i_{r-1}=0}^{n-1}\bigoplus_{i_{r}=0}^{n-1}
                                                       \Dc_{K_{r-2}}x_{r-1}^{i_{r-1}}x_r^{i_r})
                                                       = \bigoplus_{i_j=0, j=1, \dots, r}^{n-1}\Dc_K \prod_{j=1}^r x_{j}^{i_j}.
  \end{align*}
\end{example}

We now consider inverse images of abelian Galois morphisms when one
has information about the inertia group of the module. It gives a 
more precise result than in \Theorem{\ref{cliffordtheorem}}.

Given field extensions
\begin{displaymath}
  \Spec L \xrightarrow{\pi} \Spec K \xrightarrow{r} \Spec K_0,
\end{displaymath}
put $p= r\circ \pi$, and let $G_\pi$, $G_r$ and $G_p$ denote the
corresponding Galois groups. Assume that $\pi$ and $r$ are Galois, so
that $G_\pi$ is a normal subgroup of $ G_p$ and $G_r= G_p/G_\pi$. Say
that $\pi$ is {\it minimal Galois with respect to $r$ } if for any
field $K \subset L_1 \subset L$ such that $L_1/K_0$ is Galois it
follows that $L_1= L$ or $L_1= K$. As before, if $N$ is a
$\Dc_K$-module, let $G_N$ be its inertia subgroup in $\Aut(L/k)$.
\begin{theorem}\label{abelian-clifford}  Assume that $G_\pi$ is abelian and  $N$ be a simple
  holonomic $\Dc_K$-module such that its inertia group $G_N$ contains 
  $G_\pi  $. Assume moreover that   $\pi$ is minimal Galois with
  respect to $r: \Spec K \to \Spec L^G$. Then one of the following holds:
  \begin{enumerate}
  \item $\pi^!(N)$ is simple.
  \item $\pi^!(N) = M^e$, where $M$ is simple and
    $e^2=\deg \pi = |G_M|$.
    \item
      \begin{displaymath}
        \pi^!(N)=\bigoplus_{g\in G_\pi} M_g,
\end{displaymath}
where $M$ is  simple
      and  $M_g \not \cong M_{g'}$     when $g \neq g'$.
  \end{enumerate}
\end{theorem}
\begin{proof} Let $G_M$ be the inertia group in $G_\pi$ of a simple
  module $M$ in $\pi^!(N)$. By \Theorem{\ref{cliffordtheorem}}
  \begin{displaymath}
    \pi^!(N) = \bigoplus_{i=1}^t (g_i\otimes M)^e 
  \end{displaymath}
  where $t= |G_\pi/G_M|$ and $e $ divides $|G_M|$. Since $G_M\cap G_\pi$ is normal
  in the abelian group $G_\pi$ it follows that $L^{G_M}$ is Galois
  over $L^{G}$. By the minimality condition either $L^{G_M}= L$ or
  $L^{M}= K$. In the first case $G_M= \{1\}$ and thus $ t= |G_\pi|$
  and $e=1$,  corresponding to (3).  In the  latter case, $G_M= G_\pi$,
  implying $(2)$ or (1). 
\end{proof}

\begin{corollary}\label{prime-cor}
  If the Galois group $G_\pi$ is of prime order $p$ and $N$ is a simple
  $\Dc_K$-module, then either
  \begin{enumerate}
  \item $\pi^!(N)$ is simple
    \item $\pi^! (N) = \bigoplus_{i= 1}^p N_i$, where $N_i \not \cong
      N_j$ when $i \neq j$. 
  \end{enumerate}
\end{corollary}
\begin{proof}
  Here $L/K$ is minimal Galois relative to $K=K$ so one can apply
  \Theorem{\ref{abelian-clifford}}, where case (2) cannot occur since $|G_\pi|=p$ is not a square.
\end{proof}

\begin{remark} Taken together with the equivalence
  \Theorem{\ref{equivalence}} below the above results imply
  corresponding assertions in ``Clifford Theory''
  \cite{isaacs:character}*{Th. 6.18, Cor. 6.19}, but notice that the
  scope of the above result is greater since the $\Dc_K$-modules need
  not arise from irreducible representation of a finite group.
\end{remark}

\subsection{Decomposition over liftable differential operators
}\label{sec:5} We have the subring $\Dc_Y^\pi$ of liftable
differential operators in $\Dc_Y$ (\ref{diffoperators}), and for
any coherent $\Dc_X$-module $M$ we can define the map
\begin{equation} \Theta_M: \pi_*(M) \to \pi_+(M), \quad m \mapsto
\Theta(1) \otimes 1 \otimes m,
\end{equation} where $\Theta $ is described in equation (\ref{S}).

\begin{proposition}\label{epinymous}
  The map $\Theta_M$ is a homomorphism of $\Dc_Y^\pi$-modules. On the
  étale locus, $\Theta_M$ is an isomorphism of $\Dc_{Y_0}$-modules.
\end{proposition}

\begin{proof} Clearly both sides in the map are left
$\Dc_Y^\pi$-modules. If $\partial$ is a liftable
  derivation, then $\partial (\Theta(1) \otimes 1 \otimes m)=
  \Theta(1) \otimes 1 \otimes \partial m$ (see also
  \Remark{\ref{right-left}}), therefore the restriction of $\Theta_M$
  to $Y_0$ is a homomorphism of modules over the ring $\Dc_{Y_0}^\pi =
  \Dc(T_{Y_0}^\pi)$. If $M= \Dc_X$ then $\pi_+(\Dc_X) =
  \pi_*(\omega_{X/Y}\otimes_{\pi^{-1}(\Oc_Y)} \pi^{-1}(\Dc_Y)) =
  \pi_*(\omega_{X/Y})\otimes_{\Oc_Y}\Dc_Y$ and we have the map
  \begin{displaymath}
    \psi :\pi_*(\Dc_X) \to \pi_*(\omega_{X/Y})\otimes_{\Oc_Y}\Dc_Y,
    \quad P \mapsto \phi_i \tr_\pi \otimes P_i,
  \end{displaymath} where a locally defined section in $X$ (over $Y$)
  $\sum \phi_i \otimes P_i \in \pi_* ( \Oc_X\otimes_{\pi^{-1}(\Oc_Y)}
  \pi^{-1}(\Dc_Y))\cong \Diff (\pi^{-1}(\Oc_Y), \Oc_X) $ is the
  differential operator that is induced by section $P$ of $
  \pi_*(\Dc_X)$. Since the action of a section of $\Dc_Y^\pi$ on a
  section of either side of the map $\psi$ is determined by its
  restriction to $Y_0$, and we know that the restriction of $\psi$ to
  $Y_0$ is a homomorphism, it follows that $\psi$ is a homomorphism of
  $\Dc_Y^\pi$-modules. Therefore the assertion follows when $M$ is
  locally free over $\Dc_X$. Since any coherent $\Dc_X$-module is
  locally provided with a finite locally free resolution (locally the
  category of coherent $\Dc_X$-modules is of finite global homological
  dimension), and the map $\Theta_M$ is functorial in $M$, it follows
  that $\Theta_M$ is a homomorphism for any coherent $M$. Finally, $
  \Dc_{Y_0}^\pi=\Dc_{Y_0}$, so the restriction of $\Theta_M$ is a
  morphism of $\Dc_{Y_0}$-modules $\Theta_M^0: i^*(\pi_*(M))=
  (\pi_0)_*j^*(M)\to i^!(\pi_+(M)) = (\pi_0)_+j^!(M)$. It is easy to see
  that this  is an isomorphism when $M= \Dc_X$, and again since $\Theta_M$ is
  functorial in $M$, it follows that $\Theta_M^0$ is an isomorphism.
\end{proof} 

When $M$ is a connection it can be more natural to study the $\Oc_Y$-
coherent module $\pi_*(M)$ than $\pi_+(M)$, where the latter tends to
be non-coherent when $\pi$ is ramified. On the other hand, the ring
$\Dc_Y^\pi$ is not easy to get a grip on\footnote{This is in contrast
  to the ring generated by the easily described liftable tangent
  vector fields $T^\pi_Y= T_Y(I_{D_\pi})$, formed as the vector fields
  that preserve the ideal of the disciminant; for a discussion of
  $T_Y^\pi$ see \cite{kallstrom:liftingder} and references therein.}.
However, when $X/Y$ is fairly close to being an invariant map
$X \to X^G$, generators of $\Dc_Y^\pi$ admit a direct construction,
and we shall see that in certain cases the semisimple decomposition of
$\pi_+(M)$ induces a semisimple decomposition of $\pi_*(M)$ as
$\Dc_Y^\pi$-module.

Say that a morphism $\pi: X\to Y$ is {\it uniformly ramified} if for
any prime divisor $D$ on $Y$ we have
$\pi^{-1}(D) = \sum_{i=1}^s r_iE_i$, where
$r_1 = \cdots = r_s$.\footnote{This notion was also employed in
  \cite{kallstrom:liftingder} for a different purpose.}

This condition is good for ensuring a nice Galois cover
$\tilde X \to X \to Y $ of $\pi$, as noted by F. Knop
\cite{knop:gradcofinite}; see also (\ref{galois-section}) and
(\ref{general-finite}). Uniformly ramified morphisms naturally occur
as Galois extensions $\bar X \to Y$ and more generally as
factorizations $ \tilde X \to X \to Y$, where $X= \tilde X^H$ and $H$
is a subgroup of the Galois cover $\tilde X/X/Y$ of $X/Y$ that does
not contain any pseudo reflections. It turns out that any uniformly
ramified morphism is of this form.

Let $\bar \Dc_X$ and $\bar \Dc^\pi_X$\footnote{We use the notation in
  [loc. cit.]} be the graded commutative rings associated to the order
  filtrations of the rings $\Dc_Y^\pi \subset \Dc_X$ .
\begin{proposition}\label{knop-prop} The following are equivalent for a
  finite morhism $\pi: X\to Y$ of smooth varieties:
  \begin{enumerate}
  \item $\pi$ is uniformly ramified.
  \item The induced map $\tilde X \to X$ is étale and $\tilde X$ is
  smooth.
\end{enumerate} Under these equivalent conditions we have:
\begin{enumerate}
\item $\Dc_Y^\pi = \Dc_{\tilde X}^G$ and $\Dc_X = \Dc_{\tilde X}^H$.
\item $\Dc_Y^\pi$ is a simple ring.
\item $\bar \Dc_X$ is of finite type over $\bar \Dc^\pi_X$ and $\Dc_X$
  is finitely generated over $\Dc_Y^\pi$.
\end{enumerate}
\end{proposition}

Assume that $A \to B$ is an inclusion of graded rings where $B$ is a
polynomial ring, such that $X= \Spec B \to \Spec A$ is uniformly
ramified. Then the Galois cover $B \to \bar B $ gives a map
$\bar \pi : A \to \bar B$, where $A= \bar B^G$ for a finite group $G$.
So that generators of $\Dc_A^{ \pi} = \Dc_A^{ \bar \pi} $ are given as
in \cite{levasseur-stafford:invariantdiff}*{Th. 5} and
\cite{knop:gradcofinite}*{Th. 7.3}
\begin{proof}
  See \cite{knop:gradcofinite}*{Prop 3.3} for the assertion that $\pi$
  is uniformly ramified if and only if $\tilde X/X$ is unramified in
  codimension $1$. Since $\tilde X$ is normal and $X$ is smooth, it
  follows from Zariski-Nagata's purity theorem that $\tilde X \to X$
  is étale, and as $X$ is smooth, it follows that $\tilde X$ is
  smooth. The remaining assertions  are explained in [loc. cit.].
\end{proof}

By \Proposition{\ref{knop-prop}}, (3), if $M$ is a coherent
$\Dc_X$-module, then $\pi_*(M )$ is a coherent $\Dc_Y^\pi$-module. If
$M$ is a connection along the ramification locus we can be more precise.
\begin{theorem}\label{coh-decomp} Let $\pi : X \to Y$ be a uniformly
  ramified morphism of smooth varieties and $M$ be a simple
  $\Dc_X$-module which generically is non-zero and of finite type over
 fraction field of $X$. Assume also that the stalk $M_{x}$ is of
  finite type over $\Oc_{X,x}$ when $x$ is a generic point of the
  ramification locus $B_\pi$. Let $\pi_+(M)= \oplus M_i^{n_i}$ be a
  semisimple decomposition \Th{\ref{decomposition-thm}}, where the
  $M_i$ are simple $\Dc_Y$-modules, and put
  $N_i = \Theta^{-1}(M_i) \subset \pi_*(M)$. Then
  \begin{displaymath}
    \pi_*(M) = \oplus N_i^{n_i},
  \end{displaymath}
where the $N_i$ are simple $\Dc_Y^\pi$-modules.
\end{theorem}
\begin{remark}
  \begin{enumerate}
  \item In \cite{kallstrom-bogvad:decomp} we study
    $\Dc_Y^\pi$-submodules of $\pi_*(\Oc_X)$ for invariant maps
    $\pi : X\to X^G$ of a finite group $G$ acting linearly on an
    affine space $X= \Ab^n$.
  \item Question: are the $\Dc_Y^\pi$-modules $N_i$ of finite length
    or even simple when $M$ is a torsion free holonomic
    $\Dc_X$-module?
  \end{enumerate}

\end{remark}
\begin{proof} 
  Let $N^0_i $ be a non-zero submodule of the $\Dc^\pi_Y$-module $N_i$
  and put $\bar N_i = N_i/ N_i^0$, so that by
  \Proposition{\ref{epinymous}}, $\supp \bar N_i\subset D_\pi$.
  Assuming $D_\pi$ is non-empty, let $y$ be a point in the
  discriminant $D_\pi$. Since $\pi$ is uniformly ramified we have
  $\pi^{-1}(y)= B_\pi\cap \pi^{-1}(y) $, implying in particular that
  $M_x$ is of finite type when the height $\hto(x)\leq 1$ and
  $x\in B_\pi$, and hence $M_x$ is of finite type for all points of
  height $\leq 1$ in $B_\pi$. Since $M$ is simple it follows that
  $M_x$ is of finite type for any point $x$ in $B_\pi $. Hence
  $\pi_*(M)_y$ is of finite type over $\Oc_{Y,y}$ for all points $y$
  in $D_\pi$, hence $(\bar N_i)_y$ is a $\Dc^\pi_{Y,y}$-module that is
  of finite type over $\Oc_{Y,y}$. Moreover,
  $\Ann_{\Oc_{Y,y}}((\bar N_i)_y)$ is a non-zero 2-sided ideal of the
  simple ring $\Dc_Y^\pi$ \Prop{\ref{knop-prop}}, since $\pi$ is
  uniformly ramified; hence
  $\Ann_{\Oc_{Y,y}}((\bar N_i)_y)= \Dc^\pi_{Y,y}$. This implies that
  $(\bar N_i)_y=0$ when $y\in D_\pi$, and hence $\bar N_i=0$.
\end{proof}
\Proposition{\ref{dual-iso}} can be made more precise when the morphism is
uniformly ramified.
\begin{corollary}\label{cor-duality} Let $\pi: X\to Y $ be a uniformly
ramified morphism of smooth varieties.
  The homomorphism in \Proposition{\ref{dual-iso}} induces an
  isomorphism of constant sheaves of rank 1
\begin{displaymath}
  \eta' : Hom_{\Dc_Y^\pi}(\pi_*(\Oc_X),\Oc_Y) \to
  \Hom_{\Dc_Y^\pi}(\omega_Y,\pi_* (\omega_{X})),
\end{displaymath} where on the right  (left) $\Dc_Y^\pi$ acts on $\omega_X$
( $\omega_Y$) from the right (left). The section $\Theta $ in \thetag{S} is
a global section of
\begin{displaymath}
  \Hom_{\Dc_Y^\pi}(\omega_Y,\pi_* (\omega_{X}))=
  Hom_{\Dc_Y^\pi}(\pi^{-1}(\omega_Y), \omega_{X}) \subset
  \omega_{X}\otimes_{\Oc_X}\pi^*(\omega_Y),
\end{displaymath} and forms a basis of the constant sheaf.
\end{corollary}
\begin{proof}
  It remains only to prove that $\eta'$ is an isomorphism of constant
  sheaves of rank 1. Let $D^K$ denote vector space dual
  over $K$. We have (as detailed below)
\begin{align*}
  Hom_{\Dc^\pi_Y}(\pi_*(\Oc_X), \Oc_Y) &= Hom_{\Dc_Y}(\pi_+(\Oc_X),
                                         \Oc_Y)= Hom_{\Dc_K}(L, K)
  \\&= Hom_{\Dc_K}( D^K(K) , D^K(L)) =
      Hom_{\Dc_K}(K, L)\\ & = L^{T_{K/k}} = k.
\end{align*} The first equality follows from
\Theorem{\ref{coh-decomp}}, the second since $L$ and $K$ are  semisimple
$\Dc_K$-modules and, letting $j$ be the inclusion of the generic point
in $Y$, the functor $j_{!+}$ fully faithful,  and we have
$j_{!+}(L) = \pi_+(\Oc_X)$ and $j_{!+}(K) = \Oc_Y$. The third
equality follows since $D^K$ is an equivalence of categories and the fourth follows
since $D^K(L)\cong L$ and $D^K(K)=K$. The last equality follows from
\Lemma{\ref{int-closure}} (below), since $k$ is algebraically closed in $K$.

Applying the functor $\omega_K\otimes_K\cdot $ (from left- to right
$\Dc_K$-modules), so that $\omega_L= \omega_K\otimes_KL$, we get
$Hom_{\Dc_K}(K, L)= Hom_{\Dc_K}(\omega_K, \omega_L)$. One can now work
backwards to arrive at the right side of the isomorphism $\eta'$.
\end{proof}

\section{Covering $\Dc$-modules} \label{alg-mod} Say that a
  $\Dc_L$-module $M_1$ is {\it diagonalizable} if there exists an
  isomorphism
  \begin{displaymath}
    M_1 = \bigoplus_i \Lambda_i,
  \end{displaymath}
  where the modules $\Lambda_i$ are of rank $1$. 
  \begin{definition} Let $X$ be a smooth variety with fraction field
    $K= \Oc_{X,\eta}$. A $\Dc_X$-module $M$ is a {\it covering module}
    if it is torsion free over $\Oc_X$ and there exists a finite field
    extension $L/K$ such that $M_1=L \otimes_K M_\eta $ is diagonalizable.
    Let $\Mod_{cov}(X)$ be the category of covering modules on a
    smooth variety $X$.
\end{definition} 
A {\it monomial } module is of the form $\pi_+(\Lambda)$ for some
finite map $\pi$ and simple $\Dc_X$-module $\Lambda$ of generic rank
$1$.

When the modules $\Lambda_i$ all are isomorphic to $ L$ we say that
$M$ is generically ($L$-) {\it étale trivial} (or isotrivial), meaning
that that $M$ has a complete set of algebraic solutions in $L$.
Moreover, if $L/K$ is a finite Galois extension with Galois group $G$,
it is well-known that the category of ($L$-) étale trivial
$\Dc_K$-modules is equivalent to the category of $k$-linear
representations of a finite-group $G$; this Picard-Vessiot equivalence
is described in (\ref{finite-groups}). The differential Galois group
of a covering $\Dc_K$-module is an extension of a finite group by a
product of multiplicative groups.

Any covering module $M$ is generically semisimple
\Th{\ref{semisimple-inv}}, and its socle is the minimal extension from
the generic point. If the integral closure $\pi : X_M \to X$ of $X$ in
$L$ is smooth and $\pi^+(M)= \oplus \Lambda_i$, then the socle is
contained in $\oplus \pi_+ (\Lambda_i)$, so that in particular any
simple covering module is a submodule of a monomial module.\footnote{
  If $X_M$ is non-smooth, then one can take a resolution
  $p: \tilde X_M \to X_M$ so that the socle of $M$ belongs to the
  torsion free part of $(\pi\circ p)_+(\oplus \tilde \Lambda_i)$, for
  some simple invertible $\Dc_{\tilde X_M}$ modules
  $\tilde \Lambda_i$.} Moreover, the category $\Con_{cov}(X)$ of
connections that are also covering modules is semisimple
\Th{\ref{alg-semi}}, where the category of étale trivial connections
$\Con_{et}(X)$ forms a semisimple subcategory.

Monomial modules are a great source of covering modules. On the other
hand, it is a very challenging problem to determine for which
parameters $\alpha$ a module $M_\alpha= \Dc_X/I_\alpha$ is étale
trivial, given a parametrized left ideal $I_\alpha$ of $\Dc_X$. For
easy examples, consider finite surjective maps to the complex line
$\pi : X \to \Cb$, and $t$ be a coordinate of $\Cb$. Let $\Dc_{\Cb}$
be the ring of algebraic differential operators in the variable $t$
and $M_\alpha = \Dc_{\Cb}/\Dc_{\Cb}(t\partial_t- \alpha)$,
$E =\Dc_{\Cb}/\Dc_{\Cb}(\partial_t -1)$. The module $\pi^!(M_\alpha)$
is generically invertible, and if $t^\alpha $ belongs to the function
field of $X$ it is also generically trivial; $M_\alpha$ is étale
trivial if and only if $\alpha$ is a rational number. The module
$\pi^!(E) $ is a generically invertible simple non-trivial
$\Dc$-module that is not étale trivial. 

When $K= k((t))$, then any semisimple $\Dc_K$-module is a covering
module \cite{levelt}.

We mention also the classical problem studied by Schwarz, Fuchs and
Klein to determine the étale trivial rank 2 connections on the
projective line:
$\Con^2_{et}(\Pb_{\Cb}^1\setminus S) \subset \Con^2(\Pb^1\setminus
S)$,
where $S$ i a finite set of points in $\Pb^1_\Cb$. The case $|S|=3$
results in the famous Schwarz list of étale trivial hypergeometric
modules $\Dc_{\Pb_{\Cb}}$-modules, and for general finite subsets $S$
of a smooth projective curve $C$ a decision procedure is devised in
\cite{dwork-baldassarri:2nd, baldassarri2nd} (based on Klein; see
references in [loc. cit.]) for a connection over $C\setminus S$ to be
étale trivial, and stress that one needs in particular to decide when
connections of rank $1$ are étale trivial. Therefore, as a tiny first
step to understand why a connection belongs $\Con_{et}(X)$ it is
described how one can determine the étale trivial objects in its
subcategory $\Io(X)= \Con^1(X)$ of connections of rank 1, also when
$\dim X >1$, see (\ref{etaletrivial}). The whole category $\Io(X)$ -
the basic covering modules - is studied in \Section{\ref{rank1}}, and
its set of isomorphism classes is described in
\Theorem{\ref{class-rank-1}} when $X$ is affine, in terms of a smooth
completion. Now in some sense the simplest objects in
$\Io(K)= \Io(\Spec K)$ are the exponential ones, where $K/k$ is a
finitely generated field extension, and in this context a classical
theorem by Liouville is extended by describing the objects
$M \in \Io(K)$ such that $L\otimes_KM\cong \Dc_L e^\phi$,
$\phi \in L$, for some elementary field extensions $L/K$, i.e. $M$
becomes exponential after pulling back to $L$ \Th{\ref{liouville}}.


\subsection{Global sections of differential forms and connections on a
  projective variety }
Since closed 1-forms on a projective variety $X$ are important for the
structure of $\Dc$-modules that are of rank 1 over $\Oc_X$ we insert a
general result about the cohomology of closed forms on $X$, which can
be regarded as a natural extension of the Hodge decom\-position
theorem; there is no claim for essential originality since the proof
relies on the usual Hodge theorem. In particular, global 1-forms on
$X$ are closed, and we show also that global sections of a connection
on $X$ are constant. It was surprising not to find the latter
assertion in the literature (though it must be known), and although
the former is well-known it was frustrating not to find an algebraic
argument.
\subsubsection{Global sections of differential forms.} 
Let $\Omega_X^{m,cl}$ be the subsheaf of closed $m$-forms in
$\Omega_X^{m}$, where we single out the sheaf of closed 1-forms
$\Omega_X^ {cl}= \Omega_X^{1,cl}\subset \Omega_X$.
\begin{theorem}\label{global-forms}
  \begin{enumerate}
    \item If $X/k$ is a smooth projective variety, then global m-forms are
  closed,
    \begin{displaymath}
      \Gamma(X, \Omega^m_X) = \Gamma(X, \Omega^{m,cl}_X).
\end{displaymath}
\item Let $X/\Cb$ be a smooth complex projective variety and
  $\Omega_{X_a}^{m,cl}$ be the sheaf of closed $m$-forms on the
  corresponding complex analytic Kähler manifold $X_a$. Then
  \begin{displaymath}
    H^n(X_a, \Omega_{X_a}^{m,cl}) = \bigoplus_{\substack{j\geq m\\ i+j = n+m}}
    H^i(X, \Omega^j_X).
  \end{displaymath}
  \end{enumerate}
\end{theorem}
Clearly, $\Omega_{X}^{0,cl}$ is flasque and
$\Omega_{X}^{N,cl}= \Omega_{X}^{N}$ when $N= \dim X$, but a
description of $H^n(X, \Omega_{X}^{m,cl}) $ when $0 <m < N $ and
$n\geq 1$ is more complicated due to the fact that the algebraic
de~Rham complex is not acyclic.

\begin{proof}
  (1): By the Lefschetz principle one can reduce to the case $k= \Cb$
  so that $X$ can also be regarded as a compact Kähler manifold, and
  as such it is again denoted $X_{a}$. We have isomorphisms
\begin{displaymath}
  \Gamma(X, \Omega^m_X)\cong  \Gamma (X_a, \Omega^m_{X_a}) \cong \Gamma(X_a, \Omega_{X_a}^{m,cl})
\end{displaymath}
where the first is a consequence of G.A.G.A. and it is well-known that
holomorphic forms on a compact Kähler manifold are closed. If a global
algebraic 1-form $\omega$ whose associated complex analytic 1-form is
closed, then it is closed also as algebraic 1-form.

  (2): We first prove the assertion in the theorem with indices $a$
  attached to the identities, where $X_a$ can be any compact Kähler
  manifold. The proof is by induction over $m$, where in the case
  $m=0$ and all $n\geq 0$ we have the usual Hodge decomposition
\begin{displaymath}
  H^n(X_a,\Cb)\cong H^n (X_a, \Omega_{X_a}^\bullet)\cong \bigoplus_{\substack{i+j =n\\ i \geq 0, j\geq
      0}}  H^i(X_a, \Omega^j_{X_a}).
\end{displaymath}
Assume by induction that
\begin{displaymath}\tag{*}
  H^n(X_a, \Omega_{X_a}^{r,cl}) = \bigoplus_{\substack{j\geq r\\
      i+j = n+r}}
  H^i(X_a, \Omega^j_{X_a}),
\end{displaymath}
when $r < m$ and all $n \geq 0$. To prove the assertion when $r=m$ for
all $n\geq 0$ we use induction in $n$, where the case $n=0$ follows
from (1). By Poincaré's lemma for the complex analytic de~Rham
complex, for each integer $m \geq 1$ we have the exact sequence
\begin{displaymath}
0 \to \Omega_{X_a}^{m-1,cl} \to \Omega^{m-1}_{X_a} \to \Omega_{X_a}^{m,cl}\to 0
\end{displaymath}
and hence a long exact sequence
\begin{align*}
  0\to & \Gamma(X_a, \Omega_{X_a}^{m-1,cl})\to \Gamma(X_a,
  \Omega_{X_a}^{m-1})\to \Gamma(X_a, \Omega_{X_a}^{m,cl})\to H^1(X_a, \Omega_{X_a}^{m-1,cl})\\
\cdots       & \to  H^{n_1}(X_a,
        \Omega_{X_a}^{m-1,cl})  
        \xrightarrow{A_{n_1}^m}
        H^{n_1}(X_a,
        \Omega_{X_a}^{m-1})
        \to
        H^{n_1}(X_a,
        \Omega_{X_a}^{m,cl})\\
  &\xrightarrow{B_{n_1}^m} H^{n_1+1}(X_a,\Omega_{X_a}^{m-1,cl} )\to H^{n_1+1}(X_a,
    \Omega_{X_a}^{m-1})\to 
\end{align*}
Assume by induction that \thetag{*} holds when $r=m$ and $n< n_1$.
This implies that the map $A_{n_1}^m$ is surjective so that $B^m_{n_1}$ is
injective and we get the exact sequence
\begin{displaymath}
  0 \to H^{n_1}(X_a, \Omega^{m,cl}_{X_a})\to H^{n_1+1}(X_a,
  \Omega_{X_a}^{m-1,cl})\to H^{n_1+1} (X_a, \Omega^{m-1}_{X_a}),
\end{displaymath}
where by induction
\begin{displaymath}
  H^{n_1+1}(X_a,  \Omega_{X_a}^{m-1,cl}) = \bigoplus_{\substack{j\geq m-1\\
      i+j = n_1+m}}
  H^i(X_a, \Omega^j_{X_a}),
\end{displaymath}
implying that we can add $\to 0$ to the right. This implies that the
assertion also holds for $r=m$ and $n= n_1$. Finally, by  
G.A.G.A. one can erase  $a$ from the right side of  \thetag{*}.
  \end{proof}
  \subsubsection{Connections on projective varieties and their global
    sections.}
  Given a finite-dimensional vector space $V$ over $k$ we get the
  trivial connection $\Oc_X\otimes _kV$, where
  $\partial (\phi\otimes v )= \partial (\phi)\otimes v$,
  $\partial \in T_X, \phi \in \Oc_X, v\in V $.
\begin{proposition}\label{constant-global}
  Let $M$ be a connection on a normal projective variety $X$ over an
  algebraically closed field $k$.
  \begin{enumerate}
  \item There exists a projective morphism of smooth varieties
    $\pi :\bar X \to X$ such that $\pi^!(M)$ is a trivial connection.
  \item Let $V= \Gamma(X, M)$ be the space of global sections. Then
    $\Oc_X\otimes _kV$ is isomorphic to a trivial subconnection of $M$
    on $X$. The following are equivalent:
\begin{enumerate}
\item $M$ is a trivial connection.
\item $\dim_k V \geq \rank M$.
\end{enumerate}
  \end{enumerate}
\end{proposition}
Of course, it then follows that (b) implies $\dim _k V = \rank M$.
\begin{proof}
  (1): The fraction field $K$ of $X$ is finitely generated over $k$.
  Since $\dim_K M_K < \infty$ there exists a field extension $L/K $
  (the Picard-Vessiot extension) for the $\Dc_K$-module $M_K$, where
  $L$ is finitely generated over $k$, $M_L= L\otimes_K M_K\cong L^m$
  as $\Dc_L$-module, and $m= \rank M = \dim_K M_K$. By a standard
  construction there exists a map $\pi_1 : X_1 \to X$ of normal
  projective varieties inducing the field inclusion $K\subset L$ as
  the inclusion of fraction fields. One can resolve the singularities
  of $X_1$ by a birational projective map $p : \bar X \to X_1$ from a
  smooth projective variety $\bar X$, and we put $\pi= \pi_1\circ p$.
  Then $\pi^!(M)$ is a connection that it generically trivial, hence
  it is trivial \Prop{\ref{simple-coh}},
  $\pi^!(M)= \Oc_{\bar X}\otimes_k W$, where $T_{X}\cdot W =0$, and
  $\Gamma(\bar X, \pi^!(M)) = W$, since $k$ is algebraically closed.

  (2): Since $V$ maps injectively into $W$ and $L/K$ is separable, so
  that derivations of $K$ can be lifted to derivations of $L$ (defined
  modulo $K$-linear derivations of $L$), it follows that
  $T_K \cdot V =0$. This implies that $\Oc_X\otimes_k V$ is a trivial
  subconnection of $M$. The equivalences (a-b) should now also be
  clear.
\end{proof}
There is also a relative version:
\begin{proposition}\label{global-invariant}
  Let $\pi : X\to Y$ be a smooth projective morphism of smooth
  varieties over an algebraically closed field. Then the direct image
  defines a functor
  \begin{displaymath}
    \pi_* : \Con (X) \to \Con (Y), \quad M\mapsto \pi_*(M).
  \end{displaymath}
\end{proposition}
If $M$ is a semisimple connection and $\pi$ is finite étale
\Theorem{\ref{coh-decomp}} implies that $\pi_*(M)$ is a semisimple
connection.

It follows from the proof that the connection $\pi_*(M)$ is the same
as the lowest homology group of the full direct image
$H^{-d}\pi_+(M) = H^{-d}(R\pi_*(\Omega_{X/Y}^\bullet(M)[d]))$, where
$\Omega^\bullet_{X/Y}(M)$ denotes the relative de~Rham complex (a
Gauss-Manin connection).
\begin{proof} Denote by $\bar \pi_*(M)$ the image of the canonical
  injective map $\pi^{-1}\pi_*(M)\to M$ and let $i: X_y \to X$ be the
  embedding of a fibre (which is smooth since $\pi$ is smooth) over a
  closed point $y\in Y $, so that $k_{Y,y}=k$ and $i^!(M)$ is a
  connection on $X_y$. By \Proposition{\ref{constant-global}}
  \begin{displaymath}
    i^*(T_{X/Y}\cdot \bar \pi_*(M)) = T_{X_y/k_{Y,y}}\cdot  i^*(\bar
    \pi_*(M) ) = T_{X_y/k_{Y,y}}\cdot \Gamma(X_y, i^*(M)) =0
  \end{displaymath}
  Therefore
  \begin{displaymath}\tag{*}
    T_{X/Y}\cdot \bar \pi_*(M)=0,
\end{displaymath}
where $T_{X/Y}$ is the subsheaf of relative derivations in $T_X$.
Since $\pi$ is smooth we have the exact sequence
  \begin{displaymath}
    0 \to T_{X/Y}\to T_X \xrightarrow{d\pi} \pi^*(T_Y)\to 0,
  \end{displaymath}
  so that if $\partial$ is a local section of $ T_Y$ there exists
  locally in $X$ a section $\tilde \partial$ in $ T_X$ such that
  $d\pi(\tilde \partial )= \partial$, and by \thetag{*} it follows
  that if $\tilde \partial'$ is another such lift, then
  $(\tilde \partial - \tilde \partial')\bar \pi_*(M)=0$. Therefore we
  have an action of $T_Y$ on $\pi_*(M)$, which gives $\pi_*(M)$ a
  structure of connection, as it is moreover well-known that
  $\pi_*(M)$ is coherent over $\Oc_Y$.
  \end{proof}

  \subsection{$\Dc$-modules of rank 1}\label{rank1}
  Let $j: X_0 \to X$ be an open inclusion of smooth varieties, where
  $X$ is projective and $D= X\setminus j(X_0)$ is a divisor. Then if
  $M$ is a connection on $X_0$ it follows that $j_+(M)$ is a coherent
  $\Dc_X$-module and also a coherent $\Oc_X(*D)$-module. We are
  interested in the case when $\rko (M)=1$ and $X_0$ is
  affine\footnote{A sufficient condition for $X_0$ to be affine is
    that $D$ be the support of an effective ample divisor, but in
    general $X_0$ being affine needs not imply that $X\setminus X_0$
    is the support of an effective and ample divisor
    \cite{goodman:affine}.}. We want to classify such connections by
  decomposing closed 1-forms
  $ \Gamma(X_0, \Omega_{X_0}^{cl})=\Gamma(X, \Omega_X^{cl}(*D))$ on
  $X_0$ into a sum of logarithmic forms on $X$ with poles along $D$,
  exact forms, and a number $g$ of such closed forms with vanishing
  residues along $D$, where $g = \dim_k H^1(X, \Oc_X)$.
  \subsubsection{Generalities}
  We first recall a well-known fact.
\begin{lemma}\label{simple-1}
  Any $\Dc_X$-module which is locally free of rank $1$ as
  $\Oc_X$-module is simple.
\end{lemma}
\begin{proof} Let $M$ be a
  $\Dc_X$-module that is of rank $1$ and $N$ be a non-zero coherent
  submodule. Then $N$ is a $\Dc_X$-module which is coherent over
  $\Oc_X$, hence it is locally free as $\Oc_X$-submodule and also a
  submodule of the locally free module $M$ of rank $1$; therefore
  $ M/N$ is a coherent torsion module over $\Oc_X$, which moreover is
  a $\Dc_X$-module; therefore $M/N =0$.
\end{proof}
Let $\Io(X)$ (or $\Con^1(X)$) be the category of connections on
$X$ that are of rank $1$. This is a group in categories, where the sum
of two objects $M_1, M_2\in \Io(X)$ is $M_1\otimes_{\Oc_X} M_2$
and the $\Dc_X$-action is the diagonal one. Each object in
$\Io(X)$ is simple and we denote by $\bar \Io(X)$ the set of
isomorphism classes of simples.

Let $\Omega_{X}^{cl}$ be the sheaf of closed 1-forms and $\Oc_X^*$ the
sheaf of invertible elements in $\Oc_X$. There is then a distinguished
triangle
\begin{equation}\label{exact-log-ex}
  \Oc_X^* \xrightarrow{\dlog } \Omega_{X}^{cl} \to \Cc_{l} \xrightarrow{+1}, 
\end{equation}
where the cone\footnote{ Here $\Oc_X^*[1]$ and $\Omega_X^{cl}$ are
  placed in degrees $-1$ and $0$. See \cite{kashiwara-schapira} for
  generalities about triangulated categories.} of $\dlog$ is
\begin{displaymath}
  \Cc_{l} =\Cc_l(X)= \Cone (\dlog : \Oc^*_X \to \Omega_{X}^{cl}) =
  (\Oc_X^*[1]\oplus \Omega_X^{cl}).
\end{displaymath}
We have also the distinguished triangle
\begin{displaymath}
  \Oc_{X}\xrightarrow{d}  \Omega^{cl}_X \to \Cc_d \xrightarrow{+1}, 
\end{displaymath}
and when $X_a$ is a complex analytic manifold there is third
distinguished triangle
\begin{displaymath}
  \Oc_{X_a} \xrightarrow {\exp} \Oc_{X_a}^* \to \Cc^a_{e} \xrightarrow{+1},
\end{displaymath}
where $\Cc^a_{e} \cong 2\pi i \Zb [1]$. Replacing $X$ by $X_a$ in the
two previous triangles we have $\Cc_l^a \cong \Cb^*[1]$ and
$\Cc_d\cong \Cb[1]$, and we get three distinguished triangles where
any two have a common vertex, and $d= \dlog \circ \exp$. By the
octahedral axiom for triangulated categories, there is then also a
distinguished triangle for complexes of sheaves on analytic sheaves
\begin{displaymath}
  \Cc^a_e\to \Cc^a_d  \to \Cc^a_l \xrightarrow{+1},
\end{displaymath}which is isomorphic to the triangle
$2\pi i \Zb_{Z_a}[1]\to \Cb_{X_a}[1]\xrightarrow{\exp} \Cb_{X_a}^*[1]\xrightarrow{+1} $. The cohomology of
(\ref{exact-log-ex}) gives a fragment of a long exact sequence
\begin{align}
\label{exact-picard}  0\to \bar k^* \to \Gamma (X, \Oc_X^*)  &\xrightarrow{\dlog} \Gamma (X, \Omega_X^{cl}) \to R^0 \Gamma(X,
  \Cc_{l}) \xrightarrow{\Reso_X} H^1(X, \Oc_X^*) \\ & \xrightarrow{c}  H^{1}(X,
\notag                                         \Omega_{X}^{cl})\to,
\end{align}
where $\bar k^*= R^{-1}\Gamma(X, \Cc_l)$ is the multiplicative group
of the algebraic closure of $k$ in the fraction field $K$ of $X$ (a
motivation for the notation $\Reso_X$ will appear in
\Remark{\ref{rem-residues}}).

Let $\Pic^\tau(X)$ be the subgroup of the Picard group consisting of
invertible sheaves with vanishing rational first chern class. The
following proposition is well-known.
\begin{proposition}\label{lem-rank1}
  \begin{enumerate}
  \item
    \begin{displaymath}
 \bar \Io(X)  = R^0\Gamma (X,\Cc_l) = \frac {\Gamma(X,
   \Omega_X^{cl})}{\dlog \Gamma (X, \Oc_X^*)} \oplus \Ker (c)      
\end{displaymath}

\item If the divisor class group is trivial, $\Clo(X)=0$, then
  \begin{displaymath}
  \bar \Io(X) = \frac{\Gamma (X, \Omega_X^{1,cl})}{ \dlog \Gamma (X,
  \Oc^*_X)}.
\end{displaymath}

  \item If $X_a$ is a complex analytic manifold, then
    \begin{displaymath}
    \bar \Io(X_a) = H^1(X_a, \Cb^*_{X_a}) = Hom
    (\pi_1(X_a),\Cb^*) = (\Cb^*)^{b_1(X_a)}\oplus H,
  \end{displaymath}
where $\pi_1(X_a)$ is the fundamental group (with some choice of
base point), $b_1(X_a)$ is the first betti number of $X_a$, and  $H$
is a finite abelian group.
  \item If $X$ is projective, then
    \begin{displaymath}
      \bar \Io(X) = \Gamma(X, \Omega_X) \oplus \Pic^\tau(X).
    \end{displaymath}
  \end{enumerate}
\end{proposition}
In \Theorem{\ref{class-rank-1}} $\bar \Io(X\setminus D)$ is compared to
$\bar \Io(X)$ as in (4), when $D$ is a divisor.
\begin{proof}(1): If $U_i$ is an affine covering of $X$,
$\gamma_i \in \Omega_{U_i}^{cl}$, where $M_{U_i}= \Dc_{U_i} \mu_i$ and
$\partial \cdot \mu_i = \gamma_i (\partial)\mu_i$ it follows that if
$U_{ij} = U_i \cap U_j$, then
$(\gamma_i)_{U_{ij}} - (\gamma_j)_{U_{ij}} = \dlog (\phi_{ij})$ for
some $\phi_{ij}\in \Oc_{U_i \cap U_j}^*$. Therefore
$(\gamma_i, \phi_{ij})$ defines a Cech 0-cocycle of $\Cc_l$. The
second equality follows from (\ref{exact-log-ex}).

(2): Here $H^1(X, \Oc^*_X) =0$, so the assertion follows from (1).

(3): Since $\Cc_l^a\cong \Cb_{X_a}^*[1]$, the first equality follows
from (1), which also holds with $X$ replaced by $X_a$. The remaining
two equalities are also well-known.

(4): In the exact sequence (\ref{exact-picard}) $\dlog =0$, so it
suffices to see that the image of $\Reso_X$ equals $\Pic^\tau(X)$. To
prove this in turn it seems that one needs to employ the Lefschetz
principle, G.A.G.A. and Hodge theory, so we assume $k= \Cb$ and work
with the associated complex analytic space $X_a$. By G.A.G.A.
$\Pic^\tau(X) = \Pic^\tau(X_a)=c_1^{-1}(H^2(X_a,2\pi i \Zb)^t) $,
where $c_1 : H^1(X_a, \Oc_X^*)\to H^2(X_a,2\pi i \Zb)$ is the map
occurring in the exponential sequence, and the index ${}^t$ denotes
the torsion subgroup. We have
\begin{align*}
  \Ker (c) &= \Ker  (H^1(X, \Oc^*) \to H^2(X_a, \Cb)) \\
  &= \Ker (H^1(X_a,
  \Oc_{X_a}^*) \to H^2(X_a,2\pi i \Zb) \to H^2(X_a, \Cb) ) = \Pic^\tau(X_a).
\end{align*}
\end{proof}

Let $\pi: \Spec X \to \Spec Y$ be a finite morphism of smooth
$k$-varieties. The inverse image defines a group homomorphism 
\begin{displaymath}
  \pi^! : \Io(\Dc_Y)\to \Io(\Dc_X), \quad [M] \mapsto [\pi^!(M)],
\end{displaymath}
since $\pi^!(M_1\otimes_A M_2)= \pi^!(M_1)\otimes_{B} \pi^!(M_2)$. If
locally $M= M_\gamma$ for some closed 1-form,
$\pi^!(M) = M_{\pi^*(\gamma)}$, where $\pi^*(\gamma)$ is the image of
$\gamma$ with respect to the pull-back map
$\pi^* : \Omega_{Y}^{1,cl}\to \Omega_{X}^{1, cl}$.

Let $\gamma $ be an element in the space of closed 1-forms
$\Omega_K^{cl}$, where $K/k$ is a field extension of finite type. Then
$[M_\gamma]\in \Ic(\Dc_K)$ corresponds to
$ [\gamma] \in \Omega_K^{cl}/ \dlog K$. The group $\Aut (K/k)$ acts on
$\Omega^{cl}_K$ and the de~Rham cohomology group $H^1_{dr}(K)=\Omega^{cl}/ \dlog K$ in a natural
way, and we have $G_{M_\gamma} = G_{\gamma}$, where
$G_\gamma = \{g \in G \ \vert \ g\cdot [\gamma] = [\gamma]\}$ is the
stabilisator group of $[\gamma]$.

\subsubsection{ Exponential modules.}\label{exp-mod-sec} For a finitely generated smooth
$k$-algebra $A/k$, an exponential $\Dc_A$-modules is a rank $1$ module
$M$ that is determined by an exact 1-form, $\gamma = d\psi$,
$\psi \in A$. Writing $M= \Dc_A e^{\psi} $, we have
$\partial \cdot e^\psi = \partial (\psi) e^\psi$,
$\partial \in T_{A}$. The following proposition is well-known.
\begin{proposition}\label{exp-mod}
  If $M$ is a connection that is locally exponential in the Zariski
  topology on a variety $X$, then it is globally exponential, i.e.
  there exists $\psi \in \Gamma (X, \Oc_X)$ such that
  $M = \Dc_X e^\psi$. In particular, if $M$ is a rank 1 connection an
  a projective variety $X$ that is locally exponential, then
  $M\cong \Oc_X$.
\end{proposition}
\begin{remark}
  It follows from the proof that the locally exponential
  $\Dc_{X_a}$-modules are classified by the cohomology group
  $H^1(X_a, \Cb_{X_a})$ (where now $X_a$ is a complex analytic manifold).
\end{remark}
The following lemma is close to \cite{rosenlicht}*{Prop 4}.
\begin{lemma}\label{exact-log} Let $A$ be a smooth $k$-algebra and
  $(c_i)$ be a linearly independent subset of $k$, where $k$ is
  regarded as vector space over the rational numbers $\Qb$. Then if
  $\phi_i \in A^*$
  \begin{displaymath}
    (\sum_i c_i \dlog \phi_i)\cap d A =0. 
  \end{displaymath}
In particular, 
  \begin{displaymath}
    dA \cap (\Qb \dlog A^*) =0.
  \end{displaymath}
\end{lemma}
\begin{proof} If
  \begin{displaymath}
    d\psi = \sum_i c_i \dlog (\phi_i) 
  \end{displaymath}
and $\nu$ is a discrete valuation of the fraction field  $K$ of $A$,
then by \Proposition{\ref{res-basics}},
\begin{displaymath}
  \sum_i c_i \Reso_\nu (\dlog (\phi_i) ) = \sum_i c_i \nu (\phi_i) =0.
\end{displaymath}
Since $(c_i)$ is linearly independent, it follows that
$\nu(\phi_i)=0$ for all $\nu$. This implies that $\phi_i$ belongs to the algebraic
closure $\bar k$ of $k$ in $K$, and therefore $\dlog (\phi_i)=0$.
\end{proof}

\begin{pfof}{\Proposition{\ref{exp-mod}}}
  If $\psi_i\in \Oc_X(U_i) $ are local functions such that
  $M_{U_i} = \Dc_{U_i} e^{\psi_i}$, then
  $d(\psi_i - \psi_j)= \dlog (\phi_{ij})$ for some
  $\phi_{ij}\in \Oc_X^*(U_i\cap U_j)$. By \Lemma{\ref{exact-log}} it
  follows that $d(\psi_i- \psi_j)=0$, hence
  $\psi_i = \psi_j + \lambda_{ij} $, where $\lambda_{ij}$ belongs to
  the algebraic closure of $\bar k $ of $k$ in $K$
  \Lem{\ref{adj-lemma}}. Now the constant sheaf $\bar k_X$ is flasque
  so there exist $\lambda_i \in \bar k_X(U_i)$ such that
  $\lambda_{ij} =\lambda_j - \lambda_i$, and therefore the sections
  $\psi_i- \lambda_i$ glue to a global section $\psi$ in
  $\Gamma(X, \Oc_X)$.
\end{pfof}

\subsubsection{Higher dimensional residues.}
We will later decompose $\Omega_X^{cl}(*D)$ inte an exponential and
logarithmic component and analyse the logarithmic component using
residues.

K. Saito introduced and studied the notion of logarithmic forms along
arbitrary reduced divisors $D$ in a complex manifold
\cite{saito-kyoji:log}, extending similar constructions by
Delinge\footnote{Deligne assumed that $D$ be a divisor with at most
  normal crossing singularities.}. We will further extend Saito's
treatment by defining logarithmic forms and a residue
operation also for closed differentials relative to any discrete
valuation of a field, using the classical constructions for curves. We
are here only interested in 1-forms and work algebraically.

Let $I_D$ be the ideal of the reduced divisor $D$ on a normal variety
$X$, which is regarded as a scheme.
\begin{proposition}\label{def-log-sheaf}
  Assume $I_{D,x}= (h_x)$ for some function $h$ and let
  $\gamma_x\in \Omega_{X,x}(*D)$. The following are equivalent:
  \begin{enumerate}
  \item $h_x \gamma_x \in \Omega_{X,x} $ and
    $h_x d(\gamma)_x\in \Omega^2_{X,x}$.
  \item $\gamma_x = \omega_x/h_x$, where $\omega_x\in \Omega_{X,x}$
    and $d(\omega_x) - \dlog (h_x)\omega_x\in \Omega^2_{X,x}$.
  \item there exists $g_x\in \Oc_{X,x}$ that is not a zero-divisor in
    $\Oc_{D,x}$ such that
    \begin{displaymath}
      g_x \gamma_x = a \dlog (h_x) + \eta_x,
    \end{displaymath}
    where $a\in \Oc_{X,x}$ and $\eta_x \in \Omega_{X,x}$.
  \end{enumerate}
\end{proposition}
The proof is similar to \cite{saito-kyoji:log}*{\S 1}. The sheaf
$\Omega_X(\log D) $ of {\it logarithmic differentials} is the subsheaf
of sections $\gamma$ in $\Omega_X(*D)$ that satisfy the above
equivalent conditions (1-3) at all points $x$.

Now let $\nu : K^* \to \Zb$ be a discrete valuation of the finitely
generated field extension $K/k$, $R= R_\nu$ be the corresponding
discrete valuation subring of $K$, and $k_R = R/\mf_R$ its residue
field. Let $\bar R$ be the $\mf_R$-adic completion of $R$, and
$\bar K$ be the fraction field of $\bar R$. Let
$\hat k_R\subset \bar R $ be a coefficient field, so that if $t$ is a
uniformising parameter for $R$, then $\bar R = \hat k_R [[t]]$,
$\bar K= \hat k_R ((t)) $, and we have an identification
$k_R = R/\mf_R= \bar R/ (t) = \hat k_R$. The modules of differentials
$\Omega_{\bar R/\hat k_R}$ and $\Omega_{\bar R/k}$ are not separated
in the $\mf_R$-adic topology, and are moreover not of finite
dimension, while the separated modules
\begin{displaymath}
  \bar \Omega_{R/k_R} :=\frac{ \Omega_{\bar R/\hat k_R}}{ \cap (\mf^n_R
  \Omega_{\bar R/\hat k_R})} = \bar R dt \quad \text{ and}\quad
  \bar \Omega_{ R/k}:=\frac{\Omega_{\bar R/k}}{\cap
    (\mf^n_R\Omega_{R/k})} = \bar R \otimes_R \Omega_{R/k}
\end{displaymath}
are of dimension $1$ and $\trdeg K/k$, respectively. The usual residue
map is defined by
\begin{displaymath}
\reso_\nu: \bar \Omega_{K/k_R}:= \bar K\otimes_{\bar R} \bar
\Omega_{\bar R/\hat k_R}= \bar K dt \to k_R, \quad a\mapsto a_{-1},
\end{displaymath}
where $a= \sum_i a_{i} t^i\in \bar K$, $a_i \in k_R$; the element
$a_{-1}$ is independent of the choice of uniformising parameter as is
described in \cite{serre:alg-groups} and \cite{tate}, for the given
choice of coefficient field. Now the exact sequence
\begin{displaymath}
  0 \to \bar K \otimes_{\hat k_R} \Omega_{\hat k_R/k}  \to \bar K
  \otimes_{\bar R } \bar \Omega_{R/k} \to \bar K \otimes_{\bar R}
  \Omega_{\bar R/\hat k_R}\to 0
\end{displaymath}
can be used to extend $\reso_\nu$ to the composed map
\begin{equation}\label{res-eq}
  \Reso_\nu : \Omega_{K/k}\to \bar K  \otimes_{\bar R} \bar
  \Omega_{R/k} \to  \bar K\otimes_{\bar R} \bar \Omega_{\bar R/\hat k_R}  \xrightarrow{\reso_\nu} k_R.
\end{equation}
The map $\Reso_\nu$ depends on the choice of coefficient field
$\hat k_R$ in $\bar R$ as follows. Let $\hat k'_R$ be another
coefficient field in $\bar R$ and denote by $\Reso'_\nu$ the above map
with $\hat k_R$ replaced by $\hat k'_R$. Select transcendence bases
$\{y_i\}\subset \hat k_R $ and $\{ z_i\}\subset \hat k'_R $ as field
extensions of $k$, and define the derivation
$\partial'_t\in T_{\bar K/k }$ by $\partial'_t (t)=1$ and
$\partial'_t (z_i)=0$. Any
$\omega \in \bar K \otimes_R \bar \Omega_{R/k}$ can be expressed as
\begin{equation}\label{expansions}
  \omega = a\frac{dt}t + \sum_{i=1}^r b_i d y_i = a' \frac {dt}t + \sum_{i=1}^r b'_i dz_i,
\end{equation}
for $a, a', b_i, b_i'\in \bar K $, for the different choices of
coefficient fields, so that the relation between the residues becomes
\begin{align*}
  \Reso'_\nu (\omega) &= a'_{-1} = (a + \sum_{i=1}^r
                        (b_i \partial'_t(y_i)))_{-1}  \\
  &= \Reso_\nu (\omega) + \sum_{i=1}^r\Reso'_\nu
  (b_i \partial'_t(y_i))dt).
\end{align*}
\begin{remark}\label{rem-res}
  If $ \nu (b_i) \geq 0$ then $\nu (b_i \partial'_t(y_i))\geq 0$, so
  that $\Reso_\nu(\omega)$ is independent of the choice of coefficient
  field.
\end{remark}

A {\it $\nu$-logarithmic} $1$-form $\gamma\in \Omega_{K/k}$ is a
1-form that has a logarithmic pole at $\nu$, meaning
$\mf_R \gamma \subset \Omega_{R/k}$ and
$\mf_R d\gamma \subset \Omega^2_{R/k}$. Let
$\Omega_{K/k}(\log \nu)$\footnote{It can also be denoted
  $ \Omega_{R/k}(\log \mf_R)$.} be the $k$-vector space of
$\nu$-logarithmic 1-forms. Let $\Omega^{cl}_{K/k}$ be the closed
$1$-forms and put
$\Omega^{cl}_{K/k}(\log \nu) =\Omega_{K/k}(\log \nu)\cap
\Omega^{cl}_{K/k}$.

Let $\bar k_\nu$ be the algebraic closure of $k$ in $k_R$.
\begin{proposition}\label{res-basics}
  \begin{enumerate}
  \item The map $\Reso_\nu$ in (\ref{res-eq}) restricts to a map
      \begin{displaymath}
        \Omega_{K/k}(\log \nu) \to k_R
      \end{displaymath}
      that is independent of the choice of coefficient field.
  \item The map $\Reso_\nu$ restricts to a map
    \begin{displaymath}
      \Omega^{cl}_{K/k} \to \bar k_\nu  
    \end{displaymath}
    that is independent of the choice of coefficient field.
  \item $\Reso_\nu (dK)=0$.
  \item $\Reso_\nu(\dlog \phi) = \nu(\phi)$.
  \end{enumerate}
\end{proposition}
\begin{remarks}\label{res-rem}
  \begin{enumerate}
  \item Saito proves (1) in the complex analytic situation, where
    $k = \bar k_\nu = \Cb$, $k_R$ is the field of meromorphic
    functions on a divisor $D$ in a complex manifold, and $R $ is a
    ring of germs of meromorphic functions that do not have poles
    along $D$ \cite{saito-kyoji:log}.
  \item Let $D= V(f)$ be an irreducible germ of a complex analytic
    hypersurface and $\nu$ be its discrete valuation. If $\omega$ is a
    closed meromorphic 1-form with poles only along $D$ we have
    $\Reso_\nu (\omega)= \int_{c}\omega$, where $c$ is a generator of
    the cyclic homology group $H_1(U \setminus D, \Zb)$ for
    sufficiently small open sets $U$.
  \item By Artin's approximation theorem there exists a finite field
    extension $L/K$ such that the algebraic closure $\bar k_\nu $ is
    isomorphic to the algebraic closure of $k$ in $L$.
  \end{enumerate}

\end{remarks}
\begin{proof}
  (1): Consider the expressions (\ref{expansions}). Since
  $\omega \in \Omega_{K/k}(\log \nu)$ it follows that
  $t\omega \in \Omega_{R/k}$, so that $\nu(a)\geq 0$, and since
  $td\omega \in \Omega^2_{R/k}$, implying $\nu(b_i)\geq 0$, the
  assertion follows by \Remark{\ref{rem-res}}. We can see also that
  $\Reso_{\nu}(\omega) = \bar a= a(\omod \mf_R)$, where now $\bar a$
  is evaluated without selecting a coefficient field. Not knowing that
  $\Reso_{\nu}(\omega)$ arises also from $\reso_\nu$ by a choice of
  $\hat k_R$, one instead can check that $\bar a$ does not depend on
  the representation of $\omega$. Suppose we have two decompositions
  \begin{displaymath}
    \gamma =  a   \frac {dt}{t} + \gamma_0 = a'  \frac {dt}{t} + \gamma'_0
  \end{displaymath}
  where $\nu(a), \nu(a')\geq 0$ and
  $t\gamma_0, t\gamma'_0 \in \Omega_{R/k}$.
Hence
  \begin{displaymath}
    (a-a')\frac {dt}{t}+ \gamma_0 - \gamma'_0 =0.
  \end{displaymath}
  Since $t d\gamma \in \Omega^2_{R/k}$,
  $\gamma_0, \gamma'_0 \in \Omega_{R/k} $, and therefore
  $\gamma_0 - \gamma'_0 \in \Omega_{R/k}$. It follows that
  $\nu(a-a')>0$, whence $\bar a' = \bar a$.

  (2): Select a coefficient field $\hat k_R$ of $\bar R$, let $t $
  be a uniformising parameter, and $(y_2, \ldots , y_3)$ be a
  transcendence basis of $ k_R/k$. Then the differentials
  $\{dt, dy_i\}$ form a basis of
  $\bar \Omega_{\bar K_\nu/k} = \oplus \bar K dy_i$, and we let
  $\{\partial_t, \partial_i\}$ be the dual basis. If
\begin{displaymath}
  \omega = a dt +\sum_{i=1}^r b_i dy_i
\end{displaymath}
and $d\omega =0$, we get
\begin{displaymath}
  \sum_{i=1}^r (\partial_{y_i}(a)- \partial_{t}(b_i)) dy_i dt =0
\end{displaymath}
and, since the differentials $dy_i dt \in \bar K \otimes_K \Omega^2_{K/k}$,
$i=1, \ldots, n$ are linearly independent over $\bar K $,
\begin{displaymath}
  \partial_{y_i}(a)- \partial_{t}(b_i) =0.
\end{displaymath}
Writing $a= \sum_i a_i t^i$,
$b_i = \sum_{j} b_{ij}t^j\in \bar K = \hat k_R ((t))$,
$a_i, b_{ij}\in \hat k_R $, we get $\partial_{y_i}(a_{-1})=0$,
$i=1, \ldots , r$. Hence by \Lemma{\ref{int-closure}} (below) $ a_{-1} $
belongs to the algebraic closure $\bar k$ of $k$ in $\hat k_R$. Since
the image of $\bar k$ under an isomorphism $\hat k_R\cong k_R$ (fixing
$k$) coincides with the algebraic closure $\bar k_\nu$ of $k$ in
$k_R$, it follows that
$\Reso_\nu(\omega) = a_{-1} \in \bar k_\nu \subset k_R $.

(3): If $\phi \in K$, then
$d_{\bar R/k}(\phi)= d_{\bar R/\bar k_R}(\phi)$, as elements in
$\Omega_{\bar R/\hat k_R}$ equals, and it is well-known that
$\reso_\nu(d_{\bar R/\bar k_R}(\phi))=0$.

(4): If $\phi = t^mu$, where $\nu(t)=1$, $\nu(u)=0$, then
$\dlog (\phi) = m d(t)/t + d(u)/u$. This implies the assertion.
\end{proof}

Saito's residue map is given by 
\begin{equation}\label{res-sheaf}
  \Reso : \Omega_X(\log D) \to K_D= \bigoplus_{i=1}^r K(\Oc_{D_i}),  \quad
 \gamma \mapsto  \frac {\bar a_x}{\bar g_x} = \bigoplus_{i=1} ^r\reso_{\nu_i}(\gamma),
\end{equation}
(in the notation of \Proposition{\ref{def-log-sheaf}})
where $K_D$ is the total field of fractions of $D$, $K(\Oc_{D_i})$ the
fraction field of $D_i$, and  $\nu_i$ is the discrete valuation that is
determined by $D_i$.  There exists an exact sequence 
\begin{equation}\label{zero-resdues}
0 \to \Omega_X \to \Omega_X(\log D) \xrightarrow{\Reso}  \Rc_D \to 0.
\end{equation}
Here $\Rc_D$ is a subsheaf of $ K_D $ that contains
$c_*(\Oc_{\tilde D})$, where $c: \tilde D \to D $ is the normalization
map (see \cite{saito-kyoji:log}). The residue map restricts to a map
\begin{equation}
  \Reso: \Omega^{cl}_{X}(\log D) \to  \bigoplus_{i=1}^r\bar k_{\nu_i},\quad
  \gamma \mapsto  \bigoplus_{i=1} ^r\reso_{\nu_i}(\gamma),
\end{equation}
where $\bar k_{\nu_i}$ is regarded as a constant subsheaf of
$K(\Oc_{D_i})$ (see \Proposition{\ref{res-basics}}), and we also have
a map
\begin{displaymath}
  \Reso: \Omega^{cl}_{X}(*D) \to  \bigoplus_{i=1}^r\bar k_{\nu_i},\quad
  \gamma \mapsto  \bigoplus_{i=1} ^r\reso_{\nu_i}(\gamma),
  \end{displaymath}

    


\subsubsection{Decomposition inte logarithmic and exponential modules}
Let $D$ be any reduced divisor in a smooth projective variety $X$ and
$\Omega_X(*D)$, $\Oc_X(*D)$ the sheaf of differentials and functions,
respectively, with arbitrary poles along $D$, and $d\Oc_X(*D)$ be the
image sheaf of the differential $d : \Oc_X(*D)\to \Omega_X(*D) $. We
have the exact sequence
\begin{displaymath}
  0\to \Omega^{cl}_X \to d\Oc_X(*D) \oplus \Omega_X^{cl}(\log D)
  \xrightarrow{p} \Omega_X^{cl}(*D),
\end{displaymath}
where clearly the map of sheaves $p$ need not be surjective, even when
$D$ is smooth. On the other hand, if one replaces $X$ by its
associated complex analytic manifold $X_a$ there is a kind of Poincaré
lemma, which surely is well-known but seems not available in the
literature.
\begin{lemma}\label{exact-sequence} We have an exact sequence 
  \begin{equation}\label{exact-log-seq}
    0\to \Omega^{cl}_{X_a} \to d\Oc_{X_a}(*D) \oplus \Omega_{X_a}^{cl}(\log D)
    \xrightarrow{p} \Omega_{X_a}^{cl}(*D)\to 0.
\end{equation}
\end{lemma}
\begin{proof} Assume that $ D= \cup_{i=1}^r D_i$ where the $D_i$ are
  irreducible. Let $\omega$ be a section of $\Omega_{X_a}^{cl}(*D)$
  and put $\lambda_i = \Reso_{\nu_i} (\omega)\in \bar k$, where
  $\nu_i$ is the discrete valuation of the fraction field of the ring
  of complex convergent power series, corresponding to the divisor
  $D_i$. Then
  \begin{displaymath}
\omega^0=    \omega - \sum_{i=1}^r \lambda_i \frac {df}f   
  \end{displaymath}
  is a closed meromorphic 1-form with vanishing residues along $D$.
  There exist a subset $C $ of $D$ such that $D\setminus C$ is smooth
  and $\codim_{X_a}C \geq 2$. Let $j : U_a = X_a\setminus C_a\to X_a$
  be the open inclusion. Then there exists an open covering $\{U_i\}$
  of $U_a$ and holomorphic functions $\phi_{U_i}$ such that
  $d\phi_{U_i} = \omega^0\vert_{U_i}$ and
  $d(\phi_{U_i}\vert_{U_j}- \phi_{U_j}\vert_{U_i})=0$. Since
  $R^1j_*(\Cb_{U_a}) =0$ (see \Remark{\ref{vanishing-remark}}), it follows that locally in $X_a$ there
  exists a function $\phi$ such that $\omega^0 = d\phi$.
\end{proof}

\begin{remarks}\label{vanishing-remark}
  \begin{enumerate}
\item It is well-known, and can be proven with a topological
    argument, that
    \begin{displaymath}
      R^ij_*(\Cb_{U_a}) =0,
    \end{displaymath}
    when $i=1,2$ and $\codim_{X_a}C \geq 2$. We indiciate a different
    known proof, based on Grothendieck's comparison theorem
    \begin{displaymath}\tag{*}
      R\Gamma_C\circ DR (\Oc_{X_a}) = DR\circ R\Gamma_{[C]}(\Oc_{X_a}),
    \end{displaymath} 
    where $DR$ is the de~Rham functor and $R\Gamma_{[C]}$ denotes
    tempered local cohomology; the latter occurs in a  distinguished
    triangle
    \begin{displaymath}\tag{**}
R\Gamma_{[C]}(\Oc_{X_a}) \to \Oc_{X_a}\to       Rj_+j^!(\Oc_{X_a})\xrightarrow{+1},
    \end{displaymath}
    (see \cite{bjork:analD}*{Th. 5.4.1}). The homology of the complex
    $R\Gamma_{[C]}(\Oc_{X_a})$ is concentrated in degrees
    $\geq \codim_{X_a} C $ and has support in $C$, so by \thetag{*}
    the homology of
    $ R\Gamma_{C}(\Cb_{X_a})= R\Gamma_{C}DR (\Oc_{X_a})= DR\circ
    R\Gamma_{[C]}(\Oc_{X_a}) $ is concentrated in degrees
    $\geq 2 \codim_{X_a} C$. By applying $DR$ on \thetag{**} we
    get the distinguished tringle
    \begin{displaymath}
      R\Gamma_C(\Cb_{X_a}) \to \Cb_{X_a} \to    Rj_*j^*(\Cb_{X_a})\xrightarrow{+1},
    \end{displaymath}
    and therefore
    \begin{displaymath}
R^kj_*(\Cb_{U_a})=      R^kj_*j^*(\Cb_{X_a})= 0,
    \end{displaymath}
    when $k=1, \ldots , 2 (\codim_{X_a} C -1)$. Keeping the assumption
    $\codim_{X_a} C \geq 2$, the exact sequence
    $0 \to \Cb_{X_a}\to \Oc_{X_a}\to \Omega^{cl}_{X_a}\to 0$ gives
    \begin{displaymath}
      R^1j_*(\Omega^{cl}_{U_a}) = R^1j_*(\Oc_{U_a}) = R^2\Gamma_{C}(\Oc_{X_a}),
    \end{displaymath}
    which is moreover $0$ when $\codim_{X_a} C > 2$ .
  \item \Lemma{\ref{exact-sequence}} can also be proven using
    integrals (see \Remark{\ref{res-rem}}, (2)) and Riemann's
    extension theorem, stating that $\Oc_{X_a} = j_*j^*(\Oc_{X_a})$
    when $C$ is of codimension $\geq 2$.

  \end{enumerate} \end{remarks}

Assume that $D$ is a divisor such that $X_0=X\setminus D$ is affine
and $M$ be a $\Dc_X$-module whose restriction to $X_0$ is a connection
of rank $1$. It follows that $\Gamma(X, M)\neq 0$, and any global
non-zero section $\mu$ of $M$ restricts to a cyclic generator,
$M_{\vert X_0}= \Dc_{X_0} \mu$, where
$\partial \cdot \mu = \omega(\partial) \mu$, $\partial \in T_X$ for
some
$\omega \in \Gamma(X, \Omega_X^{cl}(*D))=\Gamma(X_0,
\Omega_{X_0}^{cl})$.

\begin{proposition}\label{ample}  Let $X$ be a smooth projective
  variety and $D$ be a divisor such that $X_0= X\setminus D$ is
  affine. Then
  \begin{equation}\label{decom-closed-one}
    \Gamma(X_0,\Omega_{X_0}^{cl}) = \Gamma(X,\Omega_X^{cl}(\log
    (D))) \oplus   d\Gamma(X_0, \Oc_{X_0}) \oplus H^1(X,
    \Oc_X).
  \end{equation}
\end{proposition}
For the proof we need to work with the associated complex analytic
manifold $X_a$. Consider the exact sequence
\begin{equation}\label{important-exact}
  0 \to \Cb_{X_a} \to \Oc_{X_a}(*D_a)\xrightarrow{d} d\Oc_{X_a}(*D_a) \to 0.
\end{equation}
We have therefore the exact sequence (which will be explained in the
proof below)
\begin{equation}\label{important-conseq}
  0 \to d \Gamma (X_0, \Oc_{X_0}) \to  \Gamma(X_a, d\Oc_{X_a}(*D_a)) \to
  H^1(X, \Oc_X) \oplus \Gamma(X, \Omega_X)\to 0.
\end{equation}
\begin{proof} By Lefschetz' principle we can assume that $k= \Cb$. We
  first prove that in the transcendental topology we have the exact
  sequence
  \begin{equation}\label{proof-exact}
  \begin{aligned}[b]
   0  & \to  \Gamma(X_a,\Omega^{cl}_{X_a}) \to     \Gamma(X_a,\Omega_{X_a}^{cl}(\log (D_a))) \oplus 
         \Gamma(X_a,d\Oc_{X_a}(* D_a)) \\
       & \to      \Gamma(X_a,\Omega_{X_a}^{cl}(*D_a))\to 0.
  \end{aligned}
  \end{equation}
  By \Lemma{\ref{exact-sequence}} we have the exact seqence
\begin{displaymath}
  0 \to \Omega_{X_a}^{cl}\to \Omega^{cl}_{X_a}(\log D_a) \oplus
  d\Oc_{X_a}(*D_a) \to  \Omega_{X_a}^{cl}(*D_a)\to 0,
\end{displaymath}
whose associated long exact sequence in homology contains the map
\begin{displaymath}
  H^1(X_a, \Omega_{X_a}^{cl}) \to  H^1(X_a, \Omega_{X_a}^{cl}(\log D_a))\oplus
  H^1(X_a, d\Oc_{X_a}(*D_a)).
\end{displaymath}
To prove the exactness of (\ref{proof-exact}) it suffices to see that
the composition to a map on the second factor
\begin{equation}\label{first-map}
  H^1(X_a, \Omega_{X_a}^{cl}) \to   H^1(X_a, d\Oc_{X_a}(*D_a))
\end{equation}
is injective, which is done by employing the exact sequence
(\ref{important-exact}). Since $X \setminus D$ is affine, by G.A.G.A.
$H^i(X_a, \Oc_{X_a}(*D_a) ) = H^i(X, \Oc_{X}(*D) ) =0$, $i\geq 1$,
hence by the Hodge theorem we get
\begin{align*}
  H^1(X_a, d\Oc_{X_a}(*D_a)) &= H^2(X_a, \Cb_{X_a})\\
  &=H^2(X_a, \Oc_{X_a}) \oplus H^1(X_a, \Omega_{X_a})\oplus \Gamma(X_a,\Omega^2_{X_a}).
\end{align*}
Again by G.A.G.A.  we get the sequence (\ref{important-conseq}), so that
\begin{equation}\label{second}
  \Gamma (X_a, d\Oc_{X_a}(*D)) = d \Gamma (X, \Oc_X(*D)) \oplus H^1(X, \Oc_{X})   \bigoplus \Gamma(X,
      \Omega_{X}).
\end{equation}
Since moreover 
\begin{displaymath}
  H^1(X_a, \Omega_{X_a}^{cl}) = H^1(X_a, \Omega_{X_a})\oplus
  \Gamma(X_a, \Omega^2_{X_a})
\end{displaymath}
\Prop{\ref{global-forms}},  (\ref{first-map})  defines  a map
\begin{displaymath}
  H^1(X_a, \Omega_{X_a})\oplus
  \Gamma(X_a, \Omega^2_{X_a}) \to  
  H^2(X_a, \Oc_{X_a}) \oplus H^1(X_a, \Omega_{X_a})\oplus \Gamma(X_a,\Omega^2_{X_a}),
\end{displaymath}
whose composition to a map
$ H^1(X_a, \Omega_{X_a})\oplus \Gamma(X_a, \Omega^2_{X_a})
\to H^1(X_a, \Omega_{X_a})\oplus \Gamma(X_a, \Omega^2_{X_a})$
is an isomorphism, it follows that (\ref{first-map}) is injective.

By G.A.G.A. we can first erase $a$ on the left in (\ref{proof-exact}),
noting also that $\Gamma(X, \Omega^{cl}_X) = \Gamma(X, \Omega_X)$
\Th{\ref{global-forms}}, and since
\begin{align*}
  \Gamma(X_a, \Omega^{cl}_{X_a}(*D_a))&\subset \Gamma(X_a,
  \Omega_{X_a}(*D_a))\cap \Gamma(X_a, \Omega^{cl}_{X_a}(*D_a)) =
  \Gamma(X, \Omega^{cl}_{X}(*D)) \\
  &= \Gamma(X_0, \Omega^{cl}_{X_0}),
\end{align*}
we can erase it on the right too, and
$\Gamma(X_a,\Omega_{X_a}^{cl}(\log (D_a))) =
\Gamma(X,\Omega_{X}^{cl}(\log (D))$. Then the proof is complete by
(\ref{second}).
\end{proof}

Put $g= \dim_k H^1(X, \Omega_X)$ and select a subset
\begin{equation}
    \label{g-forms}
    \{\omega_i\}_{i=1}^g\subset \Gamma(X_a, d\Oc_{X_a}(*D_a)) \subset \Gamma(X_a, \Omega_{X_a}^{cl}(*D_a))=\Gamma(X_0, \Omega_{X_0}^{cl}),
  \end{equation}
  that maps to a basis of $H^1(X, \Oc_X)$ in (\ref{important-conseq}).
  By \Proposition{\ref{ample}} a closed 1-form $ \omega$ on $X_0$ has a
  unique decomposition
\begin{equation}\label{rep-rat-form}
  \omega = \gamma + d\psi + \sum_{i=1} ^g\alpha_i \omega_i,
\end{equation}
where $\psi\in \Gamma(X_0,\Oc_{X_0})$,
$\gamma \in \Gamma(X, \Omega^{cl}_X(\log D))$, and $\alpha_i\in k$.
Let $D= \cup_{j=1}^r D_j$ be a decomposition into irreducible
components and $D^s$ be the singular locus of $D$. Any point in
$X \setminus D^s$ is contained in an affine open neighbourhod $U$ such
that if $f_j \in \Oc_X(U)$ defines $D_j$, $D_j\cap U = V(f_j)$, then
\begin{displaymath}
  \gamma\vert_{U }  = \sum_{j=1}^r \lambda_j \frac {df_j}{f_j} + \gamma_0,
\end{displaymath}
where $\gamma_0 \in \Omega(U, \Omega_X^{cl})$ and
$\lambda_j\in \bar k$.\footnote{Such a represention of a logarithmic
  form $\gamma$ is in general not possible at singular points of $D$
  (see (\ref{res-sheaf}) and  \cite{saito-kyoji:log}).}

Since locally in the analytic topology, in sufficiently small open
sets $U_a$, $\omega = \sum \lambda_i df_i/f_i + d\psi$, for some
meromorphic functions $f_i, \psi \in \Gamma (U_a, \Oc_{X_a}(*D_a))$,
$M$ can be presented as follows
\begin{displaymath}
  M(U_a)= \Dc_X(U_a) e^{\psi}  \prod_{j} f_{j}^{\lambda_j}, \quad \partial \cdot
  e^{\psi}  \prod_{j}  f_{j}^{\lambda_j} = (\partial (\psi ) + \sum_j\lambda_j\frac {\partial(f_j)}{f_j}) e^{\psi}\prod_{j}  f_j^{\lambda_j},
\end{displaymath}
where $\partial \in T_{X_a}(U_a)$. Since $M$ is a rational connection this
determines the action on $\mu$ also over small neighbourhoods $U_a$ of
points in $D^s\subset X_a$.
\begin{remark}
  In classical terminology the differentials $\omega_i$ represent a
  basis of the differentials of the second kind
  $\Gamma(X_a, d\Oc_{X_a}(*D_a))$ (i.e. differential forms on $X_0$
  with zero residues along $D$) modulo differentials of the first kind
  and exact differentials. Since no differentials of the first kind
  are exact when $X$ is projective, it follows also that
  $\dim_k \Gamma(X_a, d\Oc_{X_a}(*D_a))/ d\Gamma(X, \Oc_X(*D)) =
  h^{0,1} + h^{1,0}= \dim H^1(X_a,\Cb)$. Elliptic curves:
  $X_0= V(y^2-( x^3+ ax +b)) \subset \Ab^2_k\subset \Pb^2_k$,
  $4a^3 + 27b^2\neq 0$. Then $\omega= dx/y$ and $\omega_1 = xdx/y$ are
  differential of the first and second kind, respectively.
  K3-surfaces: Then
  $h^{0,1}=h^{1,0}=0 $, and in the decomposition (\ref{rep-rat-form})
  we only have $\omega= \gamma + d\psi$.
\end{remark}

\subsubsection{Classification of rank 1 connections on affine
  varieties}
Let $D$ be a hypersurface in a smooth projective variety $X$ and let
$j: X_0 = X\setminus D \to X $ be the open inclusion of its
complement, where we assume that $X_0$ is affine. Let $\Io_{conn}(X)$
denote the category of connections on $X$ of rank $1$ and
$\bar \Io_{conn}(X)$ its isomorphism classes. The restriction defines
a functor
 \begin{displaymath}
   j^!: \Io_{conn}(X)\to \Io_{conn}(X_0), \quad M\mapsto j^!(M),
 \end{displaymath}
 and we let $\Io^0_{conn}(X)$ be the subcategory of modules such that
 $j^!(M)\cong \Oc_{X_0}$. Denote by
 $\bar \Io_{conn, reg}(X,D)$ the isomorphism classes of connections of
 rank 1 with regular singularities along $D$; see
 \cite{hotta-takeuchi-tanisaki, borel:Dmod} for the notion of regular
 singular $\Dc$-modules. Letting $\hat \pi_1(U_a )$ be the set of
 linear characters of the fundamental group of $U_a$, by the
 Riemann-Hilbert correspondence and G.A.G.A. we have an equivalence of
 categories
\begin{displaymath}
  \bar \Io_{conn, reg}(X,D)= \hat \pi_1(U_a )
\end{displaymath}
(see [loc cit]).

As a
divisor, write $D= \sum_{i=1}^r D_i$, where $D_i$ are irreducible
hypersurfaces, and denote by $\Cb_{D_i}$ and  $(\Cb/\Zb)_{D_i}$ the direct image sheaf
on $X$ of the constant sheaves  on $D_i$ with values in $\Cb$ and
$\Cb/\Zb$, respectively.

Let $X_a$ denote the complex analytic manifold and $D_a$ the complex
analytic divisor that can be naturally associated with $X$ and $D$. K.
Saito defined the analytic residue map (we denote it by the same
symbol as in (\ref{res-sheaf}))
    \begin{displaymath}
  \begin{tikzcd}
    \Omega_{X_a}(\log D_a) \arrow[r, "\Reso "] & \Mcc_{D_a} \\
    \Omega^{cl}_{X_a}(\log D_a) \arrow[r,"\Reso"]\arrow[u ]& \oplus_{i=1}^r\Cb_{D_{i,a}},
    \arrow[u]
\end{tikzcd}
  \end{displaymath}
  where $\Mcc_D$ is the sheaf of  meromorphic functions on $D_a$, which
  contains the locally constant sheaf $\oplus_{i=1}^r\Cb_{D_{i,a}}$ (see
  \cite{saito-kyoji:log}*{Def 2.2}). 

  Similarly to
  \Lemma{\ref{exact-sequence}} (with a similar proof) we have:
  \begin{lemma} The following sequence is exact
    \begin{equation}\label{exact-dlog}
  0 \to \dlog \Oc^*_{X_a}(*D_a) \to \Omega^{cl}_{X_a}(\log D_a) \xrightarrow{\reso}
  \bigoplus_{i=1}^r (\frac{\Cb}{\Zb})_{D_{a,i}}\to 0. 
\end{equation}
  \end{lemma}
  The corresponding long exact sequence in cohomology contains the map
\begin{displaymath}
\psi :   \bigoplus_{j=1}^r \frac{\Cb}{\Zb}\to H^1(X_a,  \dlog \Oc^*_{X_a}(*D_a)).
\end{displaymath}

\begin{theorem}\label{class-rank-1} Assume that $X$ is a smooth
  complex projective variety and $D$ be a divisor such that
  $X_0 = X\setminus D$ is affine. Then
  \begin{enumerate}
  \item \begin{displaymath} \bar \Io_{conn}(X,D) = \Ker (\psi)\oplus
      d\Gamma(X_0, \Oc_{X_0})\oplus \Cb^g\oplus \bar \Io^0(X).
    \end{displaymath}
\item 
\begin{displaymath}
  \bar \Io_{conn, reg}(X,D) =\Ker (\psi) \oplus \bar \Io^0(X).
\end{displaymath}
  \end{enumerate}
\end{theorem}
Compare also  with \Proposition{\ref{lem-rank1}}.
\begin{proof} 
  Since $ (\dlog \Oc^*_{X_0}) \cap d\Oc_{X_0} = 0$
  \Lem{\ref{exact-log}}, \Propositions{\ref{ample}}{\ref{lem-rank1}}
  give
    \begin{displaymath}
    \bar \Io^{r}_{conn}(X,D) = \frac {\Gamma(X, \Omega^{cl}_X(\log D))}{\dlog \Gamma(X_0,
      \Oc^*_{X_0})}\oplus d\Gamma(X_0, \Oc_{X_0}) \oplus \Cb^g,
  \end{displaymath}
and
    \begin{displaymath}
      \bar \Io^{r}_{conn,reg}(X,D) = \frac {\Gamma(X, \Omega^{cl}_X(\log D))}{\dlog \Gamma(X_0,
        \Oc^*_{X_0})},
  \end{displaymath}
  since there can be no exponential component in a rational connection
  with regular singularities and also since the rational forms
  $\omega_i$ forming a basis for $H^1(X,\Omega_X)$ do not define
  regular singular connections (see (\ref{g-forms})). The sequences
  (\ref{exact-dlog}) and
  \begin{equation}\label{basic-exact}
  0 \to \Cb^*_{X_a}\to \Oc^*_{X_a}(*D_a) \xrightarrow{\dlog} \dlog (\Oc^*_{X_a}(*D_a)
  ) \to 0,
\end{equation}
now  imply
  \begin{align*}
    \frac {\Gamma(X, \Omega^{cl}_X(\log D))}{\dlog \Gamma(X,
    \Oc^*_X(*D))} &= \frac {\Gamma(X_a, \Omega^{cl}_{X_a}(\log D_a))}{ \Gamma(X_a, \dlog
                    \Oc^*_{X_a}(*D_a))}\oplus \frac {\Gamma(X_a, \dlog
                    \Oc^*_{X_a}(*D_a))}{\dlog \Gamma(X_a, 
                    \Oc^*_{X_a}(*D_a))}\\
                  &= \Ker (\psi) \oplus \bar \Io^0(X); 
  \end{align*}
  the last equality follows since $ \bar \Io^0(X) = \Ker (C)$, where
  the map
  \begin{displaymath}
C: H^1(X_a, \Cb^*_{X_a})= \bar \Io(X)   \to H^1(X_a,\Oc^*_{X_a}(*D_a)) =  \Pic (X_0),
    \end{displaymath}
    occurs in the long exact sequence of the cohomology of
    (\ref{basic-exact}), and the last equality follows from G.A.G.A.;
    see also \Proposition{\ref{lem-rank1}}.
\end{proof}

When all rank 1 connections on the affine variety $X_0$ are trivial
then obviously $\bar \Io^0 (X) = \bar \Io(X)$; if $X$ is simply
connected (\Section{\ref{sc-section}}), then $\bar \Io^0 (X) $ is a
singleton set; see also \Corollary{\ref{triv-ranke-1}}. We now
show how the subgroup $\Ker (\psi)$ of $(\Cb/\Zb)^r$ can be
successively computed using the exact sequences (\ref{basic-exact})
and
\begin{equation}
  \label{nonbasix-exact}
0 \to \Oc_{X_a}^* \to \Oc_{X_a}^*(*D_a)  \xrightarrow{\nu_D} \Zb_{D_a}\to 0,
\end{equation}
where $\Zb_{D_a}$ is the direct image on $X_a$ of the constant sheaf
on $D_a$ with values in the integers $\Zb$, and $\nu_D$ is the order
valuation along $D_a$. These sequences are exact also after erasing
the index $a$, where the first then implies
\begin{displaymath}
  \Gamma(X, \dlog (\Oc^*_{X}(*D) ) ) = \dlog \Gamma (X_0, \Oc_{X_0}^*),
\end{displaymath}
since $\Cb^*_{X}$ is flasque\footnote{If $\Zb_{D}$ also is flasque,
  i.e. $D$ is a a disjoint union of irreducible hypersurfaces, one also
  gets the well-known fact that the natural map
  $H^1(X, \Oc_X^*)\to H^1(U, \Oc_X^*) $, $U = X \setminus D$, is
  surjective.}. The sequences (\ref{basic-exact}) and
(\ref{nonbasix-exact}) give the fragments of a long exact sequences
\begin{displaymath}
\to  H^1(X_a,\Oc^*_{X_a}(*D_a)) \xrightarrow{b}
  H^1(X_a,\dlog (\Oc^*_{X_a}(*D_a)) \xrightarrow{c} H^2(X_a, \Cb^*_{X_a})\to,
\end{displaymath}
and
\begin{displaymath}
\to   \Zb^r \to  H^1(X, \Oc_X^*)\to H^1(X_a, \Oc^*_{X_a}(*D_a))\xrightarrow{d}
H^1(D_a,\Zb_{D_a})\to,
\end{displaymath}
where $H^1(X_a, \Oc_{X_a}^*)= H^1(X, \Oc_X^*)$ by G.A.G.A.. Let
$\Cb^r \to (\Cb/\Zb)^r$, $\lambda \mapsto \bar \lambda$ be the natural
projection. Given $\lambda \in \Cb^r $ there exist an open covering
$\{U_\alpha\}$ of $X_a$, elements
$\gamma_i \in \Omega^{cl}_{X_a}(U_\alpha)$ with
$\reso (\gamma_i)= \bar \lambda$, and
$\partial \gamma_i = \dlog \phi_{\alpha, \beta}$, where
$[\phi_{\alpha, \beta}]$ is a 1-cycle with values in
$ \Oc^*_{X_a}(X_a\setminus D_a)$, so that
$\psi (\bar \lambda) = [\dlog \phi_{\alpha, \beta}]$. Put
$c_{\alpha, \beta, \gamma }= \partial (\phi_{\alpha, \beta })$,
defining a 2-cycle with values in $\Cb^*$, and we have the map
\begin{displaymath}
  \psi_0= c\circ \psi : \bigoplus_{j=1}^r \frac{\Cb}{\Zb}  \to  H^2(X_a,
  \Cb^*),\quad 
  \bar \lambda \mapsto [c_{\alpha,\beta, \gamma}].
\end{displaymath}
Therefore there exists a map
$\psi^0 : \Ker (\psi_0)\to H^1(X_a,\Oc^*_{X_a}(*D_a)) $ such that the
restriction of $\psi$ to $\Ker (\psi_0)$ equals $ b\circ \psi^0 $.
Defining the map
\begin{displaymath}
  \psi_1 = d\circ \psi^0 : \Ker (\psi_0) \to H^1(D_a, \Zb_{D_a})
\end{displaymath}
there exists a map 
\begin{displaymath}
  \psi_2 :  \Ker (\psi_1) \to H^1(X, \Oc_X^*),
\end{displaymath}
such that
\begin{displaymath}
\Ker \psi = \Ker \psi_2 \subset \Ker \psi_1 \subset \Ker
(\psi_0).
\end{displaymath}
For example, if $X$ is a smooth complex projective curve and $D$ a set
of points in $X$, then $\psi_0: (\Cb/\Zb)^r \to \Cb^*$,
$\lambda \mapsto e^{2\pi i \sum_j \lambda_j}$, so that
\begin{displaymath}
  \Ker (\psi_1) = \Ker (\psi_0) =  \{ \bar \lambda \ \vert \sum_{i=1}^r \bar \lambda_i =0\}
\end{displaymath}
and
\begin{displaymath}
  \Ker (\psi) = \Ker ( \psi_2 : \{ \bar \lambda \ \vert \sum_{i=1}^r \bar \lambda_i =0\} \to H^1(X, \Oc^*_X)) 
\end{displaymath}
In particular, 
\begin{displaymath}
 \bar \Io^{r}_{conn, reg}(\Pb^1_\Cb,D)=\Ker (\psi)\oplus \{[\Oc_{\Pb^1_\Cb}]\}= \{ \bar
 \lambda \ \vert \sum_{i=1}^r \bar \lambda_i =0\}\oplus \{[\Oc_{\Pb^1_\Cb}]\}.
\end{displaymath}
Assuming that the support of $D= \sum D_i$  is 
disjoint from the point at infinity, $\Cb \subset \Pb^1= \Cb\cup
\{\infty\}$, the simple module corresponding to $\bar \lambda$ is $ \Dc_\Cb \prod_{i=1}^r (x-
x_i)^{\lambda_i}$, where $D_i = \{x_i\}$.  
\subsection{Étale trivial connections of rank 1 and a residue
  theorem}\label{etaletrivial}
\subsubsection{A residue theorem}
Let $\Divo_k(X)$ be the group of $k$-divisors, i.e. the abelian group
of formal sums $\sum \alpha_x x $, where the sum runs over points $x$
of height $1$ in a normal projective variety $X$, $\alpha_x\in k$, and $\alpha_x$ is $0$ except
for finitely many points $x$. Selecting an embedding $X\to \Pb^n$ the
degree $\deg (x)$ is the degree of the closure of $x$ in $\Pb^n$, and
one can define the degree map
\begin{displaymath}
  \deg : \Divo_k(X) \to k, \quad \sum_{\hto (x)=1}  \alpha_x x \mapsto
  \sum_{\hto (x)=1}  \alpha_x \deg (x)
\end{displaymath}
Let $\Divo_{0,k}(X)$ be the group of $k$-divisors of degree $0$ and
$\Clo_{k,0}(X)$ the group $k$-divisor classes modulo linear
equivalence.

Define the {\it residue map }
\begin{displaymath}
\Reso_K:  \Omega_K^{cl}\to \Divo_k(X), \quad \gamma \mapsto
\sum_{\hto(x)=1} \reso_x(\gamma) x.
\end{displaymath}

One can argue that the following result is a generalization of the
classical residue theorem for projective curves to higher dimensional
projective varieties.

\begin{theorem}\label{residue-theorem} Let $X$ be a normal projective
  variety over a field of characteristic $0$ and $K$ be its fraction
  field. The degree of the residue of a differential is $0$, so we
  have a map
  \begin{displaymath}
    \Reso_K : \Omega^{cl}_K  \to \Divo_{k,0}(X)
  \end{displaymath}
  and  a map
  \begin{displaymath}
    \overline{ \Reso}_K : \bar \Io(\Dc_K)\cong    H^0(\Cc_{\dlog}(K)) \to \Clo_{k,0}(X).
  \end{displaymath}
\end{theorem}
\begin{remark}\label{rem-residues}
  In (\ref{exact-picard}) we have the map
  $\Reso_X :\Io(X) \to \Pic^0(X)\cong \Clo^0(X)$, the integral
  Weil divisors that are algebraically equivalent to a principal
  divisor. We have inclusions
  $\bar \Io(X) \subset \bar \Io(\Dc_K)$ and
  $\Clo^0(X)\subset \Clo_0(X)\subset \Clo_{k,0}(X)$. The map
  $\overline{\Reso}_K$ in \Theorem{\ref{residue-theorem}} then extends
  $\Reso_X$.

\end{remark}
Let us first sketch a way of proving \Theorem{\ref{residue-theorem}}
by relying on the residue theorem for curves. Given a closed
1-form $\omega$ with polar locus $D$ there exists an irreducible curve
$C$ that does not belong to $D$. Its restriction $\omega\vert_C $ is a
1-form on $C$ that has residues along $C\cap D \subset C$ which are
related to the residues of $\omega$ along $D$ by the intersection
multiplicity between $C$ and $D$. The result can then be concluded
from Bezout's theorem and the fact that the sum of residues on a curve
is $0$.

The detailed proof below is more uniform in that higher-dimensional varieties
are not treated differently from curves. It follows the same general
line of argument as Hasse's proof for curves as presented in
\cite{serre:alg-groups}*{\S 11}, where one starts with the projective
line, but here the treatment of higher-dimensional projective spaces
is more complicated. For curves the hardest case is when $k$ has
postitive characteristic so one should really extend the above result
also to positive characteristic, but notice that we do not require
that $k$ be algebraically closed.

\begin{proof}
  (a) We first prove the assertion when $X= \Pb^m_k$, where the degree
  map $\deg : \Divo_k(\Pb_k^m) \to k$ is defined by linear extension
  from the usual degree $\deg (V) $ of a prime divisor $V$ in
  $\Pb^m_k$. Let $\omega$ be a closed differential form with poles
  along a hypersurface $D= \cup_{i=1}^r D_i$ in $\Pb^m_k$, where the $D_i$
  are irreducible, and let $H$ be a linear hypersurface in $\Pb_k^m$
  that is not a component of $D$. Put
  $\Reso_{D_i}(\omega) = \alpha_i D_i$, and set
  $\bar \alpha_i = \alpha_i/n_i$, where $n_i = [\bar k_i:k]$ and
  $\bar k_i$ is the algebraic closure of $k$ in the fraction field of
  $D_i$. Let $P_i\in \Gamma(\Pb^m_k\setminus H, \Oc_X)$ be
  irreducible polynomials that define $D_i \cap X_0$. The  differential
  \begin{displaymath}
    \omega'= \omega - \sum_i \bar \alpha_i \frac{dP_i}{P_i}\in \Gamma(X,
    \Omega^{cl}_X(*(H\cup D)) ),
\end{displaymath}
and its residues along any divisor different from $H$ is $0$. We will
prove that the residue of $\omega'$ along $H$ is 0. Since
$X_0=\Pb^m_k\setminus (D\cup H)$ is affine, by
\Proposition{\ref{ample}} (see (\ref{rep-rat-form}))
$\omega'=\gamma + d\psi$, where
$\gamma\in \Gamma(\Pb^m_k, \Omega^{cl}_{\Pb^m_k}(\log (H)))$ and
$\psi \in \Gamma(X_0,\Oc_{X_0})$. Since $\Reso_W(d\psi)=0$ for any
hypersurface $W$, we get
\begin{displaymath}
  \Reso_H(\omega') = \Reso_H(\gamma).
\end{displaymath}
Now since (this is detailed below)
\begin{equation}\label{log-single}
  \gamma \in \Gamma (\Pb^m_k,\Omega_{\Pb^m_k}^{cl}(  \log (H)))=
  \Gamma(\Pb^m_k, \Omega_{\Pb^m_k}^{cl}) =   \Gamma(\Pb^m_k, \Omega_{\Pb^m_k}),
\end{equation}
we get $ \Reso_H(\gamma)=0$ and hence $\Reso_H(\omega')=0$, so that
 \begin{displaymath}
   \deg  \Reso(\omega ) = \sum_{i=1}^r \alpha_i \deg \Reso(\frac {dP_i}{P_i}) = \sum_{i=1}^r
   \alpha_i \deg (P_i) =0,
\end{displaymath}
  since the degree of a principal divisor is $0$.
  
  It remains to see (\ref{log-single}). Select homogeneous coordinates
  $\bar t_i$ such that $H=V(\bar t_0)$. Put $t_i= \bar t_i/\bar t_0$
  and $s_i = \bar t_i /\bar t_j$, $j\neq 0$, so that $t_i= s_i/s_0$,
  $i\neq j$, $t_j = 1/s_0$, and $H\cap ( \bar t_j \neq 0) = V(s_0)$;
  put also  $A= \Gamma(\Pb^m_k\setminus H, \Oc_{\Pb^m_k})= k[t_1, \ldots ,
  t_m]$. We have
\begin{displaymath}
  \gamma = \sum_{i=1}^m a_i dt_i =  a_j d(\frac {1}{s_0}) +
  \sum_{i=1, i\neq j}^m a_i  d(\frac{s_i}{s_0})  =-\frac 1{s_0^2}(a_j
  +\sum_{i=1, i\neq j}^m s_ia_i )  ds_0 + \sum_{i=1, i\neq j}^m \frac {a_i}{s_0}  ds_i,
\end{displaymath}
where $a_i \in A$. On the right,  consider the expression
\begin{displaymath}
\psi_j=  a_j +\sum_{i=1,
    i\neq j}^m s_ia_i = a_j +\sum_{i=1,
    i\neq j}^m \frac {t_i}{t_0}a_i.
\end{displaymath}
Since $\nu_H(t^\alpha) = -|\alpha| $ for each monomial in $A$, each
monomial term in $a_j$ and $\frac {t_i}{t_0}a_i$ has a non-positive
value with respect to the valuation $\nu_H$ of $H$. This implies that
\begin{displaymath}
\Reso_H(\gamma)=  (-\frac {\psi_j}{s_0^2})_{-1} =0.
\end{displaymath}
Here the residue is computed by expanding in $K(\bar R_\nu)$, where
$\bar R_{\nu_H}$ is the adic completion of the local ring
$ R_{\nu_H} $ of $H$, $s_0$ is a uniformizing parameter, and the $s_i$
determine a coefficient field in $\bar R_{\nu_H}$. A logarithmic
differential whose residues are zero is regular (\ref{zero-resdues}),
so that $\gamma\in \Gamma(\Pb^m_k, \Omega_{\Pb^m_k}^{cl})$ . The last
equality in (\ref{log-single}) follows from
\Theorem{\ref{global-forms}}.

(b) We can regard $X$ as a subvariety $X\subset \Pb^n_k$ for some
integer $n$. Put $m= \dim X$ and let $L$ be a maximal linear subspace
of $\Pb^n_k$ that is disjoint with $X$ (so that $\codim L = m+1$) and
$p: \Pb^n_k \setminus L \to \Pb^m $ be the projection map. The
restriction of $p$ to the map $\pi: X\to \Pb_k^m$ is then finite (see
\cite{mumford:redbook}*{\S 7, Prop 6 }). If $x$ is a point of height
$1$ in $X$ and $y=\pi(x)$, then
\begin{equation}\label{deg-points}
  \deg (x)= [k_x: k_y]\deg y
\end{equation}
(see \cite{derksen-kraft}*{Prop 8.3}). The trace $\Tr : L \to K$,
where $L$ is the fraction field of $X$ and $K$ that of $\Pb_k^m$, defines a map
$\Omega_L = L\otimes_K\Omega_K \to \Omega_K$ that restricts to a map
\begin{displaymath}
\Tr:  \Omega_L^{cl} \to \Omega^{cl}_K.
\end{displaymath}
By (a) the proof will be completed if we prove that
\begin{displaymath}
  \reso_y(\Tr(\omega))\deg (y) =  \sum_{x\to y}\reso_x(\omega)\deg (x),
\end{displaymath}
which by (\ref{deg-points}) is equivalent to proving
\begin{displaymath}
  \reso_y(\Tr(\omega)) =  \sum_{x\to y}\reso_x(\omega)[k_x: k_y].
\end{displaymath}
We have
\begin{displaymath}
  L \otimes_{ K}\hat K_y = \prod_{x\to y}  \hat  L_x
\end{displaymath}
where the $\hat {\ } $ denotes $\mf$-adic completions of the fields, and 
\begin{displaymath}
  \Tr (f) = \sum_{x\to y} \Tr_x(f), \quad f\in L,
\end{displaymath}
where $ \Tr_x$ denotes the trace from $\hat L_x $ to $\hat K_y$  (see
\cite{serre:alg-groups}*{II, \S 12}). Let $t_y$ and $t_x$ be  local parameters of
  the discrete valuation ring $\Oc_{Y,y}$ and $\Oc_{X,x}$,
  respectively, so that $t_y= u t_x^e$ where $u$ is a unit in
  $\Oc_{X,x}$ and $e$ is some integer $\geq 1$. By the additivity of
  $\reso_y$
  \begin{displaymath}
    \reso_y(\Tr (\omega)) = \sum_{x\to y}  \reso_y(\Tr_x (\alpha_x
    \frac {dt_x}{t_x} + \omega_x^0) ) = \sum_{x\to y}  \reso_y(\Tr_x (\alpha_x
    \frac {dt_x}{t_x}) ).
  \end{displaymath}
By (\ref{deg-points}) it now suffices to prove
\begin{displaymath}
  \reso_y(\Tr_x (\alpha_x    \frac {dt_x}{t_x}) ) =
[k_x:k_y]  \reso_x(\alpha_x\frac {dt_x}{t_x}) )= [k_x:k_y]\alpha_x.
\end{displaymath}
Since
\begin{displaymath}
  \frac {dt_y}{y_y} = \frac {du}u + e \frac {dt_x}{t_x},
\end{displaymath}
we get
\begin{displaymath}
  \alpha_x \frac{dt_x}{t_x} = \frac {\alpha_x}e  \frac {d
    t_y}{t_y} - \frac {\alpha_x}e \frac {du}u. 
\end{displaymath}
Since moreover $\Tr_x(1) = [L_x:K_y]= e [k_x:k_y]$ and $\Tr(du/u))$
has no pole at $y$, and therefore $\reso_y(\Tr(du/u))=0$, we get
\begin{displaymath}
  \reso_y (\Tr_x (\alpha_x \frac {dt_x}{t_y})) = \reso_y(\Tr_x (\frac {\alpha_x}e  \frac {d
    t_y}{t_y} - \frac {\alpha_x}e \frac {du}u))= \reso_y ([k_x:k_y] \alpha_x  \frac {d
    t_y}{t_y} ) = [k_x:k_y] \alpha_x.
\end{displaymath}
\end{proof}
The set of linear characters
$\hat \pi_1(X)= Hom_\Cb(\pi_1(X_a), \Cb^*) \cong \overline \Io(X)$ of the
fundamental group of a smooth projective $X $ was related to
$\hat \pi_1(X_0\setminus H)$ in \Theorem{\ref{class-rank-1}} (2),
where $H$ is a hypersurface. By \Theorem{\ref{residue-theorem}},
$\Gamma(X,\Omega_{X}(\log H))= \Gamma(X,\Omega_X)$ for an irreducible
hypersurface $H$, so that the map $\psi$ in
\Theorem{\ref{class-rank-1}} is injective. Therefore, removing $H$
from $X$ will not change the linear characters.
\begin{corollary}\label{triv-ranke-1}
  Let $X$ be a smooth projective variety and $H$ be an irreducible
  hypersurface in $X$. Then
\begin{displaymath}
      \hat \pi_1(X)= \hat \pi_1(X\setminus H) = \overline \Io(X)=
      \overline \Io(X\setminus
      H).
    \end{displaymath}
\end{corollary}
This can also be expressed by saying that that if $M$ is connection of
rank 1 on $X\setminus H$ with regular singularities along $H$, then in
fact $M$ is the restriction of  a connection on $X$.
\subsubsection{Degree of Hodge classes}
If the Hodge conjecture holds on a smooth projective variety $X$ for
the Hodge classes $H^{p,p}(X_a,\Cb) $ in $H^n(X_a, \Cb)$, a natural
question to ask is whether the notion of degree of a projective
variety also can be attached to a Hodge class. We shall see that this
is possible at least when $p=1$.

By Lefschetz $(1,1)$-theorem the cycle map
\begin{displaymath}
  \Clo(X) \to H^2(X_a, \Cb)= H^2(X, \Oc_X)\oplus
H^1(X,\Omega_X)\oplus H^0(X,\Omega^2_X )
\end{displaymath}
composed with the projection to the middle term gives a surjective map
to the integral Hodge classes
\begin{displaymath}
  \Clo(X) \to H^2(X_a,\Zb)\cap H^1(X,\Omega^1_X) \to 0,
\end{displaymath}
which also implies  that we get a surjective map
\begin{displaymath}
  c:   \Clo_{\Cb}(X) \to H^1(X,\Omega^1_X) \to 0.
\end{displaymath}

\begin{corollary}\label{hodge-degree} The degree map $\deg : \Clo_{\Cb} (X)\to \Cb$ can be
  factorized as $\deg = \overline{ deg} \circ c$, where
  $\overline{ \deg}$ is a map
  \begin{displaymath}
    \overline{ \deg}:     H^1(X,\Omega^1_X) \to \Cb.
  \end{displaymath}
\end{corollary}
\begin{proof} It suffices to prove that if
  $\Lc = \sum_{i=1}^r \alpha_i[D_i]$, where $D_i$ are prime divisors
  in $X$, belongs to the kernel of the cycle map, $c(\Lc)=0$, then
  $\deg (\Lc)=0$. Let $\Cc$ be the direct sum of the direct images of
  the sheafs $\Cb_{D}$, whith respect to the inclusion maps $D\to X$,
  where $D$ runs over all prime divisors in $X$. Let
  $\Omega_{X_a}^{cl}(\log )$ be the sheaf of closed holomorphic
  logarithmic differentials. We have then the  exact sequence
  \begin{displaymath}
    0 \to \Omega^{cl}_{X_a}\to \Omega_{X_a}^{cl}(\log )\xrightarrow{\reso}  \Cc \to 0,
  \end{displaymath}
  hence also the exact sequence
\begin{displaymath}
  \Gamma(X_a,\Omega_{X_a}^{cl}(\log )) \xrightarrow{\reso} \Divo_{\Cb}   \xrightarrow{\delta}  H^1(X_a,
  \Omega_{X_a}^{cl}) \to .
\end{displaymath}
Then the cycle map $c= \delta \circ p_1 $, where $p_1$ is the
projection from the Hodge decomposition
$H^1(X_a, \Omega_{X_a}^{cl}) = H^1(X, \Omega_X )\oplus H^0(X,
\Omega^2_X) \to H^1(X, \Omega_X ) $ \Th{\ref{global-forms}}. Hence by the above cited Lefschetz
theorem we have the exact sequence
\begin{displaymath}
  \Gamma(X_a,\Omega_{X_a}^{cl}(\log)) \xrightarrow{\reso} \Divo_{\Cb}    \xrightarrow{c}  H^1(X,
  \Omega_{X}) \to  0.
\end{displaymath}
Therefore $\Lc = \reso (\gamma)$ for some
$\gamma \in \Gamma(X_a,\Omega_{X_a}^{cl}(\log)) $ so that by
\Theorem{\ref{residue-theorem}}, $\deg (\Lc) = \deg \reso(\gamma) =0$.
\end{proof}

\subsubsection{ Étale trivial rank 1 connections.} We will classify
the étale trivial objects in $\Io(\Dc_K)$ when $K/k$ is a finitely
generated field extension.    
\begin{proposition}\label{rat-ext} Let $M=M_\gamma$ be a $\Dc_K$-module such
  that $\dim_K M =1$, determined by a closed 1- form
  $\gamma \in \Omega^{cl}_K$. The following conditions are
  equivalent:
    \begin{enumerate}
    \item There exists a finite extension $\pi:K\to L$ such that
      $\pi^!(M)\cong L$.
    \item There exists a radical extension $\pi : K \to L$ such that
      $\pi^!(M)\cong L$.

    \item There exists a finite extension $\pi:K\to L$ such that
      $\gamma \in d\log (L)$.
    \item There exists an integer $l$ such that $l\gamma \in \dlog K$.
    \end{enumerate}
  \end{proposition}
  See also \Remark{\ref{rem-ab-ext}}.
  \begin{proof} $(1)\Leftrightarrow (3)$ and $(2)\Rightarrow (1)$ are
    evident. $(4)\Rightarrow (2)$: If $\gamma = \dlog (y)/l$ for some
    $y\in K$, putting $x^l= y$ we get $\gamma = \dlog (x)$. It then
    suffices to prove $(3)\Leftrightarrow (4)$. Let $\{y_i\}\subset K$
    be a transcendence basis of $K$, and $\{\partial_j\} $ derivations
    such that $\partial_j(y_i)= \delta_{ij}$.

  $ (4)\Rightarrow (3) $:  There exists $a\in K$   such that 
  \begin{displaymath}
    \gamma = \sum_i \gamma_i  d y_i  = \frac 1l  d \log (a) = \sum_i \frac 1l
    \frac {\partial_i(a)}{a} d y_i.
\end{displaymath}
Let $L$ be the extension that is generated by $K$ and elements $x_i$
such that $x_i^l- a=0$, so that
$\partial_{i}(x_i)/x_i = 1/l (\partial_i (a)/a) $. Then
\begin{displaymath}
  \gamma = \sum_i  \frac {\partial_{i}(x_i)}{x_i} dy_i= d \log (\prod_i x_i) \in d \log (L).
\end{displaymath}
$(3)\Rightarrow(4):$ Assume that $\gamma= d\log (b)$, $b\in L$, where
$L/K$ is finite, and let
$Q[X] = X^n + Q_{n-1}X^{n-1}+ \cdots + Q_0\in K[X]$ be the minimal
polynomial of $b$, $Q(b)=0$. Put $dQ = \sum_{i=0}^{n-1} dQ_i X^i$,
$Q' = n X^{n-1}+ \cdots + Q_1$, and
$\partial_j Q = \sum_{i=0}^{n-1} \partial_j (Q_i) X^i$. Then
$d Q[b] + Q'[b]d b =0$, so that $db = -d Q/ Q'(b)$, and writing
$\gamma = \sum \gamma_j dy_j$, $\gamma_j\in K$, we get
\begin{displaymath}
  bQ'(b) \gamma_j = - \partial_j(Q)(b).
\end{displaymath}
Using the fact that $Q$ is the unique non-zero monic monomial of
minimal degree such that $Q(b)=0$, a straightforward computation gives
\begin{displaymath}
  \frac {\partial_j(Q_i)}{Q_i}  = i \gamma_j,
\end{displaymath}
whenever $Q_i \neq 0$. There exists $l$ such that $Q_l\neq 0$, and
then putting $a= Q_l$ we get
\begin{displaymath}
  l \gamma = d \log (a) \in \dlog K. 
\end{displaymath}
Alternative: Let $\bar L/L/K$ be a Galois cover of $L/K$, let
$\Tr:\bar L\to L$ and $ N: \bar L^*\to K^*$ be the trace and norm
maps, respectively. The trace extends to a map
$\Tr: \Omega_{\bar L} \to \Omega_K$, and
$\Tr (\dlog (\alpha)) = \dlog (N(\alpha))$, for $\alpha \in \bar L$.
Since $\gamma = \dlog \alpha \in \Omega_K$, we get
$ [\bar L, K]\gamma =\Tr (\gamma) = \dlog (N(\alpha))\in \dlog K$.
\end{proof}
  Thus if $M_\gamma$ is an étale trivial $\Dc_K$- module of rank $1$ 
  there exists an element $\phi$ in $ K $ and an integer $l$ such that
\begin{displaymath}
  M_\gamma \cong \Dc_K \phi ^{1/l}
\end{displaymath}
and $\gamma = \frac 1l \dlog (\phi)$. The isomorphism classes of étale
trivial modules $ \Io^{et}(K)$ can be identified with a subspace of
$ \Omega^{cl}_K/\dlog K$, and doing this we assert that
\begin{displaymath}
\psi:   \Io^{et}(K)\to \frac {\Qb}{\Zb}\otimes_{\Zb} \dlog (K), \quad \gamma
  \mapsto \frac 1l\otimes \dlog \phi
\end{displaymath}
is a well-defined map. To see this, if
$\gamma_1 = \gamma + \dlog (\nu) $, then
$\psi (\gamma_1)= \psi (\gamma)$ since $1\otimes \dlog (\nu)=0$ in
$\frac {\Qb}{\Zb}\otimes_{\Zb} \dlog (K)$. If
$1/l \dlog (\phi) = 1/{l_1}\dlog (\phi_1)$, then
\begin{displaymath}
  \frac 1{l_1}\otimes \dlog (\phi_1) = \frac 1{ll_1}\otimes l\dlog
  (\phi_1) = \frac 1{l l_1}\otimes l_1\dlog (\phi) = \frac 1l \otimes
  \dlog \phi. 
\end{displaymath}

\begin{theorem}\label{etale-trivial-iso} The map $\psi$ is an
  isomorphism of groups. Therefore each element in $\Io^{et}(K)$ has a
  representation of the form
  \begin{displaymath}
    \Dc_K \prod_i a_i^{r_i}
  \end{displaymath}
  where $r_i \in \Qb$ and $a_i \in K$.  Moreover, 
\begin{displaymath}
  \Dc_K \prod_i a_i^{r_i} \cong   \Dc_K \prod_i b_i^{r_i'}
\end{displaymath}
if and only  if 
\begin{displaymath}
  \frac{\prod_i a_i^{r_i}}{\prod_i b_i^{r_i'}}  \in K.
\end{displaymath}

\end{theorem}
\begin{proof}
  It is an easy verification to see that $\psi$ is a homomorphism. If
  $1/l\otimes \dlog (\phi)=0$, then
  $\dlog (\phi) = l_1 \dlog (\phi_1)$, where $l$ divides $l_1$, hence
  $1/l \dlog (\phi) = \frac {l_1}l\dlog (\phi_1) = \dlog
  (\phi_1^{l_1/l})$ which maps to the identity in $\Io^{et}(K)$. The
  surjectivity should now also be clear. It is straightforward to see
  also the remaining assertions.
\end{proof}

Let $X$ be a normal projective model for $K$, so that $K$ is the
fraction field of $X$. One can determine whether $M_\gamma$ is étale
trivial from knowledge of the torsion subgroup $\Clo^t(X)$ of the
divisor class group,
\begin{displaymath}
  \Clo^t(X)\subset \Clo_0(X) \subset \Clo(X) .
\end{displaymath}
Evidently (and by \Theorem{\ref{etale-trivial-iso}}), if $M_\gamma$ is
étale trivial, then $\gamma \in \Gamma(X, \Omega_X^{cl}(\log D))$ for
some divisor $D$ on $X$. For a given
$\gamma \in \Gamma(X, \Omega_X^{cl}(\log D))$  there is the following
procedure to determine
whether $M_\gamma \in \Io^{et}(K)$:
\begin{enumerate}
\item  Check if $\Lc= \Reso(\gamma)$ has rational
  coefficients, and if so one can find an integer $l_1 $ such that
\begin{displaymath}
\Lc_1=   l_1 \Lc \in \Clo_0(X).
\end{displaymath}
\item  Determine if $\Lc_1 \in \Clo^t(X)$. 
\item Determine an integer $l$ and a rational function $\phi \in K $
  such that $l\Lc_1 = ll_1 \Lc \sim (\phi)$. Then
\begin{displaymath}
\omega=   ll_1\gamma - \dlog \phi\in \Gamma(X, \Omega_X)
\end{displaymath}
and since $\dlog K \cap \Gamma(X, \Omega_X) =0$, we get
$M_\gamma \in \Io^{et}(K) $ if and only if $ll_1\gamma \in \dlog K$.
\end{enumerate}

We can formulate this  more succinctly. Define the map
\begin{displaymath}
  \Reso :   \Ic^{et}(K) \to  \Clo^t_{\Qb}(X),\quad  [\gamma]\mapsto [\Reso(\gamma)],
\end{displaymath}
which is well-defined since $\Reso (\dlog (\phi )) = (\phi )$.
Define also the map
\begin{displaymath}
  \Reso^{-1} : \Clo^t_{\Qb}(X) \to \Ic^{et}(K), \quad  [D]\mapsto [\frac 1r
  \dlog \phi],
\end{displaymath}
where $rD\sim (\phi)$ for some $r\in \Qb^*$ and $\phi\in K$. This map
is also well-defined, since if $r_1D \sim (\phi_1)$, $r_1 \in \Qb^*$,
$\phi \in K$, then
\begin{displaymath}
  \Reso (  \frac 1r   \dlog \phi- \frac 1{r_1}  \dlog \phi_1) \sim
  [D]- [D]=0
\end{displaymath}
hence there exists $\psi \in K$ such that  
\begin{displaymath}
  \frac 1r   \dlog \phi- \frac 1{r_1}  \dlog \phi_1 = \dlog \psi
\end{displaymath}
implying that $[\frac 1r \dlog \phi] = [\frac 1{r_1} \dlog \phi_1]$.
It is straightforward to see that $\Reso$ and $\Reso^{-1}$ define
homomorphisms of groups.
\begin{theorem}\label{etale-trivial-th} 
  The residue map defines an isomorphism of groups
  \begin{displaymath}
\Reso:    \Ic^{et}(K) \to  \Clo^t_{\Qb}(X)
  \end{displaymath}
  whose inverse is given by $\Reso^{-1}$.
\end{theorem}

\begin{proof}
  If $M_\gamma$ is étale trivial, then there exists an integer $l$
  such that $l \gamma = \dlog (\phi)$, $\phi \in K$. We have
\begin{displaymath}
  \Reso^{-1}\circ \Reso ([\gamma]) = \Reso^{-1} (\Reso (\frac 1l \dlog
  (\phi)))  =  \Reso^{-1}(\frac 1l [\phi]) = [\frac 1l \dlog \phi]= [\gamma].
\end{displaymath}
If  $[D]\in \Clo^t_{\Qb}(X) $ and $r[D] = (\phi)$, then
\begin{displaymath}
    \Reso \circ \Reso^{-1} ([D]) =     \Reso (\frac 1r \dlog (\phi)) =
    \frac 1r (\phi) = [D].
\end{displaymath}
\end{proof}

\begin{remark} If $\Clo^t(X)=0$, for instance when $\Clo(X)=\Zb$, (2)
  simply means $\Lc_1=0$. This implies that
  \begin{displaymath}
    \Clo_Q(X)  = \Clo_Q^t(X) = \Io^{et}(K)=    \Qb/\Zb\otimes \dlog K. 
\end{displaymath}
This is the case when $X$ is a smooth complete intersection of
dimension $\geq 3$ by the Grothendieck-Lefschetz theorem, and by the
Noether-Lefschetz theorem it also holds when $X$ is a hypersurface in
$\Pb^3_{\Cb}$ of degree $\geq 4$ having ``general moduli'', see
\cite{griffithsharris:math-ann}.
\end{remark}

Let us for example spell out what
all this means when $X= \Pb^n_k$ and $K$ is its fraction field. We
assert that for any $\gamma \in \Omega^{cl}_K(\log )$ with rational
residues, $\Reso (\gamma)\in \Clo_{\Qb}(\Pb^n_k)$, we have
$M_\gamma \in \Io^{et}(K)$. To see this, let $l_1$ be an integer such
that $\Reso(l_1\gamma)\in \Clo_0(\Pb^n_k) =0$. Hence there exists
$a \in K$ such that $\Reso(l_1\gamma) = (a) $, now regarded as divisors
in $\Pb^n_k$. Since
\begin{displaymath}
  l_1\gamma - \dlog (a) \in \Gamma(\Pb^n_k, \Omega_{\Pb^n_k}) =0
\end{displaymath}
we get that $M_\gamma \in \Io^{et}(K)$. Since $K$ is the fraction
field of a polynomial ring $A$, which is a unique factorization
domain, we have $a=\prod_i P_i^{n_i}$ for irreducible polynomials
$P_i$ in $A$, where $n_i\in \Zb$ and the $P_i^{n_1}$ are unique up to
a multiplicative constant. Thus
\begin{displaymath}
M_\gamma = \Dc_K \prod_i P_i^{n_i/l_1}.
\end{displaymath}
Any étale trivial $\Dc_K$-module of dimension $1$ over $K$ is of this
form. If $Q_i$ is another set of irreducible polynomials representing
$M_\gamma$ as above, then $(P_i/Q_i)^{n_i}\in K^{l_1}$. Moreover,
there exists a radical field extension $L=K[b]= K[X]/(X^n-a)$,
$a\in K$, such that $L\otimes_K M_\gamma \cong L$. The integer $n$ can
be selected to be the smallest integer such that
$n\reso_x(\gamma) \in \Zb$ for all points $x$ of height $1$ in $\Pb^n_k$.

To answer (2) when $\Clo^t(X)\neq 0$ one needs to look for a bound on
the order of $\Lc_1$, which is a problem that is not addressed here.
When $X$ is a smooth curve such a method has been devised by Risch
\cite{risch:bulletin}.

\subsection{Liouville's theorem}\label{liouville-section} One can ask when $M_\gamma$,
$\gamma \in \Omega_K^{cl}$, becomes exponential after a field
extension $E/K$, i.e. \/
$E\otimes_KM_\gamma \cong M_{d\beta} = \Dc_E e^\beta $ for some
$\beta \in E$ (\ref{exp-mod-sec}). This gives the condition
$\gamma = d(\beta) + \dlog (\phi)$ for some $\phi ,\beta \in E$. The
minimal such extension is $E_1=K(x_1)$ where
$ d(x_1) = \dlog (\phi)\in \dlog E$, so that $\gamma$ becomes exact in
$\Omega_{E_1}$, $\gamma = d(\beta + x_1)$. Making a further extension
$E_2= E_1(x_2)/E_1/K$ such that $d(x_2) = x_2 d(\beta)$, we get
$E_2\otimes_K M_\gamma \cong E_2$.

A field extensions $E/K$ is {\it
  elementary} if it is a tower of simple extensions $E= K(x)$, which
are either (i) finite, (ii) exponential, $\dlog x \in dK$, or (iii)
logarithmic, $d x \in \dlog (K)$.

We show in \Lemma{\ref{int-closure}} that the algebraic closure
$\hat k$ of $k$ in $E$ equals the field of constants of $E$.
\begin{theorem}\label{liouville} Let $M_\gamma\in \Io(K)$. The following are
  equivalent:
  \begin{enumerate}
  \item There exists an elementary extension $E/K$  such that
    $E\otimes_KM_\gamma \cong \Dc_E e^\beta$ for some $ \beta \in E$,
    i.e. $\gamma = d(\beta) + \dlog (\phi)$ for some $\phi \in E$.
  \item $\gamma \in (\hat k \dlog (\hat k K) + dK)\cap \Omega_K$, where
    $\hat k$ is the algebraic closure of $k$ in $E$.
  \end{enumerate}
\end{theorem}
It is straightforward to see that (2) implies (1). That (1) implies
(2) is contained in Liouville and Ostrowski's classical theorem about
elementary integrable functions, when $K$ is the field of meromorphic
functions in 1 variable over the complex numbers; their proof is
analytic. Rosenlicht \cite{rosenlicht} gave an algebraic proof of the
univariate case in terms of differential algebra, assuming
$\hat k \subset K$, and in \cite{risch:integration} this univariate
case was extended by allowing that the algebraic closure of $k$ in $K$
and $E$ be different (non-regular extensions). This was extended to
several derivations in \cite{caviness-rothstein:liouville}, using a
similar method.

In the proof below the rather complicated study of partial fractions
in \cite{rosenlicht} is replaced by a use of discrete valuations,
non-regular elementary extensions are dealt with in a direct way, and
several derivations are allowed.

\begin{proof} So assume (1). By induction over the length of the tower of
  elementary extensions in $E/K$, it suffices to prove that if
  $\pi: K \to E= K(x)$ is a simple extension of the types (i-iii) and
  \begin{displaymath}\tag{*}
  \gamma = d\beta +\sum_i c_i \dlog \alpha_i \in \Omega_K,
\end{displaymath}
$\alpha_i, \beta \in E$, $c_i \in \hat k$, then actually
$\gamma \in \hat k \dlog (\hat k K) + dK$. (This is the same strategy
as in \cite{rosenlicht}, but there it is assumed that
$\hat k\subset K$.) In other words, we need to prove that
  \begin{displaymath}
    (\hat k \dlog (E) + dE) \cap \Omega_K= \hat k \dlog (\hat k K) + dK
  \end{displaymath}
  whenever $E/K$ is a simple extension that is either finite,
  exponential, or logarithmic, and we need only prove the non-trivial
  assertion that if $\gamma$ is an element of the left side, then it
  is also a member of the right.

  (i) $E/K$ is finite: Let $E_1 = K\hat k$ be the composite of
  $\hat k $ and $K$ in $E$, so we have the tower $E/E_1/K$. We prove
  the assertion first for $E_1/K$ and then $E/E_1$. Assume that
  $\beta, \alpha_i \in E_1 $. Let $V_c$ be the finite-dimensional
  vector space over $\Qb$ that is generated by the coefficients $c_i$
  in \thetag{*}. Then the $c_i$ can be expressed as linear
  combinations of elements in a basis of $V_c$ with coefficients in
  the rational numbers $\Qb$, and clearing denominators we see that in
  the right hand side of \thetag{*} one can assume that the elements
  $(c_i)$ are linearly independent over $\Qb$. Now by
  \Lemma{\ref{exact-log}} it follows that $d(\beta) \in \Omega_K$ and
  by \Proposition{\ref{van-trace-prop}}, (1),
  $d(E_1)\cap \Omega_K = d(K)$; hence $d(\beta) = d(\beta_1)$ for some
  $\beta_1\in K$. Since $\alpha_i \in E_1 = K \hat k  $ this gives the
  assertion for $E_1/K$.

  Considering $E/E_1$, let $\bar E /E/E_1$ be a Galois closure. Then,
  since $c_i\in \hat k \subset E_1$,
  \begin{align*}
[E,K] \gamma &=    \Tr (\gamma) = d (\Tr (\beta)) + \sum_i  (\Tr_{\bar
               E/E_1} (c_i \dlog
(\alpha_i) ) \\ &=   d (\Tr_{\bar E/E_1} (\beta)) + \sum_i  c_i \dlog
                  (\Norm_{\bar E/E_1} (\alpha_i)),
  \end{align*}
  hence $\gamma \in \hat k \dlog E_1 + d E_1$.

  It follows that for the remaining cases we can assume that $E/K$ is
  a regular extension, so that $\hat k$ equals the algebraic closure
  of $k$ in $K$ and moreover $E = K(x)$ where $x$ is transcendental
  over $K$. We are thus given rational functions
  \begin{displaymath}
    \alpha_i = \alpha_i^0 x^{n_0} \prod_{j=1}^{r} p_{j}^{n_{ij}},\quad  \beta= \beta^0x^{l_0} \prod_{j=1}^r
    p_j^{l_j}\in K(x),
\end{displaymath}
where $\alpha_i^0, \beta^0\in K$, $p_{i} $  are distinct monic
irreducible polynomials in $K[x]$, none equal to $x$, and
$n_{ij}, l_i\in \Zb$. Thus
\begin{align*}
  \gamma=  d\beta + \sum_i c_i \dlog \alpha_i &= x^{l_0} \prod_{j=1}^r
                                                p_j^{l_j} d(\beta^0) + l_0 \beta^0 x^{l_0-1}\prod_{j=1}^r
                                                p_j^{l_j} d(x) +  \beta^0x^{l_0} d(\prod_{j=1}^r
                                                p_j^{l_j})  \\
                                              &+  \sum_i c_i (\dlog (\alpha_i^0) + n_0 \dlog (x) + \sum_{ij} n_{ij}\dlog (p_{j}) ).
\end{align*}
Let $\nu: E \to \Zb$ be a discrete valuation such that $\nu(a)=0$ when
$a\in K$. Selecting a transcendence basis $(y_i)$ of $K$ over $k$, one
can extend $\nu$ to a map $E\otimes_K\Omega_K \to \Zb$ by putting
$\nu (\sum l_i dy_i ) = \min_i \nu(l_i)$. Let $\nu_0, \nu_j,$ be the
discrete valuations of $E=K[x]$ corresponding to the elements $x$ and
$ p_i$, and let them also denote the extensions to maps
$ E\otimes_K\Omega_K \to \Zb$. Then we have
 \begin{align*}
   \nu_0(\dlog (p_{j}))&= 0, \quad \nu_{j}(\dlog p_{j'}) =-\delta_{j,
     j'},\quad  \nu_0(x^{l_0})=l_0,\\ & \nu_0(\prod_{i=1}^r p_j^{l_j})=0,
   \quad \nu_j(\prod_{j=1}^r p_j^{l_j})= l_j.
 \end{align*}
 
 There are two cases to consider: (i) $\dlog (x)\in d( u)$, hence
 $d(x)= x d(u)$, and (ii) $d(x)= \dlog (u) $, hence
 $\dlog (x)= \dlog (u)/x $, for some $u\in K$.

 (ii)(adding an exponential $x=e^u$): We get 
 \begin{align*}
   \gamma  &=  x^{l_0} \prod_{j=1}^r
   p_j^{l_j} (d(\beta^0) + l_0 \beta^0 d(u)) +  \beta^0x^{l_0} d(\prod_{j=1}^r
   p_j^{l_j})  \\
  &+  \sum_i c_i (\dlog (\alpha_i^0) + n_0d(u) + \sum_{ij} n_{ij}\dlog (p_{j}) ),
 \end{align*}
 and, since
 $ \sum_i c_i (\dlog (\alpha_i^0) + n_0d(u)) \in \hat k \dlog K + dK$,
 it suffices to prove that if
\begin{displaymath}
  \gamma_0=  x^{l_0} \prod_{j=1}^r
  p_j^{l_j} (d(\beta^0) + l_0 \beta^0 d(u)) +  \beta^0x^{l_0} d(\prod_{j=1}^r
  p_j^{l_j}) + \sum_{ij} c_in_{ij}\dlog (p_{j})\in \Omega_K,
 \end{displaymath}
 then $ n_{ij}= 0$, and if $\beta_0 \neq 0$, then also $l_j=0$. If
 $\beta_0 =0$, then if $n_{ij}\neq 0$, it follows that
 $\nu_{j}(\gamma_0) = -1$, and therefore
 $\gamma_0 \not \in \Omega_K $; hence $n_{ij}=0$. Now assume that
 $\beta_0 \neq 0$. If $l_i <0$, then $ \nu_i(\gamma_0) = l_i-1\neq 0$,
 and therefore $\gamma_0\not \in \Omega_K $; hence $l_i \geq 0$. Then
 if $n_{ij}\neq 0$ so that $\nu_i(\gamma_0) = -1 $, and one gets
 $\gamma_0\not \in \Omega_K$; hence $n_{ij}=0$.

 It remains to eliminate the case $l_i>0, n_{ij}=0$. For
 $\omega = \sum_i r_i d(y_i) \in K[x]\otimes_K \Omega_K$, let
 $\deg \omega = \max \deg r_i$, where $\deg r_i$ is the usual degree
 of a polynomial. We have then $\deg d(p) < \deg (p)$, when $p$ is a
 monic polynomial, and therefore
 \begin{displaymath}
   \deg (x^{l_0} \prod_{j=1}^r  p_j^{l_j}) > \deg (x^{l_0} d(\prod_{j=1}^r
   p_j^{l_j})),
\end{displaymath}
implying that if $l_i >0$ for some $i \geq 0$, then
$\gamma_0 \not \in \Omega_K$.

(iii) (adding a logarithm, $x= \log (u)$): We get
 \begin{align*}
   \gamma  &=   (n_0(\sum_i c_i) + l_0 \beta^0  x^{l_0}\prod_j p_j^{l_j}
             )\frac {\dlog u}{x} 
             + x^{l_0} \prod_j p_j^{l_j} d(\beta^0) + \beta^0 x^{l_0}d(\prod_j
             p_j^{l_j})\\
           &   + \sum_i c_i \dlog (\alpha_i^0) + \sum_{ij} c_i n_{ij} \dlog p_{ij}.
 \end{align*}
 Since $\sum_i c_i \dlog (\alpha_i^0) \in \hat k \dlog K + dK$, it suffices to
 prove that if
\begin{align*}
\gamma_0  &=(n_0(\sum_i c_i) + l_0 \beta^0  x^{l_0}\prod_j p_j^{l_j}
             )\frac {\dlog u}{x} 
   + x^{l_0} \prod_j p_j^{l_j} d(\beta^0) + \beta^0 x^{l_0}d(\prod_j
     p_j^{l_j}) \\ &+  \sum_{ij} c_i n_{ij} \dlog p_{ij} \in \Omega_K,
\end{align*}
then we have:  $n_{ij}=n_0=0$; if $d(\beta^0)\neq 0$, then $l_j=0$, $j\geq 0$;
if $d(\beta^0) = 0$, then $0 \leq l_0\leq 1$ and $l_j =0$, $j \geq 1$.

If $\beta^0=0$ and $n_{ij}\neq 0$, then $\nu_j(\gamma_0)=-1$; hence
$n_{ij}=0$. Similarly we get $n_0=0$. Now assume that
$\beta^0 \neq 0$. If $l_j < 0$, then $\nu_j(\gamma_0)= l_j-1\neq 0$;
hence $l_j\geq 0$. Then if $l_j \geq 0$ (all $j$) and $n_{ij}\neq 0$
(some $ij$), we get $\nu_j(\gamma_0)=-1\neq 0$; hence $n_{ij}=0$.
Similarly, if $n_0 \neq 0$, then $\nu_0(\gamma_0) = -1$; hence
$n_0=0$. Again the remaining case is $l_j >0$ (some $i$), $n_{ij}=0$
(all $ij$), and $n_0=0$. That is we have
\begin{displaymath}
  \gamma_1= l_0 \beta^0  x^{l_0}\prod_j p_j^{l_j}
             \frac {\dlog u}{x} + x^{l_0} \prod p_j^{l_j} d(\beta_0) + \beta_0 x^{l_0}d(\prod_j
     p_j^{l_j}) \in \Omega_K.
\end{displaymath}
If $d(\beta_0)\neq 0$ there is a single monomial term of degree
$l_0 + \sum l_j\deg p_j $; therefore $ l_j=0$, $j\geq 0$.

Now assume $d(\beta^0)=0$, so that $\beta^0 \in \hat k$
\Lem{\ref{int-closure}}. If $l_0 >1$, then
$\nu_0(\gamma_1) =l_0-1 >0 $, so that $\gamma_1\not \in \Omega_K$;
hence $0 \leq l_0 \leq 1$. If $l_0=0$ and some $l_j >0$, then since
\begin{displaymath}
  \deg \prod_j p_j^{l_j} > \deg d (\prod_j p_j^{l_j})
\end{displaymath}
this implies that $\gamma_1 \not \in \Omega_K$; hence $l_j=0$ for all
$j$. Finally, if $l_0=1$, so that $\gamma_1 = d(\beta^0 xp)$, where
$p= \prod_jp_j^{l_j}$, we claim that if $p\in K[x]$ and
\begin{displaymath}
  d(\beta^0 x p) \in  \Omega_K,
\end{displaymath}
then $p\in \hat k$. Since $\beta^0$ is a constant, we can assume that
$\beta^0=1$ and if $m =\deg p\geq 1 $, then $p$ is monic. If
$m\geq 1$, then the coefficient of the leading term in $d(xp)$ is
$\dlog (u)$; therefore $m=0$. If $p\in K$, then
\begin{displaymath}
  d(xp) = p d(x) +  x d(p) \in \Omega_K
\end{displaymath}
if and only if $d(p)=0$; hence $p\in \hat k $.
\end{proof}

\subsection{Representations of finite groups}\label{finite-groups} If
$Y= \Spec K$ and $M$ is an étale trivial $\Dc_Y$-module then the
horizontal sections $ (M^*)^{T_Y}$ of the dual module, forming the
sheaf of solutions of $M$, is a locally constant étale sheaf on $Y$.
Locally constant étale sheaves are classified by representations of
the absolute Galois group of $K$, and the representations in question
will be the ones that act through a finite quotient of the absolute
Galois group. This section contains a (very) down-to-earth description
of this machinery, to be used in the next section to give explicit
forms in some cases of the decomposition theorem over differential
rings on a variety $X$.

There is a general correspondence between representations of the
differential Galois group of a $\Dc_K$-module $M$ and the tensor
category of modules generated by $M$, defined using the existence of a
Picard-Vessiot extension $L$ of $K$. The Picard-Vessiot theory in one
variable is well presented in \cite{put-singer}, including also a
sketch of the several variables theory in its appendix D, while
referring to \cite{kolchin:diffalgebra} for a more complete treatment.
There one starts with a $\Dc$-module and to get a representation of a
group, but we will conversely start with a finite subgroup $G$ of the
automorphisms group $\Aut_k(L)$ of a field $L$ that fixes a subfield
$k$ and a representation of $G$, and end up in a $\Dc_K$-module, where
$K= L^G$ is the fixed subfield.\footnote{Non-finite groups are of
  course also possible, but we are here only interested in the finite
  ones.} Let $l$ be the algebraic closure of $k$ in $L$. The main
assertion is that the category $\Mod(l [G])$ of modules over the group
algebra of $l[G]$ is equivalent to the category of $\Dc_K$-modules
that have trivial inverse image as $\Dc_L$-module. Although this
certainly is well-known we lack  precise and accessible references
for our purposes, so we include a complete proof, which moreover is in
keeping with the language of $\Dc$-modules; we also {\it do not}
assume $k= l$ (compare \Theorem{\ref{liouville}}) or even that $l$ is
a subset of $K$, so that $l$ may not be central in $l[G]$. A notable
simplification is the use of \Lemma{\ref{int-closure}} to get that the
functor $s^+$ below is fully faithful.

Put $X= \Spec L$, $Y=\Spec K$ so that the map $\pi: X\to Y$ that is
induced by the inclusion of fields $K\subset L$ is smooth. We also
have $\Dc_X= \Dc_L$, $\Dc_Y = \Dc_K$, and the category of connections
$\Con(X)$ coincides with the category $ \Mod(\Dc_L)$ of $\Dc_L$-modules
$M$ such that $\dim_L M < \infty$. Since $L/K$ is finite (and
separable), we have $\Dc_L = L\otimes_K \Dc_K$ and there exists a
canonical inclusion $\Dc_K \subset\Dc_L$.

The direct image functor $\pi_+:\Mod_{fd} (\Dc_L)\to \Mod_{fd}(\Dc_K)$
is simply $\pi_+(N)= N$ considered as $\Dc_K$-module under the
restriction to the subring $\Dc_K \subset \Dc_L$, and the inverse
image functor $\Mod_{fd}(\Dc_K)\to \Mod_{fd} (\Dc_L)$ is
$\pi^+(M)= L\otimes_KM$.

Let $\Mod^L(\Dc_L)$ be the category of $\Dc_L$-modules that are
isomorphic to a direct sum $L^n$, for some $n$, and put
\begin{displaymath} \Mod^{L}(\Dc_L[G]) = \Mod^{L} (\Dc_L) \cap
\Mod(\Dc_L[G]).
\end{displaymath} Let $\Mod^L(\Dc_K) $ be the category of
$\Dc_K$-modules such that $\pi^!(M)\in \Mod^{L} (\Dc_L) $.

Define the functors
\begin{displaymath} s^+: \Mod (l[G])\to \Mod^L_{fd}(\Dc_L[G]),
  \quad V\mapsto L \otimes_{l} V,
\end{displaymath} where $G$ acts diagonally on $L\otimes_{l} V$ and
$\Dc_L$ acts only on first factor, and 
\begin{align*}
  s_+&: \Mod^L_{fd}(\Dc_L[G]) \to \Mod (l[G]),\\ M &\mapsto
  Hom_{\Dc_L}(L, M)= M^{T_L}.
\end{align*}

Let $\Dbb$ denote the Poincaré dual in a category of holonomic
$\Dc$-modules and $*$ the ordinary dual in the category of
finite-dimensional $G$-representations over $k$.
\begin{lemma}\label{adj-lemma} 
  \begin{enumerate}
  \item $ \Dbb_L\circ s^+ = s^+\circ *$ and $ *\circ s_+ = s_+\circ
  \Dbb_L$.
  \item We have a triple of adjoint functors $(s_+, s^+, s_+)$.
  \item The functors $s^+$ and $s_+$ are fully faithful.
  \end{enumerate}
\end{lemma} 

\begin{lemma}\label{int-closure}
  Let $C\to B$ be an inclusion of integral domains such that the
  extension of fraction fields $k(B)/k(C)$ is separable and finitely
  generated. Let $T_{B/C}$ be the $C$-linear derivations of $B$ and
  $\bar C$ be the integral closure of $C$ in $B$. Then
\begin{displaymath}
  \bar C = B^{T_{B/C}} := \{b \in B \ \vert \ \partial (b)=0, \partial
  \in T
_{B/C}\}.
\end{displaymath}
\end{lemma} We remark that \Lemma{\ref{int-closure}} implies the
central result \cite{put-singer}*{Lemma 1.17}, while the proof in loc.
cit. is more complicated.
\begin{proof}
  It is easy to see that $ \bar C \subset B^{T_{B/C}}$, so it suffices
  to prove that if $b \not\in \bar C $ then there exists
  $\partial \in T_{B/C}$ such that $\partial (b)\neq 0$. There exist
  separable field extensions $k(C )\subset L \subset k(B)$ where
  $\trdeg (L/k(C)) \geq 1$, $b\in L= k(C)(b)$, and therefore
  $T_{L/k(C)}(b)\neq \{0\}$. The map
  $T_{K(B)/K(C)}\to K(B)\otimes_{L} T_{L/K(C)}$ is surjective since
  $k(B)/L$ is separable, hence there exists an element
  $\partial_1 \in T_{K(B)/K(C)}$ such that $\partial_1 (b) \neq 0$ in
  $k(B)$. Multiplying by a suitable element $c\in B $ we get
  $\partial =c \partial_1 \in T_{B/C}$ and $\partial (b) \neq 0$.
\end{proof}
\begin{pfof}{\Lemma{\ref{adj-lemma}}}
  (1): This follows from \Lemma{\ref{dual-con}}.

  (2): It is evident that $(s^+, s_+)$ is an adjoint pair, and (1)
  implies that $(s_+, s^+)$ is also an adjoint pair.

  (3): $s_+$ is fully faithful: If
  $M_1, M_2 \in \Mod^L_{fd}(\Dc_L[G])$, so that $M_i\cong L^{n_i}$ as
  $\Dc_L$-module, we have
  \begin{displaymath}
    Hom_{\Dc_L[G]}(M_1, M_2) = Hom_{l [G]}(M_1^{T_L}, M_2^{T_L}) =
    Hom_{l[G]}(s_+(M_1), s_+(M_2)).
  \end{displaymath}
  $s^+$ is fully faithful: for a $l[G]$-module $V$
  \begin{displaymath} s_+s^+(V)=Hom_{\Dc_L} (L, L \otimes_{l} V) =
    Hom_{\Dc_L}(L, L)\otimes_{l} V = L^{T_{L/k}}\otimes_{l} V = l
    \otimes_{l} V =V,
\end{displaymath}
where the penultimate step follows from \Lemma{\ref{int-closure}}.
\end{pfof}

The Picard-Vessiot equivalence for finite field extensions can now be
deduced from the Galois descent equivalence $(g_+, g^+)$ in
\Proposition{\ref{morita}}. The former is given by the functors
\begin{displaymath}
  \Delta = g_+ \circ s^+ : \Mod(k[G])\to \Mod^L (\Dc_K),
\end{displaymath} and
\begin{displaymath} \Loc = s_+ \circ g^+: \Mod_{fd}(\Dc_K) \to \Mod
(k[G]).
\end{displaymath}

 \begin{theorem} \label{equivalence}
 \label{prop:equiv}
   The functor
   \begin{displaymath} \Delta: \Mod (l[G]) \to \Mod_{fd} (\Dc_{K})
   \end{displaymath} is fully faithful, and defines an equivalence of
   categories $ \Mod (l[G]) \to \Mod_{fd}^{L}(\Dc_{K}).$ A
   quasi-inverse of $\Delta$ is given by the functor $\Loc :
   \Mod^{L}_{fd}(\Dc_{K}) \to \Mod(l[G])$. Moreover,
   \begin{displaymath}
     \Dbb_K\circ \Delta = \Delta \circ * \quad {\text and } \quad *
     \circ \Loc = \Loc \circ \Dbb_K.
   \end{displaymath}
\end{theorem}
\begin{remarks}\label{rem-ab-ext}
  \begin{enumerate}
  \item If one restricts to abelian Galois extensions $L/K$, then all
    the simples in $\Mod^L (\Dc_K)$ are of rank $1$. Moreover, if $M$
    is such a simple $L$-trivial $\Dc_K$-module, there exists a
    radical extension $L_1/K$ such that $L_1\otimes_K M\cong L_1$ (see
    \Proposition{\ref{rat-ext}}).
  \item It is easy to see that $(\Delta, \Loc )$ are functors between
    tensor categories. One can notice, however, that the Tannaka
    theorem is not invoked in the proof of \Theorem{\ref{equivalence}}.
  \end{enumerate}

\end{remarks}
\begin{pfof}{\Theorem{\ref{equivalence}}}
  By \Proposition{\ref{morita}} and \Lemma{\ref{adj-lemma}} it
  follows that we have the adjoint pair of functors $ (\Delta, \Loc)$,
  and also that we have
  \begin{displaymath}
      \Loc \circ \Delta = s_+\circ g^+\circ g_+\circ s^+ = s_+\circ
      s^+ = \id.
  \end{displaymath} so that $\Delta$ is fully faithful.
We also have, again by \Proposition{\ref{morita}} and \Lemma
{\ref{adj-lemma}},
\begin{displaymath}
  \Delta \circ \Loc = g_+\circ s^+ \circ s_+ \circ g^+ = \id.
\end{displaymath}
\end{pfof}

\Theorem{\ref{equivalence}} implies that $ \Mod^L_{fd} (\Dc_{K})$ is a
semisimple category, by Maschke's theorem. This also follows from
\Theorem{\ref{semisimple-inv}}, and similarly, the following corollary
to \Theorem{\ref{equivalence}}, giving a decomposition of $\pi_+(L)$,
also follows from \Corollary{\ref{direct-L}}; in both cases one thus
did not need to employ the equivalence $\Delta$.
\begin{corollary}
  $\pi_+(L)$ is semisimple, more precisely,
  $\pi_+(L) = \Delta (l[G]) = g_+(L\otimes_k k[G]) = (L[G])^G$.
  Moreover,
\begin{displaymath}
    Hom_{\Dc_{K}}(\pi_+(L),\pi_+(L)) = l[G].
  \end{displaymath}
\end{corollary}
\begin{proof} We have, since $k[G]^* \cong k[G]$ (as $k[G]$-module),
  \begin{align*}
    \Delta(l[G]) & = g_+ s^+(l[G]) = (L\otimes_{l}\bar
                        k[G])^G =  (L\otimes_kk[G])^G = Hom_{k[G]}(k,
                        L\otimes_k k[G]) \\ & Hom_{k[G]}(k, L\otimes_k( k[G])^*) = Hom_{k[G]}(k[G],
                                              L ) =L= \pi_+(L),
  \end{align*}
  where $L$ is regarded as a $\Dc_K$-module under the inclusion
  $\Dc_K \subset \Dc_L$. The last assertion follows since $\Delta$ is
  fully faithful.
\end{proof}
In \Theorem{\ref{equivalence}} we ended up in the category of
$l[G]$-modules, where $l[G]$ is not in general a group ring when
$l\not\subset K$ (see also \Section{\ref{liouville-section}}). The
situation is described by a diagram of field extensions:
\begin{displaymath}
\begin{tikzcd}[every arrow/.append style={dash}]
  &L               \\
  & K_1\arrow{u}   \\
  l\arrow{ur} & K\arrow{u}                \\
 k,\arrow{u} \arrow{ur}  &  
\end{tikzcd}
\end{displaymath}
where $l$ is the algebraic closure of $k$ in $L$, and $K_1$ the
composite of $K$ and $l$ in $L$. Since $L/K$ is Galois it follows that
all extensions in fact are Galois. Let $p$  denote the map $\Spec K_1
\to \Spec K$. 
  \begin{proposition}\label{non-closed}
  \item The functor $p^!:\Mod^L(\Dc_K)\to \Mod^L(\Dc_{K_1})$ is fully
    faithful.
  \item If $M \in \Mod^L(\Dc_{K_1}) $ is simple, then $p_+(M) = N^r$,
    where $r= [l:k_1]$ and $k_1$ is the algebraic closure of $k$ in
    $K$.
  \end{proposition}
On the group side  we have the exact sequence
\begin{displaymath}
  0 \to l[G_1] \to l[G] \to l[H] \to 0,
\end{displaymath}
where $G_1$ and $H$ are the Galois groups of $L/K_1$ and $K_1/K$,
respectively ($H$ is also the Galois group of $l/k$). A similar
argument as in the proof below implies that the restriction functor
$\reso : \Mod(l[G])\to \Mod(l[G_1])$ (which corresponds to $p^!$) also
is fully faithful (or apply \Theorem{\ref{equivalence}}).
\begin{proof} (1): Our categories are semisimple so it suffices to
  prove that an isomorphism $\phi :p^!(N_1)\to p^!(N_2)$ gives an
  isomorphism $N_1 \to N_2$, when $N_1$ and $N_2$ are simple. The map $\phi$ induces the homomorphism
  $\hat \phi$ of $\Dc_K$-modules below
  \begin{displaymath}
    N_1 \xrightarrow{\hat \phi} K_1\otimes_K N_2 = l\otimes_k N_2 = N_2^r,
  \end{displaymath}
  where $r= [l:k_1]$. Since $N_1$ and $N_2$ are simple, this implies
  the assertion.

  (2): We have, for any non-zero element $m\in M $,
  $M= \Dc_{K_1} m = l\otimes_k \Dc_K m =\oplus^r \Dc_Km$. By (1)
  $\Dc_Km$ is simple.
\end{proof}
\subsection{The Galois correspondence}\label{galois-corr} Fix the
fields $K$ and $ L$ as above, assume for simplicity that $l =k$,
and consider a tower of intermediate fields
\begin{displaymath} K \subset E_2 \subset E_1 \subset L
\end{displaymath} corresponding in the usual Galois correspondence to
a tower of subgroups
\begin{displaymath}
  \{e \}\subset H_1 \subset H_2 \subset G
\end{displaymath} such that $E_i = L^{H_i}$. We have now equivalences
$(\Delta_1, \Loc_1)$ and $(\Delta_2, \Loc_2)$
\begin{displaymath}
  \Mod (k[H_1])\cong \Mod_{fd}^L(\Dc_{E_1}), \quad \Mod (k[H_2])\cong
  \Mod_{fd}^L(\Dc_{E_2}).
\end{displaymath} To the map $p : \Spec E_1\to \Spec E_2$ we have the
adjoint triple $({p}_+, p^+,p_+ )$ of functors
\Th{\ref{adj-triple-th}},
\begin{displaymath} \Mod^{L}_{fd}(\Dc_{E_1})
\adj{p_+}{p^+}\Mod^{L}_{fd}(\Dc_{E_2})
\end{displaymath} That the direct image functor takes $L$-trivial
connections to $L$-trivial connections is easily seen if one recalls
that the direct image is simply by restriction to a subring
$\Dc_{E_2}\subset \Dc_{E_1}$.

We regard $k[H_2]$ as a $(k[H_2], k[H_1])$-bimodule. The restriction
functor
$$
\reso_{21}: \Mod(k[H_2])\to \Mod(k[H_1]), \quad V\mapsto
Hom_{k[H_2]}(k[H_2], V)= V,
$$
where the action on $V$ is by the inclusion $k[H_1]\subset k[H_2]$,
has a left adjoint, which is the induction functor
$$
\ind_{12}: \Mod(k[H_1])\to \Mod(k[H_2]),\quad V\mapsto k[H_2]
\otimes_{k[H_1]}V.
$$
We have the duality functor $* : \Mod(k[G]) \to \Mod(k[G])$, $V\mapsto
V^* = Hom_k(V, k)$, where the $k[G]$-action is determined by $g\cdot
v^* (v)= v^*(g^{-1}v)$, $v\in V, v^* \in V^*$. We use the same
notation $*$ for the duality on $\Mod(k[H])$, when $H $ is a subgroup
of $G$. Since clearly $* \reso_{21} * = \reso_{21}$, the right adjoint
is the coinduction
\begin{align*}
  \coind_{12}: \Mod(k[H_1])&\to \Mod(k[H_2]),\\ V&\mapsto (k[H_2]
  \otimes_{k[H_1]}V^*)^*= Hom_{k[H_1]}(k[H_2] ,V),
\end{align*} where the $k[H_2]$-action is determined by $h_2 \cdot
\phi (r) = \phi (h_2^{-1}r)$, when $h_2\in H_2$, $\phi \in
Hom_{k[H_1]}(k[H_2] ,V)$, and $r\in k[H_2]$. We have a canonical
homomorphism
\begin{displaymath}
  k[H_2]^* \otimes _{k[H_1]} V \to (k[H_2] \otimes_{k[H_1]}V^*)^*
\end{displaymath} which is actually an isomorphism since the $(k[H_2],
k[H_1])$-bimodule $k[H_2]$ is self-dual; hence $\coind_{12} =
\ind_{12}$.

Since the two adjoints of $\reso_{21}$ are isomorphic we will only use
the coinduction in the proposition below (this is also reflected in
the fact that $p_+= p_{!}$ \Th{\ref{adj-triple-th}}).

\begin{proposition}\label{ind-direct} The adjoint pairs $(p^+, p_+)$
  and $(\reso_{21}, \coind_{12})$ correspond in the sense:
\begin{enumerate}
  \item $\Loc_1 \circ \ p^+\circ \Delta_{2} =\reso_{21}$.
  \item $\Loc_2 \circ\ p_+ \circ \Delta_1=\coind_{12}$.
\end{enumerate}
\end{proposition} It follows also that
\begin{align*}
  p^+&= \Delta_1 \circ \reso_{21} \circ \Loc_2 : \Mod^L(\Dc_{E_2}) \to
  \Mod^L(\Dc_{E_1}) \\ p_+ &= \Delta_2 \circ \coind_{12} \circ \Loc_1:
  \Mod^L(\Dc_{E_1}) \to \Mod^L(\Dc_{E_2}).
\end{align*}

\begin{proof} Since we are dealing with two adjoint pairs and since we have
equivalences of categories,
  it suffices to prove (1). Let $(s_i)_+:\Mod^L(\Dc_L[H_i]) \to
  \Mod(k[H_i])$ and $(g_i)_+: \Mod^L(\Dc_L[H_i]) \to
  \Mod^L(\Dc_{E_i})$, $i=1,2$ be the functors corresponding to $s_+$
  and $g_+$, as described above, so that $\Loc_1 = (s_1)_+ g_1^+$ and
  $\Delta_2 = (g_2)_+ s_2^+$. Therefore by \Proposition{\ref{morita}}
  \begin{displaymath}
    \Loc_1 \circ p^+ \circ \Delta_2 = (s_1)_+\circ g_1^+\circ p^+\circ
    (g_2)_+\circ s_2^+ = (s_1)_+\circ g_2^+\circ (g_2)_+\circ s_2^+ =
    (s_1)_+\circ s_2^+ = \reso_{21}.
  \end{displaymath}
\end{proof}
It is of some interest to connect with Mackey's restriction theorem in
group theory (see \cite{serre:lin-rep}*{Prop. 22}), using the above
relation between the pairs inverse/direct images and
restriction/induction, and thereby showing that
\Theorem{\ref{inv-dir}} can be regarded as a generalization of
Mackey's theorem. Let $H_1$ and $H_2$ be subgroups of a finite group
$G$, and define $H(s)= sH_2s^{-1}\cap H_1$ and $S$ as before
\Theorem{\ref{inv-dir}}. Let $V$ be a representation of $H_2$ and
denote by $V_s$ the representation of $H(s)$ that is formed by the
composition
$H(s) \to H_2 \xrightarrow{\rho} \Glo (V), x\mapsto s^{-1}x s\mapsto
\rho(s^{-1}x s)$. Mackey's restriction theorem states the following:
\begin{corollary}\label{mackey}
  \begin{displaymath}
    \reso_{H_1} \ind^G_{H_2}V = \bigoplus_{s\in S} \ind_{H(s)}^H V_s.
  \end{displaymath}
\end{corollary}
\begin{proof} The group $G$ can be realized as a subgroup of $\Glo(V)$
  for som $k$-vector space $V$, so that it acts faithfully on the fraction
  field $L= k(\So(V))$. Put $K= L^G$, $L_1= L^{H_1}$, and
  $L_2= L^{H_2}$. In the notation of \Proposition{\ref{inv-dir}} we
  have
\begin{displaymath}
  L^{H_1}\otimes_K L^{H_2} = \bigoplus_{s\in S} L^{H(s)}.
\end{displaymath}
By, in turn, \Proposition{\ref{ind-direct}},
\Theorem{\ref{equivalence}} and \Proposition{\ref{inv-dir}}
\begin{align*}
      \reso_{H_1} \ind^G_{H_2}V &=  \Loc_{H_1} \circ p_1^!\circ  \Delta_G \circ
      \Loc_G \circ (p_2)_+ \circ \Delta_{H_2} (V) \\ & =  \Loc_{H_1}\circ p_1^!\circ
                               (p_2)_+\circ \Delta_{H_2}(V)
\\ &  = \bigoplus_{s\in S} \Loc_{H_1}\circ  (l_s)_+\circ ( L^{H(s)} \otimes_{L^{H_2}}
    \Delta_{H_2} (V))\\ & = \bigoplus_{s\in S} \Loc_{H_1}\circ  (l_s)_+\circ \Delta_{H(s)}(V_s)
    =  \bigoplus_{s\in S} \ind_{H(s)}^{H_1} V_s,
\end{align*}
and the last step again follows from  \Proposition{\ref{ind-direct}}.
\end{proof}
\subsection{Inverse images of covering modules}


Using the holomorphic notion Higgs bundle, it can be proven that if
$\pi : X\to Y $ is a morphism of quasi-projective algebraic complex
manifolds, then $\pi^!(M)$ is semisimple when $M$ is a semisimple
connection, and if $M'$ is a second semisimple connection, then
$M\otimes_{\Oc_X}M'$ is again a semisimple connection (see
\cite{simpson-higgs}).\footnote{The latter follows from the former by
  applying the former assertion to the diagonal mapping.} Our argument
for the weaker assertions below has the merit of being algebraic; it
would of course be interesting if one could extend the result to
semisimple connections with any reductive differential Galois group.

Neither the category $\Mod_{cov}(\Dc_X)$ nor $\Con (X)$ are
semisimple, and although both categories are stable under taking
tensor products, the tensor product of simples in $\Mod_{cov}(\Dc_X)$
need not remain semisimple. In contrast, the common subcategory
$ \Con_{cov} (X) = \Con (X)\cap \Mod_{cov}(\Dc_X) $ of covering
connections behaves  better.

\begin{theorem}\label{alg-semi}
  \begin{enumerate}
  \item $ \Con_{cov} (X) $ is a semisimple category.
  \item Let $M_1$ and $M_2$ be covering connections. Then
  $M_1\otimes_{\Oc_X}M_2$ is also a
  covering connection.
\item If $M_1$ is a covering connection and $M_2$ is any semsimple
  connection, then $M_1\otimes_{\Oc_X}M_2$ is a semisimple connection.
  \end{enumerate}
\end{theorem}
It follows that $\Con_{cov} (X)$ forms a neutral tensor category,
where a fibre functor is given by
\begin{displaymath}
  M \mapsto Hom_{\Dc_L}(\bigoplus_i \Lambda_i, L\otimes_K M) =
  \bigoplus_i Hom_{\Dc_L} (\Lambda_i,  \Lambda_i^{n_i}) =
  \bigoplus k^{n_i} ,
\end{displaymath}
for a Galois field extension $L/K$ such that
$L \otimes_{K}M = \oplus \Lambda^{n_i}$, for some mutually
non-isomorphic $\Dc_L$-modules $\Lambda_i$ of rank $1$ over $L$.
Fixing the field extension $L/K$ and invertible $\Lambda$ and letting
$\Con^{\Lambda}_{cov}(X)$ be the subcategory of modules $M$ in
$\Con_{cov} (X)$ such that $L\otimes_KM\cong \Lambda^n$ for some
integer $n$, $\Con^\Lambda_{cov}(X)$ is equivalent with the neutral
tensor category of semisimple finite-dimensional representations of
$\Aut (L/K)$.

\begin{proof} Let $K=\Oc_{X,\eta}$ denote the local ring at the
generic point $\eta$ of $X$.

(1): Let $M$ be a covering connection and $M_\eta$ be the
$\Dc_K$-module at the generic point. There exists a finite field
extension $L/K$ such that $L\otimes_K M_\eta$ is diagonalizable so
that in particular it is semisimple. Since there exists an injective
homomorphism $M_\eta \subset \pi_+(L\otimes M_\eta) $,
\Lemma{\ref{ss-field}} implies that $M_\eta$ is a semisimple, and as
$M$ is a connection it equals the minimal extension of $M_\eta$, hence
$M$ is semisimple.

(2): By (1) it suffices to prove that the connection
$M_1\otimes_{\Oc_X}M_2$ is a covering module. There exists a common
finite field extension $L/K$ such that $L\otimes_K((M_1)_\eta)$ and
$L\otimes_K((M_2)_\eta)$ are diagonalizable. Now there exists a
canonical injective homomorphism of $\Dc_K$-modules
\begin{displaymath}
  L\otimes (M_1)_\eta  \otimes_K(M_2)_\eta = L\otimes_K (M_1)_\eta
  \otimes_L L\otimes_K (M_1)_\eta. 
\end{displaymath}
Since the tensor product of modules of generic rank $1$ again is a
module of generic rank $1$, it follows that the above module is
diagonalizable.

(3): Let $L/K$ be a finite field extension such that $L\otimes_K M_1 $
diagonlizable. We have
\begin{displaymath}
L\otimes_K M_1 \otimes_K M_2 = L\otimes_K M_1\otimes_L L \otimes_K
M_2. 
\end{displaymath}
Since $L \otimes_K M_2$ is semisimple by
\Theorem{\ref{semisimple-inv}}, $L\otimes_K M_1$ is diagonlizable, and
the tensor product of a rank $1$ module and a simple module is again
simple, it follows that the above module is semisimple. Hence by
\Theorem{\ref{decomposition-thm}} it is also simple as $\Dc_K$-module,
and we denote it as such by
$(\pi_0)_+(L\otimes_K M_1 \otimes_K M_2 )$. Letting $j: \Spec K \to X$
be the inclusion of the generic point we have
  \begin{displaymath}
        M_1 \otimes_{\Oc_X}M_2  \subset j_{!+}((\pi_0)_+ ((L\otimes_K M_2)^n)),
  \end{displaymath}
  where on the right we have the minimal extension module, which is
  semisimple. This implies the assertion.
  \end{proof}
\begin{remark}
  Another proof of \Theorem{\ref{alg-semi}}, (3), follows by taking
  $M_1\otimes_k M_2$ on $X\times_k X$, pulling back by the diagonal
  map $X \to X\times_k X$, and applying \Theorem{\ref{inv-conn}} below.
\end{remark}

  \begin{theorem}\label{inv-conn} Let $\pi: X\to Y$ be a morphism of
  smooth varieties and $M\in \Con_{cov} (Y)$. Then $\pi^! (M) \in
  \Con_{cov} (Y) $, and in particular, $\pi^!(M)$ is semisimple.
\end{theorem}

\begin{example}
  Let $M = \Dc_{\Ab^2} \sqrt{f}$, where $f= y-x^2 \in k[x,y]$. This is
  a simple $\Dc_{\Ab^2}$-module, and we give two morphisms such that its
  inverse image is not semisimple.
  \begin{enumerate}
  \item If $\pi: \{y=0\}\to \Ab^2_k$ is the inclusion of the $x$-axis,
    then $\pi^!(M) = k[x,1/x]$.
  \item If
    $\pi: X = \Spec k[x,z]\to \Ab^2_k, (x,z)\mapsto (x, z^2+ x^2)$ is
    the integral closure of $\Ab^2_k$ in the field extension
    $k(x,y)[\sqrt{f}]/k(x,y)$, then $\pi^!(M) = k[x,z,1/z]$.
\end{enumerate}
\end{example}

\begin{proof} Since $\pi^!(M)\in\Con (X)$ it suffices, by
  \Theorem{\ref{alg-semi}}(1) to prove that $\pi^!(M)$ is a covering
  module. One can factor $\pi$ into a closed embedding and a
  projection, so it suffices to prove the assertion separately when
  $\pi$ is one of these types. When $\pi$ is a projection this is a
  consequence of \Theorem{\ref{semisimple-inv}}, so what remains is
  the case when $\pi$ is a closed embedding. Moreover, a closed
  embedding can be factorized into closed embeddings $X\to Y $ such
  that $\pi (X)$ is a smooth hypersurface in $Y$. The proof therefore
  follows if we prove that $\pi^!(M)$ is a covering module when $\pi$
  is an embedding of a smooth hypersurface.

  By assumption there exists a morphism of varieties $p': Y' \to Y$
  such that $p^!(M)_\xi$ is diagonalizable, where $Y'$ is the integral
  closure of $Y$ in som field extension of the fraction field of $Y$.
  Let $\bar Y \to Y'$ be a desingularization of $Y'$ and
  $\bar p : \bar Y \to Y' \to Y$ be the composed map. Then
  $\bar p^!(M) $ is a connection whose generic stalk is
  diagonalizable, hence it equals a direct sum $ \oplus \Lambda_i$ of
  connections $\Lambda_i$ of rank $1$. Let $p_1: \bar X \to X$ be the
  base change of $\bar Y \to Y $ over $X \to Y$, and
  $j: \bar X \to \bar Y$ be the projection on the second factor. Then
  $j$ is a closed embedding and $j(\bar X)$ is a divisor in $\bar Y$.
  Since $\pi \circ p_1 = \bar p \circ j $ we get
  \begin{displaymath}
   p_1^! \pi^!(M)= j^!\bar p^!(M) = j^!(\bigoplus \Lambda_i)=
   \bigoplus j^!(\Lambda_i).
 \end{displaymath}
 Therefore $\pi^!(M)$ is a covering module.
\end{proof}

\section{Simply connected varieties}\label{sc-section}
Simply connected varieties are approached algebraically by studying
categories of connections instead of fundamental groups.
\subsection{Generalities} Let $X/k$ be a smooth quasi-projective
variety over an algebraically closed field $k$ of characteristic $0$, and
\begin{displaymath}
  \Con_{et}(X) \subset \Con_{rs}(X)\subset  \Con (X)  
\end{displaymath}
be the categories of étale trivial, regular singular, and all
connections on $X$, respectively. The left inclusion is immediate when
$X$ is a quasi-projective curve, and the general case follows from
this, using ``curve regularity'' as definition of regular
singularities \cite{borel:Dmod}. All three categories form $k$-linear
rigid neutral abelian tensor categories, and are therefore equivalent
to representation categories of certain affine (Tannaka) group schemes
with maps
\begin{displaymath}
\pi_1^{alg}(X,p)\to \pi^{rs}_1(X,p) \to   \pi^{et}_1(X,p),
\end{displaymath}
where $p$ is a choice of geometric point in $X$; the first map is
surjective and $\pi^{et}_1(X,p)$ is the pro-finite completion both of
the first and second terms. Here $\Con_{et}(X)$ is equivalent to the
category of finite-dimensional representations of the étale
fundamental group $\pi_1^{et}(X,p)$, $\Con(X)$ is equivalent to the
category of finite-dimensional representations of the algebraic
fundamental group $\pi_1^{alg}(X,p)$, and $\Con_{rs}(X)$ is equivalent
to the category of finite-dimensional representations of
$\pi^{ rs}_1(X,p)$. We will simply write
$\pi_1(X,p) = \pi_1^{rs}(X,p)$; notice that if $X$ is proper, then
$\Con(X) = \Con_{rs}(X)$ so that $\pi_1^{alg}(X,p) = \pi_1(X,p)$. See
\cite{esnault:flat-bund-charp} for a more detailed discussion. When
$k= \Cb$ we have $\Con_{rs}(X)\cong \Con_{rs}(X_a)\cong \Con(X_a)$ for
a quasi-projective manifold $X$, where $X_a$ is its associated complex
analytic manifold, and $\pi_1(X,p)$ is isomorphic to the topological
fundamental group  $\pi_1^{top} (X_a,a)$ \cite{borel:Dmod}.

Say $X$ is {\it étale simply connected} (abbr. e.s.c.) if any object
$M$ in $\Con_{et}(X)$ is trivial, i.e. $M\cong \Oc^n_X$ for some
integer $n$. Similarly, $X$ is {\it simply connected} (s.c.) if any
object in $\Con_{rs}(X)$ is trivial.\footnote{One can also ask that
  $\Con(X)$ only contain trivial objects when $X$ quasi-projective.
  But if $X= \bar X \setminus D$ for a smooth projective variety
  $\bar X $ of dimension $\geq 1$ and $D$ a divisor such that $X$ is
  affine, then $\Con(X)$ will contain non-trivial objects of rank $1$
  (exponential modules); see (\ref{exp-mod-sec}). } Now we have
\begin{displaymath}
  X \text { is s.c.} \Leftrightarrow X \text{ is e.s.c.},
\end{displaymath}
where the arrow to the right is evident while the opposite direction
follows from the Grothendieck-Malcev theorem
\cites{malcev:matrix,grothendieck:profinie}, whose proof is based on
the fact that $\pi^{top}_1(X_a,a)$ is finitely generated. We will
below only use the trivial implication, but it would nevertheless be
great to see an algebraic proof that e.s.c. implies s.c.
\begin{remarks}\label{rem-general}
  \begin{enumerate}
  \item Let $X/k$ be a smooth quasi-projective over a subfield
    $k \subset \Cb$, so that $k \subset l \subset \Cb$ where $l$ is
    the algebraic closure of $k$ in the field of complex numbers
    $\Cb$. Let $X_{l}$ be the base change of $X$ over $l/k $, and
    $X_a$ be its complex analytification. We have natural functors
    \begin{displaymath}
      \Con_{rs}(X)\to  \Con_{rs}(X_{l}) \to  \Con_{rs}(X_{\Cb})\cong  \Con_{rs}(X_a),
    \end{displaymath}
    where the equivalence follows from G.A.G.A. Now it can happen that
    $\pi^{top}_1(X_a,p)\neq 0$ or $\pi_1^{et}(X,p)\neq 0$ while $X$ is
    simply connected in our sense. One cause for this is that
    $\pi^{top}_1(X_a, p)$ need not be residually finite
    \cite{toledo:nonres} (even when $X$ is projective), another that
    $k\neq l$, so that even if $\pi^{top}_1(X_a,p)$ is residually
    finite, not all its representation need come from $\Con_{rs}(X)$.
    We also have
    \begin{displaymath}
      \Mod (l [\pi^{et}_1(X)])\cong\Con_{et}(X)\subset \Con_{et}(X_{\bar
        k})\cong \Mod (l [\pi_1^{et}(X_{\bar
        k})])
  \end{displaymath}
  (see \Proposition{\ref{non-closed}}). For example, let $\Pb^n_k$ be
  the projective space over the field $k$, where $n\geq 1$. Then
  $\pi_1^{top}(\Pb^n_{\Rb},p)$ is non-trivial while $\Pb^n_\Rb$ is
  simply connected (all non-trivial real-analytic connections on
  $\Pb^n_{\Rb}$ are non-algebraic). We also have
  $\pi_1^{et}(\Pb^n_\Cb,p)= \{1\}$, 
  $\pi_1^{et}(\Pb^n_{\Rb},p)= C_2$, and
  $\Con_{et} (\Pb^n_{\Rb})\cong \Mod (\Cb[C_2])\cong \Mod (\Rb)$, in
  accordance with \Proposition{\ref{raional-dsc}} below.
\item Simple étale trivial connections $M$ are submodules of monomial
  modules, i.e. there exists an étale map $\pi: X' \to X$ such that
  $M \subset \pi_+(\Oc_{X'})$ \Th{\ref{riemann}}.
\item If there exists an étale map $p: \bar X \to X$ where $\bar X$ is
  e.s.c., then $\Con_{et}(X)$ is generated by the simples in
  $p_+(\Oc_X)$. The étale fundamental group is finite if and only if
  $\Con_{et}(X)$ contains finitely many isomorphism classes of
  simples.
\item On a smooth variety $X$ the inclusion
  $\Con_{et}(X)\subset \Con (X)$ is in general strict also when $X$ is
  projective. For example, by \Lemma{\ref{lem-rank1}}
  $\bar \Io(X) = \Gamma(X, \Omega_X) \oplus \Pic^0 (X)$, and
  $\bar \Io_{et}(X)\subset \Pic^0(X)$. This follows since although
  global 1-forms on a projective variety are closed, they do not
  belong to $\dlog (L)$ (nor $d(L)$) for any finite field extension
  $L$ of the fraction field of $X$. For instance, if $X$ is a smooth
  projective curve of genus $g \geq 1$, then
  $\Con_{et}(X)\neq \Con (X)$.

  \end{enumerate}
\end{remarks}
A connection on a smooth variety is determined by its local structure
in the following strong sense: if $M_1$ and $M_2$ are connections that
are isomorphic at the generic point of $X$, then $M_1\cong M_2$
\Prop{\ref{simple-coh}}. We have also the following useful remarks:
\begin{remarks}\label{purity}
  \begin{enumerate}
  \item A simple connection $M$ is trivial if and only if there exists
    generators $\partial_i$ of $T_K$ (as Lie algebroid) such that the
    invariant space $\cap_i(K\otimes_{\Oc_X}M)^{\partial_{i}}\neq 0$.
  \item Let $U$ be an open dense subset of $X$ and $j: U\to X$ be the
    open inclusion. If $U$ is e.s.c. (s.c.) then $X$ is e.s.c. (s.c.).
    Put $ Z= X\setminus U$ and assume that $\codim_X Z\geq 2$. If
    $M\in \Con(U)$ it follows that $j_+(M)\in \Con(X)$. Therefore $X$
    is e.s.c (s.c.) if and only if $U$ is e.s.c (s.c.). This plays the
    counterpart of the Zariski-Nagata purity theorem for finite maps.
  \end{enumerate}
\end{remarks}
Now allow $X/k$ to be a singular variety (reduced) with singular locus
$S$ and $j: X_{reg}= X\setminus S\to X $ be the inclusion of its
regular locus. Then we have the restriction functor
\begin{displaymath}
      j^!: \Con(X)\to \Con(X_{reg}),
\end{displaymath}
which clearly is fully faithful. If $X$ is normal and
$M\in \Con(X_{reg})$, then $N=j_+(M)$ is a connection (by
Grothendieck's finiteness theorem \cite{SGA2}*{Prop. 3.2 }) such that
$\depth_S N \geq 2$, and in fact we have an equivalence between
$\Con(X_{reg})$ and the subcategory $\Con^2(X)$ of objects $N$ in
$\Con(X)$ such that $\depth_S N \geq 2$. Consider also the subcategory
\begin{displaymath}
\Con^{f}(X)\subset \Con^2(X) \subset \Con(X)
\end{displaymath}
of locally free connections on $X$, and say that $X$ is s.c. when any
object in $\Con^{f}(X)$ is isomorphic to $\Oc_X^m$ for some integer
$m$. In the complex analytic case, the category of finite-dimensional
representations of $\pi^{top}_1(X_a,p)$ is equivalent to $\Con^{f}(X_a)$.
Notice that if $S\neq \emptyset$, then $\Con(X_a)$ always contains non
trivial objects (for example, the Jacobian ideal is preserved by
$T_X$). Finally, if $X_{reg}$ is s.c. it follows that any object in
$\Con^{f}(X)$ is of the form $\Oc_X^m$, so that $X$ is also s.c..
\begin{remark}
  In \cite{greb-kebekus-peternell:fundgroups} Kawamata log terminal
  singularities are allowed on $X$. They prove that there exists a
  finite map $\tilde X \to X$, étale over points if height $\leq 1$,
  such that $\Con(\tilde X_{reg})\cong \Con^{f}(\tilde X)$.
\end{remark}

It seems to be well-known that $\pi_1(X_{reg},p)$ is trivial when
$X/k$ is a normal rational projective variety and $k$ is an
algebraically closed field, see \cite{SGA1}*{XI, Cor. 1.2} and
discussion in \cite{120442}. For a proof that $\pi_1^{et}(\Pb^n_k,p)$ is
trivial when $k$ is algebraically closed of characteristic $0$, see
\cite{SGA1}*{Exp. XI} (and the discussion in \cite{62282}), but to get
that $\pi_1^{et}(\Ab^n_k,p)$ is trivial by ``algebraic means'' is not
entirely straightforward when $n>1$.

Granting the equivalence between $\Con_{et}(X_{reg})$
($\Con_{rs}(X_{reg})$) and the category of finite\--dimensional
representations of $\pi_1^{et}(X_{reg},p)$ ($\pi_1(X_{reg},p)$) these
results can be concluded by instead showing that $X_{reg}$ is e.s.c.
(s.c.)\footnote{To get that $\pi_1(X_{reg},p)$ is trivial from the
  triviality $\Con_{rs}(X_{reg})$ one also needs to know that
  $\pi_1(X_{reg},p)$ is residually finite.}, which we achieve below by
using the well-known fact that the tangent sheaf of $\Pb^n_k$ is
generated by its global sections and that
$\Gamma(\Pb^n_k, T_{\Pb^n_k}) \cong \Sl_n(k)$.
\begin{proposition}\label{raional-dsc}
    \begin{enumerate}
    \item The affine space $\Ab^n_k$ and projective space $\Pb^n_k$
      are both simply connected.
    \item The smooth locus $X_{reg}$ of a normal projective rational
      variety $X$ is simply connected, and therefore $X$ itself is
      simply connected.
  \end{enumerate}
\end{proposition}

\begin{proof} 
  (1): We need the well-known fact that the algebraic de~Rham
  cohomology of affine space is concentrated in degree 0, in
  particular
  \begin{displaymath}\tag{*}
    \Ext^1_{A_n(k)}(k[x], k[x]) = H^1_{DR}(\Ab^n_k)=0,
  \end{displaymath}
  where $k[x]=k[x_1, \ldots, x_n]$ is a polynomial ring and $A_n(k)$
  is its Weyl algebra. By \thetag{*} it suffices to prove that if $M$
  is simple regular singular $\Dc_{\Pb^n_k}$-module whose restriction
  $M_{\Ab^n_k}$ to an affine subset $\Ab^n_k$ is a non-zero
  connection, then $M_{\Ab^n_k} = k[x]$. For this it suffices to prove
  that the restriction $M_K$ of $M$ to the generic point is isomorphic
  to $K$, and for this in turn it suffices (see
  \Remark{\ref{purity}},(1)) to prove that the space of global
  sections
\begin{displaymath}
  V= \Gamma(\Ab^n_k, M_{\Ab^n_k}),
\end{displaymath}
contains a non-zero constant vector, i.e.
$\cap V^{\partial_{x_i}}\neq 0$. Let $(\hat x_i)$ be homogeneous
coordinates of $\Pb^n_k$ and put $x_i= \hat x_i/\hat x_0$. The space
$V$ forms a representation of
$ \Sl_n(k)= \Gamma(\Pb^n_k, T_{\Pb^n_k})$ and is moreover a finite
module over
$A= k [x_1, \ldots , x_n] = \Gamma(\Ab^n_k, \Oc_{\Pb^n_k})$. Put
$\nabla = -\hat x_0 \partial_{\hat x_0}= \sum_{i=1}^n x_i
\partial_i\in \Sl_n(k) $.
Since $M$ has regular singularities, $V$ is locally finite over the
polynomial ring $R= k[\nabla]$, so that if $v$ is a vector in $V$,
then $\dim_k R\cdot v < \infty$. It follows that there exists a
finite-dimensional $R$-submodule $V_0$ such that $AV_0 =V$. The
support $\Lambda = \supp V_0 \subset \Spec R$ is a finite set
$\Lambda= \{\lambda_1, \ldots , \lambda_r \}$, where $\lambda_i$ are
prime ideals. Since $\supp A = \Nb \subset \Spec R $ we have
\begin{displaymath}
  \supp V = \supp A V_0= \bigcup_i (\lambda_i + \Nb)
\end{displaymath}
(if $m$ is an integer, the prime ideal $\lambda + m$ is a translate of
$\lambda$). Therefore there exists an integer $i$ such that
$\lambda_i -1 \not \in \supp V$. Since
$\supp \partial_i V_0 \subset \supp V$ it follows that
$\partial_i \cdot v =0$ if $\supp R\cdot v = \lambda_i$. Therefore
$V_i=V^{\partial_{i}}\neq 0 $. Now $V_i$ is a finitely generated
module over $A_i = k[x_1, \ldots, x_{i-1}, x_{i+1}, \ldots , x_n ]$,
and since $[\partial_i, \partial_j]=0$ it is also a representation of
$\Sl_{n-1}(k)= \Gamma(\Pb^{n-1}_k, T_{\Pb^{n-1}_k})$, which again is
locally finite over $\nabla$. One can therefore iterate and conclude
that $\cap_{i=1}^n V^{\partial_i}\neq 0$.

(2): Let first $M$ be a connection on a smooth rational projective
variety $X$. If $K$ and $K_X$ are the fraction fields of $\Pb^n_k$ and
$K_X$, respectively there exists a birational morphism $U \to U_X $
where $U$ is a subset of $\Pb^n_k$ and $U_X$ an open subset of $X$,
such that $\codim_{\Pb^n_k}(\Pb^n_k \setminus U)\geq 2 $, and inducing
an isomorphism $K \cong K_X$. The inverse image $M_{\Pb^n_k}$ therefore
defines a connection on $\Pb^n_k$ so that by (1)
$M_{\Pb^n_k}\cong \Oc^m_{\Pb^n_k}$ for some integer $m$. Therefore at
the generic points we have isomorphisms of $\Dc_K$-modules
\begin{displaymath}
  K_X\otimes_{\Oc_X} M \cong K\otimes_{\Oc_{\Pb^n_k}}
  \Oc^m_{\Pb^n_k}\cong K^m \cong K^m_X,
\end{displaymath}
hence $M \cong \Oc^m_X$.

Now assume that $X$ is a normal rational projective variety over $k$,
$j: X_{reg}\to X$ be the inclusion of its smooth locus, and $M$ be a
connection $X_{reg}$. Since $X$ is normal, it follows that
$N=j_+(M)\in \Con^2(X)$. There exists a desingularization
$\pi: X'\to X$ that is an isomorphism over $X_{reg}$. Again since $X$
is normal it satisfies Serre's condition $(S_2)$, hence
$j_*j^*(T_X) = T_X$, so that we get a well-defined tangent map
$T_{X'}\to \pi^*(T_X)$. Therefore there exists a well-defined inverse
image $\pi^!(N)$ forming a connection on the smooth rational variety
$X'$, hence it is locally free over $\Oc_{X'}$. But by the previous
paragraph we even have $\pi^!(N)\cong \Oc_{X'}^m$ for some integer
$m$, hence its restriction to the generic point
$K_X\otimes_{\Oc_X}N \cong K^m_X$, so that
$M=j^!(N)\cong \Oc^m_{X_{reg}}$. This implies that $X_{reg}$ is s.c..
\end{proof}
\begin{remark}
  Using another well-known fact, that connections on $\Pb^n_k$ are
  generated by their global sections \cite{borel:Dmod}*{VII, Prop
    9.1}, \Proposition{\ref{constant-global}} also implies that
  $\Pb^n_k$ is simply connected.
\end{remark}
\begin{proposition}\label{sc-proper}
  Let $\pi: X\to Y$ be a smooth proper map of smooth connected
  varieties with connected fibres. The following are equivalent:
  \begin{enumerate}
  \item $X$ is simply connected.
  \item $Y$ is simply connected and a closed fibre of $\pi$ is simply
    connected.
  \end{enumerate}
\end{proposition}
\begin{proof}
  $ (1)\Rightarrow (2)$: Let $N$ be a connection on $Y$ of rank $n$.
  It follows that $\pi^!(N)$ is a trivial connection on $X$. Since
  $\pi$ is proper with connected fibres, it follows that
  \begin{displaymath}
    N= \pi_*\pi^!(N) = \pi_*(\Oc_X^n) = \Oc_Y^n. 
  \end{displaymath}
  Therefore $Y$ is simply connected.

  $ (2)\Rightarrow (1)$: If $M$ is a connection on $X$, then
  $\pi_*(M)$ is a trivial connection on $Y$
  \Prop{\ref{global-invariant}} and $\pi^!(\pi_*(M))$ is a trivial
  connection on $X$. The canonical map
  \begin{displaymath}
    \Phi:     \pi^!(\pi_*(M)) \to M 
  \end{displaymath}
  is a morphism of locally free $\Oc_X$-modules, which we assert is an
  isomorphism. Let $j:F \to X$ be the inclusion of a fibre over a
  point in $Y$, which is an embedding of a smooth proper variety,
  since $\pi$ is smooth and proper. It suffices to see that $\Phi$
  induces an isomorphism $ j^!\pi^!(\pi_*(M))\to j^!(M)$, where by
  assumption both sides are trivial connections. Now conclude from the
  base change teorem \cite{borel:Dmod}*{VI, \S8, Th. 8.4}.
\end{proof}
\subsection{The Grothendieck-Lefschetz theorem for connections}
We will make some remarks pertaining to the Grothen\-dieck-Lefschetz
theorem about the étale fundamental group, as expounded in \cite{SGA2}
and \cite{hartshorne:ample}*{Ch 4, \S 2}. But here instead of studying
étale covers our objective is to compare the category $\Con(X)$
($\Con_{et}(X)$) to $\Con(Y)$ ($\Con_{et}(Y)$), where $Y$ is a smooth
subvariety of a smooth projective variety $X$ over a field $k$ of
characteristic $0$, and we can work with either $\Con(X)$ or
$\Con_{et}(X)$, while [loc. cit] only considers the latter category
(but see \Remark{\ref{rh-proof-g-l}}). The central idea to factorize
over a completion is due to Grothendieck, and the contribution by
Hartshorne by simplifying the treatment of an important special case
is much acknowledged.

So we are given a closed embedding
\begin{displaymath}
  \phi: Y\hookrightarrow X \subset \Pb^N_k, 
\end{displaymath}
and put $n= \dim X$ and $r= \codim_X Y$. The main question is to
determine when the restriction functor
\begin{equation}\label{lefschets-func}
  \phi^! : \Con_{(et)} (X)\to \Con_{(et)} (Y), 
\end{equation}
is fully faithful and even an equivalence.\footnote{ Of course,
  according to Kashiwara's theorem $\Con(Y)$ is equivalent to the
  subcategory of holonomic $\Dc_X$-modules $M$ such that $\supp M = Y$
  and $\phi^!(M)$ is a connection.} Below we will only consider
$\Con(X)$ since $\Con_{et}(X)$ is treated in the same way.

To attain our objective we follow the track of Grothendieck by
factoring over the formal completion $\hat X$ of $X$ along $Y$, using
the maps
\begin{displaymath}
   Y \xrightarrow{\hat \phi}  \hat X \xrightarrow{\nu} X,
 \end{displaymath}
 so that $\phi= \nu \circ \hat \phi $ and $\phi^! = \hat \phi^! \circ
 \nu^!$, where we have the usual inverse image functors
 \begin{displaymath}
   \Con(X) \xrightarrow{\nu^!} \Con(\hat X) \xrightarrow{\hat{ \phi}^!}
   \Con(Y). 
 \end{displaymath}
 To emphasise, $\Con(\hat X)$ is the category of
 $\Dc_{\hat X }$-modules that are coherent over $\Oc_{\hat X}$.
 Putting
 $E= H^r_Y(\Oc_{\hat X})= \hat \phi_+\hat \phi^!(\Oc_{\hat X})$,
 define the functor
\begin{displaymath}
\psi : \Con(Y)\to \Con(\hat X), \quad   \psi (N) =   Hom_{\Oc_X}(\hat \phi_+(N^*), E).
\end{displaymath}

  \begin{proposition}\label{formal-eq}
    The functors $(\hat \phi^!, \psi)$ form mutually inverse
    equivalences of categories
    \begin{displaymath}
      \Con(Y)\cong \Con(\hat X ). 
    \end{displaymath}
  \end{proposition}
  We make some preparations before the proof, by extending to a
  non-local situation some results related to Matlis duality. Put for
  an $\Oc_{\hat X}$-module $M$, $M'= Hom_{\Oc_{\hat X}}(M,E)$.
  \begin{lemma}\label{matlislemma}
    \begin{enumerate}
    \item If $N$ is  a coherent $\Oc_Y$-module, then 
      $\hat \phi_*(N)'= \hat \phi_*(N^*)$, so that if $N$ is locally free, then
      $\hat \phi_*(N)^{''} = \hat \phi_*(N)$.
    \item $E'= \Oc_{\hat X}$.
    \item If $M$ is a coherent $\Oc_{\hat X}$-module, then $M'= M^*$, so that in
      particular if $M$ is also locally free, then $(M')'=M$.
    \end{enumerate}
  \end{lemma}
 Put
\begin{displaymath}
M_n= M^{I_Y^n} = \{m \in M \ \vert \ I^n_Ym
=0\}  = Hom_{\Oc_{\hat X}}(\frac {\Oc_{\hat X}}{I_Y^n}, M).
\end{displaymath}
  \begin{remark}
It follows from the proof that if $M= M^{I_Y^n}$ for some high
    $n$, and if either $\oplus M_n/M_{n-1}$ or $\oplus I^nM/I^{n+1}M$ are
    locally free over $\Oc_Y$, then $M= M^{''}$.
\end{remark}

  \begin{proof}
    (1): Let here $\hat \phi^!$ denote the derived inverse image, so that if $W$
    is a $\Dc_{\hat X}$-module with support in $Y$, then
    $\hat \phi^!(W) = W^{I_Y}[-r]$ (a single degree complex concentrated in the
    degree $r$) and if $W$ is torsion free, then $\hat \phi^!(W)= W/I_YW$ (a
    single degree complex concentrated in the degree $0$). We recall Kashiwara's
    equivalence, that for a holonomic $\Dc_Y$-module $\Nc$,
    $\hat \phi^!\hat \phi_+(\Nc) [r] =\Nc $. This gives
    \begin{align*}
      &      Hom_{\Oc_{\hat X}}(\hat \phi_*(N), E) =Hom_{\Oc_{\hat X} } (\hat
        \phi_*(N),  (\phi_+\hat
        \phi^!(\Oc_{\hat X}))^{I_Y}) 
      \\ &= Hom_{\Oc_{\hat X}} (\hat \phi_* (N), \hat \phi_*(\hat \phi^! \hat
           \phi_+(\Oc_Y) [r]
           ))= Hom_{\Oc_{\hat X}}(\hat \phi_*(N),\hat \phi_*(\Oc_Y))
           = \hat
           \phi_*(N^*).
    \end{align*}

(2):
Since $E = \varinjlim_n E_n$ we get (where the fourth equality is
detailed below)
\begin{align*}
  Hom_{\Oc_{\hat X}}(E,E) &= Hom_{\Oc_{\hat X}}(\varinjlim_nE_n,E) =
                            \varprojlim_n Hom_{\Oc_{\hat X}}(E_n, E) \\
                          &= \varprojlim_n
                            Hom_{\Oc_{\hat
                            X}}(Hom_{\Oc_{\hat X}}(\frac {\Oc_{\hat X}}{I^n_Y},
                            E), E)
                            = \varprojlim_n \frac{\Oc_{\hat
                            X}}{I^n_Y} = \Oc_{\hat X}.
\end{align*}
We need to prove the assertion
$((\Oc_{\hat X}/I^n_Y)')'= \Oc_{\hat X}/I^n_Y $. To see this we first
include an aside:
\begin{description}
\item[A] Let $M$ and $N$ be coherent $\Oc_{\hat X}$-modules such that
  $\supp N \subset Y$ and the homology of the complex $\RG_Y(M)$ is
  concentrated in the degree $ r$. Then the homology of
  $RHom_{\Oc_{\hat X}}(N, \RG_Y(M))$ is also concentrated in the
  degree $ r$.
\end{description}
To see this first note that the assumption implies that the homology
of $RHom_{\Oc_{\hat X}}(N,$ $ M)$ is concentrated in the degree $r$.
Next applying $RHom_{\Oc_{\hat X}}(N, \cdot)$ to the distinguished
triangle $\RG_Y(M)\to M\to Rj_*j^*(M)\to $, one gets
$RHom_{\Oc_{\hat X}}(N, \RG_Y(M)) = RHom_{\Oc_{\hat X}}(N, M)$,
implying {\bf A}. Therefore by {\bf A}, $Y$ being a local complete
intersection of codimension $r$, the homology of
$RHom_{\Oc_{\hat X}}(N, \RG_Y(\Oc_{\hat X}))$ is concentrated in the
degree $r$; in particular,
\begin{displaymath}
  Ext^1_{\Oc_{\hat X}}(N, E)=0.
\end{displaymath}
Applying $'$ twice to the exact sequences  
\begin{displaymath}\tag{*}
0 \to \frac{I_Y^n}{I_Y^{n+1}}\to \frac {\Oc_{\hat X}}{I_Y^{n+1}} \to \frac{\Oc_{\hat
  X}}{I_Y^{n}}\to 0,
\end{displaymath}
we therefore  get the exact sequences
\begin{displaymath}\tag{*'}
  0 \to   (\frac {I_Y^n}{I_Y^{n+1}})^{''}\to (\frac {\Oc_{\hat X}}{I_Y^{n+1}}) ^{''} \to (\frac{\Oc_{\hat
  X}}{I_Y^{n}}) ^{''}\to 0, 
\end{displaymath}
so the assertion follows by induction and (1), since we have a canonical map
from \thetag{*} to $\thetag{*'}$, and the modules  $I^n_Y/I^{n+1}_Y $ are locally free over
$\Oc_Y$.

(3): It follows by  (2) that $M' = E'\otimes_{\Oc_{\hat X}} M^* =M^*$, then
evidently if $M^{**}=M$, it follows that $M^{''}=M$.
\end{proof}
\begin{pfof}{\Proposition{\ref{formal-eq}}}
  Let $M\in \Con (\hat X)$ and $N\in \Con(Y)$. We have
  $\hat \phi^!(M) = M/I_YM$ and locally $\hat \phi_+(N^*) $ is
  isomorphic to $E^n$, where $n= \rank N$, so that by
  \Lemma{\ref{matlislemma}},  $Hom_{\Oc_{\hat X}}(\hat \phi_+(N^*),E )$
  is locally isomorphic to $\Oc_{\hat X}^n$. Thus
    \begin{align*}
      \hat \phi^!(\psi(N)) &=  \hat \phi^! (Hom_{\Oc_{\hat X}}(\hat \phi_+(N^*),E
                             )) \\
                           &= Hom_{\Oc_Y} (N^*,\Oc_Y) =N.
    \end{align*}
    Since $M^*$ is locally isomorphic to $\Oc_{\hat X}^m$, $m= \rank
    M$, so that $\hat \phi^!(M^*)$ is locally isomorphic to $\Oc_Y^m$,
    we get as before that $\hat \phi_+(\hat \phi^!(M^*))$ is locally
    isomorphic to $E^m$. Therefore, again by
  \Lemma{\ref{matlislemma}}, 
\begin{displaymath}
  \psi(\hat \phi^!(M)) = Hom_{\Oc_{\hat X}}(\hat \phi_+\hat \phi^!(M^*), E)=   M. 
\end{displaymath}
\end{pfof}

We recall the reformulation of the condition that $\nu^*$ and $\nu^!$
be fully faithful (or an equivalence) by Grothendieck. Say that $\phi$
satisfies the Lefschetz condition $\Lef(X, Y)$ if for locally free
$\Oc_X$-modules $M$ we have
\begin{displaymath}
  \Gamma(\hat X,\nu^*(M) ) = \Gamma(X, M). 
\end{displaymath}
Say also that $\phi$ satisfies $\Leff(X,Y)$ if $\Lef(X,Y)$ holds and moreover
for any locally free $\Oc_{\hat X}$-module $N$ there exists a locally free
$\Oc_X$-module $M$ such that $N= \nu^*(M)$. Similarly, $\phi$ satisfies
$\Lef^c(X,Y)$ and $\Leff^c(X,Y)$, respectively, if the above conditions are
satisfied when $\nu^*$ is replaced by $\nu^!$ and $M$ and $N$ are connections.
\begin{remark}
  We do not need to consider open neighbourhoods of $Y$ in $X$, so
  that our effective Lefschetz condition $\Leff(X,Y)$ is stronger than
  the one used in \cite{SGA2}*{Exp X, \S 2} and
  \cite{hartshorne:ample}*{Ch 4, \S 1}. 
\end{remark}

\begin{lemma}\label{lef-lemma}
  \begin{enumerate}
  \item $\nu^*$ ($\nu^!$) is fully faithful if and only if $\Lef(X,Y)$
    ($\Lef^c(X,Y)$) holds.
  \item $\nu^*$ ($\nu^!$) is an equivalence if and only if  $\Leff(X,Y)$ ($\Leff^c(X,Y)$) holds.
  \end{enumerate}
\end{lemma}
\begin{proof} We consider only $\nu^!$, and note that one implication is
  evident. (1): If $M_1 ,M_2\Con(X)$, then
  $\Nc = Hom_{\Oc_X}(N_1, N_2)\in \Con(X)$. So by $\Lef^c(X,Y)$,
  \begin{displaymath}
    \Hom_{\Dc_X}(M_1, M_2) = \Gamma(X, \Nc) = \Gamma(\hat X, \nu^!(\Nc)) =
    \Hom_{\Dc_{\hat X}}(\nu^!(M_1), \nu^!(M_2)).  
  \end{displaymath}
  (2):Evident.
\end{proof}
Let $ j : U=X\setminus Y \to X $ be the open inclusion. The coherent
cohomological dimension for $U$, $\cdo(U)$, satisfies $\cdo(U)\leq r$
if the homology of the complex $\RG(U,M_U)$ is non-zero only in
degrees $[0,r]$ when $M_U$ is a coherent $\Oc_U$-modules. Similarly,
we have the flat coherent cohomological dimension $\cdo_c(U)$, if we
also require that $M_U$ be a connection. In general we have
$ \cdo_c(U) \leq \cdo(U)$, where strict inequality holds for instance
when $U$ is simply connected and
$\max \{i \vert H^i(U, \Oc_U)\neq 0\} < \cdo(U)$.
  \begin{proposition}
  \begin{enumerate}
  \item If $\cdo(U) < n-1$, then $\Lef (X,Y)$ holds.
  \item If $\cdo_c(U) < n-1$, then $\Lef^c(X,Y)$ holds, and $\phi^!$
    is fully faithful.
  \end{enumerate}
\end{proposition}
We remark right away that $\cdo(U) < n-1 $ if $Y$ is a complete
intersection, $Y = X \cap H_1\cap H_2\cap \cdots \cap H_r$,
$r \leq n$, where the $H_i$ are hypersurfaces in $\Pb^N_k$, since then
$U$ is a union of at most $n-1$ affine sets, see
\cite{hartshorne:ample}*{Cor 1.2}.
\begin{proof}
  For (1), see \cite{hartshorne:ample}*{Ch. 4, Prop 1.1}; one gets the assertion
  about $\nu^!$ in (2) in a similar way. To see that $\phi^!$ is fully faithful
  it is convenient to instead prove that $\bar \phi^!$ is fully faithful, where
  the inverse image functor for right connections (right $\Dc$-modules) is
  $N \to \bar \phi^!(N) = \omega_Y\otimes_{\Oc_Y}\phi^!(\omega_X^*
  \otimes_{\Oc_X}N) $.
  This implies also that $\phi^!$ is fully faithful since the pair
  $(\omega_X\otimes_{\Oc_X}\cdot, \omega_Y\otimes_{\Oc_Y}\cdot)$ defines
  equivalences between the categories of left and right connections on $X$ and
  $Y$, respectively. First, as in the proof of (1), if $M$ is a connection and
  since $\cdo_c(U)< n-1$,
  \begin{displaymath}
    H^n(\RG(X, \RG_Y(M))) =  H^n(X, M) = H^0(X, M^*\otimes_{\Oc_X}\omega_X)',
  \end{displaymath}
  where now $N= M^*\otimes_{\Oc_X}\omega_X$ is a right connection. By
  Kashiwara's theorem $\phi_+$ defines an equivalence between
  quasi-coherent $\Dc_Y$-modules and quasi-coherent $\Dc_X$-modules
  whose support belongs to $Y$, and $\RG_Y(M)= \phi_+\phi^!(M)[-r]$
  (see \cite{hotta-takeuchi-tanisaki}*{Th. 1.6.1}). Moreover, a
  $\Dc_X$-module whose support belongs to $Y$ is quasi-isomorphic to a
  complex of quasi- coherent flasque sheaves of $\Dc_X$-modules with
  support in $Y$. This implies
  \begin{align*}
    H^n(\RG(X, \RG_Y(M))) &= H^{n-r}(X, \phi_+\phi^!(M)) = H^{n-r}(Y, \phi^!(M)) \\
                          & =  H^0(Y, \omega_Y\otimes_{\Oc_Y}\phi^!(M)^*)',
  \end{align*}
  where the last step follows from Serre duality. Since
  $\bar \phi^!(N) = \omega_Y\otimes_{\Oc_Y}\phi^!(M^*) =
  \omega_Y\otimes_{\Oc_Y}\phi^!(M)^* $, we get
  $H^0(X, N) = H^0(Y, \bar \phi^!(N)) $, so that similarly to 
  \Lemma{\ref{lef-lemma}} one concludes that $\bar \phi^!$ is fully faithful.
\end{proof}

The Grothendieck-Lefschetz theorem for connections takes the following
form.
\begin{theorem}\label{groth-lef} Let $Y$ be a smooth subvariety of a smooth projective
  variety $X$.
  \begin{enumerate}
  \item If $Y$ is a complete intersection, then $\phi^!$ is fully
    faithful.
  \item If (1) holds and moreover $\dim Y \geq 2$, then $\phi^!$ is an
    equivalence of categories.
  \end{enumerate}
\end{theorem}
\begin{proof}
  It remains to see when $\nu^!$ is an equivalence. Let $\Oc_X(1)$ be
  an ample line bundle on $X$, $\Oc_{\hat X}(1) = \nu^*(\Oc_X(1))$,
  and for an $\Oc_{\hat X }$-module $\bar M$, put
  $\bar M(m) = \Oc_{\hat X}(m)\otimes_{\Oc_{\hat X}} \bar M$. Since
  $Y$ is a set-theoretic complete intersection it follows that
  $\cdo(U) < n-1$ (see \cite{hartshorne:ample}), so that $\nu^*$ is
  fully faithful. If $M\in \Con(X)$, then $M(m)$ is generated by its
  global sections, and therefore $\nu^!(M)(m)$ is also generated by
  its global sections, for $m\gg 0$. To get the converse, assume the
  following condition:
  \begin{description}
  \item [(G)] The category of locally free $\Oc_{\hat X}$-modules is generated
    by the invertible sheaves $\{\Oc_{\hat X}(m)\}_{m\in \Zb}$.
  \end{description}
  The point is that $\Oc_{\hat X}(m) = \nu^*(\Oc_X(m))$. It follows that there
  exists an exact sequence of locally free $\Oc_{\hat X}$-modules
  \begin{displaymath}
    \bar \Fc_1 \xrightarrow{\bar h} \bar \Fc_0 \to \bar M
  \end{displaymath}
such that $\nu^*(\Fc_1)= \bar \Fc_1$ and $\nu^*(\Fc_0)=\bar \Fc_0$ for some locally free $\Oc_X$-modules $\Fc_0, \Fc_1$. Since
$\nu^*$ is fully faithful we have 
\begin{displaymath}
  \Gamma(X, Hom_{\Oc_X}(\Fc_1, \Fc_0)) = \Gamma(\hat X,
  \nu^*(Hom_{\Oc_X}(\Fc_1, \Fc_0)) ) =   \Gamma(\hat X, Hom_{\Oc_{\hat
    X}}(\bar \Fc_1, \bar \Fc_0))
\end{displaymath}
so that there exists an extension $h: \Fc_1 \to \Fc_0$ of $\bar h$
such that putting $M= \Coker (h)$, then $\nu^*(M)= \bar M$ (this
argument is well-known, see \cite{hartshorne:ample}*{Th. 1.5 }). It remains
to see that $M$ is a connection, which again follows since $\nu^*$ is
fully faithful. A connection is determined by a map
\begin{displaymath}
  \nabla \in \Hom_{\Oc_X}(M, \Pc^1_X(M)) = \Hom_{\Oc_{\hat X}} (\nu^*(M), \nu^*(\Pc^1_X(M))=
  \Hom_{\Oc_{\hat X}} (\bar M , \Pc^1_{\hat X}(\bar M)),
\end{displaymath}
where $\Pc^1_X(M)$ is the sheaf of first order principal parts (see
\cite{EGA4:4}*{\S 16}). Therefore, the connection on $\bar M$
determines uniquely a connection on $M$. Considering the curvature map
$\Omega_X(M)\to \Omega^2_X(M)$, one proves similarly that this map is
$0$, so the connection on $M$ is integrable.

In general (G) need not hold (see \cite{hartshorne:ample}), contrary
to the corresponding statement for $X$ (by Serre's theorem). But if
$\dim Y \geq 2$, then (G) holds, which is the content of the following
lemma, forming a main step in the proof of Grothendieck's theorem in
\cite{hartshorne:ample}*{Ch 4, \S 2}.
\end{proof}
\begin{lemma}[\cite{SGA2}*{Exp IX, Th. 2.2},\cite{hartshorne:ample}*{Prop
    1.3}]\label{globalsections} Assume that $\nu^*$ is fully
  faithful ($\Lef (X,Y)$) and that $\dim Y \geq 2$. If $\bar M$ is a locally free
  $\Oc_{\hat X }$-module, then the space of global sections
  $\Gamma(\bar X, \bar M(m))$ generates $\bar M(m)$ when $m\gg 0$. In particuar,
  (G) holds.
\end{lemma}
\begin{remark}\label{rh-proof-g-l}
  When $X_a$ is a projective complex analytic manifold the assertion in
  \Theorem{\ref{groth-lef}} follows from the Lefschetz theorem about
  the isomorphism between $\pi_1(Y_a,p)$ and $\pi_1(X_a,p)$
  (\cites{bott-lefschetz,lefschetz-selected}), combined with G.A.G.A.
  and the Riemann-Hilbert equivalence. Regarding
  \Proposition{\ref{formal-eq}}, we know that the category of étale
  coverings of $Y$ and $\hat X$ are equivalent \cite{SGA1}*{8.4}.
\end{remark}
\subsection{Differential coverings}
A tool is introduced to analyze $\Con(X)$ by cutting out points
by subvarieties.

Let $\{C_{\lambda}\}_{\lambda\in \Lambda}$ be a family of subvarieties
of $X$, with defining ideals $I_\lambda$, and for a point $x$ in $X$
let $\Lambda_x$ be the subset of $\lambda$ in $\Lambda$ such that $x$
belongs to $C_\lambda$. Let $T_X(I_\lambda)\subset T_X$ be the
subsheaf of derivations $\partial$ that are preserve the ideal
$I_\lambda$, so that $\partial (I_\lambda) \subset I_\lambda$ (i.e.
$\partial$ is tangential to $C_\lambda$).
\begin{lemma}\label{cut-lemma}
  The following are equivalent for a smooth point $x$ in $X$:
  \begin{enumerate}
  \item \begin{displaymath} T_{X,x} = \sum_{\lambda \in \Lambda_x}
      T_{X,x}(I_{\lambda}).
\end{displaymath}
\item The map of tangent vector spaces
  \begin{displaymath}
    \bigoplus_{\lambda\in \Lambda_x} k_{C_\lambda,
      x}\otimes_{\Oc_{C_\lambda, x}}T_{C_\lambda, x} \to k_{X,x}\otimes_{\Oc_{X,x}}T_{X,x}    
  \end{displaymath}
  is surjective.
\end{enumerate}
If moreover $x$ is a smooth point in the varieties $C_\lambda$, then
(1-2) are equivalent to the map
  \begin{displaymath}
  \frac{\mf_{X,x}}{\mf^2_{X,x}}\to \bigoplus_{\lambda \in \Lambda_x}
  \frac {\mf_{C_\lambda, x}}{\mf^2_{C_{\lambda,x}}}
\end{displaymath}
being injective.
\end{lemma}
\begin{proof}
  Assume (1). Since $x$ is a smooth point, so that $T_{X,x}$ is free,
  then $k_{X,x}\otimes_{\Oc_{X,x}}T_{X,x} $ is isomorphic to the
  tangent space $Hom_{k_{X,x}}(\mf_{X,x}/\mf_{X,x}^2, k_{X,x})$, the
  map $T_{X,x}(I_\lambda)$
  $\to k_{C_\lambda,x}\otimes_{\Oc_{C_\lambda}, x} T_{C_\lambda,x}$ is
  surjective, and we have a well-defined tangent map
  $k_{C_\lambda, x}\otimes_{\Oc_{C_\lambda ,x}}T_{C_\lambda,x}\to
  k_{C_\lambda, x}\otimes_{\Oc_{X,x}} T_{X,x}$. These remarks imply
  (2). The last assertion follows since $T_{C_\lambda,x}$ is free.
\end{proof}
If the condition (1) in \Lemma{\ref{cut-lemma}} is satisfied we say
that $\{C_\lambda\}_{\lambda\in \Lambda}$ {\it cuts out} the point
$x$. Let $p_\lambda : \hat C_\lambda\to X$ be the normalization map of
$C_\lambda\subset X$. It is important to realize that in general the
sheaf of derivations ``from $\Oc_X $ to $\Oc_{\hat C_\lambda}$''
$T_{\hat C_\lambda \to X}$ does not equal $p^*_{\lambda}(T_X)$ when
$ C_\lambda\cap S\neq \emptyset$ ($S$ is the singular
locus of $X$), and therefore there exist no tangent mapping from
$T_{\hat C_\lambda}$ to $p^*_\lambda(T_X)$, so that even if $M$ is a
connection, $p^*_\lambda(M)$ need not be a connection. Therefore
instead of asking not only that $\hat C_\lambda$ be s.c., we say that
$p_\lambda$ is simply connected if $\hat C_\lambda$ is s.c. and
$p^*_\lambda(M)$ indeed is a connection for any $M\in \Con^f(X)$.

\begin{definition}\label{dc-def}
  A \emph { differential covering} of a variety $X$ is a family of
  subvarieties $\{C_\lambda\}_{\lambda \in \Lambda}$ that cut out the
  points in a dense subset of $X$. Say that $X$ is {\it differentially
    simple} (d.s.) if $X$ can be a provided with a differential
  covering such that the normalization maps  $p_\lambda$ are simply
  connected.
\end{definition}


\begin{proposition} Assume that $X$ is a normal variety that is
  provided with a differential covering
  $\{C_\lambda\}_{\lambda \in \Lambda}$ that cut out the points in an
  open subset $X_0$ such that $\codim_{X}(X\setminus X_0)\geq 2$. Let
  $M$ be a coherent $\Oc_X$-submodule of a connection $N$ such that
  the inverse image $p^*_\lambda(M)$, $\lambda \in \Lambda$, is a
  connection. Then $j_*j^*(M)$ is a subconnection of $j_+j^!(N)$,
  where $j: X_0\to X$ is the open inclusion.
\end{proposition}
Of course, if $X$ is smooth, then $j_+j^!(N)=N$.
\begin{proof} Since $N$ is a connection and $X$ is normal (so that
  $j_*j^*(\Oc_X)=\Oc_X
  $), it suffices to prove that
  $T_{X_0}\cdot M_{X_0}\subset M_{X_0}$. If $\partial_x \in T_{X,x}$
  and $x\in X_0$, by \Lemma{\ref{cut-lemma}} $\partial_x$ defines a
  map
  \begin{displaymath}
    \partial_x: M_x \to k_{X,x}\otimes_{\Oc_{X,x}}M_x.
  \end{displaymath}
  Hence by Nakayama's lemma the $\Oc_{X_0}$-submodule
  $T_{X_0} \cdot M_{X_0} \subset N_{X_0} $ maps to $0$ in
  $N_{X_0}/M_{X_0}$, and therefore
  $T_{X_0}\cdot M_{X_0}\subset M_{X_0}$.
\end{proof}

\begin{theorem}\label{dsc-thm}
  Differentially simple  varieties are simply connected.
\end{theorem}
The \hyperref[proof of diffsimple]{proof} is postponed. 
\begin{corollary}
  If $X$ and $Y$ are (étale) simply connected normal varieties, then
  $X\times_k Y$ is also (étale) simply connected.
\end{corollary}
When one of the factors is proper this follows from \cite{SGA1}*{Cor
  1.7} and also from \Proposition{\ref{sc-proper}}. It needs not hold
in positive characteristic when one of the factors is non-proper
[loc.cit, Rem. 1.10].

\begin{proof}
  The maps $p_y: X\times \{y\}\to X\times Y$ and
  $p_x: \{x\}\times Y\to X\times Y$, $x\in X_{reg}$, $y \in Y_{reg}$,
  form differential covering of $X\times Y$ where each $p_x$ and $p_y$
  is simply connected. Now conclude from \Theorem{\ref{dsc-thm}}.
\end{proof}

One says that a projective variety $X$ is rationally connected if any
pair of points in $X$ belongs to $p(\Pb^1_k)$ for some map
$p: \Pb^1_k \to X$. In some senses the category of rationally
connected varieties is more natural than its subcategory of rational
varieties, as proposed in \cite{kollar-miyoka-mori:rational} and
expounded in \cites{debarre:higher,kollar:rarionalcurves}. We want to
apply \Theorem{\ref{dsc-thm}} to such $X$, but relax the condition
that $X$ be projective, compensating by instead using maps from the
affine line $p : \Ab^1_k \to X$.

Let $j: X\to \bar X$ and $i: \Ab^1_k \to \Pb^1_k$ be a smooth
completion of $X$ and $\Ab^1_k$, respectively, identifying $X$ and
$\Ab^1_k$ with open subsets in $\bar X$ and $\Pb^1_k$. An equivalent
definition for $\bar X$ being rationally connected is that there
exists a variety $M$ and an (evaluation) map
$ e:\Pb^1_k \times M \to \bar X$ such that the naturally induced map
$c:\Pb^1\times \Pb^1 \times M\to \bar X\times \bar X$ is dominant (see
\cite{kollar:rarionalcurves}*{Def. 3.2}, \cite{debarre:higher}*{Def
  4.3}). Let $M_0$ be the projection of
$ e^{-1}(X)\cap (\Ab^1_k\times M)$ on the second factor of
$\Pb^1_k\times M$, so that we get a map $e_0:\Ab^1_k\times M_0 \to X$.
If any two points in a non-empty open subset of $X$ belong to a common
affine line $p(\Ab^1_k)$, the restriction of $c$
\begin{displaymath}
c_0:  \Ab_k^1\times \Ab_k^1 \times M_0 \to X\times X 
\end{displaymath}
is again dominant.

Say that a quasi-projective variety $X$ is {\it rationally connected}
if there exists a variety $M_0$ and a map
$e_0 : \Ab^1_k \times M_0 \to X$ such that the induced map
$c_0 : \Ab^1_k \times \Ab^1_k\times M_0 \to X\times X  $ is dominant. If $X$ is
rationally connected ($c_0$ is dominant) then $\bar X$ is rationally
connected in the sense of \cite{debarre:higher}*{Def 4.3} ($c$ is
dominant), but the converse does not hold.
\begin{corollary}\label{rat-conn-cor}
  Let $X$ be a smooth quasi-projective rationally connected variety of
  characteristic $0$. Then $X$ is differentially simple and is
  therefore simply connected.
\end{corollary}
\begin{remark}
  A proof that $\pi_1(X,p)$ is trivial when $X$ is a complex smooth
  projective rationally (chain) connected variety is presented in
  \cite{debarre:higher}*{Cor. 4.18}, where the idea to use Hodge
  theory was employed already by Serre \cite{serre-unirational} to
  prove that unirational projective complex manifolds are simply
  connected. The application of Hodge theory requires that $X$ be
  projective, and only implies that $\Con_{et}(X)$ is trivial; to get
  that $\Con(X)$ is trivial one also needs that $\pi_1(X_a)$ is
  finite, see \cite{debarre:higher}*{Cor 4.18}, using an idea from
  \cite{campana:twistor}\footnote{One can instead invoke the
    Grothendieck-Malcev theorem.}. In \Corollary{\ref{rat-conn-cor}}
  it is not required that $X$ be projective and its proof is
  algebraic.
\end{remark}

\begin{proof} Since $c_0$ is dominant, by generic smoothness there
  exists an open subset $(X\times X)_0$ over which $c_0$ is smooth.
  The image $X_0$ of $(X\times X)_0$ with respect to the projection
  $p_2: X\times X \to X$ on the second factor is an open subset of
  $X$. For $x\in X_0$ put $X_x^0= p_2^{-1}(x) \subset (X\times X)_0 $,
  which we identify with an open subset of $X$. The base change of
  $c_0$ over the inclusion
  $X_x^0 \to (X\times X)_0 \subset X\times X $ defines a smooth map
  $e^0_x: \Ab^1_k \times M^0_x \to X_x^0$. Again by generic smoothness
  there exists a non-empty open subset
  $X_x^{00}\subset X_x^0\subset X$ over which $e^0_x$ is smooth; let
  $e^{00}_x: \Ab^1_k\times M_x^{00}\to X_x^{00} $ be the base change,
  which is thus a smooth map. For each $f\in M_x^{00} $ denote also by
  $f: \Ab^1_k\to X_x^{00} $ the restriction of $e^{00}_x$ to
  $\Ab^1_k\times \{f\}$. Now there exists an open subset
  $X_{00}\subset X_0$ of points $x$ such that the tangent mapping of
  $f: \Ab^1_k\to X$ is injective at points in $\Ab^1_k$ above $f(x)$.
  Assume from the beginning that $x\in X_{00}$. Since the condition
  that the tangent map of $f$ is injective above points of $x$ is open
  in $M_x$, there exists a non-empty open subset $M^{s,00}_x$ of
  $M_x^{00}$ such that the restriction
  \begin{displaymath}
    e^{s,00}_x : \Ab^1_k\times M^{s,00}_x\to X_x^{00}
  \end{displaymath}
  is smooth, and if $f \in M^{s,00}_x$, then the tangent map of $f$ is
  injective at points in $\Ab^1_k$ above $f(x)$. This implies that the
  curves $f(\Ab^1_k)$, $f \in M^{s,00}_x$, cut out the point $x$.
  Therefore the set of maps $f: \Ab^1_k\to X$ cut out an open subset
  of $X$, hence by \Theorem{\ref{dsc-thm}} and
  \Proposition{\ref{raional-dsc}} $X$ is simply connected.
\end{proof}

\begin{question}
  Assume that $X$ is a smooth (quasi-)projective complex variety that is
  provided with a differential covering by rational curves. Is then
  $X$ rationally connected?
\end{question}

Being differentially simple is a birational property in a rather
stronge sense.
\begin{lemma}\label{lem-ds}
  Let $\pi: X\to Y$ be a dominant morphism of quasi-projective
  varieties.
\begin{enumerate}
\item If $X$ is normal and d.s and $Y$ is smooth, then $Y$ is d.s..
\item Asume that $\pi$ is proper, $X$ is smooth, and $Y$ is normal. If
  $Y$ can be provided with a differential covering such that the
  smooth locus $\hat C^{reg}_\lambda$ is s.c., then $X$ is d.s.
\end{enumerate}
\end{lemma}
Often the coverings $p_\lambda : \hat C_\lambda \to Y$ are such that
$\hat C_\lambda$ is smooth (e.g. $C_\lambda$ is a curve), so that the
auxilliary condition in (2) is satisfied.

\begin{proof}
(1):  By generic smoothness ($\Char k =0$), a differential
  covering $ p_\lambda: \hat C _\lambda \to X$ gives one on $Y$,
  $\bar p_\lambda = \pi \circ p_\lambda: \hat C_\lambda \to Y $.
  Since $Y$ is smooth it follows that if $p_\lambda$ is simply connected,
  then so is $\bar p_\lambda$.

  (2): Since $\pi$ is proper and $\hat C_\lambda$ is normal
  $\bar p_\lambda : \hat C_\lambda \to Y$ lifts to maps of schemes
  $p_\lambda : \hat C^1_\lambda \to X$, where $\hat C^1_\lambda$ is a
  subscheme of $ \hat C_\lambda$ such that
  $\codim_{\hat C_\lambda} (\hat C_\lambda \setminus \hat
  C^1_\lambda)\geq 2$. Again by generic smoothness the maps
  $p_\lambda$ cut out a dense subset of $X$. If $M$ is a connection on
  $X$, then since $X$ is smooth, $p^!_\lambda (M)$ is a connection on
  $\hat C^1_\lambda$. Since
  $\hat C^{reg}_\lambda\subset \hat C^1_\lambda$ and
  $\hat C^{reg}_\lambda$ is s.c., it follows that $p^!_\lambda (M)$ is
  trivial.
\end{proof}
We can go further than \Lemma{\ref{lem-ds}}, (1).   
\begin{corollary}\label{dom-cor}
Assume that 
  \begin{displaymath}
   \pi: X\to Y,
  \end{displaymath}
  is a dominant morphism of normal varieties and let $S_Y$ be the
  singular locus of $Y$. If $X$ is d.s. (e.g. $X= \Ab^n_k$)
  and $\codim_X \pi^{-1}(S_Y)\geq 2$, then $Y$ is d.s. and hence simply
  connected.
\end{corollary}
\Corollary{\ref{dom-cor}} generalizes, by not requiring properness,
results of similar nature in \cites{gurjar:topaffine,gurjar:dominated,
  kumar:conner}.
\begin{proof}
  Since $\codim_X \pi^{-1}(S)\geq 2$ it follows that if $M$ is a
  connection, then $\pi^*(M)$ is a connection on $X$. Since $X$ is
  d.s. it follows that $X$ is s.c. \Th{\ref{dsc-thm}}, so that
  $\pi^*(M)$ is a trivial connection. If
  $p_\lambda : \hat C_\lambda \to X$ is a differential covering such
  that $p_\lambda$ is simply connected, then
  $\pi \circ p_\lambda \to Y$ is a differential covering and
  $(\pi \circ p_\lambda)^*(M) = p^*_\lambda (\pi^*(M))$ is a trivial
  connection on $\hat C_\lambda$. This implies that $Y$ is d.s., and
  hence simply connected.
\end{proof}
\begin{pfof}{\Theorem{\ref{dsc-thm}}}\label{proof of diffsimple}
  Recall that a connection $M$ is trivial if it is trival at the
  generic points of $X$, and $\{C_\lambda\}_{\lambda \in \Lambda }$ is
  a differential covering if and only if it is a differential covering
  of each component of $X$. Therefore it is sufficient to prove the
  assertion when $X$ contains only one generic point. There exists
  then an integer $m$ and a non-empty open subset $U$ of $X_{reg}$
  such that $M_U \cong \Oc_U^m$ as $\Oc_U$-module, and the connection
  is determined by defining a
  map $\nabla = d+ \Gamma: \Oc_U^m \to \Omega_U(\Oc_U^m)$, where $d$
  is the trivial connection on $\Oc_U^m$ and $\Gamma $ is section of
  $\Omega_U(\End_{\Oc_U}(\Oc^m_U))=
  \Omega_{U}\otimes_{\Oc_U}\End_{\Oc_U}(\Oc^m_U)$\footnote{The
    endomorphism $\Gamma$ also satisfies an integrability condition,
    but this will not play any role below, besides the fact that its
    restrictions to the varieties $C_\lambda$ do gives rise to
    integrable connections.}. Then $M$ is trivial if and only if
  $\Gamma$ belongs to the image of the map
\begin{displaymath}
  \dlog:  \Aut_{\Oc_X}(\Oc_X^m)(U) \to  \Omega_X(
  \End_{\Oc_X}(\Oc_X^m))(U),\quad  \phi\mapsto \phi^{-1}  d(\phi),
\end{displaymath}
where $d(\phi) (\partial)= [\partial,\phi]\in \End_{\Oc_U}(\Oc_U^m)$
and $\partial\in T_X(U)$. Here $ \Aut_{\Oc_X}(\Oc_X^m)(U)$ denotes
an $\Oc_U$-linear automorphism of $\Oc_U^m$. Put
\begin{displaymath}
  \Omega_X^\lambda(\End (\Oc^m_X))(U) = \{ \Gamma \in \Omega_X(\End
  (\Oc_X^m))(U) \ \vert \  \Gamma (T_X(I_{\lambda}))(U) \subset I_{\lambda}\End (\Oc_X^m)(U)\}.
\end{displaymath}
The inverse image $p^!_\lambda(M)$ has the connection
\begin{displaymath}
  \nabla_\lambda : \Oc_{\hat C_\lambda}^m \to \Omega_{\hat
    C_\lambda}(\Oc_{\hat C_\lambda}^m) 
\end{displaymath}
where $\nabla_\lambda = d+ \Gamma_\lambda $; the connection matrix
$\Gamma_\lambda$ is given by $\Gamma$ as
$\Gamma_\lambda(\partial_\lambda) (a\otimes m) = a \otimes
\Gamma(\tilde \partial_\lambda)m $, where $a\in \Oc_{\hat C_\lambda}$,
$m\in \Oc_{\hat C_\lambda}^m $,
$\partial_\lambda \in T_{\hat C_\lambda}$, and a lift
$\tilde \partial_\lambda \in T_X(I_\lambda)$ of $\partial_\lambda$
exists since $U$ is smooth. Since $p^!_\lambda(M)$ is trivial there
exists $\phi_\lambda \in \Aut(\Oc_{\hat C_\lambda}^m)$ such that
\begin{displaymath}
  \Gamma_\lambda = \dlog(\phi_\lambda).
\end{displaymath}
There exists $\tilde \phi_\lambda \in \Aut(\Oc_U^m)$ that restricts to
$\phi_\lambda$. It follows that
$\Gamma = \dlog (\phi_\lambda) + \Gamma^\lambda$, where
$\Gamma^\lambda\in \Omega_X^\lambda(\End(\Oc_X^m))(U)$,  hence
  \begin{displaymath}
    \Gamma \in \bigcap_{\lambda \in \Lambda} \left(\dlog \Aut_{\Oc_X}(M)(U) + \Omega^{\lambda}_X(
    \End_{\Oc_X}(M))(U)\right).
\end{displaymath}
For any $\lambda$ in $\Lambda$ there is the exact sequence of sheaves
on $U$
  \begin{align*}
    0\to& \dlog (\Aut_{\Oc_U}(\Oc_X^m)) \to \bigcap_{\lambda\in \Lambda} (\dlog
    \Aut_{\Oc_U}(\Oc_U^m) + \Omega^{C_\lambda}_U(
    \End_{\Oc_U}(\Oc_U^m)) ) \\ &\xrightarrow{\psi_\lambda}\frac {\Omega^{C_\lambda}_X(
                                  \End_{\Oc_U}(\Oc_U^m))} {\dlog \Aut_{\Oc_U}(\Oc_U^m)},
  \end{align*}
where   the image
\begin{displaymath}
  \Imo(\psi_\lambda) = \frac{\bigcap_{\lambda \in \Lambda} \Omega^{C_\lambda}_X(
    \End_{\Oc_U}(\Oc_U^m)) } {\dlog
  \Aut_{\Oc_U}(\Oc_U^m)}.
\end{displaymath}
Since $\{C_\lambda\}_{\lambda \in \Lambda}$ is a differential covering
there exists a dense subset $X_0$ of $X$ such that if $x$ is a point
in $X_0$, then
\begin{displaymath}
  T_{X,x}= \sum_{\lambda \in \Lambda_x} T_{X,x}(I_\lambda).
\end{displaymath}
Thus if $\partial$ is a section of $T_X$ that is defined at a point
$x\in X_0\cap U$, then
$\partial_x = \sum_{\lambda \in \Lambda_x} \delta_{\lambda,x}$, where
$\delta_{\lambda,x} \in T_{X,x}(I_\lambda)$. Applying $\partial_x$ to
$\Imo(\psi_\lambda)_x$, we get
\begin{align*}&\\
&  \frac{\bigcap_{\lambda \in \Lambda_x} <\Omega^{C_\lambda}_{X,x}(
                 \End_{\Oc_U}(\Oc_U^m)),\partial_x> } {<\dlog
                 \Aut_{\Oc_U}(\Oc_U^m)_x, \partial_x>}\subset \frac{\sum_{\lambda \in \Lambda_x} (\bigcap_{\lambda \in \Lambda_x} <\Omega^{C_\lambda}_{X,x}(
                 \End_{\Oc_U}(\Oc_U^m)_x,\delta_{\lambda,x}>) } {<\dlog
                 \Aut_{\Oc_U}(\Oc_U^m)_x, \partial_x>}\\[2ex] &
                                                           \subset \frac{\bigcap_{\lambda \in \Lambda_x}
                                                           I_\lambda \End_{\Oc_U}(\Oc_U^m)_x } {<\dlog
                                                           \Aut_{\Oc_U}(\Oc_U^m)_x, \partial_x>}\subset    \frac{ \mf_x \End_{\Oc_U}(\Oc_U^m)_x} {<\dlog
                                                           \Aut_{\Oc_U}(\Oc_U^m)_x,
                                                                \partial_x>}.\\
  &
\end{align*}
Since $X_0\cap U \subset U$ is dense and $\Omega_X(\End(\Oc_X^m))$ is
coherent and torsion free, we get that 
\begin{displaymath}
  \Gamma \in  \dlog \Aut_{\Oc_U}(\Oc_U^m) + \bigcap_{x\in
    X_0\cap U} \mf_x \Omega_X(\End(\Oc_X^m))= \dlog \Aut_{\Oc_U}(\Oc_U^m).
\end{displaymath}
This completes the proof that $M$ is trivial.
\end{pfof}

\section{The decomposition of $\pi_+(\Oc_X)$}\label{decomp-O}
Perhaps the most central problem is to decompose the monomial
$\Dc_Y$-module $N=\pi_+(\Oc_X)= \Oc_Y \oplus \Tc_\pi$ for a finite
surjective map of smooth varieties $\pi: X\to Y$. Although such a
decomposition is determined already at the generic point of $Y$,
eeping track of the ramification of $\pi$ is important for actually
constructing submodules of $N$. Still, explicit generators of the
simple constituents of $\Tc_\pi$ are hard to get, which is one reason
for studying a certain canonical filtration $\{N_i\}$ of $N$. The idea
is to use a natural stratification $(\{X_{ij}\},\{Y_i\})$ of $\pi$ so
that the submodule $N_i\subset N$ is defined by requiring the
vanishing of local cohomology along the stratum $Y_i$. For this
purpose {\it totally ramified} maps are important, where the residue
field extension at the generic point of the deepest stata of the
ramification locus $B_\pi$ is trivial, since then it turns out that
the vanishing trace module $\Tc_\pi$ has vanishing local cohomology
along the deepest stratum in $\{Y_i\}$ \Th{\ref{vanishing}}. Using
this observation we reach the main result of this section
\Th{\ref{cor-can-filt}}, namely that generators of $N_i$ are
determined by vanishing trace conditions over morphisms $X/Z_{ij}$
that are totally ramified along $X_{ij}$ and fitting in a
factorisations $X\to Z_{ij} \to Y$ so that $Z_{ij}/Y$ is étale over
$Y_j$. The existence of such factorizations may be of independent
interest \Th{\ref{factor-theorem}}.

The Galois situation is of course interesting, so that $Y= X^G$ for a
finite group $G$ and $\pi $ is the invariant map. The problem of
finding generators of the simple constituents appearing in
\Theorem{\ref{galois-direct}} corresponds to a branching problem for
representations of groups, but the goal here is to geometrically
construct the simples without relying on group theory, to eventually
understand also non-Galois maps. This problem was resolved in
\cite{bogvad-kallstrom} when $G$ is the (generalized) symmetric group
acting an a finite-dimensional vector space $V$ and
$X= \Spec \So(V) $, where Young or Specht polynomials form generators
of the simples of $N$, and these polynomials are naturally constructed
from the canonical stratification of $\pi$.
In \Section{\ref{complex-refl-section}} this will be taken a step
further.

\subsection{Total ramification}
  Let $\pi: X\to Y$ be a morphism of smooth varieties, $F$ be a closed
  subset of $Y$  and put $X_F= \pi^{-1}(F)$, so we have the diagram 
 \begin{equation}\label{tot-diagam}
   \xymatrix{
     X_F \ar[r] ^{\tilde i}\ar[d]^{p}  
     & X\ar[d]^\pi\\
     F \ar[r]^i & Y.
   }
  \end{equation}
  The following  base change property for local cohomology is surely
  well-known\footnote{\Theorem{\ref{loc-cohom-dir}} is stated without
    proof in \cite{bjork:analD}*{Th. 2.5.28}.}.
\begin{theorem}\label{loc-cohom-dir}(Independence of base)
  \begin{displaymath}
    \pi_+ \RG_{X_F} = \RG_{F} \pi_+
  \end{displaymath}
as functors on the category of quasi-coherent $\Dc_X$-modules.
\end{theorem}
\begin{proof} Let $j: Y\setminus F \to X$ and
  $\tilde j : X\setminus X_F \to X$ be the open inclusions. Since
  $\supp \pi_+ \RG_{X_F} \subset F $, so that
  $j^! \pi_+ \RG_{X_F} =0 $, applying the distinguished triangle
  $\RG_{F} \to \id \to j_+ j^! \xrightarrow{+1} $ to the functor  $\pi_+
  \RG_{X_F}$ gives
  \begin{displaymath} \RG_{F} \pi_+ \RG_{X_F} = \pi_+
 \RG_{X_F}.
\end{displaymath} This implies, after applying $\RG_{F} \pi_+ $
to the distinguished triangle
\begin{displaymath}
  \RG_{X_F} \to \id \to \tilde j_+ \tilde j^! \xrightarrow{+1}
\end{displaymath} that one gets the triangle
\begin{displaymath}
  \pi_+ \RG_{X_F} \to \RG_{F} \pi_+ \to \RG_{F} \pi_+
  \tilde j_+ \tilde j^! \xrightarrow{+1}
\end{displaymath} Finally, for the last vertex we have, since $\pi_+
\tilde j_+ = j_+ \pi_{0,+} $,
\begin{displaymath}
  \RG_{F} \pi_+ \tilde j_+ \tilde j^! = \RG_{F} j_+
  \pi_{0,+} \tilde j^! =0
\end{displaymath} where in the last step we use $ \RG_{F} j_+
=0$. This completes the proof.
\end{proof}
It is basic to our characterization of the canonical filtration of
$\pi_+(\Oc_X)$  - to be studied below- that the kernel of the trace
morphism behaves well with respect to base change. Let $M$ be a
coherent $\Dc_Y$-module and define $\Tc_{\pi}(M)$ by the distinguished
triangle
\begin{displaymath}\tag{*}
 \Tc_\pi(M) \to \pi_+\pi^!(M)\xrightarrow{\Tr } M \xrightarrow{+1}.
\end{displaymath}
Assuming $p: X_F \to F$ is a morphism of smooth varieties we  also
have the  triangle
\begin{displaymath}
  \Tc_p(i^!(M)) \to p_+p^!i^!(M) \to i^!(M) \xrightarrow{+1}.
\end{displaymath}

\begin{definition}\label{def-tot-ram}
Let $x$ be a point of $x$ and put $y=\pi(x)$, and let $F$ be a closed
subset of $Y$. Then $\pi$ is totally ramified:
\begin{enumerate}
\item at $x$ if $k_{Y,\pi(x)}= k_{X,x}$.
\item at $y$ if $\pi^{-1}(y)=\{x\}$ and $\pi $ is
  totally ramified at $x$.
\item along a closed set $F$ in $Y$ if in the diagram
  (\ref{tot-diagam}) $p$ is an isomorphism.
\end{enumerate}
Say also that $\pi$ is totally ramified if $D_\pi\neq \emptyset$ and
it is totally ramified along each deepest stratum of the discriminant
locus $D_\pi$.
\end{definition}
\begin{remark}
  If the closure $x^-$ of $x$ is normal, by \Lemma{\ref{ram-lemma},(2)} below,
  the following are equivalent:
  \begin{enumerate}
  \item $\pi$ is totally ramified at $x$.
\item  $\pi$ is totally ramified along $x^-$.
  \end{enumerate}

\end{remark}
Clearly, if $\pi$ is totally ramified along $F$, then it is totally
ramified along each closed subset of $F$; in particular $\pi$ is
totally ramified along the deepest (smooth) strata of $F$.
\begin{theorem}\label{vanishing} Keep the notation in
  \Theorem{\ref{loc-cohom-dir}} and assume that $F$ and $X_F$ are
  smooth.
  \begin{enumerate}
  \item    $i^!(\Tc_\pi(M))= \Tc_p(i^!(M))$.
  \item If $\pi$ is totally ramified along $F$, then
    $\RG_F(\Tc_\pi(M))=0$.
  \end{enumerate}

\end{theorem}
\begin{proof}
By Kashiwara's theorem
$i_+i^! (M)= \RG_F(M)$ and $\tilde i_+ \tilde i^! (\pi^!(M)) =
\RG_{X_F} \pi^!(M )$. Therefore by \Theorem{\ref{loc-cohom-dir}}
\begin{align*}
  i_+ i^! \pi_+ \pi^!(M)&= \RG_F \pi_+\pi^!(M)= \pi_+
  \RG_{X_F}\pi^!(M)\\ &=  \pi_+  \tilde i_+ \tilde i^! \pi^!(M) =  i_+
  p_+p^!  i^!(M).
\end{align*}
Therefore \thetag{*} gives the distinguished triangle
\begin{displaymath}
  i_+i^!(\Tc_\pi(M)) \to i_+ p_+p^! i^!(M) \to i_+ i^! (M) \xrightarrow{+1}.
\end{displaymath}
By Kashiwara's theorem we can erase  $i_+$ 
\begin{displaymath}
i^!(\Tc_\pi(M)) \to  p_+p^! i^!(M) \xrightarrow{\psi}  i^! (M) \xrightarrow{+1},
\end{displaymath}
where now $\psi$ is the trace morphism $p_+p^! i^!(M) \to i^! (M)$. If
$p$ is an isomorphism it follows that $i^!(\Tc_\pi(M)) =0$, and hence
$\RG_F (\Tc_\pi(M))=0$.
  \end{proof}
  \begin{example}
    Let $A\to B$ be a finite totally ramified morphism of smooth
    $k$-algebras, i.e. $k_A = k_B$, and put
    $\Tc = \Ker (\Tr : \pi_+(B)\to A)$. By \Theorem{\ref{vanishing}}
    it follows that $\RG_{\mf_A}(\Tc)=0$.
  \end{example}
  One can consider the maximal connection $N_c$ in any holonomic
  $\Dc_X$-module $N$ and define the ``coherent'' direct image functor
\begin{displaymath}
  \pi_+^c : \coh(\Dc_X)\to \Con (Y), \quad N\mapsto  \pi_+^c(N)=\pi_+(N)_c.
\end{displaymath}
When $M$ is a connection and $\pi$ is étale, then
$ \pi^c_+\pi^!(M)
 = \pi_+\pi^!(M)$ but in general it can be
difficult to determine $\pi^c_+\pi^!(M)$ when $\pi$ is  ramified. In the totally ramified case, on the other hand, we
have:
  \begin{theorem}\label{coh-connection} Let $\pi$ be a finite totally
    ramified morphism of smooth varieties.
    \begin{enumerate}
    \item If $M$ is a holonomic module such that $\pi^!(M)$ is
      semisimple (see \Theorem{\ref{semisimple-inv}}), then
      $\pi^c_+\pi^!(M) = M_c$.
      \item  The functor
\begin{displaymath}
  \pi^!: \Con(Y)\to \Con(X)
\end{displaymath}
is fully faithful so that if $M$ is connection on $Y$, then
$ \pi_+^c \circ \pi^!(M) \cong M$.
    \end{enumerate}
  \end{theorem}
  The functor $\pi_+^c$ is certainly not faithful. For example, if the
  extension of fraction fields $k(X)/k(Y)$ is Galois, $N$ is a simple
  connection with trivial inertia group, and $\pi$ is ramified, then
  $\pi_+(N)$ is a simple $\Dc_Y$-module \Prop{\ref{simpledirect}} such
  that $\pi^c_+(N)=0$ (see also \Proposition{\ref{etale-finite}}).
  Concretely, consider the cyclic extension $A= \Cb[y]\to B=\Cb[x]$,
  $y=x^n$, and the connection $E=\Dc_B e^x$ (see (\ref{exp-mod-sec})). The
  Galois group is $C_n=<\zeta>$, where $\zeta$ is a primitive root of
  unity, and we have
  $\partial(x)- \partial(\zeta^k x)= (1-\zeta^k)\partial (x)\notin
  \dlog (B)$,
  $k=1, \ldots, n-1$, so that the inertia group
  $G_E = \{e\}\subset C_n$. Then $\pi_+(E)$ is a simple $\Dc_A$-module
  that is not finite over $A$.

  Moreover, $\pi^!$ is not fully faithful when extended to a functor
  between the categories of holonomic $\Dc$-modules.

\begin{proof} 
  (1): By \Theorem{\ref{vanishing}}, $\RG_F(\Tc_\pi(M))=0$, hence
  since $\pi_+\pi^!(M)$ is semisimple \Th{\ref{decomposition-thm}},
  hence $\Tc_\pi(M)$ is semisimple, it follows that
  $\RG_F((\Tc_\pi(M))_c)=0$, and therefore by Grothendieck's
  non-vanishing theorem, $\Tc_\pi(M)_c=0$, so by \thetag{*},
  $\pi^c_+\pi^!(M)=M_c$. (2): First recall that $\pi_+$ is exact
  \Prop{\ref{flat-module}}, hence $\pi_+^c$ is exact; moreover $\pi^!$
  is exact on $\Con(Y)$; therefore $\pi_+^c\pi^!$ is exact. Hence,
  connections being of finite length as $\Dc$-modules, one can assume
  that $N$ and $M$ are simple connections. Then
  $\pi_+^c\pi^!(M)\cong M$ follows from (1) (and
  \Theorem{\ref{semisimple-inv}}).
\end{proof}
\subsection{Canonical stratifications and filtrations}
A {\it stratification} of a smooth variety $X$ is a finite collection
$\{X_i\}_{i\in I }$ of mutually disjoint locally closed smooth
subvarieties $X_i$ such that $X= \cup_ {i \in I } X_i$, and the
closure $\bar X_i = \cup_{j\leq i} X_j$ is a union of strata. Here we
order the index set $I$ by reverse domination, so that $j \geq i$ when
$X_j$ belongs to the closure $ \bar X_i$ of $X_i$; this defines a
partial order $(>, I )$ such that the canonical filtration that will
be discussed below is increasing. Since $X$ is irreducible there
exists a unique minimal index $i_s$ such that $X_{i_s}$ is dense in
$X$, and putting $X_{i_f}= \emptyset$ we have $i_s < i < i_f$, for all
$i\in I$. We say that $i_s$ and $i_f$ indexes the initial and final
stratas of $\{X_i\}_{\in I}$.

A stratification $(\{X_{i}\}_{i\in I_j}, \{Y_j\}_{j\in J})$ of a morphism
$\pi : X \to Y$ of varieties is a stratification $\{Y_j\}_{j\in J}$ of
$Y$ and a stratification $\{X_{i}\}_{j\in J, i\in I_j }$ of $X$ such
that
\begin{displaymath}
  \pi^{-1}(Y_j) = \bigcup_{i \in I_j} X_{i}.
\end{displaymath}
There exists a canonical stratification of a finite morphism
$\pi: X\to Y $ of schemes of finite type such that the restriction to
the stratas of $X$ are étale; it will be the coarsest such
stratification in a natural sense. The construction is included since
it has not been spelled out in the literature in sufficient detail for
our needs.
\begin{remark}
  In singularity theory one often requires finer stratifications
  satisfying contact conditions, such as the Whitney stratification.
  We will however work only with restrictions to étale strata in an
  {\it algebraically defined} canonical stratification in order to
  decompose $\pi_+(\Oc_X)$.
\end{remark}
We will stepwise add new locally closed smooth strata to $X$ and $Y$
so that strata added in one step are not dominated by any of the other
new ones, but instead are dominated by some stratum in the previous
step.

First make a base change $ X'\to Y^{r}$ of $\pi$ with respect to the
reduced subscheme $ Y^{r}\to Y$, let $ X^{r} $ be the reduced scheme
of $ X'$, and $\pi^{(r)}: X^{r}\to X' \to Y^{r}$ be the composed map.
Since $\pi^{(r)}$ is a finite map of reduced schemes of finite type,
defined over a field of characteristic $0$, there exists a maximal
open subset $Y^{(0)} \to Y^{r}$ such that the base change
$X^{(0)} \to Y^{(0)}$ is étale. Let $\{Y_{j}\}_{j\in J_0 }$ be the
connected components of $Y^{(0)}$, so that each base change
$X^{(0)}_{j}\to Y_{j} $ of $X^{r} \to Y^{r} $ is étale. Let
$\{X_{i}\}_{i\in I_{j} }$ be the connected components of
$X^{(0)}_{j}$, so that $\cup_{i\in I_j} X_{i}= X^{(0)}_{j}$,
$j\in J_0$.

Put $\tilde Y_{1}= Y^r\setminus Y^{(0)} $ and let
$\pi_1 : \tilde X_1 \to \tilde Y_1$ be the base change of $\pi^{(r)}$
over $\tilde Y_1 \to Y^{r}$. We can now replace $\pi$ by $\pi_1$ and
proceed inductively. Thus having defined $\tilde X_{l-1}$ and
$\tilde Y_{l-1}$, and $\pi_{l-1} : \tilde X_{l-1} \to \tilde Y_{l-1} $, we
get $\tilde Y_l = \tilde Y_{l-1}^{r} \setminus Y_{l-1}^{(0)} $,
$\pi_l : \tilde X_l \to \tilde Y_l$, which is the base change of
$\pi_{l-1}^{(r)}$ over $\tilde Y_{l}\to \tilde (Y^{l-1})^{r}$. We
get for $l =1,2,...$
\begin{align*}
  Y^{(l)}&= \tilde  Y_{l}^{(0)},\quad X^{(l)}\to Y^{(l)}, \quad   Y^{(l)} = \bigcup_{ j\in J_l}
  Y_{j}, \quad X^{(l)}_j = \bigcup_{i\in I_j}X_{i}\to Y_j
\end{align*}
where $Y_{j}$ are the connected components of $Y^{(l)}$, the base
change $X^{(l)}\to Y^{(l)}$ over $Y^{(l)}\to Y^r$ is étale and $X_i$
are the connected components of $X^{(l)}_j$. By construction, the
restriction $X_{i}\to Y_{j}$, $j\in J_l$, $i \in I_j$, is an étale map
of connected smooth varieties. Since schemes of finite type only have
finitely many connected components and since in each step the
dimensions of the new strata are of strictly lower dimension, we end
up in finite index sets $J= \cup_{l=0,1,...}J_l $ and
$I=\cup_{l=0,1,...} \cup_{j\in J_l}I_j $ such that,
$\{X_i, Y_j, i\in I_j, j\in J\}$ forms a stratification of $\pi$, where
the index sets $ (I= \cup I_j, <)$ and $(J, <)$ are partially ordered
by reversed specialization of the corresponding strata. We summarize
the above construction as follows:
 
\begin{proposition}\label{smoothstrat}
  There exists a stratification
  $(\{X_{i}\}_{i\in I_j}, \{Y_j\}_{j\in J})$ of $\pi$ such that each
  restriction $X_{i}\to Y_j $ is étale and such that if
  $(\{X_{k}\}_{k\in I'}, \{Y_l\}_{l\in J'})$ is another
  stratification of $\pi$ such that each restriction
  $X'_{k}\to Y'_l $ is étale, then $(\{X'_{k}\}, \{Y'_{l}\})$ is a
  refinement of $(\{X_{ij}, Y_k\})$, i.e.  each stratum $X_{i}$ is a union
  of strata $X'_{k}$.
\end{proposition} 
We call the family of strata $(\{X_{i}, Y_j\}_{i\in I_j, j\in J})$ in
\Proposition{\ref{smoothstrat}} the {\it canonical stratification} of
$\pi$.

Our main interest is in filtrations of $\Dc_Y$-modules $M$ defined by
requiring the vanishing of local cohomology along successively deeper
strata $Y_i$ in the canonical stratification. Such a filtration,
however, does not exist for general $M$. For if $M_1, M_2 $ are
submodules of $M$ and $S$ a closed subset of $X$ such that
$\RG_S (M_1)=\RG_S(M_2)=0$ it follows that for the sum
$M_{12}= M_1 + M_2$ the complex $\RG_S(M_{12}) $ is a translation of
$\RG_S(M_1 \cap M_2)$ one step to the left, and this complex need not
be $0$ (in the derived category). Therefore, when $M$ is not
semisimple, there may exist more than one maximal submodule with
vanishing local cohomology along $S$.

\begin{lemma}
  Let $M$ be a coherent semisimple $\Dc$-module and $S$ be a locally
  closed subset of $X$. Then there exists a unique  maximal coherent submodule
  $M_S \subset M$ such that $\RG_S (M_S)=0$.
\end{lemma}
\begin{proof}
Let $\Lambda_S$ be the set of coherent submodules
$M'$ of $M$ such that $\RG_{S}(M')=0$, and put
\begin{displaymath}
  M_S = \sum_{M'\in \Lambda_S } M'.
\end{displaymath} 
The modules $M'$ are locally generated by their sections (over a fixed
affine neighbourhood) and $M$ is coherent, so locally only a finite
number of terms $M'$ contribute to $M_S$ since $M$ is semisimple and
coherent; therefore $M_S$ is coherent. Also, since $M$ is semisimple,
it follows that $\RG_{S}(M_S)=0$. By construction $M_S$ is the unique
maximal coherent submodule with vanishing local cohomology along $S$.
\end{proof}

We can therefore make the following definition:
\begin{definition}\label{van-def} For a  coherent semisimple
  $\Dc_X$-module $M$ and a locally closed subset $S$ of $X$, let $M_S$
  be the maximal coherent submodule of $M$ whose local cohomology
  $\RG_S(M_S)=0$.
\end{definition}

\begin{remark}Assume that $S$ is closed in $X$, put $U= X\setminus S$
  and let $j: U \to X$ be the open inclusion. If $M$ is simple and
  $M_S=M$, then $M=j_+j^!(M)$ (see the appendix in \Section {
    \ref{appendix}}), but the converse does not hold, so that even if
  $M=j_+j^!(M)$ is simple, it can happen that $M_S=0$ (if
  $\codim_X S \geq 2$, take $M= \Oc_X$). The $\Dc_{\Cb^1}$-module
  $M= \Cb[x,1/x]x^\alpha$ is simple when $\alpha\not \in \Zb$ and
  $M= M_{\{0\}}$.
\end{remark}

Let $S_1$ be a locally closed subset of the closure $\bar S$ of $S$.
Then in general $\RG_{S_1} (M_S)\neq 0$, even if $M$ is simple
holonomic and $\RG_S(M_S)=0$. If $M$ is simple then
$(M_S)_{S_1}= M_{S_1}$, but in general $M_S \not \subset M_{S_1}$.
\begin{example}\label{ex-nonvancoh} Let $M$ be a simple holonomic $\Dc_X$-module, torsion
  free as $\Oc_X$-module, and $\bar S_1 \subset \bar S\subset X$ be
  closed algebraic subsets such that
  $\codim_{\bar S} \bar S_1 \geq 1$. Let
  $j: X \setminus \bar S_1 \to X $ be the locally closed inclusion
  morphism. If $j^! \RG_{\bar S}(M)=0$, then
  $\RG_{\bar S} (M)= \RG_{\bar S_1}(M)$, and the latter
  complex need not be isomorphic to $0$ even if $M$ is simple and
  holonomic. For an example, let $X= \Cb^2$, $\bar S_1= \{0\}$,
  $\bar S= V(xy(x+y))$, and put
  $\tilde M= \Dc_{\Cb^2} x^{\beta_1} y ^{\beta_2} (x+y)^{\beta_3} $,
  where $\beta_i \in \Qb \setminus \Zb$ and
  $\beta _1 + \beta_2 + \beta_3 \in \Zb$. The module $\tilde M$ is of
  length $2$ and contains a unique simple submodule $M$, torsion free
  over $\Oc_X$, such that $\supp \tilde M/M =\bar S_1$ (see
  \cite{abebaw-bogvad:arkiv}*{Th 1.3}). Since $\tilde M$ has vanishing
  local cohomology along $S= \bar S\setminus \bar S_1$ it follows that
  $\RG_{\bar S}(M) = R \Gamma_{\bar S_1}(M)\neq 0$, where the
  non-vanishing follows since the minimal extension $M$ of
  $j^!(M)= j^!(\tilde M)$ is a proper submodule of
  $j_+j^!(M)= j_+j^!(\tilde M)= \tilde M $. Here $M_{S} = M$ and
  $M_{S_1}=0$.
\end{example}

Given a semisimple $\Dc_Y$-module $N$ and the canonical stratification
$\{Y_i\}_{i \in J}$ of $Y$ we define the modules $N_i = N_{X_{j}}$ as
in \Definition{\ref{van-def}} ($N_{j_f}= N$ and $N_{j_s} = 0$),
resulting in a family of submodules $\{N_j\}_{j \in J}$ which we call
the ($\pi$-) canonical submodules of $N$. As noted above, in general
$\{N_j\}_{j\in J}$ need not form a filtration of $N$ with respect to
the partially ordered set $\{J, <\}$, i.e. $j_1 < j_2$ does {\it not}
imply $N_{j_1} \subset N_{j_2}$.\footnote{Of course, replacing $Y_j$
  by its closure $\bar Y_j $ in the definition of $N_j$ always results
  in a filtration.} It is perhaps therefore all the more suprising,
shall we see, that nevertheless the semisimple module
$N= \pi_+(\Oc_X)$ satisfies such inclusions.

\subsection{Maximal étale and minimal totally ramified factorizations}

Let
\begin{displaymath}
  (A,\mf_A, k_A)\to (B,\mf_B, k_B)
\end{displaymath}
be an injective finite homomorphism of local noetherian domains, so we
can identify $A$ with a subring of $B$. Then $A\subset B $ is totally
ramified if $\Spec B \to \Spec A$ is totally ramified, as in
\Definition{\ref{def-tot-ram}}, i.e. the induced map of residue fields
is an isomorphism, $k_A\cong k_B$. We are interested in certain
factorizations by inclusions of local rings
\begin{displaymath}
  (A,\mf_A, k_A) \subset   (C,\mf_C, k_C) \subset  (B,\mf_B, k_B).
\end{displaymath}
Next lemma is well-known.
\begin{lemma}\label{ram-et-lemma}
  \begin{enumerate}
  \item $B/A$ is étale and $C/A$ is unramified  $\Leftrightarrow$  $B/C$ and $C/A$ are étale.
     \item $B/A$ is  totally ramified $\Leftrightarrow$  $B/C$
       and $C/A$ are totally ramified.
     \item If $B/A$ is finite étale and totally ramified, then $A=B$.
  \end{enumerate}
\end{lemma}
 \begin{remark}
   \begin{enumerate}
   \item If a homomorphism of local rings $A\subset B$ is totally
     ramified and $B/C/A$ is a factorization where $C/A$ is étale,
     then $A=C$.
   \item If $A\to B$ is étale and $B/C/A$ is a factorization where
     $B/C$ is totally ramified, then $C=B$.
   \end{enumerate}
 \end{remark}
 The properties that $B/A$ be étale and totally ramified,
 respectively, are preserved in opposite directions with respect to
 generization and specialization.
 \begin{lemma}\label{ram-lemma} Let $P$ be a prime ideal in $B$ and put $Q= A\cap P$. 
   \begin{enumerate}
   \item If $A\to B$  is étale, then $A_Q \to B_P$ is étale.
   \item If $A_Q \to B_P$ is totally ramified and $A/Q$ is normal,
     then $A/Q= B/P$ and therefore $A\to B$ is totally ramified.
   \end{enumerate}
 \end{lemma}
 \begin{proof}
   (1) is well-known. (2): Since $k_Q = A_Q/QA_Q = k_P = B_P/PB_P$ it
   follows that $A/Q \to B/P$ is a finite birational map, and since
   $A/Q$ in integrally closed in its fraction field, $A/Q = B/P$.
 \end{proof}

 Say that an inclusion of local rings $A\subset C_t \subset B$ is a
 {\it minimal totally ramified factorization} if $C_t\subset B$ is
 totally ramified and if $C_t'$ is another totally ramified
 factorization $A\subset C_t' \subset B$, then there exists an
 isomorphism $C_t \to C_t''$, where $C_t''$ fits in a
 factorization $A\subset C_t'' \subset C_t' \subset B$, making the natural
 diagrams commute.

A {\it maximal étale factorization} is a local $k$-subalgebra $C_e$ of
$B$ that contains $A$, $A\subset C_e \subset B$, such that
$A\subset C_e$ is étale, and if $C_e'$ is another such étale
$A$-algebra, $A\subset C_e' \subset B$, then there exists an
isomorphism $C_e'' \to C_e$, where $C_e''$ fits in a factorization
$A \subset C_e'' \subset C_e'\subset B$, making the natural diagrams
commute.  

\begin{theorem}\label{factor}  Assume that $A$ is regular and $B$ is normal.
  There exists a maximal étale factorization $A\subset C_e \subset B$.
  The following are equivalent for a factorization
  $A\subset C \subset B$:
\begin{enumerate}
\item  $B/C/A$ is a maximal étale factorization.
\item $B/C/A$ is a minimal totally ramified factorization.
\end{enumerate}
The factorization is unique up to a choice of Galois cover of the
field extension $k(B)/k(A)$.
\end{theorem}
Thus in this situation $C_e\cong C_t$ and the maximal étale (minimal
totally ramified) factorization $A\subset C_e \subset B$ is unique up
to (nonunique) isomorphism. We often abuse the terminology and say
that the local ring $C$ itself is a minimal totally ramified (maximal
étale) factorization when its position in $A\to B$ is given from the
context.

\Theorem{\ref{factor}} is proven similarly to the following  global
version:
\begin{theorem}\label{factor-theorem}
  Let $\pi : X\to Y$ be a finite morphism of normal varieties (regarded as
  schemes) and $S$ be a  finite subset of $X$.
  Then there exists a factorization,
  $\pi= q \circ p$
 \begin{displaymath}
    X  \xrightarrow{p} Z\xrightarrow{q} Y,
  \end{displaymath}
  such that $q$ is maximally étale at all points $z_1$ of height $1$
  that specialize to $z= p(x)$. This factorization is unique up to a
  choice of Galois cover of the field extension $k(X)/k(Y)$. If $Y$ is
  smooth at $\pi(x)$, then $q$ is maximally étale at $z$ and $Z$ is
  smooth at $z$. The map $p$ is totally ramified along $S$, and if
  $X/Y$ is Galois, then $p$ is minimally totally ramified along
  $ p(S)$.
\end{theorem}
If $k(X)/k(Y)$ is not Galois, then $p$ need not be totally ramified.
Example: Consider a finite morphism of curves $\pi:X\to Y$ over an
algebraically closed field. If there exists a point $y$ in $Y$ such
$X_y\cap B_\pi $ is a proper subset of the full fibre $X_y$, then
$\pi$ it totally ramified along $X_y$ but not along $\{y\}$.

\begin{remarks}
  \begin{enumerate}
  \item If the closure $x^-$ of $x\in S$ is normal, then $p$ is minimally
    totally ramified and $q$ is maximally étale for all points in $x^-$
    \Lem{\ref{ram-lemma},(2)}. Therefore if $S$ is a closed normal subvariety of
    $X$ and $\pi= q\circ p$ is a factorization such that $p$ is minimally
    totally ramified and $q$ is maximally étale at all generic points of $S$,
    then the same holds at all points of $S$.
  \item Notice that if $B/C/A$ is a minimal (maximal) totally ramified
    (étale) factorization and $P$ is a prime ideal of $B$,
    $P_C= C\cap P, Q= P\cap A $, then we get maps
    $B_P\xrightarrow{p_P} C_{P_C}\xrightarrow{q_P} A_Q$ where $p_P$ and
    $q_P$ are totally ramified and étale, respectively, but they need
    not be a minimal (maximal) totally ramified (étale) factorization
    of $B_P/A_Q$.
  \item When $X_a$ and $Y_a$ are complex manifolds the existence of
    maximal étale factorization can be argued for topologically. The
    index of the image $\pi_*(\pi_1(X_a))\subset \pi_1(Y_a)$ of the
    fundamental group $\pi_1(X_a)$ in $\pi_1(Y_a)$ is finite. Then it
    is well-known that there exists an étale map $q: Z\to Y$ such that
    $q_*(\pi_1(Z_a))= f_*(\pi_1(X_a))$ and also that there exists a
    lifting $p: X_a\to Z_a $ such that $q\circ p = \pi$.
  \item The proof of \Theorem{\ref{factor}} combined with an
    application of G.A.G.A. gives a straightforward proof of
    \cite{hwang-kebekus-peternell}*{Th. 1.2}, which states: Given a
    finite holomorphic map of projective normal complex varieties
    $\pi: X\to Y$, there exists a factorization $\pi= q\circ p$, where
    $q: Z\to Y$ is étale at points of height $\leq 1$. Moreover, up to
    isomorphism there exists a unique maximal such map $q$.
  \end{enumerate}
\end{remarks}
Recall that for a Galois morphism $\pi: X\to Y$ with Galois group $G$,
the decomposition group $G_x$ of a point $x$ in $X$ is the subgroup of
automorphisms that stabilizes $x$, and its inertia group $I_x$ is the
subgroup of $G_x$ of elements whose induced action on the residue
field $k_{X,x}$ is trivial.
\begin{proof} 
  Let $s: \bar X\to X$ be the map from the integral closure of $X$ in
  the Galois cover of the field extension $k(X)/k(Y)$, $G$ be the
  Galois group of $k(\bar X)/k(Y)$, $H$ the Galois group of
  $k(\bar X)/k(X)$, and put $\bar \pi = \pi \circ r : \bar X \to Y$.
  Let $x_1$ be a point of height $1$ that specialises to $x$. Then
  there exist points $\bar x_1, \bar x$ in $\bar X$,  such that
  $\bar x_1$ specializes to $\bar x $, $s(\bar x)= x$,  and
  $s(\bar x_1)= x_1$. Let $I_{\bar x}$ and $I_{\bar x_1}$ be the
  inertia groups of the points and put
  $ Z^{\bar x}= \bar X^{I_{\bar x}} $,
  $ Z^{\bar x_1} = \bar X ^{I_{\bar x_1}}$. Let
  $ p_{\bar x} : \bar X \to Z^{\bar x}$ and $ p_{\bar x_1} : \bar X \to Z^{\bar
  x_1}$ be
  the corresponding invariant maps; since $I_{\bar x_1}$ is a subgroup
  of $I_{\bar x}$ there exists also a map $r_{\bar x_1} : Z^{\bar x_1} \to
  Z^{\bar x}$
  such that $ p_{\bar x} = r_{\bar x_1} \circ p_{\bar x_1}$. Let $Z^x$ be the integral
  closure of $Y$ in the intersection of the fraction fields of $X$ and
  $ Z^{\bar x}$, regarded as subfields of the fraction field of
  $\bar X$, and let $p_x: X\to Z^x$ and $q_x: Z^x \to Y$ be the natural
  maps, so that $\pi = q_x\circ p_x$, and there exists also a natural map
  $r_{\bar x}: Z^{\bar x}\to Z^x $ such that $r_{\bar x}\circ p_{\bar
    x} = p_{x}\circ s$. We have the following diagram
    \begin{equation}
\label{galois-diagram}
  \begin{tikzcd}
     \bar X \arrow[r,  "p_{\bar x}"] \arrow[d, "s"'] & \bar Z^{\bar
       x}\arrow[d,"r_{\bar x}"]
     \\ X \arrow[r,"p_x"]\arrow[d, "\pi"']&  Z^x\arrow[dl,"q_x"]
     \\ Y &
\end{tikzcd}
  \end{equation}
  The map $\bar X \to Z^x$ is Galois and its Galois group $G_{H,x}$ is
  the the subgroup of $G$ that is generated by $I_{\bar x}$ and $H$.
  Since $I_{\bar x}$ already belongs to the decomposition subgroup
  $G_{\bar x}$ of $G$ it follows that if $D_{\bar x}$ is the
  decomposition subgroup of $G_{H,x}$ and $H_{\bar x}$ is the
  decomposition subgroup of $H$, then
  \begin{displaymath}
k_{X,x}  =      k_{\bar X, \bar x}^{H_{\bar x}}=   k_{\bar X, \bar
  x}^{D_{\bar x}} = k_{Z,z}.
  \end{displaymath}
  Hence $p_x$ is totally ramified at $x$ (but not necessarily at $p_x(x)=z$
  when $Hg \neq gH$ for some $g\in I_{\bar x}$). The maps $p_{\bar x}$ and
  $ p_{\bar x_1}$ however are minimal totally ramified at $\bar z =\bar
  p_{\bar x}(\bar x)$ and
  $\bar z_1= p_{\bar x_1}(\bar x_1)$, respectively. Since $\bar x_1$ and hence also
  $\bar z_1= p_{\bar x_1} (\bar x_1)$ and
  $ q_{\bar x_1} (\bar z_1)= \pi(x_1)= y_1$ are points of height $1$ in
  $\bar X$, $Z^{\bar x_1}$, and $Y$, respectively, and all appearing
  varieties being normal, we get inclusions of discrete valuation
  rings
\begin{displaymath}
  \Oc_{Y, y_1} \subset  \Oc_{Z^{\bar x_1},
    \bar z_1} \subset \Oc_{\bar X, \bar x_1}.
\end{displaymath}
We are therefore in the situation of \cite{serre:corps}*{\S 7, Prop.
  21} and can conclude that the first inclusion is étale. Since
$ \Oc_{Y, y_1} \subset \Oc_{ Z^x, z_1} \subset \Oc_{ Z^{\bar x_1}, \bar z_1}$
and $Z^x $ is normal, it follows that $q_x$ is étale at
$z_1= r(\bar z_1)$. If $Y$ is smooth at $\pi(x)$, then $q_x$ is étale at
$z$ by the purity of branch locus, and then $Z/k$ is smooth at $z$.

There exists therefore for each point $x$ in $S$ a factorization
\begin{displaymath}
  X\xrightarrow{p_x} Z^x \xrightarrow{q_x} Y
\end{displaymath}
such that $\pi= q_x\circ p_x$ where $p_x$ is totally ramified at $x$
(and at $p_x(x)$ if $k(X)/k(Y)$ is Galois, so that $\bar X =X$) and
$q_x$ is étale at all points of height $1$ in $Z^x$ that specialises
to $p_x(x)$. If $S= \{x_1, \ldots, x_r\}$, put
$ Z= Z^{x_1} \times_Y Z^{x_2} \times_Y \ldots \times_Y Z^{x_r} $, and
let
\begin{displaymath}
  X\xrightarrow{p}Z \xrightarrow{q}  Y
\end{displaymath}
be the canonical maps. Then $q$ is maximally étale at all points of
height $1$ that specialises to a point in $p(S)$, and as before, if
$y\in \pi(S)$ is a regular point in $Y$, then $q$ is étale above $y$.
The map $p$ is totally ramified along $S$, and if $k(X)/k(Y)$ is
Galois (and thus $\bar X=X$), then $p$ is minimally totally ramified
at the points in $p(S)\subset Z$.
\end{proof}
\subsection{Étale coverings for étale trivial connections and Galois
  coverings for finite modules.}
Let $Y/k$ be a normal projective variety over an algebraically closed
field of characteristic $0$, and consider finite surjective maps of
normal varieties (assumed connected)
\begin{equation}\label{sur-gal}
  \bar \pi: \bar X \xrightarrow{p}  X \xrightarrow{\pi} Y,
\end{equation}
where the fraction field $\bar L$ if $\bar X$ is Galois over the
fraction field $L$ and $K$ of $X$ and $Y$, respectively. Let $G$ be
the Galois group of $\bar L/K$.

If $M$ is an $L$-trivial connection on $Y$, so that
$L\otimes_KM_K \cong L^m$, where $M_K$ is the generic stalk of $M$,
then it follows that $\pi^!(M)$ is a trivial connection on $X$. Assume
now that $M$ is only a locally free $\Oc_Y$-module such that
$N=\pi^*(M)$ is a trivial locally free $\Oc_X$-module, i.e.
$N\cong \Oc_X\otimes_k V $, where $V=\Gamma(X, N)$. Then $N$ can be
regarded as a connection by declaring that $T_X$ acts trivially on
$V$, and we will see that in fact $M$ itself is a connection such that
$\pi^!(M)=N$ \Th{\ref{riemann}}. In general, if $Y$ had not been
projective, it is not sufficient that $N=\pi^*(M)$ be trivial to
conclude that $M$ is a connection; one would require that $X/Y$ be a
Galois cover and that $N$ moreover be a $\Dc_X[G]$-module; see
(\ref{galois-section}).

It is natural to ask whether the maps $\pi$ (and $\bar \pi$) can be
factorized so that $M$ becomes trivial over an intermediate étale
(Galois) cover.
\begin{theorem} Let $\pi: X\xrightarrow{p} Z \xrightarrow{q} Y$ be a
  factorization of the morphism $\pi$, where $q$ is maximally étale
  (see \Theorem{\ref{factor-theorem}}). For a connection $M$ on $Y$
  the following are equivalent:
  \begin{enumerate}
  \item $\pi^!(M)$ is trivial.
    \item $q^!(M)$ is trivial.
  \end{enumerate}
\end{theorem}
\begin{proof} We only need to prove $(1)\Rightarrow{2}$, and can
  assume that $\pi \neq q$. Let
  $\bar \pi: \bar X \xrightarrow{s} X \to Y$ be a Galois cover of
  $\pi$, so that $\bar X$ is the integral closure of $X$ in the
  fraction field of $\bar X$. There exists a factorization
  $\bar X \xrightarrow{\bar p} \bar Z \xrightarrow{r} Z
  \xrightarrow{q} Y$
  such that $\bar p$ is totally ramified and $r\circ \bar p$ is
  maximally étale and a map $p: X\to Z$ such that
  $p\circ s = r\circ \bar p$ (see proof of
  \Theorem{\ref{factor-theorem}}). By construction $k(Z)$ is the
  invariant field of the group $G$ that is generated by the union of
  the Galois groups of $k(\bar X)/k(\bar Z)$ and $k(\bar X)/k(X)$.
  Putting $\bar M = q^!(M)$ we have for some integer $m$,
\begin{displaymath}
r^!(\bar M) = \bar p^c_+\bar p^!(r^!(\bar M)) =\bar p_+^c(\Oc_X^m) = p^c_+\bar
p^!(\Oc_{\bar Z}^m)= \Oc_{\bar Z}^m,
\end{displaymath}
where the first and last isomorphisms follow from
\Theorem{\ref{coh-connection}}; hence $r^!(\bar M)$ is trivial. Since
also $ p^!(\bar M) =\pi^!(M)  $ is trivial it follows that
$\bar \pi^!(M)$ is a trivial $\Dc_{\bar L}[G]$-module; hence by Galois
descent \Prop{\ref{morita}} $\bar M_{k(Z)}$ is a trivial
$\Dc_{k(Z)}$-module; hence $\bar M$ is trivial.
\end{proof}

Following Nori \cite{nori:fundgrp}, say that a locally free
$\Oc_Y$-module $M$ is {\it finite} if there exists a Galois cover
$q: Z \to Y$ such that $q^*(M)$ is a trivial vector bundle. The
assertion that a vector bundle $M$ is finite if $\pi^*(M)$ is a
trivial vector bundle was proven in
\cite{biswas-dossantos:vector-fund} ($Y$ smooth projective, using the
Tannakian formalism) and \cite{antei-mehta:vector-normal} ($Y$ normal
projective, using the Harder-Narashiman filtration).
\Theorem{\ref{riemann}} (1) contains their result when the
characteristic is $0$, with a more direct proof. The result
\cite{esnault-hai:fund-group}*{Th 2.15} is similar to (2) below, but
the argument here is less involved.

\begin{theorem}\label{riemann} Let $Y/k$
  be a normal proper variety over an algebraically closed
  field $k$ of characteristic $0$, and consider the diagram
  (\ref{sur-gal}).
  \begin{enumerate}
  \item Let $M$ be a locally free and coherent $\Oc_Y$-module, and
    assume that $ \pi^*(M)$ is a trivial $\Oc_{\bar X}$-module. There
    exists a factorization
    \begin{displaymath}
      \bar \pi: \bar X \xrightarrow{p} X \xrightarrow{r}  Z \xrightarrow{q} Y,
    \end{displaymath}
    where all varieties are normal and proper, $q$ is a Galois
    cover, and $q^*(M)$ is trivial.
  \item The category of finite vector bundles that are trivial after
    pulling back by the Galois cover $q: Z\to Y$ is the same as the
    category of $L_1$-étale trivial connections on $Y$, where $L_1$ is
    the fraction field of $Z$.
  \end{enumerate}

\end{theorem}

\begin{remark}\label{rem-stein}
  \begin{enumerate}
  \item If there exist a surjective proper morphism $\pi_1:X_1 \to Y$
    such that $\pi_1^*(M)$ is trivial and $\pi_1= \pi \circ p$ is a
    Stein factorization with $\pi$ finite and $p$ has connected
    fibres, then $\pi^*(M)$ is also trivial.
  \item A simple $\Dc_Y$-module is monomial, i.e.
    $M\subset \pi_+(\Lambda)$ for some $\pi$ and rank $1$ module
    $\Lambda$, if and only if $L\otimes_KM_K\cong \Lambda^m$ as
    $\Dc_L$-module, $m= \dim_K M_K$, and $M$ is a monomial connection
    if and only if $\pi^!(M)$ splits into a sum of rank $1$
    connections. In \Theorem{\ref{riemann}} only étale trivial
    connections are mentioned, but the assertion should be possible to
    generalize so that any covering connection is a submodule of a
    monomial module with respect to an étale map.

  \end{enumerate}\end{remark}

\begin{proof}
  (1): The Galois group of $\bar L/K$ acts on $\bar X$ and
  $\bar \pi^*(M)$, and also on the global sections
  $V= \Gamma(\bar X, \bar \pi^*(M))$. Let $H$ be the kernel of the map
  $G \to \Glo_k(V)$, put $Z= \bar X^H$, and let $p: \bar X \to Z$ be
  the invariant map. Since $\bar X$ is connected, for a closed point
  $\bar x$ mapping to the point $z$ in $Z$,
  $\Gamma (\bar X, \Oc_{\bar X}) = k_{\bar X, \bar x}=k_{Z, z}=k $,
  since $k$ is algebraically closed and $\bar X$ is proper. Since
  $ p^*q^*(M)= \bar \pi^*(M) =\Oc_{\bar X}\otimes_k V$, it follows
  that
\begin{displaymath}\tag{*}
  V=\Gamma(\bar X, \bar \pi^*(M)) \cong k_{\bar X, \bar
    x}\otimes_{\Oc_{\bar X, \bar x}}\bar \pi^*(M)=k_{Z, z}\otimes_{\Oc_{Z,z}}q^*(M) = k^m,   
\end{displaymath}
where $m= \rank M = \dim_k V $, and 
\begin{displaymath}
  \Gamma(Z, q^*(M)) = \Gamma (Z, (p^*(q^*(M)))^H) = \Gamma(Z,
  (\Oc_{\bar X}\otimes_k V)^H) = \Gamma(Z, \Oc_{\bar X}^H)\otimes_k V =
  V, 
\end{displaymath}
implying that $q^*(M)$ is also  trivial. Let $G_1= G/H$ be the Galois group
of $k(Z)/K$, so that $q: Z\to Y$ is generically Galois. To see that
$q$ is actually  étale it suffices to see that the stabilisator subgroup $G_1^z$
of any closed point $z $ in $Z$ is trivial. This follows since the map
$G_1 \to \Glo_k(V)$ is injective and $G_1^z$ acts trivially on the
right side of \thetag{*}. It remains to see that the Galois group of
$p$ belongs to $H$, so that we get a factorization $\bar X\to X \to Z$
as asserted. This follows by considering the case when $M$ is already
trivial, $M= \Oc_Y\otimes_k V$, since then the whole group $G$ acts
trivially on $V$.

(2): Since $Z$ is connected in (1), $\Gamma(Z, \Oc_Z)= k$ as $k$ is
algebraically closed. First assuming that $M$ is a locally free and
finite $\Oc_Y$-module we show that it is also an étale trivial
connection. Since $q^*(M)$ is a trivial $\Oc_Z$-module,
$q^*(M)= \Oc_Z\otimes_kV$, we have $V= \Gamma(Z, q^*(M))$, and in
particular $L_1\otimes_KM_K = L_1\otimes_k V$, where $V$ is a
representation of $G_1$, and $\Delta (V)= M_K $ in
\Theorem{\ref{equivalence}}. Moreover, since $q$ is a Galois covering,
$ M = (\Oc_Z\otimes_kV)^{G_1}$, which evidently is a connection on $Y$
\Prop{\ref{morita}}. Conversely, if $M$ is an $L_1$-étale trivial
connection, then $q^!(M)$ is a connection that is generically trivial,
and therefore trivial.
\end{proof}

\begin{corollary}\label{finite-mod-cor}
  Assume that $M$ is a locally free $\Oc_Y$-module on a simply
  connected normal projective variety $Y$ over an algebraically closed
  field of characteristic $0$. If there exists a surjective proper
  morphism $\pi: X\to Y$ such that $\pi^*(M)$ is trivial, then $M$ is
  trivial.
\end{corollary}
\begin{proof} By \Remark{\ref{rem-stein}} one can assume that $\pi$ is
  finite, and by \Theorem{\ref{riemann}} there exists a Galois cover
  $q: Z\to Y $ such that $q^*(M)$ is a trivial connection. So that if
  $G$ is the Galois group of $q$, then $ M = q^*(M)^G$ is a
  connection, which is trivial since $Y$ is simply connected.
  Alternative: $q_*q^*(M)$ is a
  connection \Prop{\ref{global-invariant}}, which is trivial since $Y$
  is simply connected; in particular it is a trivial $\Oc_Y$-module.
  The trace morphism gives a split of the adjoint morphism
  \begin{displaymath}
    M  \stackrel[i]{\frac 1n \Tr }{\leftrightarrows} q_*q^*(M),
  \end{displaymath}
  where $n $ is the degree of $q$. Therefore $M$ is a trivial
  $\Oc_Y$-module.
\end{proof}
One wants to apply \Corollary{\ref{finite-mod-cor}} to get a criterion
that a locally free $\Oc_Y$-module $M$ on a simply connected variety
be trivial. The idea is that if there exists a point $y$ in $Y$ such
that any other point in $Y$ can be connected by a subvariety
$C_\lambda $, a member of an algebraically parametrized family,
$\lambda \in \Lambda$, and if the restriction of $M$ to each
$C_\lambda$ is trivial, then $M$ should also be trivial. This idea was
considered in \cite{biswas-dossantos:ratcon} when the $C_\lambda$ are
stable rational curves and $Y$ is rationally connected, and we
therefore do not claim essential originality, the main goal being to
clarify and connect to the present context.

We have then a diagram with vertices formed by normal projective
varieties over an algebraically closed field of characteristic $0$
\begin{equation}
\label{dia-sc}
  \begin{tikzcd}
     X \arrow[r,  "\pi"] \arrow[d,xshift=0.7ex, "p"] & Y
\\ \Lambda \arrow[u,xshift=-0.7ex,"\sigma"]\arrow[ur, "g"],&
\end{tikzcd}
  \end{equation}
  where $\pi$ is a surjective, $\sigma$ a section of $p$ (so
  that $p\circ \sigma = \id_\Lambda$), and $g= \pi \circ \sigma$.
  \begin{proposition} \label{equivalent-trivial} Assume that $p$ is
    proper. The following are equivalent for a locally free
    $\Oc_Y$-module $M$:
    \begin{enumerate}
    \item $\pi^*(M)$ is a trivial locally free $\Oc_X$-module.
    \item $g^*(M)$ is a trivial $\Oc_\Lambda$-module and
      $p^*p_*\pi^*(M)= \pi^*(M)$ (so that $\pi^*(M)$ is
      trivial relative to $p$).
      \item $M$ is a connection such that $\pi^!(M)$ is trivial.
      \end{enumerate}
\end{proposition}
\begin{proof} 
  $(3)\Rightarrow (1)$ is evident and $(1)\Rightarrow (2)$ follows
  since $k$ is algebraically closed and $p$ is proper.
  $(2)\Rightarrow (3)$: Putting $\bar M= \pi^*(M)$, we have since
  $p^*p_*(\bar M) = \bar M$ and $\sigma^*\circ p^* = \id$,
  \begin{displaymath}
    g^*(M)=\sigma^*(\bar M) = \sigma^* p^*p_*(\bar M) =p_*(\bar M),
\end{displaymath}
so that $p^*(\bar M)$ is trivial; hence $\bar M$ is a trivial
$\Oc_{X}$-module. Since $X$ and  $Y$  are proper, the map $\pi$ is
also defined at points $x$ of  height $\leq 1$; since $\bar M$ is
trivial, it is the restriction    is proper.  In other words, $M$ is a finite module, so by
\Theorem{\ref{riemann}} it is a connection such that $\pi^!(M)$ is
trivial.
\end{proof}
The following result gives a sufficient condition that ensures that two locally
free modules are isomorphic.
\begin{proposition} \label{isomodules}
  Let $M$ and $N$ be locally free sheaves on a variety $X$. Assume
  that
  $N= \bigoplus N_i= N_{1}\otimes_{\Oc_X}(\oplus_{j=1}^r N_{1j})$,
  where the $N_i$ and $ N_{ij}= N_i^*\otimes_{\Oc_X} N_j$ are
  invertible sheaves. The following are equivalent:
  \begin{enumerate}
  \item $M\cong M_1\otimes_{\Oc_X}\oplus_{j=1}^r N_{1j}$, where $M_1$ is an
    invertible sheaf.
  \item We have an isomorphism of locally free sheaves 
    \begin{displaymath}
      End_{\Oc_X} (M)\cong End_{\Oc_X}(N).
    \end{displaymath}
  \end{enumerate}
\end{proposition}

\begin{proof}
  $(1)\Rightarrow (2)$ is clear. $(2)\Rightarrow (1)$: Select splittings
  \begin{displaymath}
N \stackrel[p_i]{j_{i}}{\leftrightarrows} N_i 
\end{displaymath}
so that $p_i \circ j_i$ is an automorphism of $N_i$, and put also
$\bar \phi_i = j_i\circ p_i$, which is an endomorphism of $N$ that
restricts to an automorphism of $N_i$, and
$\bar \phi_i \circ \bar \phi_j = \delta_{ij}\bar \phi_i$. Letting
$\psi$ denote an isomorphism in (2), the elements
$\phi_i = \psi^{-1}(\bar \phi_i)$ form global sections of
$End_{\Oc_X}(M) $, and we have
$\phi_i\circ \phi_j =\delta_{ij} \phi_i$. Putting
$M_i = \Imo (\phi_i)\subset M$, we then have $\phi_i(M_j)= \{0\}$ when
$i \neq j$, and $\phi_i$ defines an isomorphism $M_i \cong M_i$. and
$M_i \cap M_j = \{0\}$ when $i \neq j$. Hence the inclusions
$M_i\subset M$ induce an isomorphism
\begin{displaymath}
  \bigoplus_{i=1}^r M_i \cong M. 
\end{displaymath}
Moreover,
$End_{\Oc_X} (\oplus_{i=1}^r M_i) \cong
End_{\Oc_X}(\oplus_{i=1}^rN_i)$ implies that
$Hom_{\Oc_X}(M_i, M_j) = Hom_{\Oc_X}(N_i, N_j)=N_{ij}$, hence the
modules $M_i $ are invertible, and $M_j \cong
N_{1j}\otimes_{\Oc_X}M_1$. This implies (1).
\end{proof}
It can be hard to directly check (2) in
\Proposition{\ref{isomodules}}, but at least when $Y$ is simply
connected it can be used to reach a more precise result. This was
discussed already in the proof of \cite{biswas-dossantos:ratcon}*{Th.
  2.2}, where the main ideas were presented.
\begin{theorem}\label{invertible-inverse} Consider the diagram (\ref{dia-sc}), where moreover
  $Y$ is simply connected. If $\pi^*(M)= \bar \Lc^r$ for some
  invertible $\Oc_X$-module $\bar \Lc$, then there exists an
  invertible $\Oc_Y$-module $\Lc$ such that $M=\Lc^r$.
\end{theorem}

\begin{proof}
  By \Proposition{\ref{equivalent-trivial}} $End_{\Oc_Y}(M)$ is a
  connection, which is trivial since $Y$ is simply connected. Then the
  assertion follows from \Proposition{\ref{isomodules}}, putting
  $N= \Oc_Y^r$, and hence $N_{1j}= \Oc_Y$.
\end{proof}
\begin{corollary}
  Let $Y$ be a rationally connected smooth projective variety and $M$
  be a locally free $\Oc_Y$-module.
\begin{enumerate}
\item The following are equivalent:
  \begin{enumerate}
  \item $M$ is trivial.
  \item the pullback $\gamma^*(M)$ is trivial for each map
    $\gamma : \Pb^1_k \to Y $.
  \end{enumerate}
  \item The following are equivalent:
  \begin{enumerate}
  \item $M = \Lc^r$ for some invertible module $\Lc$ on $Y$.
  \item The pullback $\gamma^*(M)= \bar \Lc^r$, for some invertible
    module $\bar \Lc$ on $\Pb^1_k$ and each non-constant map
    $\gamma : \Pb^1_k \to Y $.
  \end{enumerate}
\end{enumerate}
\end{corollary}
\begin{proof} In either case we only need to prove
  $(b)\Rightarrow (a)$. Since $X$ is rationally connected there exists
  a diagram of the type (\ref{dia-sc}), where the fibres of $p$ form
  trees of stable rational curves, and the map $g$ is constant. Put
  $\bar M = \pi^*(M)$. (1): The condition (b) implies that the
  restriction of $M$ to a tree of stable rational curves is trivial.
  This implies that the canonical map $p^*p_*(\bar M)\to \bar M$ is an
  isomorphism, and since moreover $g$ is constant, (2) in
  \Proposition{\ref{equivalent-trivial}} is satisfied, hence $ M$ is a
  connection; hence $M$ is trivial by \Corollary{\ref{rat-conn-cor}}.
  (2): The $\Oc_X$-module $\tilde M= End_{\Oc_X}(\bar M)$ is trivial
  along each fibre of $p$, and is therefore trivial as in the proof of
  (1). Hence by
  \Propositions{\ref{equivalent-trivial}}{\ref{isomodules}},
  $\bar M = \bar \Lc^r$ for some invertible sheaf $\bar \Lc$, and then
  the assertion follows from \Theorem{\ref{invertible-inverse}}.
\end{proof}
\subsection{Trace characterization of the canonical submodules}\label{can-mod-sect}
We will study the canonical submodules $N_j$ of
$N=\pi_+(\Oc_X)$\footnote{It is possible to extend the treatment to
  modules of the form $N=\pi_+\pi^!(C)$ where $C$ is any semisimple
  connection on $Y$.} using the canonical stratification of $\pi$
\begin{displaymath}
\{X_{i}, Y_j ; j\in J, i\in I_j \}.
\end{displaymath}
To each stratum $ X_{i}$, $i \in I_j$, we can associate a
factorization
\begin{displaymath}
  \pi : X \xrightarrow{p_{i}} Z_{i}\xrightarrow{q_{i}}  Y
\end{displaymath}
as in \Theorem{\ref{factor-theorem}}; we have for corresponding
generic points $\pi(x_{i})= q_{i}(z_{i})= y_{j}$. Here $q_{i}$ is
maximally étale at the point $z_{i}= p_{i}(x_{i})$ and $p_{i}$ is
minimally totally ramified at the generic point $x_{i}$ of $X_{i}$,
but if $X/Y$ is not Galois then $p_i$ need not be totally ramified at
$p_i(x_i)\in Z_i$ (see \Definition{\ref{def-tot-ram}}), so that
\Theorem{\ref{vanishing}} is not applicable. Therefore, similarly to
the proof of \Theorem{\ref{factor-theorem}}, we need to take a field
extension $ \bar L /L= k(X)/K=k(Y)$ such that $\bar L/K$ and
$\bar L/L$ are Galois. Referring to the diagram
(\ref{galois-diagram}), we in particular have maps
\begin{displaymath}\tag{*}
\bar \pi:   \bar X \xrightarrow{s} X \xrightarrow{\pi} Y,
\end{displaymath}
where $\bar X$ is the integral closure of $X$ in $L$, so that
$\bar L = k(\bar X)$.

The {\it inertia scheme} $Z^{x_i}$ of the stratum $X_{i}$ is in
general singular, so that the notation $(p_{i})_+(\Oc_X)$ as
$\Dc_{Z_{i}}$-module need not be meaningful. However, the point
$z_{i}$ is smooth and also smooth over $y_j$ so that the
$\Dc_{Z^{x_i}, z_{i}}$-module $(p_{i})_+(\Oc_X)_{z_{i}}$ is
well-defined.

To clarify the definition of inertia modules below it can be
helpful to first consider the commutative diagram
\begin{displaymath}
\tag{*}
  \begin{tikzcd}
     \bar B & \arrow[l,  "\bar p"'] \bar C
     \\ B \arrow[u, "s"] &  \arrow[l,"p"] \arrow[u,"r"']C
\end{tikzcd}
  \end{displaymath}
  of free $C$-algebras, where $C$ is noetherian, and that the maps $s$
  and $\bar p$ make $ \bar B$ and $\bar B$ free algebras over $B$ and
  $\bar C$, respectively (the maps $p,s, \bar p, r$ are inclusion
  maps). The trace morphism $\tr_{B/C}: B \to C$ is described in
  (\ref{relative-can}) and we put $\Tc_{B/A}= \Ker \tr_{B/A}$. We have
  \begin{displaymath}\tag{*}
    \tr_{\bar B/C} = \tr_{\bar C/C} \circ \tr_{\bar B/\bar C} =
    \tr_{B/C} \circ
    \tr_{\bar B/B}.
\end{displaymath}
Let $ \overline{\tr}_{B/C}: B \to \bar C $ be the
restriction of $\tr_{\bar B/\bar C}$ to $B\subset \bar B$, so that
\begin{displaymath}
  \tr_{\bar C/C}\circ \overline{\tr}_{B/C} = h \tr_{B/C},
\end{displaymath}
where $h$ denotes the rank of $\bar B$ over $B$, and  put also
\begin{displaymath}
\overline \Tc_{B/ C} = B \cap \Tc_{\bar B/\bar C} = \Ker
(\overline{\tr}_{B/C})\subset \Tc_{B/C}. 
\end{displaymath}

Returning to sheaves, consider the trace morphism
\begin{displaymath}
  \tr_{\bar z_{i}}:(\bar p_{i})_+(\Oc_{\bar X})_{\bar z_{i}}\to
  \Oc_{Z^{\bar x_i}, \bar z_i}.
\end{displaymath}
Let $H$ be the Galois group of $\bar L/L$, and define the induced map
\begin{displaymath}
  \overline{ \tr}_{\bar z_{i}}: (p_{i})_+(\Oc_X)_{z_{i}} = (\bar p_{i})_+(\Oc_{\bar X})^H_{\bar z_{i}}\to
  \Oc_{Z^{\bar x_i}, \bar z_i}.
\end{displaymath}
In the diagram \thetag{*} we have $B= \Oc_{X,x_i}$,
$\bar B= \Oc_{\bar X, \bar x_i}$, $C= \Oc_{Z^{x_i},z_i}$,
$\bar C= \Oc_{Z^{\bar x_i},\bar z_i} $, and
$ \overline{ \tr}_{z_{i}}= \overline{\tr}_{B/C}$, extended via the map
$(p_i)_*(\Oc_{X})_{z_i}\to (p_i)_+(\Oc_{X})_{z_i}$.
\begin{remark}
  If $\phi$ is a local section of $(p_{i})_+(\Oc_X)_{z_{i}}$, its
  value
  $\overline{ \tr}_{\bar z_{i}}(\phi) = \sum_{g\in I_{\bar x_i}}g \cdot
  \phi$ (the Reynold's operator) needs not belong to
  $ \Oc_{Z^{\bar x_i}, \bar z_i}^H = \Oc_{Z^{ x_i}, z_i}$. The reason
  is that in general $hI_{\bar x_i}\neq I_{\bar x_i} h $, $h\in H$.
  However, $ \overline {\tr}_{\bar z_{i}}(\phi) =0 $ if and only if
  $ \overline{\tr}_{\bar z_{i}'}(\phi) =0$ when $\bar z_i$ and $\bar z'_{i}$
  lie over the same point $z_i$.
\end{remark}
Since $\pi_+(\Oc_X)$ is the minimal extension of $\pi_+(\Oc_X)_{y_j}$,
the image of the trace morphism
$(q_{i})_+(\overline{ \tr}_{z_{i}}): \pi_+(\Oc_X)_{y_{j}} \to
(\bar q_i)_+(\Oc_{Z^{\bar x_i}, \bar z_i}) \to \pi_+(\Oc_X)_{y_{j}} $ has a unique
extension to a homomorphism of semisimple modules
  \begin{displaymath}
    \overline{ \Tr}_{i}: \pi_+(\Oc_X)\to \pi_+(\Oc_X).
  \end{displaymath}
  To emphasize, $\overline{ \Tr}_{i}$ is the ordinary trace relative
  to $p_{x_i}$ only when $X/Y$ is Galois.

\begin{definition} The {\it inertia submodule} of $\pi_+(\Oc_X)$ corresponding to
  the stratum $X_{i}$ ($i\in I_j$) is  defined by
\begin{displaymath} 
  \Tc_{i} = \Ker (    \overline{ \Tr}_{i}).
\end{displaymath}
\end{definition}
One of our main results is that the inertia modules determine the
submodules $N_j$ of $N$ and that these modules form a filtration.
\begin{theorem}\label{can-filt}
  \begin{displaymath}
    N_j = \bigcap_{i\in I_j } \Tc_{i}.
  \end{displaymath}
\end{theorem} 

\begin{corollary}\label{cor-can-filt} 
  The sequence of modules $\{N_j\}_{j=j_s}^{j_f}$ forms an increasing
  filtration of $N=\pi_+(\Oc_X)$, so that $N_{j_1} \subset N_{j_2}$ when
  $j_2 \geq j_1$.
\end{corollary} 
We call $\{N_j\}_{j\in J}$ the {\it canonical filtration} of
$\pi_+(\Oc_X)$ with respect to the poset $(J, >)$. Notice 
that $N_{j_s} = 0$ since $N$ is torsion free and $N_{j_f}= N$
since $N(\emptyset) = \{0\}$.

\begin{remark} Let $N_1$ be a submodule of $N$, $S$ be a
  locally closed subset of $Y$, and $\bar S$ its closure.
  \Corollary{\ref{cor-can-filt}} gives the condition that if
  $R\Gamma_{S}(N_1)=0$, then $ R\Gamma_{\bar S}(N_1)=0$. For example,
  the simple module $ M$ in \Example{\ref{ex-nonvancoh}} cannot occur
  in $N$ for any finite map $\pi: X\to \Cb^2$.
\end{remark}
\begin{proof}
  Since $j_1 \leq j_2 $, so that $Y_{j_2}\subset \bar Y_{j_1}$,  we have
  \begin{displaymath}
  \bigcup_{i\in I_{j_2}}X_i = X_{Y_{j_2}} \subset X_{\bar Y_{j_1}}= \bar
  X_{Y_{j_1}}= \bigcup_{i \in I_{j_1} } \bar X_{i}. 
\end{displaymath}
Since each closure $\bar X_i = \cup_{i'\geq i} X_{i'}$ (disjoint
union) and by the construction of the canonical stratification, for
each stratum $X_{i_2}$, $i_2 \in I_{j_2} \subset I$ there exists a
stratum $X_{i_1}$, $i_1 \in I_{j_1}\subset I $, such that
$X_{i_2}\subset \bar X_{i_1}$. Therefore there exists a map
$\phi:Z_{i_1}\to Z_{i_2} $ such that $p_{i_1}= \phi \circ p_{i_2}$,
implying that $\Tc_{i_1}\subset \Tc_{i_2}$. Therefore
\begin{displaymath}
  \bigcap_{i_1 \in I_{j_1}} \Tc_{i_1} \subset  \bigcap_{i_2 \in I_{j_2}} \Tc_{i_2}, 
\end{displaymath}
so we can conclude by \Theorem{\ref{can-filt}}.
\end{proof}

\begin{corollary}
  \begin{enumerate}
\item  If $i \geq j$, then $(N_j)_{Y_i}= N_i$ and $(N_i)_{Y_j}= N_j$.
\item  If $\bar Y_i\cap \bar Y_j = \emptyset$, then $(N_i)_{Y_j} =
  (N_j)_{Y_i}= N_{Y_i\cup Y_j}$.
\end{enumerate}
\end{corollary}

\begin{proof}
  (1): This follows from \Corollary{\ref{cor-can-filt}}. (2): Since
  $\bar Y_i\cap \bar Y_j = \emptyset$ we have
  $\RG_{Y_i\cup Y_j} (\cdot) = \RG_{Y_i}(\cdot )\oplus
  \RG_{Y_j}(\cdot)$.
  Therefore the maximal submodule with vanishing local cohomology
  along $\bar Y_i\cup \bar Y_j$ is the maximal submodule with
  vanishing local  cohomology both along $Y_i$ and  $Y_j$. The
  assertion now  follows from (1).
\end{proof}

Recall that $ \pi_+(\Oc_X)= \Oc_Y \oplus \Tc_\pi$.
\begin{corollary}\label{can-decomp} Let
  $j_1 <  j_2 <  \ldots < j_s <  \ldots < j_r $ be a path in
  $J$, where $j_r \neq j_f$.
  \begin{enumerate}
  \item 
    \begin{displaymath}
      N_{j_r}\cong      N_{j_1}  \bigoplus \frac
      {N_{j_2}}{N_{j_{1}}} \bigoplus \cdots \bigoplus \frac
      {N_{j_r}}{N_{j_{r-1}}}\subset \Tc_\pi,
    \end{displaymath}
    where the semisimple module $N_{j_s}/N_{j_{s-1}} $ has no common
    (non-zero) simple component with $N_{j_{s'}}/N_{j_{s'-1}} $ when
    $j_s\neq j_{s'}$, i.e.
    \begin{displaymath}
    Hom_{\Dc_Y}(N_{j_s}/N_{j_{s-1}}, N_{j_{s'}}/N_{j_{s'-1}})=0.
  \end{displaymath}

\item If $\pi$ is totally ramified at some point that maps to the
  generic point of $Y_{j_r}$, then
  \begin{displaymath}
    \pi_+(\Oc_X) \cong    N_{j_1}  \bigoplus \frac
    {N_{j_2}}{N_{j_{1}}} \bigoplus \cdots \bigoplus \frac
    {N_{j_r}}{N_{j_{r-1}}} \bigoplus\Oc_Y.
  \end{displaymath}
  \end{enumerate}
\end{corollary}

\begin{proof}
  (1): \Corollary{\ref{cor-can-filt}} implies
  $N_{j_1}\subset \ldots \subset N_{j_r}\subset \Tc_\pi$. The direct sum
  decomposition and the fact that the modules $N_{j_s}/N_{j_{s-1}}$
  are semisimple follows since $\pi_+(\Oc_X)$ is semisimple
  \Th{\ref{decomposition-thm}}. If $N^{(j_s)}$ is a nonzero simple
  module in $N_{j_s}/N_{j_{s-1}}$, it follows that
  $\RG_{\bar Y_{j_s}}(N^{(j_s)})=0$ while
  $\RG_{\bar Y_j}(N^{(j_s)})\neq 0$, when $j < j_s$. This implies
  the remaining assertion.

  (2): Since $\pi$ is totally ramified at a point $x\in X_{i}$,
  $i\in I_{j_r}$, it follows that $\Tc_\pi \subset N_{j_r}$; since
  $\Oc_Y \cap N_{j_r}= 0$ and $N= \Oc_Y \oplus \Tc_\pi$ it follows that
  $\Tc_\pi= N_{j_r}$.  Now the assertion follows from (1).
\end{proof}
When $j, j'\in J$ are unrelated with respect to $>$, there is no
general relation between $N_j$ and $N_{j'}$. In
\Section{\ref{symmetric-groups}} we give an example where
$N_j = N_{j'}$ while $j \not \leq j'$ and $j' \not \leq j$, and one
can also have $N_j= N_{j'}$ when $j > j'$ so that the filtration needs
not be strictly increasing. Notice that there may exist several paths
in $(J, >)$ joining $j_1$ and $j_r$, and thus giving rise to different
decompositions of $N_{j_r}$.

\begin{pfof}{\Theorem{\ref{can-filt}}} (a) First assume that
  $\pi: X\to Y$ is Galois. Let $\pi = q_{i}\circ p_{i}$ be the
  factorization of the stratum $X_{i}\subset X$, $i\in I_j$, so that
  $p_{i}: X\to Z_{i}$ is (minimally) totally ramified at the generic
  point $x_{i}$ of $X_{i}$ and $q_{i}: Z_{i}\to Y$ is (maximally)
  étale at $z_{i}= p_{i}(x_{i})$ \Th{\ref{factor-theorem}}. Let
  $\bar Y_j$ be the closure of the stratum $Y_j= \pi(X_i)$ and
  $\partial Y_j = \bar Y_j \setminus Y_j$ its boundary. Put
  $X_{\bar Y_j}= \pi^{-1}(\bar Y_j) $,
  $X_{\partial Y_j}= \pi^{-1}(\partial Y_j)$, and
    \begin{displaymath}
      G_{(i)}=X_{\bar Y_j}\setminus X_i, 
    \end{displaymath}
    so that $G_{(i)}$ is a closed set and
    $X_{\partial Y_j}= \bigcap_{i \in I_j} G_{(i)}$.
    Let $\psi_{i}: X \setminus G_{(i)}\to X$ and
    $\phi_i: Y \setminus \bar Y_i \to Y$ be the open inclusions.
    Similarly, define the open inclusion
    $\tilde \psi_{i}: Z_i^0= Z_{i}\setminus p_{i}(G_{(i)})\to Z_{i}$ and the
    restriction of $p_{i}$,
    $p^0_{i}: X \setminus G_{(i)} \to Z_{i}\setminus p_{i}(G_{(i)}) $,
    so that $p_{i}\circ \psi_{i} = \tilde \psi_{i}\circ p^0_{i}$.

    Putting $N_j^0 = \bigcap_{i \in I_j} \Tc_{i} $ we need to prove
    $N_j^0 = N_j$, where $\Tc_{i}$ is the minimal extension of the
    semisimple module $ (q_{i})_+(\Tc_{i}^0)_{y_j}$ from the point
    $y_j = \pi (x_{i})$ to $Y$, and
    $\Tc^0_{i} = \Ker (\Tr : (p_{i})_+(\Oc_X))_{z_{i}}\to
    \Oc_{Z_i,z_{i}}$; notice that $\Tc^0_{i}$ is only a
    $\Dc_{Z_i,z_i}$-module. Below all functors are derived.

    $N_j^0 \subset N_j$: Since $\RG_{G_{(i)}} \RG_{\bar X_{y_j}} =
   \RG_{G_{(i)}}$ we have the distinguished triangle in local cohomology
  \begin{displaymath}
    \RG_{G_{(i)}}(\Oc_X) \to \RG_{ X_{\bar
        Y_j}}(\Oc_X)\xrightarrow{\phi_{i}} (\psi_{i})_+\psi_{i}^!
    (\RG_{X_{\bar Y_j}}(\Oc_X))  \xrightarrow{+1}, 
  \end{displaymath}
so that by \Theorem{\ref{loc-cohom-dir}} we get the triangle 
  \begin{displaymath}\tag{A}
\pi_+\RG_{G_{(i)}}(\Oc_X) \xrightarrow{r_{i}} \RG_{\bar
  Y_j}\pi_+(\Oc_X)\xrightarrow{\bar \phi_{i}}
\pi_+(\psi_{i})_+\psi_{i}^! (\RG_{ X_{\bar Y_j}}(\Oc_X))
\xrightarrow{+1},
  \end{displaymath}
  where $\bar \phi_{i} = \pi_+(\phi_{i})$. The split inclusion
  $\Tc_{i}\to \pi_+(\Oc_X)_{y_j}$ of $\Dc_{Y,y_j}$-modules induces a
  split homomorphism
  $l_{i}: \RG_{\bar Y_j}(\Tc_{i})_{y_j}\to \RG_{\bar
    Y_j}\pi_+(\Oc_X)_{y_j} $, where we assert that
  \begin{displaymath}\tag{B}
\bar \phi_{i,y_j} \circ l_{i}=0.
\end{displaymath}
Assuming \thetag{B} we get a morphism
$l^0_{i}: \RG_{\bar Y_j}(\Tc_{i})_{y_j} \to
\pi_+\RG_{G_{(i)}}(\Oc_X)_{y_j}$ such that
$l_{i}= r_{i,y_j}\circ l^0_{i}$ \footnote{This follows from one of the
  axioms of triangulated categories}. Therefore one gets a split
homomorphism
\begin{displaymath}
  \RG_{\bar Y_j}(\bigcap_{i \in I_j} \Tc_i)_{y_j} \to \pi_+\RG_{\cap_{i\in I_j} G_{(i)}}(\Oc_X)_{y_j} =
  \pi_+ (\RG_{X_{\partial Y_j}}(\Oc_X))_{y_j}= \RG_{\partial Y_j}(\pi_+(\Oc_X))_{y_j}=0,
\end{displaymath}
where in the last step \Theorem{\ref{loc-cohom-dir}} is again applied.
Therefore $\supp \RG_{\bar Y_j}( N^0_j)\subset \partial Y_j$, so that
$N_j^0 \subset N_j$.

We now prove \thetag{B}. Putting
$Z_{\bar Y_j} = q_i^{-1}(\bar Y_j)$, so that
    \begin{displaymath}
      Z_{\bar Y_j} = p_i(X_i) \bigcup p_i(G_{(i)}),
    \end{displaymath}
hence    by \Theorem{\ref{vanishing}}
\begin{displaymath}
  \RG_{Z_{\bar Y_j}} (\Tc^0_i)_{z_i} =0,  
\end{displaymath}
so according to   \Theorem{\ref{loc-cohom-dir}} 
\begin{align*}\tag{C}
  \RG_{\bar Y_j}(\Tc_{i})_{y_j} &=  \RG_{\bar
                            Y_j}((q_{i})_+(\Tc^0_{i}))_{y_j}= (q_{i})_+
                            (\RG_{Z_{\bar Y_j}}(\Tc^0_i)_{z_i})_{y_j}=0. 
\end{align*}
Let $\bar X_i$ be the closure of $X_i$ and put
$\tilde X_i = p_i(X_i)$. We have for the stalk at $y_j$ of the third
term in \thetag{A} (as detailed below)
\begin{align*}
&  \pi_+ (\psi_{i})_+\psi_{i}^! (\RG_{\bar
  X_{i}}(\Oc_X))_{y_j}   
=  (q_{i})_+ (p_{i})_+  (\psi_{i})_+\psi_{i}^!  (\RG_{\bar
  X_{i}}(\Oc_X))_{y_j} \\ &=  (q_{i})_+ (\tilde \psi_{i})_+ (p^0_{i})_+
  \psi^!_{i}\RG_{\bar X_{i}}(\Oc_{X})_{y_j} =   (q_{i})_+ (\tilde \psi_{i})_+ (p^0_{i})_+
  \RG_{ X_{i}}\psi^!_{i} (\Oc_{X})_{y_j}\\&
= (q_{i})_+ (\tilde \psi_{i})_+ \RG_{ \tilde X_{i}} (p^0_{i})_+
  \psi^!_{i}(\Oc_X)_{y_j}
= (q_{i})_+ (\tilde \psi_{i})_+ \RG_{ \tilde X_{i}}
  (\Oc_{Z^0_{i}}\oplus \Tc_{i}^0)_{y_j} \\ &= (q_{i})_+ (\tilde \psi_{i})_+ \RG_{\tilde X_{i}}(\Oc_{Z^0_{i}})_{y_j}.
\end{align*}
The first equality of the third line follows from
\Theorem{\ref{loc-cohom-dir}} and the very last step is a consequence
of \Theorem{\ref{vanishing}}. By \thetag{C} this implies that
$\bar \phi_{i,y_j} \circ l_{i}$ maps to
\begin{displaymath}
(q_{i})_+ (\tilde \psi_{i})_+ \RG_{\tilde
  X_{i}}\psi^!_{i}(\Tc_{i}^0)_{y_j}=0 \quad \Th{\ref{vanishing}}.
\end{displaymath}

$ N_j\subset N_j^0$: We need to prove that $N_j\subset \Tc_{i}$ for
all $i\in I_j$. Since $\pi_+(\Oc_X)$ is semisimple it suffices to prove that
if $M^0$ is a non-zero simple component of $ \pi_+(\Oc_X)$ and if
there exists an index $i$ such that
\begin{displaymath}
  M^0 \not \subset \Tc_{i},
\end{displaymath} 
then $\RG_{Y_j}(M^0)_{y_j}\neq 0$. The generic point $z_{i}$ of
$p_i(X_{i}) \subset Z_{i}$ maps to $y_j$ and since $\pi_+(\Oc_X)$ is
torsion free we have the non-zero $\Dc_{Y,y_j}$-submodule
\begin{displaymath}
  M^0_{y_{j}} \subset (q_{i})_+(\Oc_{Z_{i}, z_{i}})_{y_j}.
 \end{displaymath} 
 By \Theorem{\ref{factor-theorem}} $Z_{i}/k$ and $Z_{i}/Y$ are smooth
 at $z_{i}$ so that the $\Dc_{Y, y_j}$-module
 $(q_{i})_+(\Oc_{Z_{i}})_{y_j}$ is of finite type over $\Oc_{Y, y_j}$,
 hence by Grothendieck's non-vanishing theorem all its non-zero
 submodules have non-vanishing local cohomology
 \cite{bruns-herzog}*{Th. 3.5.7}. Therefore
\begin{displaymath}
  \RG_{Y_j} ( M^0)_{y_{j}} = \RG_{\mf_{y_j}}
  (M^0_{y_j}) \neq 0.
\end{displaymath}
This completes the proof when $\pi$ is Galois.

(b) When $\pi$ is not Galois we consider a Galois cover and refer to
the diagram (\ref{galois-diagram}). If $H$ is the Galois group of
$\bar X/X$, putting $\bar N = \bar \pi_+(\Oc_{\bar X})$, we have
$N= \bar N^H$ (see also \Corollary{\ref{non-galois}}), and
$N_j = (\bar N_j)^H$. The set $s^{-1}(X_i)= \cup \bar X^\alpha$ is a
union of locally closed irreducible sets
$\bar X^\alpha\subset \bar X$; let $\bar x_i$ be a generic point of
$\bar X^\alpha$. To $\bar x_i$ there is associated a canonical
factorization
$\bar X \xrightarrow{\bar p} \bar Z^{\bar x_i} \xrightarrow{r} Z^{x_i}
\xrightarrow{q} Y$, and a map $p: X\to Z^{x_i} $ so that
$p\circ s = r\circ \bar p$ (we refer to the proof of
\Theorem{\ref{factor-theorem}} for more details and the notation). The
map $\bar p_{\bar x_i}$ is totally ramified along $\bar z_i$, and
therefore by (a)
\begin{displaymath}
  \bar N_j = \bigcap_{i\in I_j} \bar \Tc_{ \bar x_i},
\end{displaymath}
so we get 
\begin{displaymath}
  N_j = \bar N_j^H= \bigcap_{i\in I_j} (\bar \Tc_{\bar x_i})^H =  \bigcap_{i\in I_j} \Tc_i,
\end{displaymath}
where we note that if $\bar x_i$ and $\bar x'_i$ both lie above $x_i$, then
$\bar \Tc_{ \bar x_i}^H= \bar \Tc_{ \bar x'_i}^H$.
\end{pfof}

\subsection{Generators and decomposition  of the inertia modules}
Regard the partially ordered set $ (I= \cup I_j, <)$ that is
determined by the canonical stratification of $\pi: X\to Y $ as a
directed graph, with a directed edge $ i_1\xrightarrow{e} i_2$ when
$i_1 > i_2$, and if $i_2 \geq i \geq i_1 $, then either $i=i_1$ or
$i= i_2$. This directed graph lies over the directed specialization
graph of the stratifaction $\{Y_j\}$ of $Y$. If $i_1 \in I_{j_i}$,
$i_2 \in I_{j_2}$ and there is an edge $ i_1\xrightarrow{e} i_2$, we
get an edge $j_1 \to j_2$ in the directed graph $(J, <)$. The
factorizations $\pi = q_{i}\circ p_{i}= q_{i_1}\circ p_{i_1}$ are
compatible in the sense that there exists a map
$r_e:Z_{i_1} \to Z_{i_2}$, corresponding to the edge $e$, such that
$p_{i_2}= r_e\circ p_{i_1}$ and $q_{i_1}= q_{i_2}\circ r_e$, so that
we have the maps
  \begin{displaymath}
  \pi : X \xrightarrow{p_{i_1}} Z_{i_1} \xrightarrow{r_e}Z_{i_2}\xrightarrow{q_{i_2}}  Y.
\end{displaymath}
Assume now for simplicity that $\pi$ is Galois (otherwise use the
Galois cover $\bar X \to X$ as described in (\ref{can-mod-sect})), and
recall first how the inertia module $\Tc_i$ corresponding to the
stratum $X_i$, $i \in I_j $, is defined. We have
$(p_i)_+(\Oc_X)_z = \Oc_{Z_i,z}\oplus (\Tc^0_{i})_z$. If $Z_i$ is
singular, this is defined as a $\Dc_{Z_i, z}$-module only at points
$z$ that specialize to $p_i(x_i)$, where $x_i$ is the generic point of
$X_i$. Put $S_i = \pi( x_i)^-- \{\pi(x_i)\}\subset \Spec A= Y$. Then
$\Tc_i$ is the minimal extension of $(q_i)_+(\Tc_i^0)$ from
$Y \setminus S_i$ to the whole of $Y$.

\begin{theorem} Assume that $X_{i_1}$ is a stratum of codimension
  $\geq 2$. Then there exists an exact sequence of $\Dc_Y$-modules
  \begin{displaymath}
0 \to \bigoplus_{i\to i_1} (\bigoplus_{i'\to i_1, i'\neq i} \Tc_{i}\cap
\Tc_{i'})\to   \bigoplus_{i\to i_1}\Tc_{i}\to  \Tc_{i
  _1} \to 0.
  \end{displaymath}
  \end{theorem}
  \begin{proof} Consider points $z_{1}\in Z_{i_1}$ lying below  points
    $z_i$ that specialize to $p_i(x_i)$, for some $i\to i_1$.    Then 
    \begin{displaymath}
(\Tc_{i_1})_{z_1} = (r_e)_+(\Tc_{i})_{z_1}\oplus (\Tc_{r_e})_{z_1}
\end{displaymath}
If a simple module $W\not \subset (r_e)_+(\Tc_{i})$ for all
$i \to i_1$, then since $r_e$ is étale at all points of height
$\leq 1$ that specialize to $p_i(x_i)$, $W_{z'_{i_1}}$ is of finite
type over $\Oc_{Z_{i_1}, z'_i}$ when $\hto(z'_i)\leq 1$. Therefore,
since $W$ is torsion free, by Grothendieck's finiteness theorem $W$ is
coherent over $\Oc_{Z_i}$(\cite{SGA2}*{Prop. 3.2 }). Hence by
Grothendieck's non-vanishing theorem
$\RG_{\mf_{z_{i_1}}}(W_{z_{i_1}})\neq 0$, and therefore
$W_{z_{i_1}} \not \subset (\Tc_{i_1})_{z_{i_1}}$. This implies that
the sequence is exact on the right. The exactness on the left is
straightforward.
\end{proof}

\begin{theorem} The following are equivalent:
\begin{enumerate}
    \item $\Tc_i \subset N_j$.
    \item For all $i' \in I_j$ we have $i \to i'$.
\end{enumerate}
\end{theorem}

  \begin{proof}
    By \Theorem{\ref{can-filt}}
      \begin{displaymath}
        N_j = \bigcap_{i'  \in I_j} \Tc_{i'},
      \end{displaymath}
      and $\Tc_i\subset \Tc_{i'} $ if and only if $i \to i'$. This
      implies the assertion.
  \end{proof}

  \section{Complex reflection groups }\label{complex-refl-section}
  Let $ A= k[y_1, \ldots , y_n]\to B = k[x_1, \ldots , x_n]$ be a
  finite injective homomorphism of polynomial rings and
  $\pi: X=\Spec B \to Y=\Spec A $ be its associated map. The main
  problems concerned with the $\Dc_Y$-module $\pi_+(\Oc_X)$ are to
  determine: (i) a set of generators; (ii) a presentation; (iii) its
  canonical filtration; (iv) a semisimple decomposition. In this
  generality (for all such $\pi$), this project is quite formidable,
  and in this section we will merely get some answers to (i-iv) for
  invariant maps $X= \Ab^n_k\to Y= X^G$ for a complex reflection group
  $G$, where by a theorem of Shephard and Todd $Y= \Ab^n_k$.

  In \Section{\ref{symmetric-groups}} the canonical filtration of
  $N= \pi_+(\Oc_X)$ (\ref{can-mod-sect}) is explicitly decomposed into
  simple modules when $G$ is the symmetric group $S_n$,
  \Section{\ref{pres-exp}} contains a presentation of the direct image
  $\pi_+(E_\lambda)$ of certain exponential modules $E_\lambda$ (where
  $E_0 = \Oc_X$), and \Section{\ref{simple-imprimitive}} is about
  decomposing $\pi_+(\Oc_X)$ when $G$ is an imprimitive complex
  reflection group, using natural generalizations of Young or Specht
  polynomials.

  \subsection{Generators}
    First in greater generality, provide $B$ with the usual grading so
  that the $x_i$ have degree $1$, and assume that $\pi$ is {\it
    homogeneous}, so that the elements $y_i$ are homogeneous
  polynomials of some degrees $d_i = \deg (y_i)$ as elements in $B$.
  Then the Jacobian $J_\pi = \det \partial (y_i)/\partial (x_j) $ is
  homogeneous of degree $\deg J_\pi = \sum_{i=1}^n (d_i -1)$.

  Let $(\{X_{i}, Y_j\}, i \in I_j, j \in J )$ be the canonical
  stratification of $\pi$ and
  $X \xrightarrow{p_{i}} Z_{i}\xrightarrow{q_{i}}Y$ be a factorization
  such that $p_{i}$ is minimally totally ramified along the generic
  point $x_i$ of $X_{i}$ and $q_{i}$ is étale along the image of $x_i$
  in $\tilde X_{i}=p_{i}(X_{i})\subset Z_{i}$. In general the variety
  $Z_i$ need not be smooth.

  Assume that $\pi : A= B^G\to B$ is the inclusion of the invariant
  ring with respect to a group $G$ that is generated by pseudo
  reflections of $X$. Now the stratification $\{X_{i}\}$ are formed as
  {\it flats} in an arrangement of hyperplanes; see
  \cite{orlik-terao:arrangements}*{Theorem 6.27(1)}. Thus $X_i$ is
  formed as an intersection of hyperplanes after removing higher
  codimensional intersections. Also, in this case, the inertia
  subgroups are themselves complex reflection groups ([ibid, Theorem
  6.25], \cite[Prop. 26.6]{Kane}), and this implies, by Chevalley's
  theorem (\cite[Props. 18.3,18.5]{Kane}), that all varieties $Z_{i}$
  that occur in a minimal totally ramified factorizations
  $X\to Z_{i}\to Y$ again are isomorphic to $\mathbb C^n$. Moreover,
  the branch locus $B_\pi$ is a union of hyperplanes
  $B_\pi = \cup_{H\in \Ac} H$ and the Jacobian is of the form
\begin{displaymath}
  J_\pi =  \prod_{H\in \Ac} \alpha_H^{e_H-1},
\end{displaymath}
where $e_H$ is the order of the inertia group (which is a cyclic
group) of the hyperplane $H$ that is defined by the linear form
$\alpha_H$.

Letting $\Theta_\pi$ be the set of all divisors of $J_\pi$ it is
natural to look in $\Theta_\pi$ for generators of $\pi_+(B)$ and/or
its simple submodules.
 
\begin{proposition}
  \begin{displaymath}
    \pi_+(B)= \sum_{\theta \in \Theta_\pi} \Dc_A \theta \tr_{B/A}\otimes 1
    \otimes 1.
\end{displaymath}
\end{proposition}
\begin{remark}
  When $G= S_n$ it suffices to take only the subset of $\Theta_\pi$
  formed by the divisors that occur as Jacobians of the totally
  ramified morphisms $X\to Z_i $; these are the Specht polynomials,
  and the simples in $\pi_+(B)$ are precisely the modules that are
  generated by such polynomials. On the other hand, if $G$ is a Weyl
  group of type $D_{2n}$, then there exist simple modules that are not
  generated by a divisor of $J_\pi$ (see \Example{\ref{weyl-d4n}}).
\end{remark}
\begin{proof}Let $k[\partial_1, \ldots , \partial_n]$ be the ring of
  differential operators of $B$ with constant coefficients. Then the
  space of harmonic polynomials is of the form
  \begin{displaymath}
    \Hc =  k[\partial_1, \ldots , \partial_n]\cdot J = \frac{B}{B_+^G},
  \end{displaymath}
  so that $B= A \Hc$. Now the assertion follows since the set
  $ k[\partial_1, \ldots , \partial_n] \cdot J$ is contained in the
  $k$-linear hull of the set of all divisors of $J$.\footnote{This
    argument is due to R. Bögvad.}
\end{proof}

\subsection{Canonical filtrations for symmetric
  groups}\label{symmetric-groups}
Let $\pi: X= \Cb^n\to Y = \Cb^n$ be the invariant map
$A= B^{S_n} \to B = \Cb[x_1, \ldots , x_n]$. Let $\Tc_n$ be the set of
partitions $P= \{P_i\}_{i=1}^r$ of the set $[n]= \{1, \ldots, n\}$, so
that $P_i \cap P_j = \emptyset$, $i \neq j$, $\cup P_i = [n]$, and
$P_i\neq \emptyset$, $1\leq i \leq r$. Let $\Pc_n$ be the set of
partitions $\lambda$ of the integer $n$, $\lambda\vdash n$. So that
given $P\in \Tc_n$ we get $\lambda_P\in \Pc_n$, where
$\lambda_P(i) = |P_i|$. We order the partitions
$P= (P_1, \ldots , P_r)\in \Tc_n$ and $\lambda_P \in \Pc_n$ so that
$\lambda_P(i)\geq \lambda_P(i+1) \geq 1$, so that $P$ is given by a
Young tableaux of shape given by the Young diagram of $\lambda_P$.

Say that $P\in \Tc_n$ is a {\it refinement} of $Q\in \Tc_n$ if there
exists a partition $I=\{I_j\}$ of $[r]$ such that
$Q_j = \cup_{i \in I_j} P_{i} $, and then write $Q \succ P$. This
defines a partial order in $\Tc_n$. Similarly, $ \mu\succ \lambda$ if
for a there exists a partition $I$ such that
$\mu_j = \sum_{i \in I_j} \lambda_i $. Then clearly,
$Q\succ P\Rightarrow \lambda_Q \succ \lambda_P$. We call the partial
order $\succ$ (on $\Tc_n$ and $\Pc_n$, respectively) the
specialization order, as explained by the proposition below.

Given $P\in \Tc_n$ we can associate the linear subspace
$\bar X_P = \cap_iX_{P_i}$, where
$\bar X_{P_i}= \{x_k= x_l\}_{k,l \in P_i}$. Then
$\pi (\bar X_Q) = \pi(\bar X_P)$ if and only if $ Q= g P$ for some
element $g\in S_n$, and we put $\bar Y_\lambda = \pi(\bar X_P) $,
where $\lambda = \lambda_P$.

Recall that the dominance order $ \lambda \unlhd \mu$ is defined by
the condition 
$\sum_{1\leq i \leq j}\mu_i\geq \sum_{1\leq i \leq j}\lambda_i$ for
all $j \geq 1 $.

\begin{proposition}
  \begin{enumerate}
  \item The closure of the stratas of the canonical stratification of
    $\pi$ are given by $\{\bar X_{P}\}_{P\in \Tc_n}$ and
    $\{\bar Y_\lambda\}_{\lambda \in \Pc_n}$, where
    $\bar X_\lambda = \pi^{-1}(\bar Y_\lambda)= \cup_{\lambda_P =
      \lambda} \bar X_{P}$, and $\pi(\bar X_{P})= \bar Y_{\lambda}$.
  \item The specialization order of the strata
    $\{\bar X_P\}_{P\in \Tc_n}$ is the same as the specialization
    order on $\Tc_n$, that is
    \begin{displaymath}
      \bar X_P \subset \bar X_Q  \quad \iff \quad Q \succ P.
\end{displaymath}
\item The specialization order on $\{Y_\lambda\}_{\lambda \in \Pc_n}$ is
  the same as the specialization order on $\Pc$, that is
\begin{displaymath}
  \bar Y_\lambda\subset \bar
  Y_\mu \quad \iff \quad  \lambda \succ \mu.
\end{displaymath}
\item The specialization order on $\Pc_n$ is coarser
  then the dominance order, so that if $ \lambda \succ \mu$, then
  $ \lambda \unlhd \mu$.
  \end{enumerate}
\end{proposition}
We leave out the straightforward proof.

The semisimple decomposition of the $\Dc_Y$-module $N= \pi_+(\Oc_X)$
can be described as follows. The Specht polynomials, which are
jacobians of the map $X\to Z_Q$ computed in canonical homogeneous
coordinates, is of the form
$s_Q(x) =j_{X/Z_Q}= \prod \Delta_{Q_i}(x)$, where the van der Monde
determinant $\Delta_{Q_i}(x)= \prod_{r,s\in Q_i, r< s} (x_r-x_s)$,
corresponding to the partition $Q$. The partition has the shape
$s(Q)= \bar \lambda$, i.e. $\bar \lambda_i = |Q_i|$, so that, as is
customary, the sets $Q_i$ form the columns of a tableaux of shape
$\lambda$.\footnote{This means that $ks_Q$ is the signature
  representation of the Young group $G_Q$.} Here $\bar \lambda$ is the
conjugate partition of $\lambda$, so that the columns of the Young
diagram of $\lambda$ are the rows of the diagram of $\bar \lambda$.
Putting $\omega_Q = s_Q\tr_{X/{Z_Q}} \otimes 1\otimes 1 \in N $,
$\Dc_Y \omega_Q$ is a simple $\Dc_Y$-submodule of $N$, and
  \begin{displaymath}
    N= \bigoplus_{\lambda \vdash n} M_\lambda =     \bigoplus_{\lambda \vdash n}\bigoplus_{s(Q)= \lambda} \Dc_Y \omega_Q,
\end{displaymath}
where $M_\lambda$ is the isotypical component of $N$ that corresponds
to the irreducible representation $V_\lambda = k[S_n] s_P$ (see
\cite[Prop. 4.6 ]{kallstrom-bogvad:decomp},
\Theorems{\ref{galois-direct}}{\ref{imprimitive}}).

Our main interest in this section is in the condition
\begin{displaymath}
  M_\lambda \subset N_\mu,
\end{displaymath}
where $N_\mu$ is the canonical module corresponding to the stratum $Y_\mu$, so
we want to describe a combinatorial relation between the partititions
$\lambda = (\lambda_1, \ldots , \lambda_t), \mu = (\mu_1, \ldots, \mu_s)\in
\Pc_n$, written in decreasing order.

If $k$ is an integer such that $0\leq k \leq \lambda_1$, we get a partition
$\lambda(k)\in \Pc_{n-k}$ by replacing $\lambda_1$ by $\lambda_1-k$, and then reorder in
decreasing order. In more detail,
\begin{displaymath}
\lambda(k) = (\lambda_2, \ldots ,\lambda_i,  \lambda_1 -
k,\lambda_{i+1}, \ldots , \lambda_t ),
\end{displaymath}
where $i$ is the greatest index such that $\lambda_i > \lambda_1 - k $, or if
$\lambda_i \leq \lambda_1-k$, $2 \leq i \leq t$, then
$\lambda(k) = (\lambda_1 - k, \lambda_2, \ldots ,\lambda_i, , \ldots , \lambda_t
)$.
Now if $0 \leq \mu_1 \leq \lambda_1$, put
$\lambda^{(1)}= \lambda (\mu_1)\vdash n- \mu_1$, and inductively, if
$0 \leq \mu_i \leq \lambda^{(i-1)}_1$, put
\begin{displaymath}
  \lambda^{(i)} = \lambda^{(i-1)}(\mu_i) \vdash n- (\mu_1+ \ldots + \mu_i).
\end{displaymath}
We emphasise that the sequence $ \lambda^{(i)}\in \Pc_{n_i }$ of partitions of
the integers $n_i = n- (\mu_1+ \ldots + \mu_i)$ depend on $\mu$, and the process
stops when $\mu_i > \lambda_1^{(i-1)}$. Define the subset
\begin{displaymath}
  \Pc_n ^\mu  = \{\lambda \in \Pc \ \vert \ \exists i \quad
  \text{s.t.}\quad \mu_i > \lambda^{(i-1)}_1 \}.
\end{displaymath}
For example, if $\mu_1 > \lambda_1$, then $\lambda \in \Pc_n^\mu$, while
$\mu \not \in \Pc_n^\mu$. 
\begin{proposition}\label{prop-domorder}
  \begin{displaymath}
\text{If} \quad \lambda \unlhd \mu \quad \text{and}\quad   \lambda \neq \mu \quad
\text {then}\quad \lambda \in \Pc_n^\mu.
  \end{displaymath}
\end{proposition}
\begin{proof}
  By assumption there exists an integer $l$ such that
  $\sum_{1\leq i < l}\mu_i = \sum_{1\leq i < l}\lambda_i$ and
  $\sum_{1\leq i \leq l}\mu_i > \sum_{1\leq i \leq l}\lambda_i$. This
  implies that $\mu_l > \lambda_l$ and $\mu_i= \lambda_i$, $1 \leq i <
  l$. It now easily follows that $\mu_l > \lambda_1^{(l-1)}$.
\end{proof}
\begin{remark}  
Since for example $\mu_1 > \lambda_1$ does not imply
  $ \lambda \unlhd \mu $ one cannot  turn around the implication in \Proposition{\ref{prop-domorder}}. 
\end{remark}

\begin{theorem}\label{symmetric-canonical}
  Let $M_\lambda$ be the isotypical component of $N= \pi_+(\Oc_X)$
  that corresponds to the partition $\lambda\vdash n$ and $N_\mu$ be
  the maximal submodule of $N$ such that $\RG_{Y_\mu}(N_\mu)=0$,
  where $Y_\mu$ is the stratum corresponding to the partition
  $\mu \vdash n$. Then
\begin{displaymath}
  N_\mu = \bigoplus_{\bar \lambda \in \Pc_n^\mu } M_\lambda.
\end{displaymath}
\end{theorem}
\begin{proof}
By \Theorem{\ref{can-filt}}
\begin{displaymath}
  N_\mu = \bigcap _{s(P)= \mu} \Tc_P,
\end{displaymath}
where $P$ is a standard Young tableaux of shape $\mu$,  indexing the flats $X_P$ lying over
$Y_\mu$. The inertia module 
\begin{displaymath}
  \Tc_P = \{m \in \pi_+(\Oc_X) \ \vert \ \Tr_{G_P}(m)=0 \},
\end{displaymath}
where the inertia group $G_P$ of $X_P$ is the Young group
$G_P =S (P_1)\times \cdots \times S(P_n)$.  Therefore
\begin{align*}
  \Dc_Y \omega_Q\subset N_\mu \quad &\iff \quad  \Tr_{G_P}(\omega_Q)=0 \quad
  \text{when} \quad s(P)= \mu \\&\iff   \quad  \Tr_{G_P}(s_Q)=0\quad  \text{when} \quad s(P)= \mu.
\end{align*}

We assert:
\begin{displaymath}
  \Tr_{G_P}(s_Q)=0 \iff \exists (i,j)\quad  \text{s.t.}\quad |P_j\cap Q_i|\geq 2.
\end{displaymath}
If $|P_j\cap Q_i|\geq 2$, then $G_P\cap G_Q$ contains a transposition
$\tau$ of two elements in $Q_i$, so that $\Tr_{C_2}(s_Q)=0$ if
$C_2= <\tau>$, implying that $\Tr_{G_P}(s_Q)=0$.

Conversely, if $ |P_j\cap Q_i|\leq 1$ for all $(i,j)$ and
$\sigma\cdot s_Q\neq \sigma_1\cdot s_Q$, then
$\sigma\cdot s_Q = \prod_j \sigma \cdot \Delta_{Q_j} $ contains a
polynomial factor of degree $\geq 1$ that does not occur in
$\sigma_1 s_Q$. This implies that for each $i$
\begin{displaymath}
\Tr_{S(P_i)}(s_Q)= \sum_{\sigma \in S(P_i) } \sigma \cdot s_Q \neq 0, 
\end{displaymath}
and therefore, since $P_i\cap P_{i'}= \emptyset$ when $i\neq i'$,
\begin{displaymath}
  \Tr_{G_P}(s_Q)=  \Tr_{S(P_n)}\Tr_{S(P_2)}\Tr_{S(P_1)}(s_G)\neq 0.
\end{displaymath}
It remains to prove
\begin{displaymath}
\bar \lambda \in \Pc_n^\mu \iff  \exists (i,j)\quad  \text{s.t.}\quad |P_j\cap Q_i|\geq 2 .
\end{displaymath}
$\Rightarrow$: Assume that $ |P_l\cap Q_j| \leq 1 $ and that
$ \mu_j \leq \bar \lambda_1^{(j-1)}$ when $j < i$ and all $l$, while
$\bar \lambda^{(i-1)}_1 < \mu _i$. Let $\bar P = \{\bar P_s\}$ be the
conjugate partition of $P$; this is the partition of $[n]$ that is
formed from the columns of the tableaux of $P$ (where the rows of the
original tableaux are formed from the subsets $P_i$). The assumption
implies that each set $Q_j$, $j < i$, belongs to a single subset
$\bar P_r$. Now the condition $\bar \lambda^{(i-1)}_1 < \mu _i$
implies that $Q_i $ cannot belong to a single subset $\bar P_r$ for
all $r$, and therefore there exists an integer $r$ such that
$|P_r \cap Q_i|\geq 2$.

$\Leftarrow$:    That $ |P_j\cap Q_i| \geq 2$  for some  $(i,j)$
implies that $Q_i$ does not belong to a single subset $\bar P_r$ for
all $r$.   This implies that $\mu_i > \bar \lambda ^{(i-1)}_1$. 
\end{proof}
We can now determine which isotypical components are added in the
canonical filtration upon adding a deeper stratum. These are described
combinatorially by subsets of $\Pc_n$ of the form
\begin{displaymath}
\Phi_\mu = \Pc_n^\mu \setminus \bigcup_{\mu' \succ \mu} \Pc_n^{\mu'}.
\end{displaymath}

\begin{corollary}\label{cor_canonical}
  Put $N_{\mu' \succ \mu} = \sum_{\mu ' \succ \mu} N_{\mu'} $, the
  maximal submodule of $N=\pi_+(\Oc_X)$ with vanishing local
  cohomology along all strata $Y_{\mu'}$ that dominate $Y_\mu$ (and
  $Y_{\mu'} \neq Y_{\mu}$). Then $N_{\mu' \succ \mu} \subset N_\mu$
  and we have
  \begin{displaymath}
    \frac {N_\mu}{N_{\mu' \succ \mu}} = \bigoplus_{\bar \lambda \in \Phi_\mu
    } M_{\lambda}.
  \end{displaymath}
\end{corollary}
One would perhaps expect that the isotypical decomposition of $N$ in
some way can be mapped to its canonical decomposition, but this turns
out not to be the case, since the canonical filtration need not be
strictly increasing and therefore more than one isotypical component
may be added when including a deeper stratum. The smallest $n$ when
this happens is $n=6$.
\begin{example} A partition $\lambda$ is the index of a stratum
  $Y_\lambda$ in the canonical stratification, where the dimension
  $\dim Y_\lambda $ of a stratum is given by the length of $\lambda$,
  so that e.g. $\dim Y_{(2,2)}= 2$. For $n= 4$, the specialization
  order is described by the following diagram:
  
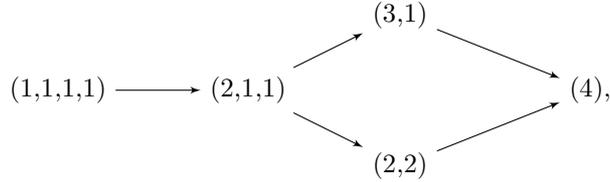
\begin{figure}[!h]
    \centering
    \begin{tikzpicture}
    \tikzset{edge/.style = {->,> = latex'}}
    \node (a) at (0,0) {(1,1,1,1)}; \node (b) at (2.5,0)
    {(2,1,1)}; \node (c1) at (4.5,1) {(3,1)}; \node (c2) at
    (4.5,-1) {(2,2)}; \node (d1) at (7,0) {(4),};
    \draw[edge] (a) to (b); \draw[edge] (b) to (c1); \draw[edge] (b)
    to (c2); \draw[edge] (c1) to (d1); \draw[edge] (c2) to (d1);
    \end{tikzpicture}
  \end{figure}

  \noindent
  where the leftmost partition corresponds to the generic point in
  $Y$, the next is the generic point of the discriminant locus, and so
  on. The arrows denote specialization, where we notice that two of
  the strata are related in the dominance order but not the
  specialization order. To exemplify
  \Theorem{\ref{symmetric-canonical}}, we have
  $N_{(3,1)}= M_{(2,2)}\oplus M_{(3,1)}\oplus M_{(4)}$ and
  $N_{(2,2)}= M_{(2,1,1)}\oplus M_{(3,1)}\oplus M_{(4)}$. Here
  $(2,2) \unlhd (3,1)$ while $(2,2) \not \succ(3,1)$. Here
  $\Phi_{(4)}= \{(3,1),(2,2)\}$ while $\Phi_\mu$ consists of a single
  element when $\mu \neq (4)$, so that one then adds a single
  isotypical component in \Corollary{\ref{cor_canonical}}. For $n=5$
  we have, e.g.,
  $N_{(3,2)} = \oplus_{\lambda \in \Pc_{(3,2)}} M_{\bar \lambda}$ ,
  where $ \Pc_{(3,2)}= \{(1,1,1,1,1),(2,1,1,1), (3,1,1),(2,2,1)\} $.
  Here $(2,2,1) \unlhd (3,1,1) $, while $(2,2,1) \not \succ (3,1,1) $.
  The graph describing the specialization order for $n= 6$ is as
  follows:
  \begin{figure}[!h]
  \centering
  \begin{tikzpicture}
    \tikzset{edge/.style = {->,> = latex'}}
    \node (a) at (0,0) {(1,1,1,1,1,1)}; \node (b) at (2.5,0)
    {(2,1,1,1,1)}; \node (c1) at (4.5,1) {(3,1,1,1)}; \node (c2) at
    (4.5,-1) {(2,2,1,1)}; \node (d1) at (7,1.5) {(4,1,1)}; \node (d2)
    at (7,0) {(3,2,1)}; \node (d3) at (7,-1.5) {(2,2,2)}; \node (e1)
    at (9,1.5) {(5,1)}; \node (e2) at (9,0) {(4,2)}; \node (e3) at
    (9,-1.5) {(3,3)}; \node (f) at (11,0) {(6).};
    \draw[edge] (a) to (b); \draw[edge] (b) to (c1); \draw[edge] (b)
    to (c2); \draw[edge] (c1) to (d1); \draw[edge] (c1) to (d2);
    \draw[edge] (c2) to (d2); \draw[edge] (c2) to (d3); \draw[edge]
    (d1) to (e1); \draw[edge] (d1) to (e2); \draw[edge] (d2) to (e1);
    \draw[edge] (d2) to (e2); \draw[edge] (d2) to (e3); \draw[edge]
    (d3) to (e2); \draw[edge] (d3) to (e3); \draw[edge] (e1) to (f);
    \draw[edge] (e2) to (f); \draw[edge] (e3) to (f);
  \end{tikzpicture}
\end{figure}

\noindent
By \Theorem{\ref{symmetric-canonical}}
we have, for example,
\begin{align*}
  N_{(4,1,1)}&= \bigoplus_{\lambda \in \Pc_{(4,1,1)}} M_{\bar \lambda},
               \quad \text{where}\\
   \Pc_{(4,1,1)} &= \{(1,1,1,1,1,1), (2,1,1,1,1),
(3,1,1,1),(2,2,1,1),(3,2,1), (2,2,2),(3,3)\}.
\end{align*}
Notice that even though $\mu=(2,2,2) $ and $\lambda =(3,1,1,1) $ are
unrelated with respect to the dominance order we have
$\lambda \in \Pc_{\mu}$ ($\lambda^{(1)}= (1,1,1,1)$, so that
$\lambda^{(1)}_1 = 1 < 2= \mu_2$) and therefore
$M_{\bar \lambda}= M_{(4,1,1)} $ is contained in $ N_{(2,2,2)}$.
In \Corollary{\ref{cor_canonical}} one adds two isotypical components
upon adding the stratum $Y_{(2,2,2)}$, since $\Phi_{(2,2,2)}= \{(2,2,1,1), (3,1,1,1)\}$.
\end{example}

\subsection{Presentation of exponential modules for complex reflection
  groups}\label{pres-exp}
Let $G$ be a complex reflection group in $\Glo(V)$, where $V$ is a
finite-dimensional $k$-space, and put $X= \Spec \So(V)$ and $Y= \Spec \So(V)^G$.
An element $\lambda \in V^*$ defines a maximal ideal
$\mf_\lambda \subset \So(V)$ and the exponential $\Dc_X$-module
$E_\lambda= \Dc_X e^\lambda = \Dc_X/ (\Dc_X \mf_\lambda) $ (here
$e^\lambda = 1 \omod \Dc_X\mf_\lambda$), which is an invertible $\Oc_X$-module.
Now put $\pfr_\lambda = \So(V)^G\cap \mf_\lambda \subset \Dc_Y$ and
\begin{displaymath}
M_\lambda = \frac {\Dc_Y}{\Dc_Y \pfr_\lambda }.
\end{displaymath}

\begin{theorem}\label{semsimplemod} If  $G$ is  a complex reflection
  group, then 
  \begin{displaymath}
M_\lambda    \cong  \pi_+(E_\lambda).
  \end{displaymath}
\end{theorem}
By a theorem of Steinberg \cite{steinberg:differential} the isotropy
group $G_\lambda$ of $\lambda$ is again a complex reflection group.
Then \Theorem{\ref{galois-direct}} gives an abstract decomposition of
$M_\lambda$, where the inertia group $G_{E_\lambda} = G_\lambda $,
which in (\ref{simple-imprimitive}) will be made explicit when $G$ is
an imprimitive complex reflection group and $\lambda=0$.

\Theorem{\ref{semsimplemod}} complements a result due to Levasseur and
Stafford \cite [Th. 4.4]{levasseur-stafford:semisimplicity}, stating
that $\Dc_A^\pi/ \Dc_A^\pi\pfr_\lambda\cong \pi_*(E_\lambda)$ as
$\Dc_A^\pi$-modules. In fact, a transcription of the first part of the
proof below also recovers [loc. cit.] in a rather direct and explicit
way.\footnote{In step (1) one can work with $A\subset B$ and
  $\Dc_A^\pi\subset \Dc_B$ instead of $K \subset L$ and
  $\Dc_K \subset \Dc_L$, proving that
  $B\otimes_A \Dc^\pi_A/(\Dc^\pi_A\nf_\lambda) = k[G_\lambda]\otimes_k
  E_\lambda $ and then
  $\Dc_A^\pi/(\Dc_A^\pi \pfr_\lambda) = \pi_*(E_\lambda) $. Steps (2)
  and (3) are then not required.}

\begin{remark}
  \begin{enumerate}
  \item It follows from the proof that there exists an injective
    homorphism
      \begin{displaymath}
      \pi_+(L\otimes_BE_\lambda)\hookrightarrow \frac {\Dc_K}{\Dc_K \pfr_\lambda},
    \end{displaymath}
    for {\it any} subgroup $G \subset \Glo (V)$, and that this is an
    isomorphism if and only if the inertia group $G_\lambda$ is
    generated by pseudo-reflections of $V$. The proof also gives that
    the inverse image $L\otimes_K\Dc_K/(\Dc_K \pfr_\lambda)$ is
    semisimple, and therefore $\Dc_K/(\Dc_K \pfr_\lambda)$ is always
    semisimple \Th{\ref{semisimple-inv}}.
  \item We do not get a  canonical isomorphism in
    \Theorem{\ref{semsimplemod}} or even a canonical cyclic generator
    of $\pi_+(E_\lambda)$. For instance, the canonical homomorphism
    \begin{displaymath}
      M_\lambda \to \pi_+(E_\lambda), \quad 1\omod \Dc_Y
      \pfr_\lambda \mapsto \tr_{X/Y}\otimes e^\lambda
    \end{displaymath}
    is not an isomorphism since
    $(\Dc_X \mf_\lambda) \cap \Dc_Y \not \subset \Dc_Y \pfr_\lambda$. In
    the case $\lambda=0$ one asks for a canonical generator of
    $\pi_+(\Oc_X)$ as $\Dc_Y$-module, see \Theorem{\ref{normalbasis}},
    and
    $(\Dc_X \mf_0) \cap \Dc_Y = \Dc_Y T_Y \not \subset \Dc_Y \pfr_0 =
    \Dc_Y \So(V)^G_+$.
  \end{enumerate}

\end{remark}

\begin{lemma}\label{decomp-lemma}
  Let $\mf$ be a maximal ideal in $\So(V)\subset \Dc_B$, $J$ be an
  $\mf$-primary ideal, and put $t= \dim_k \So(V)/J$. The dimension of
  the invariant space $(\Dc_B/\Dc_B J)^{\mf}$ equals $t$, and if
  $v_1, \ldots , v_t$ is a basis of $(\Dc_B/\Dc_B J)^{\mf}$, then
  \begin{displaymath}
    \frac {\Dc_B}{\Dc_B J}=   \bigoplus_{i=1}^t \Dc_B v_i  = \bigoplus_{i=1}^t
    \frac {\Dc_B}{\Dc_B \mf},   
  \end{displaymath}
  where $\Dc_B/(\Dc_B \mf)$ is simple.
\end{lemma}
It would be interesting to find explicit differential operators 
$P_i\in \Dc_B$ that represent the basis $v_i$ and thus an explicit
decomposition of $\Dc_B/(\Dc_B J)$ when $J= \So(V)\nf_\lambda$ and
$\nf_\lambda = \So(V)^{G_\lambda}\cap \mf_\lambda$. In
\cite{kallstrom:arkiv} this is accomplished in terms of ``Pochhammer''
differential operators for modules of the form $\Dc_B/(\Dc_B \mf^n)$.

\begin{proof} 
  The decomposition follows from Kashiwara's restriction theorem, see
  \cite{kallstrom:arkiv}. Letting $M$ be any $\mf$-primary
  $\So(V)$-module of finite dimension $t$ we prove by induction in $t$
  that the length of $\Dc_B\otimes_{\So(V)}M$ equals $t$. Let
  $\{M_i\}_{i=1}^t$ be a decomposition series of $M$, so that
  $\dim_k M_{i+1}/M_i =1$ (when $M$ is Gorenstein, as in our main
  example, then the first term is the socle $M_1= M^\mf$). The
  $\So(V)$-module $\Dc_B$ is free (see [loc. cit.]), so that if we
  apply the exact functor $\Dc_B\otimes_{\So(V)}\cdot $ to the exact
  sequence
  \begin{displaymath}
0 \to     M_i \to M_{i+1} \to \frac {M_{i+1}}{M_i} \to 0
  \end{displaymath}
  and notice that the modules
  $\Dc_B\otimes_{\So(V)}M_1 = \Dc_B \otimes_{\So(V)} M_{i+1}/M_i
  =\Dc_B/(\Dc_B\mf) $ are simple; since $\dim M_i < t$, the assertion
  $\ell(\Dc_B\otimes_{\So(V)}M) = \dim_k M$   follows by induction.
\end{proof}

\begin{lemma}\label{jacobson}
  Let $\{n_1, \ldots , n_r\} \in \So(V)^G$ be a linearly independent
  subset and $\{m_1, \ldots , m_r\}\subset \So(V)^G $ be another
  subset. Then there exists $a\in \So(V^*)^G$ such that $n_i(a)= m_i$,
  $i= 1, \ldots , r$.
\end{lemma}
\begin{proof}
  The action of constant coefficient differential operators $\So(V)$
  on $\So(V^*)$ also gives $\So(V)$ a structure as simple
  $\So(V^*)$-module.  Since  $\{n_i\}$ is linearly independent it
  follows from the density theorem that there exists $b\in \So(V^*)$
  such that $n_i(b)= m_i$.  Since $n_i, m_i \in \So(V)^G$ it follows that
$n_i(a^g)= m_i $. So that putting 
  $ a =\frac 1{|G|} \Tr_G(a)= \sum_{g\in G } a^g\in \So(V^*)$, we also
  have  $n_i(a)= m_i$.
\end{proof}

\begin{pfof}{\Theorem{\ref{semsimplemod}}}
  Since $G_\lambda$ is a complex reflection group $\So(V)^{G_\lambda}$
  is again a polynomial ring and
  $\nf_\lambda = \So(V)^{G_\lambda} \cap \mf_\lambda $ is a maximal
  ideal in $\So(V)^{G_\lambda}$ such that the fibre
  $\frac{\So(V)^{G_\lambda}}{\nf_\lambda}\otimes_{\So(V)^{G_\lambda}}
  \So(V) = k[G_\lambda]$. Moreover, the map
  $\So (V)^G \to \So(V)^{G_\lambda}$ is étale of degree
  $r= |G/G_\lambda|$ and
  $\So(V)^{G_\lambda}\pfr_\lambda = \prod_{g_i\in G/G_\lambda}
  \nf_{g_i \cdot \lambda}$, where $\nf_{g_i \cdot \lambda}$ are
  translations of the ideal
  $\nf_\lambda= \So(V)^{G_\lambda}\cap \mf_\lambda \in \Spec
  \So(V)^{G_\lambda}$ and the $g_i$ are $G_\lambda$ transversals in
  $G$. This implies
  \begin{align*}
    \frac{    \So(V)}{\So(V)\pfr_\lambda} &= \frac{\So(V)^G}{\pfr_\lambda} \otimes_{\So(V)^G} \So(V)^{G_\lambda}
                                            \otimes_{\So(V)^{G_\lambda}}    \So(V)
                                            =(\frac{\So(V)^{G_\lambda}}{
                                            \So(V)^{G_\lambda}\pfr_\lambda })
                                            \otimes_{\So(V)^{G_\lambda}}
                                            \So(V)\\
                                          & = \bigoplus_{g_i \in
                                            G/G_\lambda}
                                            (\frac{\So(V)^{G_\lambda}}{\nf_{g_i
                                            \cdot \lambda} }) \otimes_{\So(V)^{G_\lambda}}
                                            \So(V) =  \bigoplus_{g_i \in
                                            G/G_\lambda} \frac
                                            {\So(V)}{\So(V)\nf_{g_i\cdot
                                            \lambda}}. \tag{\&}
  \end{align*}
  
  Next, \Lemma{\ref{decomp-lemma}} implies, noting that
  $\Dc_B= B\otimes_k \So(V)$ as $(B, \So(V))$-bimodule,
  \begin{align}\notag
\Dc_L\otimes_{\So(V)}\frac{\So(V)}{\So(V)\nf_\lambda} &=  L\otimes _B
B\otimes_k \So(V)\otimes_{\So(V)}\frac{\So(V)}{\nf_\lambda}= L\otimes_B \frac{\Dc_B}{\Dc_B
  \nf_\lambda}\\ &=   \bigoplus_{i=1}^t  L\otimes_B \Dc_B v_i =
                   \bigoplus_{i=1}^t  L\otimes_BE_\lambda, \tag{\#}
\end{align}
where $t= \dim_k \So(V)/(\So(V)\nf_\lambda) \geq |G_\lambda|$ and
$v_1, \ldots, v_t$ is a $k$-basis of the $\mf$-invariant space
$(\Dc_B/\Dc_B \nf)^{\mf}$. Here $t= |G_\lambda| $ if and only if
$G_\lambda$ is a complex reflection group (see \cite{Kane}*{Sec.
  17.5}). Since by assumption $G$ is a complex reflection group,
$G_\lambda$ is a complex reflection group.
  
At the generic point $\eta$ in $\Spec \So(V)^G$ we have
$(M_\lambda)_\eta = \Dc_K/\Dc_K \pfr_\lambda$. By \thetag{\&} and
\thetag{\#} its inverse image to
$\Dc_L$-module is
  \begin{align} 
    \pi^!(\frac
    {\Dc_K}{\Dc_K \pfr_\lambda})&=     L\otimes_K\frac {\Dc_K}{\Dc_K \pfr_\lambda} = L\otimes_K\Dc_K \otimes_{\So(V)^{G}}
                                                                                    \frac{\So(V)^{G}}{\pfr_\lambda} \notag
                                                                                    = L \otimes_k \So(V) \otimes_{\So(V)^{G}}
                                                                                    \frac{\So(V)^{G}}{\pfr_\lambda}
    \\ &= L\otimes_k \So(V) \otimes_{\So(V)} \frac{\So(V)}{\So(V)\pfr_\lambda}
         =  \bigoplus_{g_i \in
         G/G_\lambda}  \Dc_L\otimes_{\So(V)}\frac
         {\So(V)}{\So(V)\nf_{g_i\cdot
         \lambda}} \notag \\ &= \bigoplus_{g_i \in
                        G/G_\lambda}\bigoplus_{i=1}^t
                               L\otimes_BE_{g_i\cdot \lambda}.     \tag{*}
  \end{align}
Together with \Theorem{\ref{inv-dir}} we get
  \begin{align*}
    \pi^!\pi_+(E_\lambda)&=  \bigoplus_{g\in G} (L\otimes_K
                           E_\lambda)_g = \bigoplus_{g_i  \in G/
                           G_\lambda}  \bigoplus_{j=1}^t (L\otimes_K
                           E_\lambda)_{g_i} = \bigoplus_{g_i  \in G/
                           G_\lambda} \bigoplus_{j=1}^t (L\otimes_K
                           E_{g_i \cdot \lambda})\\
    &=     \pi^!(\frac      {\Dc_K}{\Dc_K \pfr_\lambda}).
  \end{align*}
  Therefore, by adjointness,
  \begin{displaymath}
    Hom_{\Dc_K}(\pi_+(L\otimes_BE_\lambda), \pi_+(L\otimes_BE_\lambda)) =     Hom_{\Dc_K}(\pi_+(L\otimes_BE_\lambda), \frac {\Dc_K}{\Dc_K \pfr_\lambda}).
  \end{displaymath}
Since $K\otimes_A\pi_+(E_\lambda)= \pi_+(L\otimes_B E_\lambda)$ is semisimple this implies that
$K\otimes_A\pi_+(E_\lambda)$ can be identified with a submodule of $\Dc_K/\Dc_K \pfr_\lambda$.  Since  moreover
\begin{displaymath}
  \rank_K  K\otimes_A\pi_+(E_\lambda) = \rank_L  \pi^!\pi_+(L\otimes_BE_\lambda) = \rank_L
  \pi^!(\frac {\Dc_K}{\Dc_K \pfr_\lambda}) = \rank_K (\frac {\Dc_K}{\Dc_K \pfr_\lambda}),
\end{displaymath}
it follows that  $K\otimes_A\pi_+(E_\lambda) =\Dc_K/(\Dc_K \pfr_\lambda) $.

(2) $M_\lambda$ is torsion free: We need to prove that if
  $f\in A, P\in \Dc_A $ and $fP\in \Dc_A\nf_\lambda$, then
  $P\in \Dc_A\nf_\lambda$. In other words, if
  \begin{displaymath}
    fP = \sum_i Q_i s_i, \quad s_i \in \nf_\lambda, 
  \end{displaymath}
  where we can assume that $\{s_i\}$ is a linearly independent subset
  of $\So(V)$, we need to prove that $Q_i \in f\Dc_A$.

  By Chevalley's theorem there exists a subset
  $\{y_i\}\subset \So(V^*)^G$ such that $A= \So(V^*)^G\cong k[y_i]$;
  let $\{\partial_{y_i}\}\subset T_A$ such that
  $\partial_{y_i}(y_j) = \delta_{ij}$. When $\alpha$ and $\beta $ are
  multiindices $ [n]\to \Nb$ we write $\beta < \alpha$ when $\alpha $
  is greater in the lexicographic ordering, and put
  $y_1^\alpha \cdots y_n^{\alpha_n}\in A$,
  $\partial^\alpha
  = \partial_{y_1}^{\alpha_1}\cdots \partial_{y_n}^{\alpha_n}\in
  \Dc_A$, and $\alpha ! = \alpha_1 ! \cdots \alpha_n! $. Now expand
  \begin{displaymath}
    Q_i  = \sum_{\alpha \in \Ac_i} b_\alpha \partial^\alpha, 
  \end{displaymath}
  where $\Ac_i$ is a set of multiindices, and $b_{\alpha } \in A$. By
  \Lemma{\ref{jacobson}} there exist for each monomial
  $y^{\alpha} \in \So(V^*)^G=A$ an element
  $a_{\alpha}(i) \in \So(V^*)^G= A$ such that
  $s_j(a_{\alpha}(i))=y^{\alpha}\delta_{i,j} $. Let $\alpha$ be the
  minimal element of $\Ac_i$, so that $\alpha < \alpha'$ for all
  $\alpha \neq \alpha' \in \cup_i \Ac_i$. Then
  \begin{displaymath}
    f P (a_{\alpha (i)})=  \sum_{j} Q_j s_j(a_\alpha(i)) = \sum_{\beta
      \in \Ac_i} b_\beta\partial^\beta( y^\alpha) =
    b_\alpha \alpha !,
\end{displaymath}
and therefore $ b_{\alpha}\in A f$. Now assume that
$ b_{\alpha}= f \bar b_{\alpha} $ for some $\bar b_{\alpha} \in A$
when $\alpha < \beta \in \Ac_i$. As before select
$a_{\beta(i)} \in \So(V^*)^G= A$ such that
$s_j(a_{\beta}(i))=y^{\beta}\delta_{i,j} $. It follows that
  \begin{displaymath}
    fP (a_{\beta}(i))= \beta ! b_{\beta} + \sum_{ \alpha < \beta
    } f \bar b_{\alpha} \partial^\alpha (y^\beta),
\end{displaymath}
where $ \partial^\alpha (y^\beta) \in A$; hence $b_{\beta} \in f A$.
By induction it follows that $b_{\beta}\in A f$ for all
$\beta \in \Ac_i$, and therefore $Q_i \in f \Dc_A$.

(3): We have localization maps $u: M_\lambda \to K\otimes_AM_\lambda$
and $v: \pi_+(E_\lambda)\to K\otimes_ A\pi_+(E_\lambda)$. The map $u$
is injective by (2), and by \Corollary{\ref{supp-simples}}
$\pi_+(E_\lambda)$ is torsion free, hence $v$ is also injective. Put
$\mu = 1 \omod \Dc_A \pfr_\lambda$, so that $M_\lambda = \Dc_A \mu$
and $\Dc_K u(\mu) = K\otimes_A M_\lambda$. By (1) there exists an
isomorphism
\begin{displaymath}
  \psi_K : K\otimes_A M_\lambda \cong  K\otimes_A \pi_+(E_\lambda),
\end{displaymath}
so that $\Dc_K \psi_K(\mu) = K\otimes_A \pi_+(E_\lambda)$. Since $v$
is injective there exists a non-zero element $h\in A$ such that
$h\psi_K (\mu) \in \pi_+(E_\lambda)$ maps to $h\psi_K(\mu)$. Now since
$\Dc_A h\mu \subset M_\lambda$ and the latter module is semisimple, so
that $\mu$ has a non-zero projection to each simple module in a
semisimple decomposition, we get $\Dc_A h\mu = M_\lambda $, hence the
map that sends $h\mu $ to $ h \psi_K\circ u (\mu)$ defines an
isomorphism $M_\lambda \cong \pi_+(E_\lambda)$.
\end{pfof}

\subsection{Simple $\Dc$-modules for imprimitive complex reflection
  groups}\label{simple-imprimitive}
The irreducible imprimitive reflection subgroup are of the form
$G=G(de,e, n)= A(de,e,n)\rtimes S_n\subset \Glo(V) $, so that if
$(x_i)_{i=1}^n$ is an imprimitive basis of $V$ and $m= de$,
$A(de, e,n)$ is the group of diagonal matrices with entries in the
group $\mu_m$ of $m$th roots of unity whose determinant belongs to
$\mu_d$, while $S_n$ permutes the $x_i$. Letting as before $B= \So(V)$
and $A= B^G$, we aim to explicitly describe the semisimple
decomposition of $\pi_+(B)$ when $\pi : A \to B$ is the inclusion map,
and at the same time present an equally explicit description of
representatives of $\hat G$ as subrepresentations of $B$. 

Let $L$ be the fraction field of $B$, $L_1=L^{A(de,e,n)}$, and
$K= L^G$ be the fraction field of $A$.\footnote{ The invariant ring
  $B^{^{A(de,e,n)}}$ is singular when $e>1$ so it is here natural to
  work with the fraction fields.} We have a factorization
$\pi_0 = q\circ r$
\begin{displaymath}
 \Spec L \xrightarrow{r} \Spec L_1 \xrightarrow{q} \Spec K
\end{displaymath}
of the restriction of $\pi$ to the generic point in $X= \Spec B$.
There exists an exact sequence
$1 \to A(de,e,n) \to \Zb_m^n \to C_e\to 1$, where $m=de$, and
therefore, since $\hat C_e = C_e $ and $\widehat {\Zb_m^n}= \Zb_m^n$,
also the exact sequence
$1 \to C_e \xrightarrow {d} \Zb_m^n \to A(de,e,n)\to 0$, where
$d(1) = (d,\ldots, d)\in \Zb_m^n$. The irreducible
$A(de,e,n)$-representations are thus parametrized by elements
$\alpha \in \Zb_m^n$ modulo the relation
$\alpha \sim \alpha + j d(1)$, for some integer $j$; let $\bar \alpha$
denote the class of $\alpha$ in $\hat A(de,e,n)$. Regarding $\alpha$
as a function $\alpha : [n]\to \{0, 1, \ldots , m-1\}$ we get the
monomial $x^\alpha = \prod_{i=1}^n x_i^{\alpha (i)}$, and can define
the $\Dc_{L_1}$-module $\Lambda_\alpha = \Dc_{L_1} x^\alpha $. Then
$\Lambda_\alpha$ is simple and $\Lambda_\alpha \cong \Lambda_\beta $
if and only if $\beta \sim \alpha$. Concretely, the inclusion
$\Dc_{L_1} x^\beta= \Dc_{L_1} \Phi ^{dj} x^\alpha \subset \Dc_{L_1}
x^\alpha $
is an isomorphism, where $\Phi = \prod_{i=1}^n x_i$, so that
$\Phi^d \in L_1$. It is therefore unambiguous to write
$\Lambda_{\bar \alpha}= \Lambda_\alpha$. It follows that:

\begin{lemma}  
  \begin{displaymath}
    r_+(L) = \bigoplus_{\bar \alpha \in \hat A(de,e,n)} \Lambda_\alpha.
  \end{displaymath}
\end{lemma}
The map $q$ is also Galois, with Galois group $S_n$, so to compute
$(\pi_0)_+(L) = \oplus q_+(\Lambda_\alpha)$ we need the inertia group
$G_\alpha \subset S_n$ of $\Lambda_\alpha$. The symmetric group $S_n$
acts on the elements $\alpha$ by $(g\cdot \alpha)(i) = \alpha (g(i))$,
and $g\cdot \Lambda_\alpha = \Lambda_{g\cdot \alpha}$. Then
\begin{displaymath}
  G_\alpha = \{g \in S_n \ \vert \ \Lambda_{g\cdot\alpha} \cong
             \Lambda_\alpha \} = \{ g \in S_n \ \vert \ g\cdot \alpha - \alpha \in d(C_e)\}.
\end{displaymath}
Let
$Y_\alpha =S(\alpha^{-1}(0))\times S(\alpha^{-1}(1)) \times \cdots
\times S(\alpha^{-1}(m-1))$ be the product of the symmetric groups of
the sets
$\alpha^{-1}(i)= \{l \in [n]\ \vert \ \alpha(l)= i\}\subset [n]$
($Y_\alpha$ is a Young subgroup of $S_n$). Put
$n_i= |\alpha^{-1}(i)|$, so that $\sum n_i = n$.
\begin{lemma}\label{in-lemma}  Put $b_\alpha = \lcm (b_i)$,
  where $b_i$ is the smallest non-zero positive integer such that
  $n_{i + b_i d}= n_i$. There exists an exact sequence
  \begin{displaymath}
1 \to  Y_\alpha \to  G_\alpha   \to C^{\alpha} \to 1
\end{displaymath}
where $C^{\alpha } = <b_\alpha>$ is a subgroup of $ C_e$.
\end{lemma}

Any irreducible representation of $Y_\alpha$ is of the form
$V_P= V_{\lambda_0}\otimes_k \cdots \otimes_k V_{\lambda_{m-1}}$, where for
each integer partition $\lambda_i \vdash n_i$ the irreducible
$S(\alpha^{-1}(i))$-representation $V_{\lambda_i}$ can be realized as
the vector space with basis
$\{s_{P_i}\}\subset \Cb [y_1, \dots , y_{n_i}]$, indexed by standard
Young tableaux $P_i$ of shape $\lambda_i$. The polynomials $s_{P_i}$
are often selected to be Specht or Young polynomials. For a bijection
$[n_i]\to \alpha^{-1}(i), l\mapsto j$, put $y_l= x_j^{m}$,
$s^{(i)}_{P_i}(x) = s_{P_i}(y )$, and
\begin{displaymath}
  s_{P_\alpha} = \prod_{i=0}^{m-1} s^{(i)}_{P_i}.
\end{displaymath}

Let $t_\alpha$ be a lift of $b_\alpha$ to $G_\alpha$ in
\Lemma{\ref{in-lemma}} and 
\begin{displaymath}
  k[t_\alpha] x^\alpha = \bigoplus_{i=0}^{b_\alpha-1} k f_{\alpha,i}  \cong k[t_\alpha]
\end{displaymath}
be a decomposition as $k[t_\alpha]$-module. The simple
$G_\alpha$-modules are then of the form
\begin{displaymath}
V_{P_\alpha,\alpha, i}=kf_{\alpha,i}\otimes_k V_{P_\alpha}.
\end{displaymath}

Now define the following $\Dc_A$-submodules of $\pi_+(B)$:
\begin{displaymath}
  N_{P_\alpha, \alpha, i} = \Dc_A \tr_{B/A}\otimes    f_{\alpha,i}  s_{P_\alpha},  \quad   0 \leq i < b_\alpha.
\end{displaymath}

\begin{theorem}\label{imprimitive}
  The modules  $ N_{P_\alpha, \alpha, i}$ are simple and 
  \begin{displaymath}
    \pi_+(B) = \bigoplus_{\bar \alpha \in \hat A(de,e,n)}\bigoplus_{i=0}^{b_\alpha-1}\bigoplus_{s(P)=\lambda} N_{P_\alpha, \alpha, i},
  \end{displaymath}
  where the sum runs over all sets of standard tabeaux $P$ of shape
  $s(P)= \lambda$, for a set of partitions
  $\lambda= (\lambda_0, \cdots \lambda_{m-1}) $,
  $\lambda_i \vdash \alpha^{-1}(i)$.
\end{theorem}
\begin{proof} Consider first the fraction fields $K\subset L$, and put
  $m_\lambda = \dim_k V_{P_\alpha}$. Then
\begin{align*}
  q_+(\Lambda_\alpha)&=  \Dc_{L_1}^r \otimes_{\Dc_{L_1}[G_\alpha]} \Dc_{L_1}[G_\alpha]
  \otimes_{\Dc_{L_1}}  \Lambda_\alpha = \bigoplus_{i=0}^{b_\alpha-1}\bigoplus_{s(P)=\lambda} \Dc_{L_1}^r
  \otimes_{\Dc_{L_1}[G_\alpha]}V_{P_\alpha,\alpha, i}^{m_\lambda}
                       \otimes_k \Lambda_\alpha\\
  &= \bigoplus_{i=0}^{b_\alpha-1}\bigoplus_{s(P)=\lambda} \Dc_{L_1}^r   \otimes_{\Dc_{L_1}[G_\alpha]} (\Dc_{L_1}[G_\alpha]
    f_{\alpha,i} s_{P_\alpha})^{m_\lambda},
\end{align*}
where $\Dc_{L_1}[G_\alpha] f_{\alpha,i} s_{P_\alpha}$ is a simple
$\Dc_{L_1}[G_\alpha]$-module, and therefore
$\Dc_K f_{\alpha,i} 1\otimes s_{P_\alpha}$ is a simple component of
$q_+(\Lambda_\alpha)$. This implies the assertion for the fraction
fields, so to complete the proof it suffices to see that
$N_{P_\alpha, \alpha, i}$ is a simple $\Dc_A$-module, and for this
purpose we employ the canonical inclusion $\pi_*(B)\subset \pi_+(B)$
(\ref{sec:5}). The $\Dc_A^\pi$-module $ \pi_*(B)$ is semisimple by
\Theorem{\ref{coh-decomp}} (see also
\citelist{\cite{levasseur-stafford:semisimplicity}*{Th. 3.4} \cite
  {kallstrom-bogvad:decomp}*{Prop. 2.2}}). Since $\pi_*(B)$ is finite
over $A$, its semisimple decomposition is determined by its
decomposition as $\Dc_K$-module at the generic point. Since
$f_{\alpha,i} s_{P_\alpha}\in \pi_*(B) $ and
$\Dc_Kf_{\alpha,i} s_{P_\alpha} $ is simple, it follows that
$\Dc_A^\pi f_{\alpha,i} s_{P_\alpha} $ is simple. Now as the image of
$\Dc_A^\pi f_{\alpha,i} s_{P_\alpha}$ in $\pi_+(B)$ generates
$N_{P_\alpha, \alpha, i}$, \Theorem{\ref{coh-decomp}} implies that
$N_{P_\alpha, \alpha, i}$ is simple.
\end{proof}
Since any vector in $N_{P_\alpha, \alpha, i}$ generates an irreducible
representation (see \Theorem{\ref{galois-direct}}) we also get: 

\begin{corollary}\label{imprimitive-repr} The representations
  \begin{displaymath}
    V_{P_\alpha, \alpha, i}= k[G(de,e, n)]f_{\alpha,i} s_{P_\alpha}
\end{displaymath}
give all the irreducible representations of $G(de,e, n)$, and
$V_{P_\alpha, \alpha, i}\cong V_{P'_\alpha, \alpha', i'}$ if and only
if $\alpha= \alpha'$, $i=i'$, and $P_\alpha$ and $P_{\alpha}'$
have the same shape $\lambda= (\lambda_0, \cdots , \lambda_{m-1})$,
$\lambda_i \vdash |\alpha^{-1}(i)|$.
\end{corollary}

\begin{example}\label{weyl-d4n}
  The groups $G(2,2,n)$ are Weyl groups of type $D_n$. The possible
  $\bar \alpha \in \hat A(2,2,n)$ are represented by functions
  $\alpha : [n]\to \{0,1\}$ and $\beta \sim \alpha$ if and only if
  $\beta$ arises from $\alpha$ be switching $0$s and $1$s, e.g.
  $(1,0,0,1)\sim (0,1,1,0)$ ($n=4$). For such a sequence $\alpha$, put
  $n= n_0+ n_1$, where $n_0$ ($n_1$) is the number of $0$s ($1$s) in
  $\alpha$. If $n_0 \neq n_1$, then $G_\alpha = Y_{\alpha}$, so that
  we have only $f_{\alpha, 0}=1$. If $n= 2n_1$ is an even number, one
  can have $n_0 = n_1$, so that $\alpha$ has an equal number of $0$s
  and $1$s. This implies that $C^\alpha = C_2$ and
  $k[t_\alpha]x^\alpha = k f_{\alpha, 0} + k f_{\alpha,1}$, where
  \begin{displaymath}
    f_{\alpha,0} = \prod_{i=1}^{n_1} x_i +  \prod_{i=n_1+1}^{n}
    x_i\quad \mbox{and}\quad      
    f_{\alpha,1} = \prod_{i=1}^{n_1} x_i -  \prod_{i=n_1+1}^{n} x_i.
  \end{displaymath}
\end{example}
\begin{remark} An explicit construction of the irreducible
  representations of $G$ is given in
  \citelist{\cite{ariki-hecke}\cite{ariki-koike-hecke} } (see also
  \cite{marin-michel}) by studying the restriction of representations
  of $G(d,1,n)$ to its subgroup $G(de,e, n)$, which in turn are
  presented in a ``classical'' and computational way using vector
  spaces with bases indexed by Young tableaux. In fact, in [loc. cit]
  one defines the action of generators of $G(d,1,n)$ on such vector
  spaces as a limiting case of an action that give representations of
  an associated Hecke algebra. In \Corollary{\ref{imprimitive-repr}}
  we work with the quite concrete polynomial ring $B$ and use a more
  direct approach instead of restrictions from $ G(d,1,n)$. Knowing the
  action of Coxeter generators of $S_n$ on a polynomial basis
  $\{s_P\}_{s(P)\vdash n}$ (the indices $P$ are standard Young
  tableaux of given shape) of an irreducible representation also gives
  the action of Coxeter generators of $G(de,e, n)$ in the basis
  $\{f_{\alpha,i} s_{P_\alpha}\}$ of $V_{P_\alpha, \alpha, i}$. Thus
  if $s_P$ are Young polynomials one gets a counterpart for
  $G(de,e, n)$ of Young's seminormal presentation of representations
  of $S_n$.

  Notice also that \Theorem{\ref{imprimitive}} describes all the
  simple $\Dc_A$-modules $N$, torsion free over $A$, such that
  $\pi^!(N)\cong L^r$ ($r= \rank N$).
\end{remark}

\section{Appendix: Minimal extensions}\label{appendix}
We present a standard setup involving a triple of functors $(H,F,G)$
with certain properties, yielding a ``minimal  extension functor'' $S$.
This entails, for example, the well-known functors $(j_!,j^!, j_+)$ on
constructible sheaves, associated with a locally closed embedding
$j: U \to X$ of quasi-projective varieties, where $S= j_{!+}$ was used
in \cite{perverse} to define intersection cohomology. See
\Remarks{\ref{remark-motive2}}{\ref{gen-remark}} for further
motivation why this abstract treatment is useful.
\subsection{Abstract minimal extensions}
Let $F: \Cc_1 \to \Cc_2$ be a functor of abelian categories provided
with functors $H, G: \Cc_2 \to \Cc_1$, such that $(H,F)$ and $(F, G)$
form adjoint pairs, and therefore have adjoint morphisms
\begin{displaymath}\tag{*}
  HF \to \id \to  GF.
\end{displaymath}
Make also  the following assumptions:
\begin{enumerate}
\item $F$ is exact and essentially surjective.
\item $H$ and $G$ are fully faithful.
\end{enumerate}
We consider the subcategory of $\Cc_1$ consisting of objects that do not contain
subquotients  that are orthogonal to $F$. The precise definition is:
\begin{definition}
  An object $M$ in $\Cc_1$ is ($F$-) strict if for any subobject
  $M_1\hookrightarrow M $ and exact sequence
  $0\to K \xrightarrow{f} M_1 \to N \xrightarrow{g} 0$, upon applying
  $F$ so that one has the sequence
  $ F(K)\xrightarrow{F(f)} F(M_1)\xrightarrow{F(g)} F(N)$, we have:
  \begin{enumerate}[label=(\roman*)]
  \item $F(f)=0 \Rightarrow f=0$,
\item $F(g)=0 \Rightarrow g=0$.
  \end{enumerate}
Let $\Cc_1^{str}$ be the subcategory of strict objects in $\Cc_1$.
\end{definition}
If $\Cc_1$ is closed under the formation of extensions and
subquotients, then so is $\Cc^{str}_1$.
\begin{remark} Assume that $\Cc_1$ and $\Cc_2$ are provided with
  duality functors $\Dbb_1, \Dbb_2$ such that
  $F \Dbb_1\cong \Dbb_2 F$. Then $\Cc_1^{str}$ is preserved by
  $\Dbb_1$. If (i) holds when $M_1=M$ and $M_1= \Dbb_2 (M) $, then
  (i-ii) holds for all subobjects $M_1$.
\end{remark}

\begin{lemma}
  There exists a unique morphism of functors $\phi :H \to G$ such that
  if $N= F(M)$ for some $M$, and $\phi_N : H(N)\to G(N)$, then
  $\phi_N =v\circ u$, the composed
  morphism of the unit and counit in $\thetag{*}$,
  $H(N)\xrightarrow{u} M \xrightarrow{v} G(N)$.
\end{lemma}

\begin{proof} 
  It suffices to see that the composed morphism $f_M=v \circ u$ is
  independent of the choice of $M$ such that $N= F(M)$. So let $M'$ be
  another object in $\Cc_1$ such that $F(M')=N$, we will prove that
  $\delta = f_M - f_{M'}$ equals $0$. Put
  $C_1 = \Imo( \delta)\subset G(N) $, so $v\circ u = g \circ h $,
  where $h: H(N)\to C_1$ is surjective and $g: C_ 1 \to G(N)$ is
  injective. Now $FG= FH = \id$ since $H$ and $G$ are fully faithful,
  implying that $F(\delta) =0$, hence $F(g)\circ F(h)=0$. Since $F$ is
  exact, it follows that $F(C)=0$ and hence
  $Hom(C, G(N))= Hom (F(C), N) =0$. Therefore, since $C\to G(N)$ is
  injective, $C=0$ and hence $\delta =0$.
\end{proof}
We get now  the functor
\begin{displaymath}
  S: \Cc_2 \to \Cc_1 , \quad N \mapsto S(N) = \Imo(H(N)\to G(N))
\end{displaymath}
so that $S$ is the image functor of the morphism of functors $ H\to G$,
\begin{displaymath}
  H \to S \to G. 
\end{displaymath}
More precisely, $S$ is a functor, and for any morphism $\phi : N_1 \to N_2$
in $\Cc_2$, the morphism $S(N_1)\to S(N_2)$ in $\Cc_1$ appears in
a commuting diagram
\begin{equation}{\tag{**}}
  \begin{tikzcd}
    H(N_1)\arrow[r, twoheadrightarrow, "h_1"] \arrow[d, "H(\phi)"] &
    S(N_1)\arrow[d, "S(\phi)"] \arrow[r, hook, "g_1"] &
    G(N_1)\arrow[d, "G(\phi)"]\\
    H(N_2) \arrow[r,twoheadrightarrow, "h_2"]
    & S(N_2) \arrow[r,hook,  "g_2"] &  G(N_2),\\
      \end{tikzcd}
\end{equation}
 where the maps $h_1$ ($g_i$) are surjective (injective). We have
 also the   commutative diagram
 \begin{equation}\label{comm-diag}
   \begin{tikzcd}
     &M\arrow{dr}\arrow[dd, twoheadrightarrow]&\\
     HF(M) \arrow{ru}{} \arrow[dr, twoheadrightarrow] & & GF(M). \\
     &SF(M)\arrow[ur,hook]&
  \end{tikzcd}
\end{equation}

 Let $\Cc_1^{ss}\subset \Cc_1$
 and $\Cc_2^{ss}\subset \Cc_2$ be the additive subcategories of
 semi-simple objects in $\Cc_1$ and $\Cc_2$, respectively, and
 $\Cc_1^{ss,str}= \Cc_1^{ss} \cap \Cc_1^{str} $ be the category of
 strict semisimple objects in $\Cc_1$.

\begin{theorem}\label{simple-equivalence}
  The functor $S$ is fully faithful and therefore defines an
  equivalence of categories
\begin{displaymath}
S: \Cc_2 \to \Cc_1^{str}.
\end{displaymath}
It also restricts to an equivalence of categories
  \begin{displaymath}
  S: \Cc_2^{ss} \to  \Cc_1^{ss, str}.
\end{displaymath}
If $N$ is simple, then $S(N)$ is also simple.
\end{theorem}
\begin{pf}
  (i) $F\circ S (N) \cong N$ when $N\in \Cc_1$: Since $H$ and $G$ are
  fully faithful, we have $F\circ H(N)\cong N \cong F\circ G (N)$ and
  combining with the diagram
  $F\circ H (N)\to F\circ S(N)\to F\circ G(N)$ one gets a natural map
  $F\circ S (N)\to N $, which is an isomorphism.

(ii) $S$ is fully faithful:    Consider the maps 
\begin{displaymath}
  Hom_{\Cc_1} (N_1, N_2)  \adj{S}{F}  Hom_{\Cc_2} (S(N_1),  S(N_2)).
\end{displaymath}
Applying $F$ to the diagram \thetag{**}, since
$F\circ H(\phi) = F\circ G(\phi) = \phi$, we get by (i) that
$F\circ S = \id : Hom_{\Cc_1} (N_1, N_2) \to Hom_{\Cc_1} (N_1, N_2) $;
hence $S$ is injective. To see that $S$ is surjective, letting
$\psi \in Hom_{\Cc_2} (S(N_1), S(N_2)) $, it suffices to see that
$\delta = S\circ F(\psi) - \psi: S(N_1)\to S(N_2)$ equals $0$ and,
since $g_2$ is injective, this follows if
$\delta_1 = g_2\circ \delta : S(N_1)\to G(N_2)$ is $0$. Since
$F\circ S = \id$,
\begin{displaymath}
  F(\delta) = F\circ S\circ  F(\psi) - F(\psi)  = F(\psi)- F(\psi)=0,
\end{displaymath}
hence also $F(\delta_1)=0$. Now since $F\circ G = \id$ and
$(F,G)$ is an adjoint pair,  we have
\begin{align*}
  Hom_{\Cc_2}(S(N_1), G(N_2)) &= Hom_{\Cc_2}(S(N_1), G\circ F \circ G(N_2))\\
  &= Hom_{\Cc_1}(FS(N_1), F\circ G(N_2)).
\end{align*}
Therefore $\delta_1 =0$.

(iii) $S\circ F(M) \cong M$ when $M\in \Cc_1^{str}$: Let
$\phi : K\to M $ be the kernel of $M\xrightarrow{\psi} G\circ F(M)$.
Applying $F$ to these morphisms we get $F(\phi)=0$, hence $\phi=0$
since $M\in \Cc_2^{str}$. Hence $\psi$ can factorized over an
injective morphism $\bar \psi :M \to S\circ F(M)$. Since
$H\circ F(M)\to S\circ F(M)$ is surjective it follows that 
$\bar \psi$ is also surjective.

This completes the proof that $S$ defines an equivalence
$\Cc_2 \to \Cc_1^{str}$, with quasi-inverse $F$.

(iv) To see the remaining assertion it suffices to see that if $N$ is
simple, then $S(N)$ is simple. Let $0 \to K \to S(N)\to C\to 0$ be a
short exact sequence. Since $F$ is exact and $F(S(N))=N$ is simple, it
follows that $F(K)=0$ or $F(C)=0$. Since $S(N) $ belongs to $
\Cc_1^{str}$ it follows that $K=0$ or $C=0$. Therefore either $K=
S(N)$ or $C= S(N)$, and hence $S(N)$ is simple
\end{pf}

It is suggestive to call $G$ and $H$ the maximal and comaximal
extensions functors of $F$, respectively, and $S$ the minimal
extension functor. On  certain objects all three functors agree.

\begin{proposition}\label{max-ext} Let $\Dbb_i$ be duality functors on
  $\Cc_i$, $i=1,2$ such that $F\Dbb_1 = \Dbb_2 F$. Then
  $\Dbb_1 G \Dbb_2 = H$, and $\Dbb_1 S = S \Dbb_2$. For an object $N$
  in $\Cc_2$ the following are equivalent:
\begin{enumerate}
\item $G(N) = H(N)$,
\item $G(N) = S(N)$ and $G\Dbb_2(N) = S\Dbb_2 (N)$,
\item $H(N) = S(N)$ and $H\Dbb_2(N) = S\Dbb_2 (N)$.
\end{enumerate}
\end{proposition}
\begin{proof} We have the sequence of functors $H\to S \to G$, and
  therefore $\Dbb_1 G \Dbb_2 \to \Dbb_1 S \Dbb_2 \to \Dbb_1 H \Dbb_2$.
  Since $\Dbb_1 G\Dbb_2 = H$ and $\Dbb_1 H\Dbb_2 = G $, and
  $\Dbb_1 S \Dbb_2$ also is an image functor of $H\to G$, this gives
  $\Dbb_1 S \Dbb_2 = S$. That (1) implies (2) and (3) is evident, so
  it suffices to prove that (2) implies (1). Since
  $\Dbb_2\circ S = S\circ \Dbb_1$,
  \begin{displaymath}
    H  \Dbb_2\circ  (N) = \Dbb_1 \circ G (N) = \Dbb_1\circ S (N)  =
    S\circ  \Dbb_2 (N) = G\circ \Dbb_2 (N).
  \end{displaymath}
  Since $\Dbb_1 H \Dbb_2 = G$, this implies (1).
\end{proof}

\subsection{Holonomic minimal extensions}
This section is a complement to the discussion of the minimal
extension in \cite{perverse}*{Sec. 1.4} and
\cite{hotta-takeuchi-tanisaki}*{Sec. 3.4}, where the main model of the
above abstract setup is studied. 

Let $U$ be a locally closed smooth subscheme of a smooth $k$-scheme
$X$ and $j: U \to X$ is inclusion morphism. Let $ \hol (\Dc_U)$ be the
category of holonomic $\Dc_U$-modules and $ \hol_{\bar U} (\Dc_X)$,
the category of holonomic $\Dc_X$-modules whose support belongs to the
closure $\bar U$ of $U$ in $X$. The Poincaré duals $\Dbb_U$ and
$\Dbb_X$ act on $ \hol (\Dc_U)$ and $\hol_{\bar U}(\Dc_X)$,
respectively.

We have the inverse and direct image functors (detailed below)
\begin{align*}
  j^!&: \hol_{\bar U} (\Dc_X) \to \hol (\Dc_U), \quad M\mapsto j^!(M) \\
  j_+&: \hol (\Dc_U)\to \hol_{\bar U} (\Dc_X), \quad N\mapsto \Dc_{U
       \leftarrow X} \otimes_{\Dc_U}N
\end{align*}
This is done by considering open and closed embeddings separately.When
$j$ is closed embedding of a smooth variety $U$, with defining ideal
$I_U$, put
\begin{align*}
  j^!(M)&= M^{I_U} = \{m \ \vert \ I_U \cdot m =0\}\footnote{Here $j^!$ is to be considered an extraordinary inverse
  image, see \Remark{\ref{invim-remark}}.}\\ \intertext{and}  j_+(N)&=
  \Dc_{X \leftarrow U}\otimes_{\Dc_U}N.
\end{align*}
When $j$ is an open embedding, then $j^!(M)$ is the ordinary sheaf
inverse image and $j_+(M)$ is the ordinary sheaf direct image. In
general, $j$ can be factorized
$U \xrightarrow{j_1} \Omega\xrightarrow{j_0} X$, $j= j_0\circ j_1$,
such that $j_0$ is the inclusion of an open subset of $X$ and $j_1$ is
a closed embedding. Then put
\begin{align*}
j^!(M)&= j_1^! (j_0^!(M))\\  \intertext{and}  j_+(N) &= (j_1)_+(j_0)_+(N).
\end{align*}

\begin{lemma}
  The modules $j^!(M)$ and $j_+(M)$ do not depend on the choice of
  open set $\Omega$.
\end{lemma}

\begin{lemma}\label{keylemma}
  The functor $j^! $ is exact, essentially surjective, and
  $\Dbb_U j^! \Dbb_X = j^!$. The functor $j_+$ is a right adjoint
  functor, and therefore $j_! = \Dbb_X j_+ \Dbb_U$ is a left adjoint
  functor to $j^!$.
\end{lemma}
\begin{proof}
  We have $j= j_0\circ j_1$, where $j_0$ is an open inclusion and
  $j_1$ is a closed embdding of a smooth variety. Firstly, the
  restriction functor $j_0^+$ to an open subset is exact, and
  essentially surjective, since $N =j_0^+ (j_0)_+(N)$ and
  $(j_0)_+(N)\in \hol_{\bar U}(\Dc_X)$ (by Bernstein's theorem about
  and the existence of $b$-functions). Secondly, the functor
  $j_1^!: \hol_{\bar U\cap \Omega}(\Dc_\Omega) \to \hol(\Dc_U)$ is an
  equivalence by Kashiwara's theorem. Finally, it is well-known that
  $j_1^!$ and $j_0^!$ commute with Poincaré duality. It is also
  well-known that both $(j_0^!,(j_0)_+)$ and $(j_1^!, (j_1)_+)$ form
  adjoint pairs, which implies that that $(j^!, j_+)$ forms an adjoint
  pair. The fact that $(j_!, j^!)$ also forms an adjoint pair is a
  formal consequencce from the fact that $\Dbb_X$ and $\Dbb_U$ are
  duality functors and that $\Dbb_U j^! \Dbb_X = j^!$.
\end{proof}

Putting $F= j^!$, $G= j_+$ and $H= j_!$ we are therefore in the
situation of \Theorem{\ref{simple-equivalence}}, and we put also
$j_{!+}= S$. Let $\hol^{str}_{\bar U} (\Dc_X)$ be the subcategory of
$\hol_{\bar U} (\Dc_X)$ consisting of modules that do not contain
non-zero subquotients that have support contained in the boundary
$\partial U = \bar U \setminus U$. Let $\hol^{ss} (\Dc_U)$ be the
subcategory of $\hol (\Dc_U)$ consisting of semisimple modules, and
$\hol^{str, ss}_{\bar U} (\Dc_X) = \hol^{str}_{\bar U}(\Dc_X)\cap
\hol^{ss}_{\bar U}(\Dc_X)$ the subcategory of strict semisimple
modules in $\hol_{\bar U}(\Dc_X)$. In this notation
\Theorem{\ref{simple-equivalence}} implies:

\begin{corollary} The functor $j_{!+}$ is fully faithful and defines
  an equivalence of categories
  \begin{displaymath}
    j_{!+} : \hol(\Dc_U) \to \hol^{str}_{\bar U}(\Dc_X).
  \end{displaymath}
It also restricts to an equivalence
\begin{displaymath}
  j_{!+} : \hol^{ss}(\Dc_U) \to \hol^{ss,str}_{\bar U}(\Dc_X).
\end{displaymath}
If $N_U$ is simple, then $j_{!+}(N_U)$ is simple.
  
\end{corollary}

\begin{remark}\label{gen-remark}
  Let $\hol^{reg}(\Dc_U)\subset \hol(\Dc_U) $ and
  $\hol^{reg}_{\bar U}(\Dc_X)\subset \hol_{\bar U}(\Dc_X)$ be the
  subcategories of regular holonomic modules, which are preserved by
  the duality functors $\Dbb_U$ and $\Dbb_X$, respectively. Then
  \Lemma{\ref{keylemma}} still applies to these subcategories, so that
  in \Theorem{\ref{simple-equivalence}} one can add the superindex
  $reg$ and get another theorem, e.g.
  $j_{!+}:\hol^{reg,ss}(\Dc_U)\cong \hol^{reg,ss, str}_{\bar
    U}(\Dc_X)$. Similarly, one can replace $\hol(\Dc_U)$ and
  $\hol_{\bar U}(\Dc_X)$ with categories of perverse sheaves on $U$
  and $\bar U$, respectively, resulting in an analogous result.
\end{remark}

The actual evaluation of $ j_{!+}(N)$ can be based on the observation
that it is a prolongation $M$ of $N$ from $U$ to $X$, $j^!(M) = N$,
such that if $M_1 $ is any other such prolongation, then
$\supp (M_1+ M)/M \subset \partial U$. Since $N$ and $j_+(N)$ are of
finite length, one gets local generators of $j_{!+}(N)$ from
generators of $N$ in the following way. Working locally near a point
in $\partial U$, if $f$ is a function in the ideal of $\partial U$,
$N_0$ any coherent $\Oc_X$-submodule of $j_+(N)$ whose restriction to
$U$ generates $N$, we have
\begin{displaymath}
  j_{!+}(N) = \Dc_X f^k N_0, \quad k \gg 1.
\end{displaymath}
If $N$ is simple a lower bound for $k$ kan be determined from the
degree of the $b$-function of $\mu f^s$ where $\mu$ is any non-zero
section of $ j_+(N)$.

\Theorem{~\ref{simple-equivalence}} also contains the following
result:

A coherent $\Dc_X$-module $M$ is {\it pure} if the dimension of the
support of all its non-zero coherent submodules $M_1$ are equal, i.e.
$\dim \supp M = \dim \supp M_1$.

\begin{corollary}\label{minimalext}
  Let $X/k$ be a smooth variety and $M$ be a holonomic $\Dc_X$-module. Let $\xi$ be a generic point
  of the support $\supp M$.  There exists a unique minimal holonomic $\Dc_X$-submodule $ M(\xi)
  \subset M$ such that $M_\xi = M(\xi)_\xi$.  The module $ M(\xi)$ is pure, and $ M(\xi)$ is
  semi-simple if and only if $M_\xi$ is semi-simple.
\end{corollary}
We call $ M(\xi) = j_{!+}(M_U)$ the{ \it minimal extension} of
$M_{\xi}$, where $U$ is any locally closed subset of $X$ that contains
$\xi$ but no other generic points in $\supp M$.
\begin{proof} 
  There exists an open neighbourhood $U$ of $\xi$ such that $M_U$ is
  pure; let $j: U \to X$ be the open inclusion. Then $M_U$ is
   (semi)simple if and only if $j_{!+}(M_U)$ is (semi)simple. Since
  $M_U$ is pure it follows that $j_+(M_U)$ is pure, and hence also its
  submodule $j_{!+}(M_U)$ pure. Moreover, if $ \bar M\subset M$ is any
  coherent $\Dc_X$-submodule such that $\bar M_\xi = M_\xi$, then
  $j_{!+}(\bar M_U) \subset \bar M$.
\end{proof}

\begin{remark}\label{remark-motive2}
  For the model $(H,F, G)= (j_!, j^!, j_+)$
  \Theorem{\ref{simple-equivalence}} is well-known when $U$ is
  quasi-projective. A rather involved proof is presented in
  \cite{hotta-takeuchi-tanisaki}*{Th. 3.2.16}, using a factorisation
  $ j = p \circ i \circ g$, where $g$ is a closed immersion, $i$ is
  open, and $p$ is a projective morphism, and the functor $S= j_{!+}$
  is determined by $i_{!+}$ (resolutions of singularities is required
  for the argument). A point of \Theorem{\ref{simple-equivalence}} is
  to show that the functor $S$ is built into the categorical context
  and that the above factorisation is not required. Moreover, as
  indicated in \Remark{\ref{gen-remark}}, one also gets analogous
  results in other contexts.
\end{remark}

There is also a point-wise setup, again using
\Theorem{\ref{simple-equivalence}}. Let $\xi$ be a point in the
scheme $X$ and $k_{X, \xi}$ its residue field. Let $\Dc_{k_{X, \xi}}$
be the ring of $k$-linear differential operators on $k_{X, \xi}$,
$\hol ( \Dc_{k_{X, \xi}})$ the category of holonomic
$\Dc_{k_{X, \xi}}$-modules. Let $\hol_\xi^{strict}(\Dc_X)$ be the
category of holonomic $\Dc_X$-modules $M$ such that all its coherent
subquotients have support equal to the closure $\xi^-$ of $\xi$. Then
$\hol_\xi^{strict}(\Dc_X)$ is a full subcategory of $\hol(\Dc_X)$. The
following theorem, which essentially is a reformulation of
\cite{hotta-takeuchi-tanisaki}*{Th. 3.4.2}, asserts that simple
holonomic $\Dc_X$-modules are determined by their generic structure.

\begin{corollary}
Let $\xi$ be a point in $X$. There exists an  equivalence of categories:
  \begin{displaymath}
j^{\xi}_{!+}:    \hol(\Dc_{k_{X, \xi}}) \cong  \hol^{strict}_{ \xi}(\Dc_X).
  \end{displaymath}
  The functor $j^\xi_{!+} $ induces a bijection between the
  isomorphism classes of simple holonomic $\Dc_{k_{X, \xi}}$-modules
  and simple holonomic $\Dc_X$-module whose support contains the
  generic point $\xi$.
\end{corollary}

Let us now consider connections on a smooth variety $X$, i.e.
$\Dc_X$-modules that are coherent over $\Oc_X$. When $j: X_0 \to X$ is
an open immersion we know that $j^+(M)$ is semisimple for any
semisimple $M$. Knowing that $M$ is coherent over $\Oc_X$ outside an
exceptional set we get a converse.
\begin{proposition}\label{simple-coh}
  \begin{enumerate}
  \item Let $M_1$ and $M_2$ be connections on $X$ such that
    $j^+(M_1)\cong j^+(M_2)$. Then $M_1 \cong M_2$.
  \item Put $S=X\setminus j(X_0) $ and assume that $\codim_X S\geq 2$.
    Let $M$ be a coherent $\Dc_X$-module which is torsion free along
    $S$ such that $j^!(M)$ is a (semi)simple module which is moreover
    coherent over $\Oc_{X_0}$. Then $M = j_+j^!(M) $, and $M$ is
    coherent as $\Oc_X$-module and (semi)simple over $\Dc_X$.
  \end{enumerate}

\end{proposition}

\begin{proof}
  (1):   The assumption implies
  \begin{displaymath}
M_1 \cong     j_{!+}j^+(M_1) \cong j_{!+}j^+(M_2) \cong M_2,
  \end{displaymath}
where one sees that the first and last maps are isomorphisms as follows.
  The canonical surjective map $M_i\to j_{!+}j^+(M_i)$ (see
  (\ref{comm-diag})) is generically an isomorphisms and the kernel is
  coherent over $\Oc_X$; hence being a $\Dc_X$-module, it is also
  injective.

  (2): It follows from Grothendieck's finiteness theorem
  that $j_+j^!(M)$ is coherent over $\Oc_X$, hence we have the
  sequence of $\Dc_X$-modules
  $0 \to M \to j_+j^!(M)\to H^1_T(M)\to 0$, where $T= X\setminus X_0$,
  with two terms being coherent over $\Oc_X$, hence ($X$ being a
  noetherian space) $ H^1_T(M)$ is coherent over $\Oc_X$, and since
  $ \supp H^1_T(M)\subset T$, we get $H^1_T(M)=0 $; thus
  $M= j_+j^!(M)$. Let $X_{00}\subset X_0$ be an open subset such that
  the inclusion map $i: X_{00} \to X $ is affine. Since $M $ is a
  submodule of $i_+i^!(M)$ and $i_{!+}i^!(M)$ a submodule of $M$ such
  that $i^!(i_{!+}i^!(M)) = i^!(M)$, it follows that $M/i_{!+}i^!(M)$
  is a $\Dc_X$-module which is torsion and coherent over $\Oc_X$ and
  therefore $M/i_{!+}i^!(M)=0$; thus $M= i_{!+}i^!(M)$. By
  \Theorem{\ref{minimalext}} it follows that $M$ is (semi)simple if
  and only if $i^!(M)$ is (semi)simple.
\end{proof}

\end{document}